\documentclass[11pt]{scrartcl}

\usepackage[utf8]{inputenc}
\usepackage[T1]{fontenc}
\usepackage{lmodern}
\usepackage{alphabeta}

\usepackage{titling}
\usepackage{xspace}
\usepackage{soul} 

\usepackage[letterpaper, top=25.4mm, bottom=25.4mm, left=25.4mm, right=25.4mm, includefoot]{geometry}

\DeclareSectionCommand[%
level=4,
indent=0pt,
beforeskip=1ex plus 1ex minus .2ex,
afterskip=-1em,
font={},
tocindent=7em,
tocnumwidth=4.1em,
counterwithin=subsubsection
]{paragraph}

\usepackage[inline]{enumitem}

\usepackage{dsfont}
\usepackage{setspace}

\usepackage{bm}
\usepackage{bbold}
\usepackage{microtype}
\usepackage{amsmath}
\usepackage{amssymb}
\usepackage{amsfonts}
\usepackage{mathtools}
\usepackage[mathscr]{euscript}

\usepackage[hyphens]{url}
\usepackage{amsthm}

\usepackage{tikz}
\usepackage{xcolor}
\usepackage{subcaption}
\usepackage{graphicx}
\usepackage{wrapfig}

\usepackage{hyperref}
\usepackage{zref-clever}

\usepackage{nicefrac}
\usepackage[textwidth=1.5in]{todonotes}

\newcommand*\samethanks[1][\value{footnote}]{\footnotemark[#1]}

\colorlet{myGreen}{green!50!black}
\colorlet{myLightgreen}{green}
\colorlet{myRed}{red!90!black}
\definecolor{myBlue}{rgb}{0.25, 0.0, 1.0}
\definecolor{myLightBlue}{rgb}{0.39, 0.58, 0.93}
\colorlet{myViolet}{myBlue!55!myRed}
\definecolor{myOrange}{rgb}{1.0, 0.66, 0.07}

\definecolor{CornflowerBlue}{rgb}{0.39, 0.58, 0.93}
\definecolor{DarkGoldenrod}{rgb}{0.72, 0.53, 0.04}
\definecolor{BritishRacingGreen}{rgb}{0.0, 0.26, 0.15}
\definecolor{DarkMagenta}{rgb}{0.55, 0.0, 0.55}
\definecolor{AO}{rgb}{0.0, 0.5, 0.0}
\definecolor{BostonUniversityRed}{rgb}{0.8, 0.0, 0.0}
\definecolor{myRed}{rgb}{0.8, 0.0, 0.0}
\definecolor{DarkMidnightBlue}{rgb}{0.0, 0.2, 0.4}
\definecolor{DarkTangerine}{rgb}{1.0, 0.66, 0.07}
\definecolor{AppleGreen}{rgb}{0.55, 0.71, 0.0}
\definecolor{BrightUbe}{rgb}{0.82, 0.62, 0.91}
\definecolor{Amethyst}{rgb}{0.6, 0.4, 0.8}
\definecolor{DarkGray}{rgb}{0.52, 0.52, 0.51}
\definecolor{Gray}{rgb}{0.66, 0.66, 0.66}
\definecolor{BananaYellow}{rgb}{1.0, 0.88, 0.21}
\definecolor{Amber}{rgb}{1.0, 0.75, 0.0}
\definecolor{LightGray}{rgb}{0.83, 0.83, 0.83}
\definecolor{PrincetonOrange}{rgb}{1.0, 0.56, 0.0}
\definecolor{DeepCarrotOrange}{rgb}{0.91, 0.41, 0.17}
\definecolor{CarrotOrange}{rgb}{0.93, 0.57, 0.13}
\definecolor{MidnightBlue}{rgb}{0.1, 0.1, 0.44}
\definecolor{Magenta}{rgb}{0.50, 0.0, 0.50}
\definecolor{BrightPink}{rgb}{1.0, 0.0, 0.5}
\definecolor{BrilliantRose}{rgb}{1.0, 0.33, 0.64}
\definecolor{ChromeYellow}{rgb}{1.0, 0.65, 0.0}
\definecolor{HotMagenta}{rgb}{1.0, 0.11, 0.81}
\definecolor{Amethyst}{rgb}{0.6, 0.4, 0.8}
\definecolor{BrightGreen}{rgb}{0.4, 1.0, 0.0}
\definecolor{BrilliantLavender}{rgb}{0.96, 0.73, 1.0}

\vfuzz \hfuzz
\setlist[itemize]{topsep=0pt,partopsep=0pt,itemsep=0pt,parsep=0pt}
\setlist[itemize,1]{label={\small\textbullet}}
\setlist[itemize,2]{label={\tiny\textbullet}}
\setlist[itemize,3]{label=$\cdot$}
\setlist[enumerate]{topsep=0pt,partopsep=0pt,itemsep=0pt,parsep=0pt}
\setlist[enumerate,1]{label=\roman*)}
\setlist[enumerate,2]{label=\alph*)}
\setlist[enumerate,3]{label=\arabic*)}

\hypersetup{
colorlinks=true,
linkcolor=AO!65!black,
citecolor=AO!65!black,
urlcolor=AppleGreen!65!black,
bookmarksopen=true,
bookmarksnumbered,
bookmarksopenlevel=2,
bookmarksdepth=3
}

\theoremstyle{definition}

\newtheorem{environment}{Environment}[section]

\newtheorem{lemma}[environment]{Lemma}
\AddToHook{env/lemma/begin}{\zcsetup{countertype={environment=lemma}}}
\zcRefTypeSetup{lemma}{
Name-sg = Lemma ,
name-sg = Lemma ,
Name-pl = Lemmas ,
name-pl = Lemmas ,
}

\newtheorem*{lemma*}{Lemma}
\AddToHook{env/lemma*/begin}{\zcsetup{countertype={environment=lemma*}}}
\zcRefTypeSetup{lemma*}{
Name-sg = Lemma ,
name-sg = Lemma ,
Name-pl = Lemmas ,
name-pl = Lemmas ,
}

\newtheorem{proposition}[environment]{Proposition}
\AddToHook{env/proposition/begin}{\zcsetup{countertype={environment=proposition}}}
\zcRefTypeSetup{proposition}{
Name-sg = Proposition ,
name-sg = Proposition ,
Name-pl = Propositions ,
name-pl = Propositions ,
}

\newtheorem{corollary}[environment]{Corollary}
\AddToHook{env/corollary/begin}{\zcsetup{countertype={environment=corollary}}}
\zcRefTypeSetup{corollary}{
Name-sg = Corollary ,
name-sg = Corollary ,
Name-pl = Corollaries ,
name-pl = Corollaries ,
}

\newtheorem{theorem}[environment]{Theorem}
\AddToHook{env/theorem/begin}{\zcsetup{countertype={environment=theorem}}}
\zcRefTypeSetup{theorem}{
Name-sg = Theorem ,
name-sg = Theorem ,
Name-pl = Theorems ,
name-pl = Theorems ,
}

\newtheorem*{theorem*}{Theorem}
\AddToHook{env/theorem*/begin}{\zcsetup{countertype={environment=theorem*}}}
\zcRefTypeSetup{theorem*}{
Name-sg = Theorem ,
name-sg = Theorem ,
Name-pl = Theorems ,
name-pl = Theorems ,
}

\newtheorem{conjecture}[environment]{Conjecture}
\AddToHook{env/conjecture/begin}{\zcsetup{countertype={environment=conjecture}}}
\zcRefTypeSetup{conjecture}{
Name-sg = Conjecture ,
name-sg = Conjecture ,
Name-pl = Conjectures ,
name-pl = Conjectures ,
}

\newtheorem*{hypothesis*}{Hypothesis}
\AddToHook{env/hypothesis*/begin}{\zcsetup{countertype={environment=hypothesis*}}}
\zcRefTypeSetup{hypothesis*}{
Name-sg = Hypothesis ,
name-sg = Hypothesis ,
Name-pl = Hypotheses ,
name-pl = Hypotheses ,
}

\newtheorem{observation}[environment]{Observation}
\AddToHook{env/observation/begin}{\zcsetup{countertype={environment=observation}}}
\zcRefTypeSetup{observation}{
Name-sg = Observation ,
name-sg = Observation ,
Name-pl = Observations ,
name-pl = Observations ,
}

\AddToHook{env/example/begin}{\zcsetup{countertype={environment=example}}}
\zcRefTypeSetup{example}{
Name-sg = Example ,
name-sg = Example ,
Name-pl = Examples ,
name-pl = Examples ,
}

\AddToHook{env/remark/begin}{\zcsetup{countertype={environment=remark}}}
\zcRefTypeSetup{remark}{
Name-sg = Remark ,
name-sg = Remark ,
Name-pl = Remarks ,
name-pl = Remarks ,
}

\zcRefTypeSetup{equation}{
Name-sg = Equation ,
name-sg = Equation ,
Name-pl = Equations ,
name-pl = Equations ,
}

\zcRefTypeSetup{chapter}{
Name-sg = Chapter ,
name-sg = Chapter ,
Name-pl = Chapters ,
name-pl = Chapters ,
}

\zcRefTypeSetup{section}{
Name-sg = Section ,
name-sg = Section ,
Name-pl = Sections ,
name-pl = Sections ,
}

\zcRefTypeSetup{algorithm}{
Name-sg = Algorithm ,
name-sg = Algorithm ,
Name-pl = Algorithms ,
name-pl = Algorithms ,
}

\AddToHook{env/notation/begin}{\zcsetup{countertype={environment=notation}}}
\zcRefTypeSetup{notation}{
Name-sg = Notation ,
name-sg = Notation ,
Name-pl = Notations ,
name-pl = Notations ,
}

\newtheorem{question}[environment]{Question}
\AddToHook{env/question/begin}{\zcsetup{countertype={environment=question}}}
\zcRefTypeSetup{question}{
Name-sg = Question ,
name-sg = Question ,
Name-pl = Questions ,
name-pl = Questions ,
}

\AddToHook{env/problem/begin}{\zcsetup{countertype={environment=problem}}}
\zcRefTypeSetup{problem}{
Name-sg = Problem ,
name-sg = Problem ,
Name-pl = Problems ,
name-pl = Problems ,
}

\AddToHook{env/claim/begin}{\zcsetup{countertype={environment=claim}}}
\zcRefTypeSetup{claim}{
Name-sg = Claim ,
name-sg = Claim ,
Name-pl = Claims ,
name-pl = Claims ,
}

\newtheorem{definition}[environment]{Definition}
\AddToHook{env/definition/begin}{\zcsetup{countertype={environment=definition}}}
\zcRefTypeSetup{definition}{
Name-sg = Definition ,
name-sg = Definition ,
Name-pl = Definitions ,
name-pl = Definitions ,
}

\zcRefTypeSetup{figure}{
Name-sg = Figure ,
name-sg = Figure ,
Name-pl = Figures ,
name-pl = Figures ,
}

\usetikzlibrary{calc}
\usetikzlibrary{fit}
\usetikzlibrary{decorations}
\usetikzlibrary{decorations.pathmorphing}
\usetikzlibrary{decorations.text}
\usetikzlibrary{shapes,hobby}

\tikzset{
	position/.style args={#1:#2 from #3}{
		at=($(#3)+(#1:#2)$)
	}
}

\tikzset{
  v:main/.style = {draw, circle, scale=0.8, thick,fill=black,inner sep=0.7mm},
  v:ghost/.style = {inner sep=0pt,scale=1},
  >={latex},
  e:marker/.style = {line width=8.5pt,line cap=round,opacity=0.35,color=DarkGoldenrod},
  e:main/.style = {line width=1pt},
}

\newcommand\restrict[1]{\raisebox{-.5ex}{$|$}_{#1}}

\newcommand{\ocp}{\textsf{OCP}}
\newcommand{\ocptw}{\mathsf{OCP}\mbox{-}\mathsf{tw}}
\newcommand{\ocpw}{\mathsf{OCP}\mbox{-}\mathsf{width}}
\newcommand{\tw}{\mathsf{tw}}
\newcommand{\ocptd}{\textsf{OCP}-tree-decomposition\xspace}

\newcommand{\oddminor}{\preccurlyeq_{\mathsf{odd}}}

\newcommand{\mis}{\textsc{Maximum Weighted Independent Set}\xspace}

\title{Odd-Cycle-Packing-treewidth: On the Maximum Independent Set problem in odd-minor-free graph classes}

\predate{}
\date{}
\postdate{}

\preauthor{}
\DeclareRobustCommand{\authorthing}{
	\begin{center}
        Mujin Choi\footnote{\href{mailto:mujinchoi@kaist.ac.kr}{mujinchoi@kaist.ac.kr}}~~\!\thanks{Supported by the Institute for Basic Science (IBS-R029-C1).} \\
		{\small Department of Mathematical Sciences, KAIST, Daejeon, South Korea\\ Discrete Mathematics Group, Institute for Basic Science (IBS), Daejeon, South Korea} \\
        \medskip
		Maximilian Gorsky\footnote{\href{mailto:m.gorsky@pm.me}{m.gorsky@pm.me}}~~\!\samethanks[2] \\
		{\small Discrete Mathematics Group, Institute for Basic Science (IBS), Daejeon, South Korea} \\
        \medskip
		Gunwoo Kim\footnote{\href{mailto:gunwoo.kim@kaist.ac.kr}{gunwoo.kim@kaist.ac.kr}}~~\!\samethanks[2] \\
		{\small School of Computing, KAIST, Daejeon, South Korea\\ Discrete Mathematics Group, Institute for Basic Science (IBS), Daejeon, South Korea} \\
        \medskip
		Caleb McFarland\footnote{\href{mailto:cmcfarland30@gatech.edu}{cmcfarland30@gatech.edu}} \\
		{\small School of Mathematics, Georgia Institute of Technology, Atlanta, USA} \\
	    \medskip
		Sebastian Wiederrecht\footnote{\href{mailto:wiederrecht@kaist.ac.kr}{wiederrecht@kaist.ac.kr}} \\
		{\small School of Computing, KAIST, Daejeon, South Korea} \\
\end{center}}
\author{\authorthing}
\postauthor{}

\begin{document}
\maketitle

\begin{abstract}
We introduce the tree-decomposition-based graph parameter \emph{Odd-Cycle-Packing-treewidth} (\textsf{OCP}-tw) as a width parameter that asks to decompose a given graph into pieces of bounded odd cycle packing number.
The parameter \textsf{OCP}-tw is monotone under the odd-minor-relation and we provide an analogue to the celebrated Grid Theorem of Robertson and Seymour for \textsf{OCP}-tw.
That is, we identify two infinite families of grid-like graphs whose presence as odd-minors implies large \textsf{OCP}-tw and prove that their absence implies bounded \textsf{OCP}-tw.
This structural result is constructive and implies a $2^{\mathbf{poly}(k)}\mathbf{poly}(n)$-time parameterized $\mathbf{poly}(k)$-approximation algorithm for \textsf{OCP}-tw.

Moreover, we show that the (weighted) \textsc{Maximum Independent Set} problem (MIS) can be solved in polynomial time on graphs of bounded \textsf{OCP}-tw.
Finally, we lift the concept of \textsf{OCP}-tw to a parameter for matrices of integer programs.
To this end, we show that our strategy can be applied to efficiently solve integer programs whose matrices can be ``tree-decomposed'' into totally $\Delta$-modular matrices with at most two non-zero entries per row.
\end{abstract}
\let\sc\itshape
\thispagestyle{empty}

\newpage
\tableofcontents
\thispagestyle{empty}
\newpage

\setcounter{page}{1}

\section{Introduction}\label{sec:intro}

The \textsc{Maximum Independent Set} \emph{problem} (MIS) takes as input a graph $G$ and asks for an \textsl{independent set} -- a set of vertices that do not share any edges -- of maximum size.
MIS is known to be \textsf{NP}-hard in general and remains \textsf{NP}-hard on triangle-free graphs \cite{Poljak1974ANote}, graphs without induced $K_{1,4}$-subgraphs \cite{Minty1980Maximal}, and planar graphs of maximum degree at most three \cite{GareyJ1977Rectilinear}.

The field of parameterized algorithms aims to identify hierarchical graph classes $\mathcal{C}_1\subseteq \mathcal{C}_2\subseteq \dots $ such that a given \textsf{NP}-hard problem can be solved in polynomial time in $\mathcal{C}_k$ for every fixed $k\in\mathbb{N}$.
A simple example for a typical result in this area is the observation that, if we let $\mathcal{C}_k$ be the class of all graphs $G$ with independent sets of order at most $k$, then for every $k$ we can solve MIS on the graphs in $\mathcal{C}_k$ in time $n^{\mathcal{O}(k)}$ which is, under established complexity-theoretic assumptions, best possible (see \cite{CyganFKLMPPS2015Parameterized} for an introduction to parameterized complexity).
As the degree of this polynomial grows in the solution size, a need for different types of \textsl{parametrizations} for MIS becomes apparent.

\paragraph{Our contribution.}
In this paper we introduce the notion of \textsl{Odd-Cycle-Packing-treewidth} (\textsf{OCP}-tw) as a unifying way to understand both tree-like features and the behaviour of odd cycles in a graph.
We leverage this definition to show that MIS can be solved in time $n^{f(k)}$ on graphs of \textsf{OCP}-tw at most $k$ (see \zcref{subsec:OCPtwIntro,subsec:ResultsIntro} for more detailed expositions).
This result can be understood as a major step forward in the understanding of the tractability horizon for MIS in \textsl{odd-minor-free} graph classes\footnote{Such graph classes can be seen as generalisations of bipartite graphs. See \zcref{subsec:OCPtwIntro} for a definition of odd-minors.} (see \zcref{subsec:RelatedIntro} for a deeper discussion).
Along the way we develop a list of technical tools we believe to be invaluable in the further advancement towards this goal.

\paragraph{Structural parametrizations for MIS.}
A well-established strategy for the search for good parametrizations is the exclusion of certain substructures.
An example of such a strategy is used for minor-closed graph classes:
Via a result of Garey and Johnson \cite{GareyJ1977Rectilinear}, we know that MIS is \textsf{NP}-hard on every minor-closed graph class containing all planar graphs.
However, the celebrated Grid Theorem of Robertson and Seymour \cite{RobertsonS1986Grapha} tells us that every such graph class excluding a planar graph $H$ has \textsl{bounded treewidth},\footnote{We postpone the definition of treewidth and tree-decompositions to \zcref{subsec:OCPtwIntro}.} i.e.\ treewidth at most $k = k(|V(H)|)$.
As a consequence, there exists an algorithm for MIS on $H$-minor-free graphs running in $f(k)\cdot n$-time \cite{CyganFKLMPPS2015Parameterized} if $H$ is planar.

The underlying tree-like structure of graphs with bounded treewidth naturally lends itself to dynamic-programming and allows for the design of many parameterized algorithms \cite{RobertsonS1986Graphb,ArnborgP1989Linear,BodlaenderK2007Combinatorial,BodlaenderCKN2015Deterministic}.
A downside of minor-closed graph classes -- and therefore of treewidth as a parameter -- is that $H$-minor-free graphs are necessarily \textsl{sparse}.
Specifically, they always exclude many graphs on which MIS is naturally polynomial-time solvable such as \textsl{dense} bipartite graphs.
In recent years, several research projects were aimed at overcoming this hurdle.
On one side, the work of Fiorini, Joret, Weltge, and Yuditsky \cite{FioriniJWY2021Integer,Fiorini2025Integer} focussed on structural generalisations of bipartite graphs by proving that MIS can be solved in polynomial time on graphs with at most $k$ pairwise disjoint odd cycles, i.e.\ graphs of bounded \emph{odd cycle packing number}.
On the other hand, based on an idea of Eiben, Ganian, Hamm, and Kwon \cite{EibenGHK2021Measuring}, so-called \textsl{hybridisation techniques} have been used to combine the structural power of tree-decompositions -- and related concepts -- with the structural properties of bipartite graphs to create new width parameters \cite{JansendK2021FPTAlgorithms,JansendKW2021Vertex,GollinW2023OddMinors,JaffkeMIT2023Dyanmix}.
The goal here is to create parameters $\mathsf{p}$ such that MIS -- or any other interesting problem -- can be solved in polynomial time on all graphs $G$ with $\mathsf{p}(G)\leq k$.

There are two slightly different genres of such hybrid parameters that can be found in the literature:
One way -- pioneered by Bulian and Dawar \cite{BulianD2016Graph,BulianD2017Fixed-parameter} -- fixes a minor-monotone\footnote{A graph parameter $\mathsf{q}$ is \emph{minor-monotone} if for every graph $G$ and every minor $H$ of $G$ it holds that $\mathsf{q}(H)\leq \mathsf{q}(G)$.} parameter $\mathsf{q}$ and a graph class $\mathcal{B}$ -- for example bipartite graphs -- and then fixes $\mathsf{p}(G)$ to be the smallest $k$ for which there exists $X \subseteq V(G)$ whose \textsl{torso}\footnote{The torso $G_X$ is built from $G[X]$ by turning $N_G(V(H))\subseteq X$ into a clique for every component $H$ of $G-X$.} $G_X$ satisfies $\mathsf{q}(G_X)\leq k$, and $G-X\in\mathcal{B}$.
The case where $\mathsf{q}$ is treewidth is known as $\mathcal{B}$-treewidth.

One may also change the way a tree-decomposition is evaluated.
This approach is central to our work.
In simplified terms, the treewidth of a graph is the smallest $k$ such that a graph $G$ can be decomposed into a tree-like arrangement of vertex sets $X\subseteq V(G)$ of size at most $k+1$.
Instead of focusing on the size of the sets, we could ask for any two such sets $X$ and $Y$ to satisfy $|X \cap Y| \leq k$ and for each torso $G_X$ to satisfy a fixed property.
A classic example of such a parameter is based on a theorem of Robertson and Seymour \cite{RobertsonS1993Excluding} and asks for every $X$ to either be of size at most $k+1$ or $G_X$ to be planar (see \cite{Kaminski2021MAXCUT} for an application and \cite{ThilikosW2024Killing} for a generalisation).
Regarding MIS, the notions of \textsl{bipartite treewidth} \cite{JaffkeMIT2023Dyanmix} asking for every $G_X$ to become bipartite by deleting at most $k$ vertices, and \textsl{$\mathcal{B}$-blind-treewidth} \cite{GollinW2023OddMinors}, where $\mathcal{B}$ is the class of bipartite graphs, which only measures the treewidth of the non-bipartite blocks of $G$, fall into this category.

We further discuss some of the ideas above in relation to our results in \zcref{subsec:RelatedIntro}.

\subsection{Odd-Cycle-Packing-treewidth}\label{subsec:OCPtwIntro}
Our results are centred around a novel extension of the ideas behind bipartite treewidth and odd cycle packing number.
We combine both of these notions into a single parameter by choosing the parameter $\mathsf{q}$ in the construction above to be the \textsl{odd cycle packing number}.
If we denote the class of bipartite graphs by $\mathcal{B}$, this results in a common generalisation of $\mathcal{B}$-treewidth, bipartite treewidth, and odd cycle packing number.
To provide a full definition, we first define tree-decompositions.

A \emph{tree-decomposition} of a graph $G$ is a pair $(T,\beta)$ where $T$ is a tree and $\beta$ assigns to every $t\in V(T)$ a set $\beta(t)\subseteq V(G)$ called the \emph{bag} of $t$, such that $\bigcup_{t\in V(T)}\beta(t)=V(G)$, for each $xy\in E(G)$ there is $t\in V(T)$ with $x,y\in\beta(t)$, and for every $v\in V(G)$, the set $\{ t\in V(T) \colon v\in\beta(t) \}$ is connected.
The \emph{adhesion} of $(T,\beta)$ is the value $\mathsf{adhesion}(T,\beta) \coloneq \max_{dt\in E(T)}|\beta(d)\cap\beta(t)|$ and the \emph{width} of $(T,\beta)$ is the value $\max_{t\in V(T)}|\beta(t)|-1$.
In this language, the \emph{treewidth} of a graph $G$, denoted by $\mathsf{tw}(G)$, is the minimum width over all tree-decompositions for $G$.

We now give a formal definition of Odd-Cycle-Packing-treewidth, which we mostly shorten to \ocp-treewidth.
This variant of treewidth measures how far each bag is from being bipartite.

In the following, for any graph $G$ we denote by $\ocp(G)$ the largest integer $k$ such that $G$ contains $k$ pairwise vertex-disjoint cycles of odd length.

\begin{definition}[\ocp-treewidth]\label{def:ocptreewidth}
    An \emph{\ocptd} of a graph $G$ is a triple $\mathcal{T}=(T,\beta, \alpha)$, where $T$ is tree and $\beta, \alpha \colon V(T) \rightarrow 2^{V(G)}$, satisfying the following conditions:
    \begin{enumerate} [leftmargin=1.8cm,labelindent=16pt,label= (\ocp\,\arabic*)]
        \item \label{ocp1} $(T,\beta)$ is a tree-decomposition of $G$,
        \item \label{ocp2} $\alpha(t) \subseteq \beta(t)$ for each $t \in V(T)$, and
        \item \label{ocp3} for each $t\in V(T)$, if there is a path $P$ in $G-\alpha(t)$ such that both endpoints of $P$ are in $\beta(t)$, then $V(P)\subseteq \beta(t)$.
    \end{enumerate}
    The width of an \ocptd $\mathcal{T}$, denoted by $\ocpw(\mathcal{T})$, is the following value:
    \begin{align*}
        \max\Big\{ \mathsf{adhesion}(T,\beta),~\max_{t\in V(T)}\big(  |\alpha(t)| +  \ocp(G[\beta(t)\setminus\alpha(t)])  \big)\Big\}.
    \end{align*}
    The \emph{\ocp-treewidth} of $G$, denoted by $\ocptw(G)$, is the minimum width of an \ocptd of $G$. For $t \in V(T)$, we call $\beta(t)$ the \emph{bag of $t$} and $\alpha(t)$ the \emph{apex set of $t$}.
\end{definition}

A key feature of treewidth is that it is minor-monotone.
A similar observation can be made for \ocp-treewidth if we augment the minor relation to respect the parity of cycles.

For a set $X\subseteq V(G)$ we write $\partial_G(X)\coloneqq \{ xy\in E(G) \colon x\in X, y\notin X\}$.
A \emph{bond} in $G$ is a set $\partial_G(X)$ such that there is no $Y\subseteq V(G)$ with $\partial_G(Y)\subsetneq \partial_G(X)$.
We perform a \emph{bond contraction} in a graph $G$ by contracting all edges in a bond $\partial_G(X)$.
A graph $H$ is said to be an \emph{odd-minor} of a graph $G$ if it can be obtained from $G$ by a sequence of edge deletions, vertex deletions, and bond contractions.

The notion of odd-minors has first been explicitly studied in the context of Hadwiger's Conjecture \cite{GeelenGRSV2009Oddminor,Steiner2022Asymptotic,NorinS2022NewUpper,Liu2024Proper}.
It is a special case of the notion of minors in \textsl{signed graphs} and has strong relations with matroid minors (see \zcref{sec:integerprogramming} for an in-depth introduction).
It is easy to see that \ocp-treewidth is odd-minor-monotone which naturally raises the question of its structural behaviour under the odd-minor relation.

\subsection{Our results}\label{subsec:ResultsIntro}
We split our overview into three categories: Algorithmic results, structural results, and building blocks for more structural inquiries concerning odd-minor-free graphs.

\paragraph{Algorithmic results.}
Our main result in this category is that MIS can be solved in polynomial time on graphs of bounded \ocp-treewidth.

\begin{theorem}[see \zcref{thm:miswithboundedocptw}]\label{thm:MISIntro}
For every fixed non-negative integer $k$, there exists a polynomial-time algorithm for (weighted) MIS on the class of all graphs of \ocp-treewidth at most $k$.
\end{theorem}

This algorithm consists of two parts.
One that finds a \ocp-tree-decomposition of small width and a second that is essentially a dynamic program along this decomposition which uses the algorithm of Fiorini et al.\@ \cite{FioriniJWY2021Integer,Fiorini2025Integer} for weighted MIS in graphs of bounded odd cycle packing number as a blackbox.
Due to the running time of the algorithm by Fiorini et al.\@, the running time of our algorithm from \zcref{thm:MISIntro} is of the form $n^{f(k)}$ where $f$ is a computable function.
\smallskip

\textbf{Tree-decomposing integer programs:}
The landscape of (tree-)decomposition based parameters measuring the matrices of integer programs and allow for parameterized algorithms is relatively sparse \cite{GanianOR2017Going,DvorakEGKO2021Complexity}.
Partially this lack of parameterized tools could be attributed to the fact that one of the most popular parameters -- \textsl{treewidth} -- fails in this regime \cite{GanianO2018Complexity}.
Based on our structural insight, we lift \zcref{thm:MISIntro} into a width parameter for integer programs whose coefficient matrices correspond to the signed incidence matrices of signed graphs.
While our parameter imposes further restrictions on the \textsl{structure} of the input matrices, i.\@e.\@ it requires them to be incidence matrices of signed graphs, it guarantees efficient algorithms way beyond the boundedness of typical parameters such as treewidth.

Fiorini et al.\@ used their result to efficiently solve certain types of integers programs as follows.
Let $\Delta$ be a positive integer.
A matrix $M$ is \emph{totally $\Delta$-modular} if the determinant of every square submatrix of $M$ falls into the set $\{ -\Delta, -\Delta+1, \dots, 0,\dots , \Delta-1, \Delta \}$.
A major conjecture in the field of integer programming is that for every $\Delta$ there is a polynomial-time algorithm to solve all integer programs of the form $\max\{ w^{\text{T}}x \colon Ax\leq b, x\in\mathbb{Z}^n\}$ where $A$ is totally $\Delta$-modular.
See \cite{ArtmannWZ2017Strongly,Schrijver2003Combinatorial,AprileFJKSWY2025Integer} for further partial solutions and more background.
In \cite{FioriniJWY2021Integer,Fiorini2025Integer}, Fiorini et al.\@ proved that for every $\Delta$ there exists a polynomial-time algorithm that solves all integer programs of the form $\max\{ w^{\text{T}}x \colon Ax\leq b, x\in\mathbb{Z}^n\}$ where $A$ is totally $\Delta$-modular and has at most two non-zero entries per row.
Using the bridge between odd-minors, minors in signed graphs, and matroid minors, we are able to combine \zcref{thm:MISIntro} with some of the methods from \cite{Fiorini2025Integer} to obtain algorithms for integers programs whose coefficient matrices correspond to signed graphs excluding certain minors.
For the sake of brevity, we present here an abbreviated version of our result and refer the interested reader to \zcref{sec:integerprogramming} where it appears in the form of \zcref{cor:IPsWork}.

To keep it short, a signed graph is a pair $(G,\gamma)$ where $G$ is a graph and $\gamma\colon E(G)\to \mathbb{Z}_2$ is a labelling of the edges of $G$ with elements of $\mathbb{Z}_2$.
A cycle $C$ of $G$ is \emph{unbalanced} if $\sum_{e\in E(C)}\gamma(e)=1$.
We adapt our definition of \ocp-treewidth to the world of signed graphs by replacing, for the evaluation of the width of a decomposition $\mathcal{T}$, the odd cycle packing number of $G[\beta(t)\setminus\alpha(t)]$ with the largest integer $k$ such that $(G[\beta(t)\setminus\alpha(t)],\gamma)$ has $k$ pairwise vertex-disjoint unbalanced cycles.

\begin{theorem}\label{thm:IPsIntro}
There exist a computable function $f$ and an algorithm that given a non-negative integer $k$ and an integer program of the form $\max\{ w^{\text{T}}x \colon Ax\leq b, x\in\mathbb{Z}^n\}$ where $A\in \{ -1,0,1\}^{m\times n}$ is the signed edge-vertex incidence matrix of a signed graph $(G,\gamma)$, either decides correctly that $\ocp\text{-}\tw(G,\gamma)>k$ or solves the integer program in time $n^{f(k)}$.
\end{theorem}

The incidence matrix of a signed graph $(G,\gamma)$ is totally $\Delta$-modular for some $\Delta$ if and only if any family of pairwise vertex-disjoint unbalanced cycles in $(G,\gamma)$ has bounded size (see \cite{Fiorini2025Integer}).
Hence, \zcref{thm:IPsIntro} contains the result from \cite{Fiorini2025Integer} for $\{ -1,0,1\}$-matrices as a special case.
Moreover, if $(G_k, \gamma_k)$ is the disjoint union of $k\geq 1$ unbalanced cycles, then for every $\Delta$ there exists $k$ such that the incidence matrix of $(G_k,\gamma_k)$ is not totally $\Delta$-modular, but $\ocp\text{-}\tw(G_k,\gamma_k)=1$ for all $k$.
Thus, \zcref{thm:IPsIntro} covers a larger class of $\{ -1,0,1\}$-matrices than the work of Fiorini et al.
\smallskip

\textbf{Finding tree-decompositions of small \textsc{OCP}-width:}
The core of our work is focussed on actually finding an \ocp-tree-decomposition of bounded width or determining that the \ocp-treewidth of the input graph is larger than $k$.
A first roadblock in this endeavour is the observation that computing the \ocp-treewidth of general graphs is as hard as computing the treewidth of bipartite graphs which is known to be equivalent to computing the treewidth of general graphs and therefore \textsf{NP}-complete \cite{Arnborg1987Complexity}.

\begin{theorem}\label{thm:OCPtwNP}
Computing \ocp-treewidth is $\mathsf{NP}$-complete.
\end{theorem}

A major consequence of our structural results as explained below is the following parametrized approximation algorithm for \ocp-treewidth.

\begin{theorem}\label{thm:OCPapproxIntro}
There exists an algorithm that takes as input a non-negative integer $k$ and an $n$-vertex graph $G$ and either finds a subgraph of $G$ certifying that $\ocp\text{-}\tw(G)>k$ or computes an \ocp-tree-decomposition of width in $\mathbf{poly}(k)$ for $G$ in time $2^{\mathbf{poly}(k)}n^{\mathcal{O}(1)}$.
\end{theorem}

\paragraph{Structural results.}
There are several equivalent ways to state the celebrated \textsl{Grid Theorem} of Robertson and Seymour \cite{RobertsonS1986Grapha}.
One of them would be a twofold statement as follows:
\begin{enumerate}
\item For every $k\geq 1$, every graph containing the $(k\times k)$-grid as a minor has treewidth at least $k$.

\item There exists a function $f$ such that for every $k\geq 1$, every graph with treewidth at least $f(k)$ contains the $(k\times k)$-grid as a minor.
\end{enumerate}
A functionally equivalent statement would be:
\begin{center}
\vspace{-3mm}
\textsl{A minor-closed graph class $\mathcal{C}$ has bounded treewidth if and only if $\mathcal{C}$ excludes a planar graph.}
\vspace{-3mm}
\end{center}
\textsl{Planarity} in the statement above can be replaced by the phrase ``\textsl{... $\mathcal{C}$ excludes a graph $H$ such that there is a $k\geq 1$ for which $H$ is a minor of the $(k\times k)$-grid.}''
It just happens to be true that the class of all graphs that are a minor of some grid is exactly the class of planar graphs.

\begin{figure}[ht]
    \centering
    \scalebox{1}{
    \begin{tikzpicture}

        \pgfdeclarelayer{background}
		\pgfdeclarelayer{foreground}
			
		\pgfsetlayers{background,main,foreground}

        \begin{pgfonlayer}{background}
            \pgftext{\includegraphics[width=9.5cm]{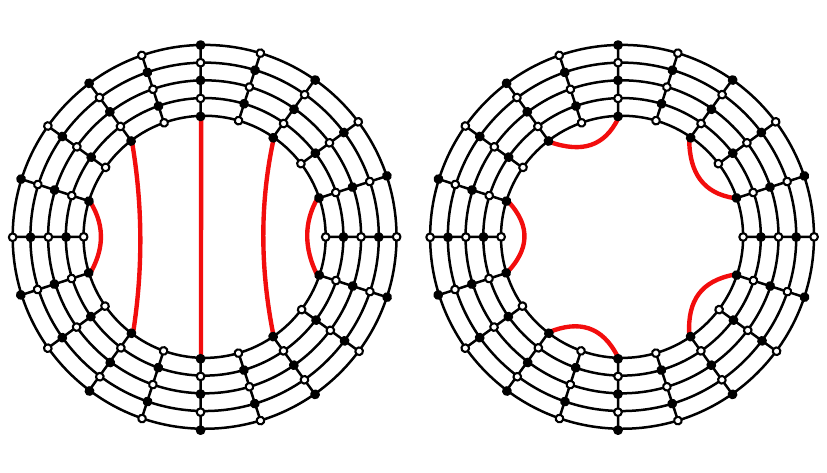}} at (C.center);
        \end{pgfonlayer}{background}
			
        \begin{pgfonlayer}{main}
        \node (M) [v:ghost] {};

        \end{pgfonlayer}{main}
        
        \begin{pgfonlayer}{foreground}
        \end{pgfonlayer}{foreground}

    \end{tikzpicture}}
    \caption{Representatives of our two obstructing families for \ocp-treewidth: The parity handle of order $5$ (left) and the parity vortex of order $5$ (right). Notice that a cycle in either of these graphs is odd if and only if it contains an odd number of the \textcolor{BostonUniversityRed}{red} edges.}
    \label{fig:ParityGridsIntro}
\end{figure}

Our structural main result is an analogue of the Grid Theorem of Robertson and Seymour for \ocp-treewidth under the odd-minor relation.
This provides a satisfying answer to the question about structural properties of graphs of bounded \ocp-treewidth as it classifies all odd-minor-closed graph classes where \ocp-treewidth is bounded, and hence, where \zcref{thm:MISIntro} can be applied.
The role of the $(k\times k)$-grid is played by two infinite families of graphs -- both consisting of planar graphs -- to which we will refer to as the \emph{parity grids} in the following.
We present both types of parity grids in their \textsl{cylindrical} form here, but they could easily be expressed as square grids as well.

Let $k\geq 1$ be an integer and $H_k$ be the cylindrical $(k \times 4k)$-grid.\footnote{The cylindrical $(n\times m)$-grid is the Cartesian (or box) product of the path $P_n$ and the cycle $C_m$.}
Moreover, let $X=\{ x_1,\dots,x_{2k} \}$ be a maximum subset of vertices of the $k$th copy of $C_{4k}$ in $H_k$ such that all vertices of $X$ are pairwise at even distance.
The \emph{parity handle} $\mathscr{H}_k$ of order $k$ is the graph obtained from $H_k$ by introducing the edges $x_ix_{2k-i+1}$ for all $i\in[k]$.
The \emph{parity vortex} $\mathscr{V}_k$ of order $k$ is the graph obtained from $H_k$ by introducing the edges $x_{2i-1}x_{2i}$ for all $i\in[k]$.
See \zcref{fig:ParityGridsIntro} for examples.

Our structural main theorem reads as follows.

\begin{theorem}\label{thm:GridThmIntro}
There exists a polynomial $f\colon\mathbb{N}\to\mathbb{N}$ such that for every integer $k \geq 1$ and every graph $G$ it holds that
\begin{enumerate}
    \item if $G$ contains $\mathscr{H}_k$ or $\mathscr{V}_k$ as an odd-minor, then $\ocp\text{-}\tw(G)\geq \frac{k}{2}$, and
    \item if $\ocp\text{-}\tw(G)\geq f(k)$ then $G$ contains $\mathscr{H}_k$ or $\mathscr{V}_k$ as an odd-minor.
\end{enumerate}
Moreover, $f(k)\in\mathbf{poly}(k)$ and there exists an algorithm that, given an integer $k$ and a graph $G$ as input, finds either a subgraph of $G$ certifying that $G$ contains $\mathscr{H}_k$ or $\mathscr{V}_k$ as an odd-minor, or produces an \ocp-tree-decomposition of width at most $f(k)$ in time $2^{\mathbf{poly}(k)}|V(G)|^{\mathcal{O}(1)}$.
\end{theorem}

Since $f$ is a polynomial, the algorithm guaranteed by \zcref{thm:GridThmIntro} in fact yields \zcref{thm:OCPapproxIntro}.

Notably, every odd-minor of a parity handle can be embedded into the plane with at most two odd faces, while every odd-minor of the parity vortex can be embedded in the plane with a unique face that meets all odd cycles in the graph.
These two properties illustrate that there exist planar graphs $H$ such that $H$-odd-minor-free graphs have unbounded \ocp-treewidth.

As \textsc{MIS} is \textsf{NP}-hard on planar graphs, the core question underlying our research is the following:

\begin{question}\label{quest:MainIntro}
For which planar graphs $H$ does there exist a polynomial-time algorithm for \textsc{MIS} on $H$-odd-minor-free graphs?
\end{question}

To give a more concrete example of the gaps in our understanding of \textsc{MIS} in planar graphs, consider the following problem concerning a class of graphs with unbounded $\mathsf{OCP}$.

\begin{question}\label{quest:BoundedOddFaces}
Let $k \in \mathbb{N}$.
Does there exist a computable function $f \colon \mathbb{N} \to \mathbb{N}$ and an algorithm for \textsc{MIS} running in $f(k)n^{\mathcal{O}(1)}$-time on the class of planar graphs with at most $k$ odd faces?
\end{question}

In the above question $k$ can clearly be taken to be even.
The case $k \leq 2$ was solved by Gerards in \cite{Gerards1989MinMax}.
All other cases are widely open.
In \zcref{quest:BoundedOddFaces} we ask explicitly for an \textsf{FPT}-algorithm, but even an \textsf{XP}-algorithm is unknown and would fully suffice to further our understanding of \zcref{quest:MainIntro}.

Beyond planar graphs the following question seems key to expanding our understanding of the tractability of \textsc{MIS}.
A graph embedded in a surface $\Sigma$ is said to have an \textsl{even-faced embedding} if none of its odd cycles describes a genus-reducing curve in $\Sigma$, equivalently no odd cycle bounds a disk in $\Sigma$.

\begin{question}\label{quest:EvenFacedTorus}
Does there exist a polynomial-time algorithm for \textsc{MIS} on graphs with a given even-faced embedding on the torus.
\end{question}

\zcref{quest:BoundedOddFaces,quest:EvenFacedTorus} are not asked idly.
We strongly believe that both of these questions must in some way be resolved to expand the structural understanding of odd-minor-free graph classes in which \textsc{MIS} is efficiently solvable \textsl{beyond} the results of this paper, that is the combination of \zcref{thm:GridThmIntro} and \zcref{thm:MISIntro}.
Notably, both questions diverge from the perspective of \cite{ConfortiFHJW2019stableset}, \cite{Fiorini2025Integer}, and \cite{AprileFJKSWY2025Integer} as the classes considered here all contain graphs of unbounded $\mathsf{OCP}$ and thus the corresponding matrices have \textsl{unbounded} subdeterminants.
As such our more structural approach to studying odd-minor-closed graph classes seems more promising.
Indeed, any positive progress on the questions above is likely to yield immediate progress towards \zcref{quest:MainIntro} due to the robustness of our structural tools as laid out below.

\paragraph{Tools for the study of the structure of odd-minor-free graphs.}
As part of the more general project outlined above, we are working on understanding the structure of odd-minor-free graphs, with a particular focus on excluding planar graphs as odd-minors.
In line with this, several of the theorems and tools in this article are geared not only towards proving the results advertised in this paper, but also for further use in future parts of this project.
We highlight here the advantages of our particular approach and why it is necessary to proceed this way if we wish to prove theorems constructively and with explicit bounds.

The established approach to proving a structural result when excluding an odd-minor (or in a closely related fashion an \textsl{oriented or unoriented group-labelled-minor}\footnote{We will not define the particulars of the group-labelled setting as they are not needed for our results.}) is to take the existing results from the Graph Minor Series due to Robertson and Seymour (especially the Graph Minor Structure Theorem (GMST) from \cite{RobertsonS2003Grapha}) and then provide a tailor-made analysis of the structure provided by the GMST to solve the specific problem under consideration.
Examples of this approach can be found in the work of Fiorini et al.\ \cite{Fiorini2025Integer} we already discussed, an approximate description of the structure of odd-minor-free graphs \cite{DemaineHK2010Decomposition}, a solution to the linkage problem for the oriented group-labelled setting \cite{Huynh2009Linkage}, a structure theorem for oriented group-labelled graphs \cite{GeelenG2009Excluding}, and the solution to the linkage problem in the unoriented group-labelled setting (together with an associated structure theorem) \cite{LiuY2025Disjoint}.
We provide a rough description of what the GMST tells us about a graph excluding a given minor in \zcref{subsec:ProofSketchIntro}.

This approach has two downsides.
First, due to using older versions of the GMST and the associated machinery, such an approach often does not yield explicit bounds for the desired structural properties.
More importantly, due to the ad-hoc nature of these individual solutions it is difficult to adapt them as tools for the further development of an efficient structure theory for odd-minors, where by efficient we mean theorems with explicit, hopefully low bounds, such as the polynomial bounds presented in this paper.
This is particularly important as future progress towards resolving \zcref{quest:MainIntro} will most likely heavily rely on such structural results.

We instead build robust tools that can be used in a more modular fashion.
The key result here is a proof of an odd-minor version of the \textsl{Society Classification Theorem} (see \zcref{thm:evensocietyclassification}), which is the engine behind a new constructive proof of the GMST with explicit bounds due to Kawarabayashi, Thomas, and Wollan \cite[Theorem 10.1]{KawarabayashiTW2021Quickly} (see also \cite[Theorem 14.1]{GorskySW2025polynomialboundsgraphminor}).
This theorem is what allows Kawarabayashi et al.\ to prove the GMST constructively with explicit bounds via an induction.
Our variant necessarily features more outcomes than the original Society Classification but importantly can be adapted to prove further results regarding excluded odd minors and (in line with the variant in \cite{GorskySW2025polynomialboundsgraphminor}) features polynomial bounds on the order of the structural elements involved.

As a second example, we provide a proof of a odd-minor variant of the \textsl{Flat Wall Theorem} (again see \zcref{subsec:ProofSketchIntro} for more details) with polynomial bounds, contrasting the worse than exponential bounds one would derive from the more general Flat Wall Theorem variant Thomas and Yoo prove in \cite{ThomasY2023PackingCycles} for group-labelled graphs.

\subsection{Related work}\label{subsec:RelatedIntro}
Let us now revisit related work on the topic of islands of tractability for the \textsc{MIS} problem.
We partition this subsection into two parts:
First we discuss how \ocp-treewidth fits into the known landscape of odd-minor-free classes and odd-minor-related parameters that allow for efficient algorithms for \textsc{MIS}.
In the second part, we briefly touch upon other structural strategies to approach \textsc{MIS} that have emerged in recent years.

\paragraph{\textsc{Maximum Independent Set} and odd-minors.}

It is easy to see that a graph is bipartite if and only if it does not contain $K_3$ as an odd-minor.
So on $K_3$-odd-minor-free graphs \textsc{MIS} is famously tractable.
Moreover, Gerards \cite{Gerards1989MinMax} has proven a min-max duality for the \textsc{MIS}-problem on $K_4$-odd-minor-free graphs that implies a polynomial-time algorithm.
For any $t\geq 5$, the class of $K_t$-odd-minor-free graphs contains the class of all planar graphs and thus, \textsc{MIS} becomes \textsf{NP}-hard.
However, on the positive side, Tazari \cite{Tazari2012Faster} showed that \textsc{MIS} admits a PTAS on $H$-odd-minor-free graph classes for \textsl{every} graph $H$. 

\begin{figure}[ht]
    \centering
    \scalebox{1}{
    \begin{tikzpicture}

        \pgfdeclarelayer{background}
		\pgfdeclarelayer{foreground}
			
		\pgfsetlayers{background,main,foreground}

        \begin{pgfonlayer}{background}
            \pgftext{\includegraphics[width=12cm]{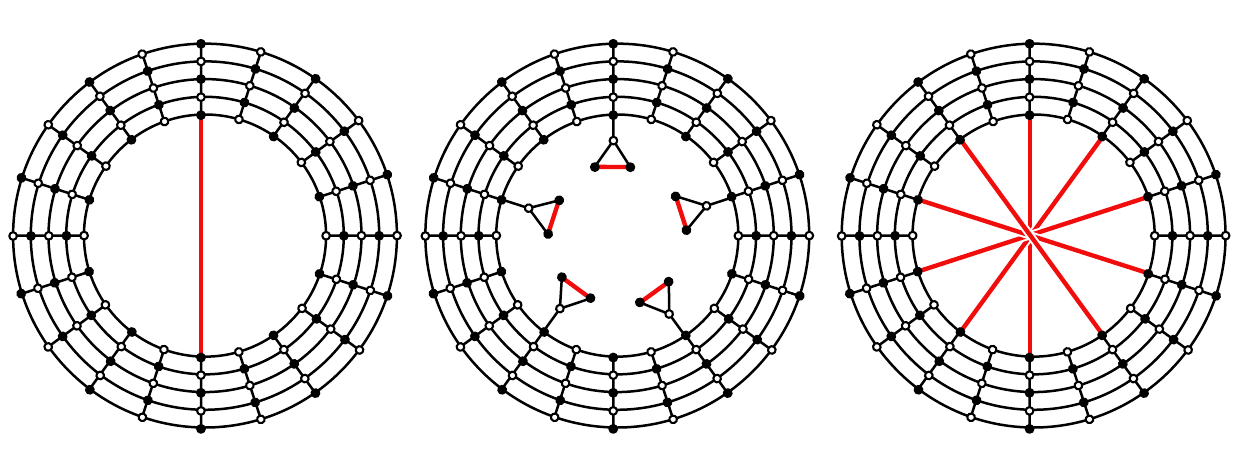}} at (C.center);
        \end{pgfonlayer}{background}
			
        \begin{pgfonlayer}{main}
        \node (M) [v:ghost] {};

        \end{pgfonlayer}{main}
        
        \begin{pgfonlayer}{foreground}
        \end{pgfonlayer}{foreground}

    \end{tikzpicture}}
    \caption{Representatives of three more types of parity grids: The single parity break of order $5$ (left), a grid with odd cycle outgrows of order $5$ (middle), and a parity crosscap, also known as an \textsl{Escher grid}, of order $5$ (right).
    As before, a cycle in either of these graphs is odd if and only if it contains an odd number of the \textcolor{BostonUniversityRed}{red} edges.}
    \label{fig:ThreeMoreGridsIntro}
\end{figure}

In the following we let $\mathcal{B}$ be the class of bipartite graphs.
Recall that for a graph $G$ we denote by $\ocp(G)$ the largest integer $k$ such that $G$ has $k$ pairwise vertex-disjoint odd cycles.
This is called the \emph{odd cycle packing number}.
The \emph{odd cycle transversal number} of $G$, denoted by $\mathsf{OCT}(G)$, is the smallest $k$ such that there is a set $S\subseteq V(G)$ of size at most $k$ where $G-S$ is bipartite.

It is clear that $\ocp(G)\leq k$ if and only if $G$ does not contain $k\cdot K_3$ as an odd-minor.
Similarly, Reed \cite{Reed1999Mangoes} proved that $\mathsf{OCT}(G)$ is bounded if and only if there exists a $k$ such that $G$ does not contain $k \cdot K_3$ nor the \textsl{Escher grid} of order $k$ (see \zcref{fig:ThreeMoreGridsIntro}) as odd-minors.
Since computing $\mathsf{OCT}(G)$ is \textsl{fixed-parameter tractable} (see for example \cite{ReedSV2004Finding}), \textsc{MIS} can be solved in polynomial-time on graphs of bounded \textsf{OCT}.
Moreover, by the result of Fiorini et al.\@ \cite{Fiorini2025Integer}, \textsc{MIS} is tractable on graph classes of bounded \ocp.
Notice that every class of bounded \textsf{OCT} is also a class of bounded \textsf{OCP} and, in turn, any class of bounded \ocp\@ is trivially also of bounded \ocp-treewidth.

The Escher grid acts as an excellent example that the parameters \textsf{OCT} and \ocp\@ are \textsl{comparable yet distinct} in the sense that for every positive integer $k$, \textsf{OCT} of the Escher grid of order $k$ is $k$, while its \ocp\@ is equal to $1$.
However, $\ocp(G)\leq \mathsf{OCT}(G)$ for all graphs $G$.
Two parameters that generalise the notion of \textsf{OCT} while maintaining both the comparability and the distinctness to \ocp\@ are the \emph{elimination distance to $\mathcal{B}$}, denoted by $\mathsf{ed}_{\mathcal{B}}$, (see \cite{BulianD2016Graph,BulianD2017Fixed-parameter}) and \emph{$\mathcal{B}$-treewidth}, denoted by $\tw_{\mathcal{B}}$ (see \cite{EibenGHK2021Measuring,JansendKW2021Vertex,JansendK2021FPTAlgorithms}).
Each of them is the minimum $k$ such that a graph $G$ has a set $X\subseteq V(G)$ such that $G-X\in\mathcal{B}$ and $\mathsf{p}(G_X)\leq k$ where $G_X$ is the \textsl{torso} of $X$ in $G$ and $\mathsf{p}$ is the \textsl{treedepth}\footnote{Since treedepth is not used anywhere in the paper, we omit its definition. See \cite{BulianD2016Graph,BulianD2017Fixed-parameter,JansendKW2021Vertex} for further information.} in the case of elimination distance to $\mathcal{B}$ and the \textsl{treewidth} in the case of $\mathcal{B}$-treewidth.
Moreover, it holds that $\mathsf{OCT}(G) \geq \mathsf{ed}_{\mathcal{B}}(G) \geq \tw_{\mathcal{B}}(G)$.

The \emph{$\mathcal{B}$-blind-treewidth}, denoted by $\mathcal{B}\text{-}\mathsf{blind}\text{-}\tw$, of a graph $G$ is the largest treewidth over all \textsl{non-bipartite} blocks of $G$.
Gollin and Wiederrecht \cite{GollinW2023OddMinors} proved an analogue of the Grid Theorem for $\mathcal{B}$-blind-treewidth where the grid is the \textsl{single parity break grid} (see \zcref{fig:ThreeMoreGridsIntro}).
Observe that the single parity break grid of order $k$ has \textsf{OCT} equal to $1$ for all $k\geq 1$.
The notion of \textsl{bipartite treewidth} as given by Jaffke et al. \cite{JaffkeMST2023Dynamic} is relatively technical, instead we give a functionally equivalent definition in terms of \ocp-tree-decompositions:
The \emph{bipartite treewidth}, denoted by \textsf{btw}, of a graph $G$ is the minimum width of an \ocp-tree-decomposition $(T,\beta,\alpha)$ where $G[\beta(t)]-\alpha(t)$ is bipartite for all $t\in V(T)$.
Notice that $\mathcal{B}\text{-}\mathsf{blind}\text{-}\tw(G) \geq \mathsf{btw}(G)$ and $\tw_{\mathcal{B}}(G) \geq \mathsf{btw}(G)$ for all graphs $G$.
In particular, for every $k\geq 1$, the grid with odd cycle outgrowths of order $k$ (see \zcref{fig:ThreeMoreGridsIntro}) has bipartite treewidth $1$, while its $\mathcal{B}$-treewidth is $k$.
So with bipartite treewidth we now have a first parameter that properly subsumes both, the \textsf{OCT}-based parameters as well as $\mathcal{B}$-blind-treewidth.

The only parameter we have not discussed yet is \ocp-treewidth itself.
From the definition one can immediately deduce that $\ocp\text{-}\tw(G) \leq \mathsf{btw}(G)$.
Moreover, the Escher grid of order $k$ can be seen to have bipartite treewidth $k$ while, as established above, its \ocp-treewidth is $1$ for all $k\geq 1$.
This makes \ocp-treewidth the \textsl{most general} odd-minor-based parameter so far when it comes to parametrizations of \textsc{MIS}.
In \zcref{fig:Classes} we give an overview of all classes and parameters discussed above and what they imply for the tractability of \textsc{MIS}.

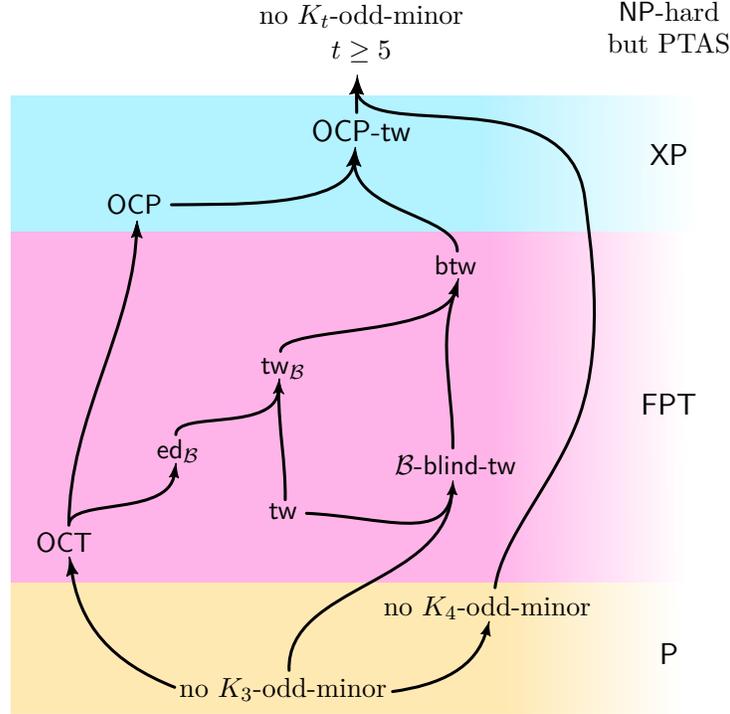
\begin{figure}[ht]
    \centering
    \scalebox{0.95}{
    \begin{tikzpicture}

        \node (C) [v:ghost] {};

        \pgfdeclarelayer{background}
		\pgfdeclarelayer{foreground}
			
		\pgfsetlayers{background,main,foreground}

        \begin{pgfonlayer}{background}
            \pgftext{\includegraphics[width=12cm]{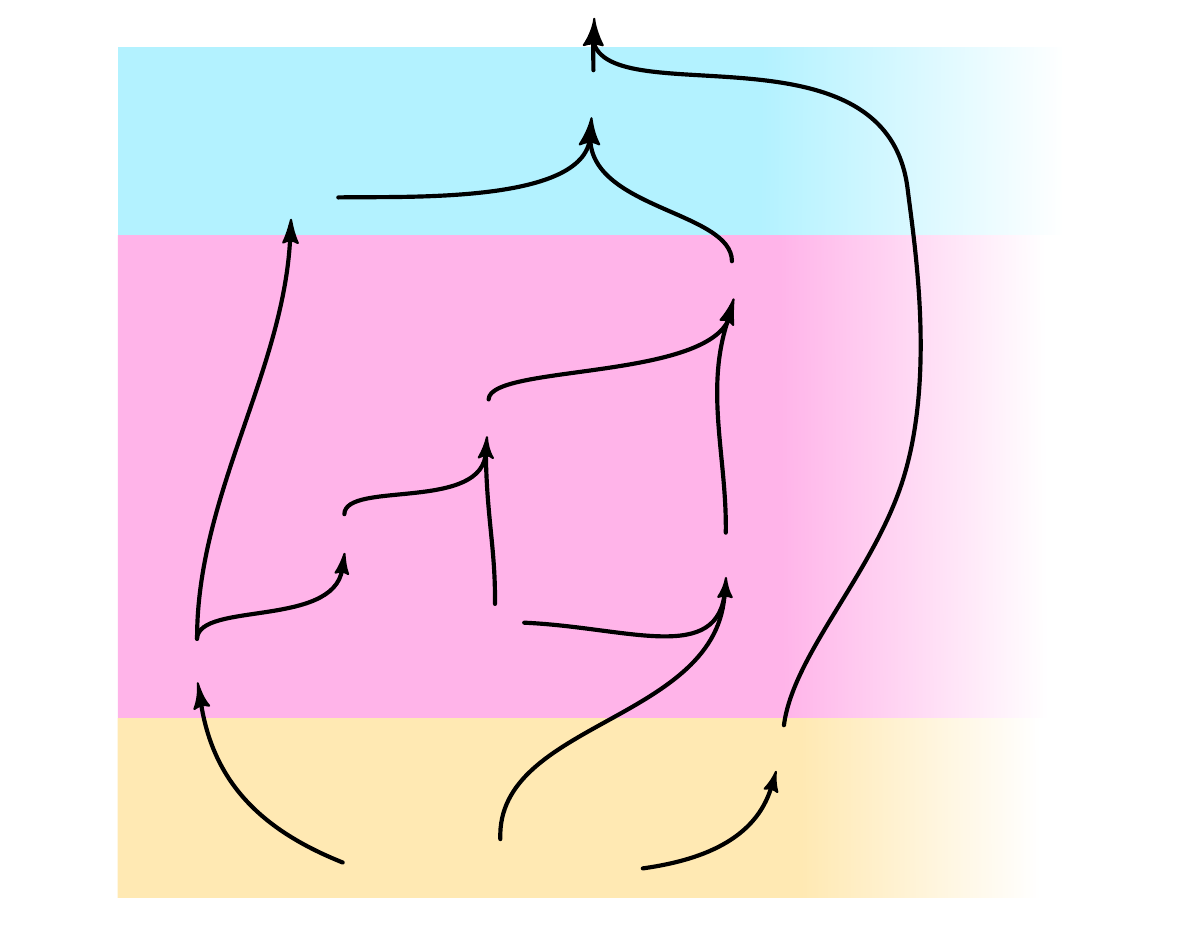}} at (C.center);
        \end{pgfonlayer}{background}
			
        \begin{pgfonlayer}{main}
        \node (M) [v:ghost,position=270:45mm from C] {};

        \node (L) [v:ghost,position=180:10mm from M] {};
        \node (R) [v:ghost,position=0:44mm from M] {};

        \node (Poly) [v:ghost,position=90:10mm from R] {\large \textsf{P}};
        \node (FPT) [v:ghost,position=90:45mm from R] {\large \textsf{FPT}};
        \node (XP) [v:ghost,position=90:80mm from R] {\large \textsf{XP}};
        \node (NPC) [v:ghost,position=90:100mm from R] {\textsf{NP}-hard};
        \node (NPC2) [v:ghost,position=270:4.5mm from NPC] {but PTAS};

        \node (K3) [v:ghost,position=90:5mm from L] {no $K_3$-odd-minor};
        \node (K4) [v:ghost,position=30:33mm from L] {no $K_4$-odd-minor};

        \node (OCT) [v:ghost,position=140:40mm from L] {\textsf{OCT}};

        \node (tw) [v:ghost,position=90:30mm from L] {\textsf{tw}};

        \node (Bblindtw) [v:ghost,position=15:25mm from tw] {$\mathcal{B}$-\textsf{blind}-\textsf{tw}};

        \node (btw) [v:ghost,position=90:28mm from Bblindtw] {\textsf{btw}};

        \node (Btw) [v:ghost,position=90:20mm from tw] {$\mathsf{tw}_{\mathcal{B}}$};

        \node (Bed) [v:ghost,position=150:17mm from tw] {$\mathsf{ed}_{\mathcal{B}}$};

        \node (OCP) [v:ghost,position=100:35mm from Bed] {$\mathsf{OCP}$};

        \node (OCPtw) [v:ghost,position=125:23mm from btw] {\large $\mathsf{OCP}$-\textsf{tw}};

        \node (Kt) [v:ghost,position=90:16mm from OCPtw] {no $K_t$-odd-minor};
        \node (Kt2) [v:ghost,position=270:5mm from Kt] {$t\geq 5$};

        \end{pgfonlayer}{main}
        
        \begin{pgfonlayer}{foreground}
        \end{pgfonlayer}{foreground}

    \end{tikzpicture}}
    \caption{The complexity landscape of \textsc{MIS} in odd-minor-closed graph classes.
    The bottom shows the two instances of $K_t$-odd-minor-free classes where \textsc{MIS} is in \textsf{P}. The two middle areas depict odd-minor-monotone parameters and the corresponding (parameterized) complexity classes for \textsc{MIS} and finally, the very top indicates the general regime of $K_t$-odd-minor-free graphs for $t\geq 5$ where \textsc{MIS} is known to be \textsf{NP}-hard and to admit a PTAS.
    An arrow from $x$ to $y$ indicates that membership in the class $x$, or having the parameter $x$ bounded also implies membership in the class $y$ or that the parameter $y$ is bounded.}
    \label{fig:Classes}
\end{figure}

With this it is now apparent that within the realm of known odd-minor-closed graph classes where \textsc{MIS} is tractable, there is a unique class, namely the $K_4$-odd-minor-free graphs, that does not fall under the regime of \ocp-treewidth.
Indeed, the parity handle grids are all $K_4$-odd-minor-free while $K_4$ is an odd-minor of $\mathscr{V}_3$.
This raises the following question, related to \zcref{quest:MainIntro}.

\begin{question}
Let $H$ be an odd-minor of the parity vortex of order $k$ for some $k\geq 1$, is there a polynomial-time algorithm for \textsc{MIS} on $H$-odd-minor-free graphs?
\end{question}

\paragraph{Non odd-minor-based approaches to \textsc{MIS}.}
To conclude this part, notice that we focussed here on the relevant literature for odd-minors.
There are several other ongoing lines of research dealing with the computational complexity of \textsc{MIS} in settings of restricted structure.
Among them is a recent interest in induced minors and the related \textsl{tree-independence number} \cite{Yolov2018Minor,DallardMS2024TreeIndependence,ChudnovskyHT2024Tree}.
Of particular interest is also the realm of $H$-vertex-minor-free graphs and classes of bounded \textsl{rankwidth} \cite{Oum2005Rankwidth,GeelenKMW2023Grid}.
In this area we have a particularly intriguing conjecture due to Geelen (see \cite{McCarty2021Local}) claiming that \textsc{MIS} should be tractable on \textsl{all} vertex-minor-closed graph classes.
Finally, there exist many different kinds of ``width'' parameters that are fit to act as possible parametrizations for \textsc{MIS} (see \cite{BergounouxKR2023} for some recent results and a good overview).

\paragraph{Existing tools for the study of odd-minor-free graphs.}
As mentioned earlier, there exist a few high-profile results on the structure of odd-minor-free and group-labelled-minor-free graphs \cite{Huynh2009Linkage,GeelenG2009Excluding,DemaineHK2010Decomposition,ThomasY2023PackingCycles,LiuY2025Disjoint}, which all feature the downsides laid out earlier, namely non-constructive proofs, non-modular proof approaches, and non-explicit or huge bounds.
Most other results concerning the structure of odd-minor-free graphs are more directly related to the odd-minor variant of Hadwiger's famous colouring conjecture \cite{Hadwiger1943Uber,JensenT1994Graph} (see also \cite{Seymour2016Hadwigers}).
Sadly, with the exception with the work of Geelen et al.\ in \cite{GeelenGRSV2009Oddminor}, these results tend to not be very helpful in pursuing the kind of characterisations of odd-minor-free classes we are after.

\subsection{An overview of the proof}\label{subsec:ProofSketchIntro}
We provide a brief overview of our proof for \zcref{thm:GridThmIntro}.
Many of our techniques are based on the recent results in \cite{GorskySW2025polynomialboundsgraphminor} on the so-called \textsl{Graph Minor Structure Theorem} (GMST).
The GMST -- originally due to Robertson and Seymour \cite{RobertsonS2003Grapha} -- gives an approximate description of $H$-minor-free graphs, by stating that each such graph has a tree-decomposition where neighbouring bags intersect in a bounded number of vertices and the torso of every bag ``\textsl{almost embeds}'' into a surface where $H$ does not embed.
Here an ``\textsl{almost embedding}'' of $H$ is roughly a drawing of $H$ into a surface $\Sigma$ after the removal of a bounded number of vertices, such that $H$ is drawn into $\Sigma$ ``up to 3-separations'' with a bounded number of special disks called ``\textsl{vortices}'' in which the drawing many include crosses but the connectivity of the part of $H$ embedded in the disk is restricted severely.
(We highly recommend \cite{GorskySW2025polynomialboundsgraphminor} for its illustrations of these concepts.)
We refine the techniques of \cite{GorskySW2025polynomialboundsgraphminor} to be sensitive also to the parity of the embedded cycles and construct a more refined form of such an almost embedding, where the only cycles that may be of odd length are those that pass through the crosscaps of the surface an odd number of times and those that interact with vortices.
Let us call this refined almost embedding an ``\textsl{even-faced almost embedding}'' for this overview.

Due to the high degree of technicality, we keep the details of the definitions involved purposefully vague and only provide a rough intuition.
The relevant definitions can be found in \zcref{sec:graphMinorPrelims}.

In the following we lay out five general steps.
The first four of these steps have as their final goal to describe the structure of graphs without $\mathscr{H}_k$ and $\mathscr{V}_k$ as odd-minors ``locally''.
Here the notion of locality is captured by being ``highly connected'' to a given grid minor (or \textsl{wall} or \textsl{mesh}).
In the theory of Robertson and Seymour this connectivity is expressed through the notion of tangles (see \cite{RobertsonS1991Graph}).
The final step then puts these local pieces together to derive \zcref{thm:GridThmIntro}.

\paragraph{A bipartite Flat Wall.}
A central theorem of Robertson and Seymour is the \textsl{Flat Wall Theorem} \cite{RobertsonS1995Graph}.
It says that given a large wall\footnote{A \emph{wall} is a subdivision of a hexagonal grid.} $W$ in a graph $G$, there either exists a $K_t$-minor highly connected to $W$, or a small set $A \subseteq V(G)$ and a smaller but still big wall $W' \subseteq W$ such that the part of $G - A$ that attaches to the ``inside''\footnote{The ``inside'' is everything except for the cycle bounding the outer face in a standard planar drawing of the wall.} of $W'$ can be almost embedded into the plane without vortices.
Such an embedding is called \textsl{flat}.
We note that this is still an almost embedding and not a drawing, since we are still only embedding the graph up to 3-separations.
Another key theorem for us is due to Geelen, Gerards, Reed, Seymour, and Vetta \cite{GeelenGRSV2009Oddminor} stating that given a large enough complete graph as a minor in $G$, the graph $G$ either contains every graph on $t$ vertices as an odd-minor, or there is a small set $A\subseteq V(G)$ such that the part of $G-A$ the contains this complete minor is bipartite.
Combining both results allows us to show that for any large enough wall $W$, we either find a set $A$ and wall $W'$ as above such that the inside of $W'$ is also bipartite, or we find $\mathscr{H}_k$ as an odd-minor.
A similar theorem was for example proven by Thomas and Yoo \cite{ThomasY2023PackingCycles}.
In the interest of deriving polynomial bounds for our results, we provide a simple, independent proof.

\paragraph{A parity-sensitive society classification.}
A core technique called \textsl{society classification} useful for proving the GMST, pioneered by Kawarabayashi, Thomas, and Wollan \cite{KawarabayashiTW2021Quickly} and refined in \cite{GorskySW2025polynomialboundsgraphminor}, is to fix our perspective onto a disk whose boundary is protected by many concentric cycles which appear together in a ``flat'' embedding on the outside of the disk, whilst on the inside of the disk we have an unembedded part of the graph.
In \cite{KawarabayashiTW2021Quickly} it is proved that in this setting we can either find a $K_t$-minor, a large linkage resembling a handle or crosscap that has its endpoints on the boundary of the disk, or we can embed the entirety of the graph drawn on this disk in a flat way up to a small number of vortices whose internal connectivity is additionally restricted.
In the last case, we in fact find an embedded, wall-like subgraph capturing each of these vortices in distinct walls that are all well-connected to a common wall living on the outside of our disk.  

We refine this result by upgrading the embedding in each case to be ``even-faced'' in the sense described above and letting a variant of our obstructions be another outcome.
In particular, we use the theorem from \cite{GeelenGRSV2009Oddminor} discussed above to process the $K_t$-minor further.

\paragraph{Building a surface.}
Our society classification variant can be used to build a ``local'' form of the structure theorem based on a wall.
First, the wall is made flat and bipartite via our modified flat wall theorem.
Then, starting on the outside of the wall, we apply our refined society classification theorem to either find one of our obstructions, turn the block associated with our wall bipartite after the removal of a small number of vertices, embed what remains unembedded of the graph up to a handful of vortices and deleted vertices, or we find a large linkage resembling a crosscap or handle after removing a few vertices.

The last option allows us to iterate, as we can use this linkage and the structure on the boundary of our disk to cut out another disk that encompasses the previous one whilst preserving a large part of the handle or crosscap linkage.
Thus we also have witnesses to the fact that we need to increase the genus of the surface we use in our even-faced almost embedding.
The supporting infrastructure these linkages provide accumulates throughout the iterations, allowing us to find one of our obstructions if this process continues for too long.
This procedure ends up either finding one of our obstructions or an even-faced almost embedding for the block associated with the remains of the initial wall.

\paragraph{Locally bounding \textsc{OCP}.}
Our cycles may be odd either when they are not entirely embedded, i.e.\ they go through a vortex, or they are embedded by going through an odd number of crosscaps.
We first remove all odd cycles arising from vortices or obtain $\mathcal{V}_t$ as an odd-minor using a general technique exemplified in \cite{ThilikosW2024Killing}.
Consequently, all remaining odd cycles go through an odd number of crosscaps, allowing us to bound the odd-cycle-packing number by the maximum size of a disjoint set of such curves in our embedding via fairly simple counting arguments.

\paragraph{Deriving a decomposition.}
Being able to ``locally'' guarantee that our graph has bounded \ocp, we modify a standard strategy stemming from one of the main proofs in \cite{RobertsonS1991Graph} to derive the GMST from the local structure theorem to finally arrive at our desired \ocptd in the absence of any obstructions.
This allows us to not only derive \zcref{thm:GridThmIntro} at this point, but also prove \zcref{thm:OCPapproxIntro} fairly directly.

\newpage
\section{Preliminaries}\label{sec:preliminaries}
We start by introducing some basic terms from graph theory and in particular graph structure theory.
Our notation mainly derives from \cite{Diestel2010Graph}.
Exceptions to this rule will be made explicit, e.g.\ we denote the complete graph on $t$ vertices as $K_t$.
Later on we will introduce a much larger and more complex set of concepts to delve into the theory surrounding the graph minor structure theorem (see \zcref{sec:graphMinorPrelims}, a variant of which we will prove ourselves.
This first set of definitions is used to introduce and discuss the basic properties of our new parameter (see \zcref{sec:ocptw,sec:dynamicprogramming}).

On many occasions we will deal with a set $\mathcal{G}$ of graphs and want to form the union of all of them.
We will often write $\bigcup \mathcal{G}$ to denote $\bigcup_{G \in \mathcal{G}} G$ for the sake of simplicity.
In particular, if $\mathfrak{G} = \{ \mathcal{G}_1, \ldots , \mathcal{G}_k \}$ is a set of sets of graphs, we further let $\bigcup \mathfrak{G}$ denote $\bigcup_{\mathcal{G} \in \mathfrak{G}} \bigcup \mathcal{G}$.

We write $\partial(X)$, for a set $X \subseteq V(G)$ in a graph $G$, to denote an \emph{edge cut (around $X$)}, which is the set of all edges in $E(G)$ with exactly one endpoint in $X$.
As a shorthand, we also write $\partial(v)$, if $X = \{ v \}$.

The \emph{Cartesian product} of two graphs $H,G$, denoted as $H \Box G$ is constructed by using $V(H) \times V(G)$ as the vertex set and making two vertices $(u,v),(u',v')$ adjacent if and only if either $uu' \in E(H)$ and $v=v'$, or $vv' \in V(G)$ and $u=u'$.

For a few of our statements, we will want access to the matrix multiplication constant to state accurate runtimes.
\begin{definition}\label{def:matrixmultconstant}
    Let $\omega_{\ref{def:matrixmultconstant}}$ be the matrix multiplication constant with $2 \leq \omega_{\ref{def:matrixmultconstant}} \leq 2.372$ (see \cite{WXXZ2024newbounds} for the current status of this constant).
\end{definition}

\paragraph{Colouring.}
Given a positive integer $k \in \mathbb{N}$ and a set $S$, a \emph{$k$-colouring (of $S$)} is a function $c \colon S \to [k]$.
We call a $k$-colouring $c$ of the vertex set of a graph $G$ \emph{proper} if we have $c(u) \neq c(v)$ for all $uv \in E(G)$.
A graph that admits a proper 2-colouring is called \emph{bipartite}.
We call the two sets of differently coloured vertices in a bipartite graph (with an implicit, fixed, proper colouring) its \emph{colour classes}.
In our figures we will often mark the two colour classes of a bipartite graph in black and white.

\paragraph{Separations.}
A \emph{separation} in a graph $G$ is a pair $(A,B)$ of vertex sets such that $A \cup B=V(G)$ and there is no edge in $G$ with one endpoint in $A\setminus B$ and the other in $B\setminus A$.
The \emph{order} of $(A,B)$ is $|A\cap B|$.
Given sets $X,Y\subseteq V(G)$ an \emph{$X$-$Y$-separation} in $G$ is a separation $(A,B)$ such that $X\subseteq A$ and $Y\subseteq B$.

\paragraph{Linkages and paths.}
We call a path $P$ \emph{even} if it contains an even number of edges and otherwise $P$ is considered to be \emph{odd}.
A \emph{linkage} $\mathcal{L}$ in a graph $G$ is a set of pairwise vertex-disjoint paths and, in a slight abuse of notation, we use $V(\mathcal{L})$ and $E(\mathcal{L})$ to denote $V( \bigcup \mathcal{L} )$ and $E(\bigcup \mathcal{L} )$ respectively.

We say that a path $P$ is \emph{internally disjoint} from a set $X$, or a subgraph $H$, if $V(P) \cap X$, respectively $V(P) \cap V(H)$, does not contain any non-endpoint vertex of $P$.
Let $A, B \subseteq V(G)$, an \emph{$A$-$B$-path} is a path that has one endpoint in $A$, the other in $B$, and is internally disjoint from $A \cup B$.
Further, an \emph{$A$-path} is a path of length at least one with both endpoints in $A$ that is internally disjoint from $A$.
If $H$ is a subgraph of $G$, we also denote $V(H)$-paths simply as \emph{$H$-paths} and we extend this notation also to $H$-$H'$-paths, if $H'$ is also a subgraph of $G$.
If $A$ contains only a single element $a \in A$, we also denote an $A$-$B$-path as an $a$-$B$-path.
Both $a$-$b$-paths and $A$-$b$-paths similarly denote paths between sets which may be singletons.
An \emph{$A$-$B$-linkage} for two vertex sets $A,B \subseteq V(G)$ is a linkage consisting of $A$-$B$-paths.

\begin{proposition}[Menger's Theorem \cite{Menger1927Zur}]\label{prop:mengersthm}
Let $G$ be a graph and $X,Y\subseteq V(G)$ be two sets of vertices.
Then the minimum order of a $X$-$Y$-separation in $G$ equals the maximum number of paths in an $X$-$Y$-linkage in $G$.
\end{proposition}

Given an $X$-$Y$-linkage of size $k$, a well-known algorithm of Ford and Fulkerson \cite{FordF1956Maximal} takes time $\mathbf{O}(|E(G)|)$ to find an $X$-$Y$-linkage of size $k+1$, if one exists.
It follows that we can find an $X$-$Y$-linkage of order $k$ in time $\mathbf{O}(k|E(G)|)$ or determine that no such linkage exists and find an $X$-$Y$-separation of minimum order instead.
We will generally use this fact implicitly when referencing \hyperref[prop:mengersthm]{Menger's Theorem}. 

If $P$ is a path and $x,y \in V(P)$ are vertices on $P$, we denote by $xPy$ the subpath of $P$ with the endpoints $x$ and $y$.
Let $P$ be a path from $s$ to $t$ and $Q$ be a path from $q$ to $p$.
If $x$ is a vertex in $V(P) \cap V(Q)$ and $sPx$ and $xQp$ only intersect in $x$, then we let $sPxQp$ be the path obtained from the union of $sPx$ and $xQp$.

\paragraph{Tree-decompositions and treewidth.}
A \emph{tree-decomposition} of a graph $G$ is a pair $\mathcal{T} = (T, \beta)$, where $T$ is a tree and $\beta: V(T) \rightarrow 2^{V(G)}$ such that
\begin{itemize}
    \item $\bigcup_{t \in V(T)} \beta(t) = V(G)$,
    \item for every edge $uv \in E(G)$, there is a node $t \in V(T)$ such that $\beta(t)$ contains both $u$ and $v$.
    \item for every vertex $v \in V(G)$, the subgraph of $T$ induced by $\{t \in V(T) : v \in \beta(t)\}$ is connected.
\end{itemize}
The \emph{width} of $\mathcal{T}$ is the maximum value of $|\beta(t)| - 1$ over all $t \in V(T)$.
The \emph{treewidth} of $G$, denoted by $\tw(G)$, is the minimum width over all tree-decompositions of $G$.

For each $t \in V(T)$, we define the \emph{adhesion sets of $t$} as the sets in $\{\beta(t) \cap \beta(d) : d \in N_T(t)\}$, and the maximum size of them is called the \emph{adhesion of $t$}. The \emph{adhesion} of $\mathcal{T}$, denoted by $\mathsf{adhesion}(\mathcal{T})$, is the maximum adhesion of a node of $\mathcal{T}$.

\paragraph{Block decomposition.}
A classic example of a tree-decomposition for a graph is found in \emph{block decompositions}, which we will need later.
We call a maximal 2-connected subgraph of a graph a \emph{block}.
Note that each edge induces a 2-connected subgraph of a graph and thus some of the blocks of a graph may be single edges.

\begin{proposition}[folklore]\label{prop:blockdecomposition}
    Let $G$ be a graph, then there exists a tree-decomposition $(T,\beta)$ such that for each block $B$ of $G$ there exists a unique $t \in V(T)$ with $\beta(t) = V(B)$ and we have $st \in E(T)$ if and only if the blocks $B_s$ and $B_t$, corresponding to $s$ and $t$ respectively, intersect in a single vertex.
    Accordingly, $(T,\beta)$ has adhesion at most 1 and is called the \emph{block decomposition} of $G$.

    In particular, the block decomposition of a given graph $G$ can be found in $\mathbf{O}(|V(G)|^2)$.
\end{proposition}

\paragraph{Minors.}
Given a graph $G$ with an edge $e = uv$, the \emph{contraction} of $e$ is the operation of identifying $u$ and $v$ into a single vertex and subsequently deleting all loops and parallel edges.
Note that, if we find a subset of vertices $U \subseteq V(G)$ that induces a connected subgraph in $G$, we can contract this set into a single vertex via repeatedly contracting edges.
We say that $H$ is a \emph{minor} of $G$ and conversely we say that $G$ has an \emph{$H$-minor} if a graph isomorphic to $H$ can be obtained from a subgraph of $G$ by repeatedly contracting edges.
Given an edge set $F \subseteq E(G)$ in a graph $G$, we denote the result of contracting all edges in $F$ as $G / F$.

A function $\varphi \colon V(H) \rightarrow 2^{V(G)}$ is called a \emph{minor model (of $H$ in $G$)} (or simply a \emph{model}) if $\varphi(V(H))$ is a set of pairwise disjoint vertex sets, each inducing a connected subgraph of $G$, and for each edge $uv \in E(H)$ there exists an edge $ab \in E(G)$ with $a \in \varphi(u)$ and $b \in \varphi(v)$.
Given this definition, we observe that a graph $H$ is a minor of a graph $G$ if and only if there exists a minor model of $H$ in $G$.
We denote by $H_\varphi$ the subgraph of $G$ that is the union $\bigcup_{v \in V(H)}\varphi(v)$ together with the edges in $E(H)$.
We call $\varphi(v)$ the \emph{branch set} of $v$.

\paragraph{Subdivisions.}
Let $G$ be a graph, let $uv\in E(G)$, and let $P$ be a path with endpoints $u$ and $v$ such that no internal vertex of $P$ belongs to $G$.
We say that the graph $G'\coloneqq (G-uv)+P$ is obtained from $G$ by \emph{subdividing} the edge $uv$.
A graph $G''$ is called a \emph{subdivision} of $G$ if it can be constructed by subdividing a subset of edges of $G$.

It is easy to observe using the notion of a minor model that, if a graph $G$ contains a graph $H$ with $\Delta(H) \leq 3$ as a minor, then $G$ also contains $H$ as a subdivision.

Another notion of subdivision we are using is \emph{even subdivision}.
Let $F\subseteq E(G)$ be a minimal edge cut and for each $uv\in F$, let $P_{uv}$ be a path of length $2$ whose endpoints are $u$ and $v$.
We say that the graph $G'\coloneqq (G-F)+\bigcup_{uv\in F} P_{uv}$ is obtained from $G$ by \emph{evenly subdividing} $F$.
A graph $G'''$ is called an even subdivision of $G$ if it can be constructed from $G$ by iteratively replacing an edge by a path of odd length and evenly subdividing minimal edge cuts.

Analogous to the case for normal minor models, if $G$ contains $H$ with $\Delta(G) \leq 3$ as an odd minor, then $G$ also contains $H$ as an even subdivision.

\paragraph{Odd-minors.}\label{def:oddminors}
We say an $H$-minor model $\varphi$ in $G$ is \emph{odd} if there is a $2$-colouring $c$ of vertices in $H_\varphi$ in such a way that $c$ restricted to $\varphi(v)$ is a proper $2$-colouring for each $v \in V(H)$, and each $e \in E(H)$ is monochromatic.
We say $H$ is an \emph{odd-minor} of $G$, and write $H \oddminor G$, if there is an odd $H$-minor model in $G$.
Equivalently, a graph $H$ is an odd-minor of $G$ if $H$ can be obtained from $G$ by iteratively removing vertices, edges, and contracting minimal edge cuts.
It is easy to see that if $H$ is an even subdivision of $G$, then $G$ is an odd minor of $H$.

On the other hand, we say an $H$-minor model $\varphi$ in $G$ is \emph{bipartite} if $\bigcup_{v\in V(H)} \varphi(v)$ is bipartite.

\begin{observation}
    Let $G$ and $H$ be graphs such that $H\oddminor G$. Then we have $\ocp(H)\leq \ocp(G)$.
\end{observation}

\begin{observation}
    For graphs $G_1,G_2,G_3$, if $G_1\oddminor G_2$ and $G_2\oddminor G_3$, then we have $G_1\oddminor G_3$.
\end{observation}

\paragraph{Grids and walls.}
Grids and graphs will similar structure are a fundamental part of structural graph theory, especially when discussing obstructions for width-parameters.
Let $n,m \in \mathbb{N}$ be two positive integers.
The \emph{$(n \times m)$-grid} is the graph $G$ with the vertex set $V(G) = [n] \times [m]$ and the edges
\begin{align*}
E(G) = \big\{ \{ (i, j) , (\ell , k) \} ~\!\colon\!~    & i, \ell \in [n], \ j,k \in [m], \text{ and } \ \\
                                                        & ( |i - \ell| = 1  \text{ and } j = k ) \text{ or } ( |j - k| = 1 \text{ and } i = \ell ) \big\} .
\end{align*}
An \emph{$n$-grid} is an $(n \times n)$-grid.
The \emph{elementary $(n \times m)$-wall} is then derived from the $(n \times 2m)$-grid by deleting all edges in the following set
\begin{align*}
\big\{ \{ (i, j) , (i+1 , j) \} ~\!\colon\!~ i \in [n - 1], \ j \in [m], \text{ and } {i \not\equiv j \mod{2} } \big\}
\end{align*}
and removing all vertices of degree at most 1 in the resulting graph.
An \emph{$(n \times m)$-wall} is a subdivision of the elementary $(n \times m)$-wall.
An \emph{$n$-wall} is an $(n \times n)$-wall.

\paragraph{Parity grids.}
In the introduction, we defined two grid-like graphs, the \emph{parity handle} and \emph{parity vortex}, that will serve as the principal obstructions to having bounded \ocp-treewidth, if we exclude them as odd minors.
Here, we define several additional features of these graphs.

\begin{definition}\label{def:grid}
    The \emph{$(k,\ell)$-cylindrical grid}, denoted by $\mathcal{C}_{k,\ell}$, is the graph obtained from the Cartesian product of a cycle $C=(u_1,u_2,\cdots,u_{\ell})$ on $\ell$ vertices with a path $P=v_1,\cdots,v_k$ on $k$ vertices.
    We say the cycles $C\openbox \{v_i\}$ are \emph{concentric cycles} and the paths $\{u_j\}\openbox P$ are \emph{radial paths}.

    Also, $\mathcal{H}_k$ and $\mathcal{V}_k$ are defined by adding edges to $\mathcal{C}_{k,4k}$, so we can define the concentric cycles and radial paths of $\mathcal{H}_k$ and $\mathcal{V}_k$ by those of $\mathcal{C}_{k,4k}$ in each graph.
    
    Note that a $(k,\ell)$-cylindrical grid is bipartite for all $k$ and $\ell$, and each edge added in parity handles and parity vortices connects two vertices in a same bipartition.
    Therefore, we call these new edges as \emph{parity breaking edges}.
\end{definition}

It is easy to see that $\ocp(\mathscr{H}_{k})=\ocp(\mathscr{V}_{k})=k$, as each odd cycle in $\mathscr{H}_k$ or $\mathscr{V}_k$ must contain an odd number of parity breaking edges.

We also define a universal parity breaking grid $\mathscr{U}_k$.

\begin{definition}\label{def:universalparitybreakinggrid}
    The \emph{universal parity breaking grid of order $k$}, denoted by $\mathscr{U}_k$, is a graph obtained from the $( 2k \times (2k+1))$-grid by subdividing the edges between $(2i-1,k)$ and $(2i-1,k+1)$ for all $i\in[k]$ so that they become a path of length $2$.
\end{definition}

\begin{figure}[!ht]
    \centering
    \begin{tikzpicture}[scale=1]
    \pgfdeclarelayer{background}
    \pgfdeclarelayer{foreground}
    \pgfsetlayers{background,main,foreground}
    \begin{pgfonlayer}{background}
        \pgftext{\includegraphics[width=6cm]{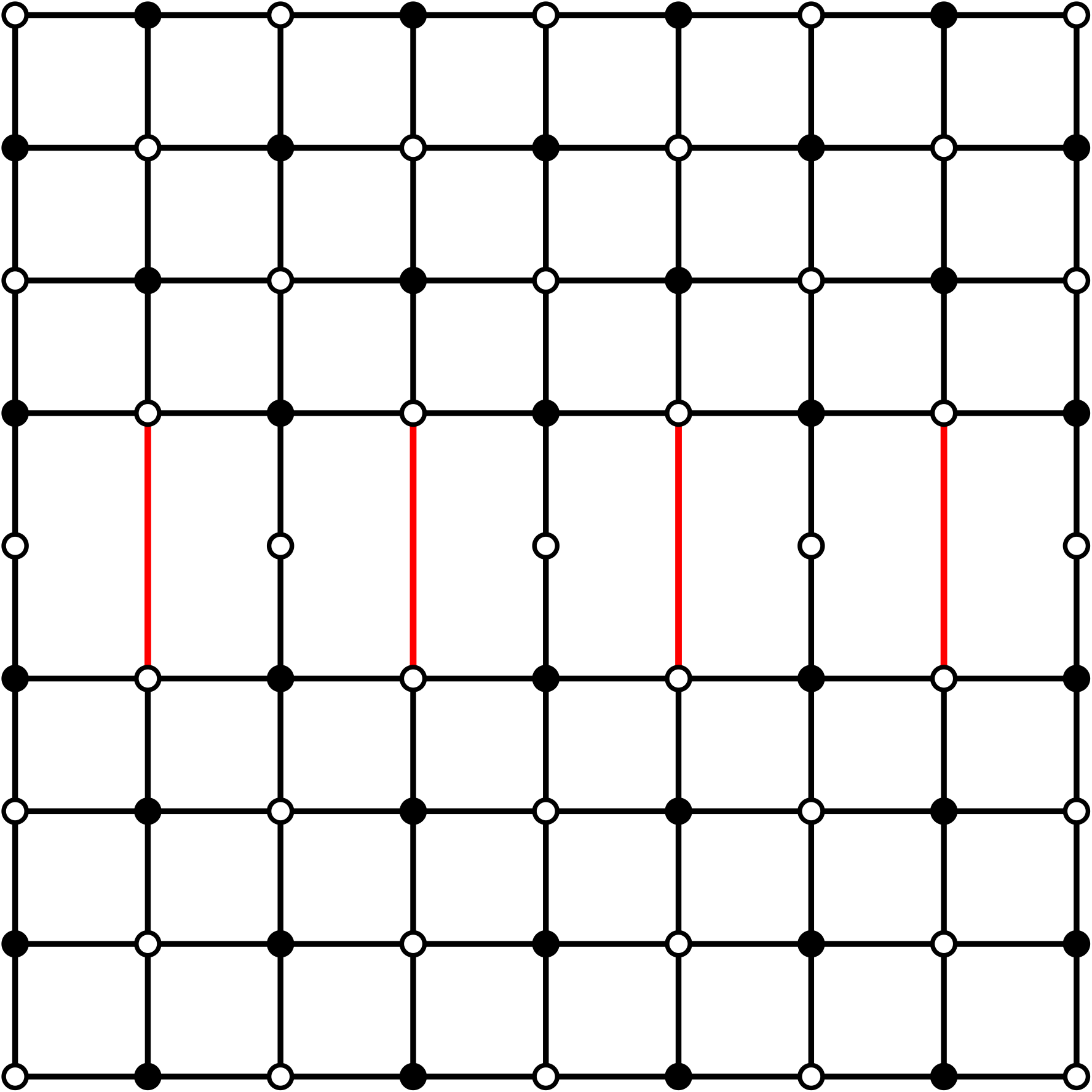}} at (C.center);
    \end{pgfonlayer}{background}
    \begin{pgfonlayer}{main}
    \end{pgfonlayer}{main}
    \begin{pgfonlayer}{foreground}
    \end{pgfonlayer}{foreground}
    \end{tikzpicture}
    \caption{The universal parity breaking grid of order $4$.
    The edges coloured by black form a bipartite graph while the red edges connect two vertices in a same bipartition.}
    \label{fig:universalparitybreakinggrid}
\end{figure}

We say the paths $(i,1)(i,2)\ldots(i,2k)$, or the path obtained by subdividing the edge between $(i,k)$ and $(i,k+1)$ if $i$ is odd, are the \emph{horizontal paths}, and we say the paths $(1,j)(2,j)\ldots(2k+1,j)$ are the \emph{vertical paths}.

It is easy to see that the universal parity breaking grid of order $3k$ contains the parity handle of order $k$ and the parity vortex of order $k$ as odd minor.
Furthermore, for each planar graph $H$, we can find a sufficiently large universal parity breaking grid containing $H$ as an odd minor using the methods Robertson, Seymour, and Thomas present to prove an analogous theorem for planar graphs found as minors of grids \cite{RobertsonST1994Quickly}.

\begin{theorem}\label{planarasoddminor}
    Let $H$ be a planar graph.
    Then there is some integer $k$ depending only on $H$ such that $\mathscr{U}_k$ contains $H$ as an odd minor.
\end{theorem}

By alternatively using the upper part and the lower part of $\mathscr{U}_{k^2}$, connected by the parity breaking edges, one can observe that $\mathscr{U}_{k^2}$ contains an odd minor isomorphic to a subdivision of a $(k \times k)$-grid whose non-outer faces are bounded by odd cycles.
As a special case of \zcref{planarasoddminor}, it is easy to see that we can use much simpler bounds if we are only interested in finding the parity handle and the parity vortex.
For later use, we state it here as a lemma.

\begin{lemma}\label{universalcontainsparityhandleandvortex}
    $\mathscr{U}_{3k}$ contains $\mathscr{H}_k$ and $\mathscr{V}_k$ as an odd minor.
\end{lemma}
\section{Odd-Cycle-Packing-treewidth}\label{sec:ocptw}
This section will be concerned with a set of fundamental results that confirm basic properties of \ocp-treewidth.
We start with an observation that solidifies our understanding of what the ``simple'' class is for our new parameter.

\begin{observation}\label{obs:ocptw0}
    For every graph $G$, it holds that $\ocptw(G) = 0$ if and only if $G$ is bipartite.
\end{observation}

Unsurprisingly, the treewidth and \ocp-treewidth of a graph are related in a simple fashion.

\begin{lemma}\label{lemma:tw>ocptw}
    For every graph $G$, it holds that $\ocptw(G) \leq \tw(G)$.
\end{lemma}
\begin{proof}
    Let $(T,\beta)$ be a tree-decomposition with width $\tw(G)$.
    Notice that we may assume $\mathsf{adhesion}(T,\beta)\leq\tw(G)$ since, if there are any two nodes $d,t\in V(T)$ such that $|\beta(d)\cap \beta(t)|=\tw(G)+1$, then, by the definition of tree-decompositions, there must exist two adjacent such nodes and, moreover $\beta(d) = \beta(t)$.
    So contracting the edge $dt$ in $T$ results in a tree-decomposition of minimum width for $G$ with strictly fewer nodes than $(T,\beta)$. 

    Now, for each $t\in V(T)$, choose a vertex $v_t \in \beta(t)$, and let $\alpha(t) = \beta(t) \setminus \{ v_t \}$.
    Given this $\alpha$, the \ocp-width of $(T,\beta,\alpha)$ is at most $\tw(G)+1$ as a single vertex does not constitute an odd cycle.
\end{proof}

As treewidth is closed under taking minors, we would expect that \ocp-treewidth would also be closed under odd minors, if it was defined sensibly, and this is indeed the case.

\begin{lemma}\label{lemma:OCPtwIsOddMinorMonotone}
    Let $H$ and $G$ be graphs such that $H \oddminor G$. Then we have $\ocptw(H) \leq \ocptw(G)$.
\end{lemma}
\begin{proof}
    Let $G$ be a graph and $\mathcal{T} = (T, \beta, \alpha)$ be an \ocptd of $G$.
    
    First, we show that the deletion of a vertex does not increase the \ocp-treewidth.
    Let $v$ be a vertex of $G$, and let $\mathcal{T}'$ be the \ocptd obtained from $\mathcal{T}$ by removing $v$ from all bags and sets $\alpha(t)$ containing $v$.
    Then $\mathcal{T}'$ is an \ocptd of $G-v$ and it is straightforward to see that $\ocpw(\mathcal{T}') \leq \ocpw(\mathcal{T})$.
    Hence, $\ocptw(G-v) \leq \ocptw(G)$.

    Similarly, we can show that the deletion of an edge does not increase the \ocp-treewidth.
    Let $e$ be an edge of $G$.
    Then $\mathcal{T}$ is still an \ocptd of $G-e$ and \ocp-width cannot have increased.
    Hence, $\ocptw(G-e) \leq \ocptw(G)$.
    
    Now, we show that contracting a minimal edge cut does not increase the \ocp-treewidth.
    Let $C$ be a minimal edge cut of $G$ and let $\phi \colon V(G / C) \rightarrow 2^{V(G)}$ be the minor model of $G / C$ in $G$.
    We construct an \ocptd $\mathcal{T}' = (T, \{ \beta'(t) \}_{t \in V(T)}, \{ \alpha'(t) \}_{t \in V(T)})$ of $G/C$, where $T$ is same tree as in $\mathcal{T}$, as follows.
    For each node $t$ of $T$, let $\beta'(t) = \{ v \in V(G/C) \mid \phi(v) \cap \beta(t) \neq \emptyset \}$ and $\alpha'(t) = \{ v \in V(G/C) \mid \phi(v) \cap \alpha(t) \neq \emptyset \}$.
    
    We claim that $\mathcal{T}'$ is an \ocptd of $G/C$ satisfying $\ocpw(\mathcal{T}')\leq \ocpw(\mathcal{T})$.
    The first and second condition of \zcref{def:ocptreewidth} are easy to check.
    To check the third condition, fix $t \in V(T)$.

    Suppose that there exists a path $P'$ in $G/C-\alpha'(t)$ such that both endpoints of $P'$ are in $\beta'(t)$.
    Let $V(P')=\{v_1,v_2,\cdots,v_\ell\}$, where indices are given according to the order of vertices in the path.
    Then $\phi(v_1)\cap \beta(t)\neq \emptyset$ and $\phi(v_1)\cap N(\phi(v_2))\neq \emptyset$, so we can find a path in $\phi(v_1)$ such that one of the endpoints is in $\beta(t)$ and the other is in $N(\phi(v_2))$.
    Note that $\phi(v_i)\cap \alpha(t)=\emptyset$ for all $i$.
    Similarly, for each $v_i$, we can find a path so that concatenating them gives a path $P$ in $G-\alpha(t)$ whose endpoints are in $\beta(t)$.
    This implies $V(P)\subseteq \beta(t)$, which gives $V(P')\subseteq \beta(t')$, which confirms the third condition.

    Next, we compare $\ocpw(\mathcal{T}')$ and $ \ocpw(\mathcal{T})$.
    It is straightforward to check that the adhesion of $(T,\{\beta'(t)\}_{t\in V(T)})$ is at most the adhesion of $(T,\{\beta(t)\}_{t\in V(T)})$.
    Also, for each $t\in V(T)$, we have $|\alpha(t)|\geq |\alpha'(t)|$ from the definition.
    Lastly, since $\beta'(t)-\alpha'(t)$ is an odd minor of $\beta(t)-\alpha(t)$, we have $\ocp(G[\beta(t)-\alpha(t)])\geq \ocp(G[\beta'(t)-\alpha'(t)])$.
    Therefore, we have $\ocpw(\mathcal{T}')\leq \ocpw(\mathcal{T})$, and consequently, this gives $\ocptw(G)\geq \ocptw(G/C)$.
    
    \medskip
    If $H$ is an odd minor of $G$, then $H$ can be obtained by a sequence of vertex deletions, edge deletions, and edge cut contraction.
    Therefore, we have $\ocptw(H)\leq \ocptw(G)$.
\end{proof}

Next, let us define a stronger version of \ocp-tree-decompositions that allows us to determine the width simply by considering the apex sets of each bag (and adding 1).
We say that an \ocp-tree-decomposition $(T,\beta,\alpha)$ of a graph $G$ is \emph{tame} if for all $dt\in E(T)$, $|(\beta(d)\cap\beta(t))\setminus\alpha(d)|\leq 1$.
The fact that we can find these is encapsulated in the next theorem (see also \zcref{thm:approxocptw}), which we will prove in \zcref{sec:localtoglobal}.

We mention this property in particular, since we will ultimately be able to prove that, in the absence of a parity handle and parity vortex, we can find such an \ocp-tree-decomposition with relatively modest width.
It will take us quite a while to get to the proof of this theorem, which will ultimately be a proven via a slightly stronger statement in \zcref{sec:localtoglobal}.

\begin{theorem}\label{thm:globalstructure}
    There exists a function $f_{\ref{thm:globalstructure}} \colon \mathbb{N} \to \mathbb{N}$ such that for every graph $G$ and positive integer $t$ one of the following holds:
    \begin{enumerate}
        \item $G$ has the parity handle of order $t$ as an odd minor,
        \item $G$ has the parity vortex of order $t$ as an odd minor, or
        \item $G$ has a tame \ocptd $(T,\beta,\alpha)$ of width at most $f_{\ref{thm:globalstructure}}(k)$ with $|V(T)| \in \mathbf{O}(|V(G)|)$.
    \end{enumerate}
    In particular, $f_{\ref{thm:globalstructure}}(t) \in \mathbf{O}(t^{3220})$ and there also exists an algorithm that, given $G$ and $k$ as above as input finds one of these outcomes in time $\max(f_{\ref{thm:ktminormodel}}(t),2^{\mathbf{poly}(t)})|V(G)|^6$.\footnote{$f_{\ref{thm:ktminormodel}}(t)$ is a computable function we inherit from a result discussed later which we use as a black box.}
\end{theorem}

\subsection{Computational hardness of \ocp-treewidth}\label{subsec:hardness}
We briefly discuss that computing the \ocp-treewidth of an arbitrary graph is $\mathsf{NP}$-complete.
This is shown via a reduction from the computation of the treewidth of a graph.

\begin{proposition}[Arnborg, Corneil, Proskurowski \cite{Arnborg1987Complexity}]\label{prop:twNPcomplete}
Computing treewidth is $\mathsf{NP}$-complete.
\end{proposition}

A well-known corollary of \zcref{prop:twNPcomplete} is that computing treewidth is $\mathsf{NP}$-complete even on bipartite graphs.
To see this simply notice that subdividing edges of a graph does not change its treewidth, and for all graphs $G$, the graph $G'$ obtained from $G$ by subdividing every edge of $G$ exactly once is always bipartite. 

\begin{corollary}[folklore]\label{cor:biptw}
Computing treewidth is $\mathsf{NP}$-complete on bipartite graphs.
\end{corollary}

Let $G$ be a bipartite graph.
We define a specific augmentation of $G$ as follows.
For every edge $e=xy\in E(G)$ introduce a new vertex $v_e$ adjacent exactly to $x$ and $y$.
We denote the resulting graph by $G^{\Delta}$.

It is easy to observe that, if $G$ is not a forest, then $\tw(G)=\tw(G^{\Delta})$.

\begin{lemma}\label{lemma:GDelta}
Let $G$ be a bipartite graph that is not a forest, then $\ocptw(G^{\Delta})=\tw(G)$.
\end{lemma}

\begin{proof}
By \zcref{lemma:tw>ocptw} we already have $\ocptw(G^{\Delta}) \leq \tw(G^{\Delta}) = \tw(G)$.
Thus, we only have to show that $\tw(G) \leq \ocptw(G^{\Delta})$.

Let $(T,\beta,\alpha)$ be an \ocp-tree-decomposition of minimum width $k$ for $G^{\Delta}$.
We may assume that for all $dt\in E(T)$, we have $\beta(d)\cap V(G)\neq \beta(t)\cap V(G)$.
Moreover, let us select an arbitrary vertex $r\in V(T)$ as a root for $T$ such that such that $(\beta(r)\setminus\beta(d))\cap V(G)\neq\emptyset$ for all $d\in N_T(r)$.
To see that this is possible for any choice of $r$ observe the following:
Suppose that $\beta(r)\cap V(G)\subseteq \beta(d)\cap V(G)$ for some $d\in N_T(r)$.
Let $x\in V(G^{\Delta})\cap \beta(r)$ be any vertex, then one of three options are possible:
\begin{itemize}
    \item $x$ appears only in $\beta(r)$,
    \item $x$ appears in $\beta(d)$, or
    \item $x$ does not appear in $\beta(d)$ but in $\beta(t)$ for some $t\in N_T(r)\setminus\{ d\}$.
\end{itemize}
If $x$ appears only in $\beta(r)$, then we must have $x = v_e \in V(G^\Delta) \setminus V(G)$ and by the axioms of tree-decompositions, both endpoints of $e$ are contained in $\beta(r)\cap V(G)\subseteq \beta(d)$.
In this case, $\alpha(d)$ must contain one of the endpoints of $e$ as otherwise \ref{ocp3} would require $x\in\beta(d)$ as well.
For every $t\in N_T(r)\setminus\{ d\}$ we have that $|\beta(d)\cap \beta(t)|\leq k$, thanks to our assumptions on $(T,\beta,\alpha)$.
In the third option, it must therefore be true that $|\beta(d)\cap \beta(t)|\leq k-h$, where $h$ is the number of such vertices $x\in \beta(r)$ that appear in $\beta(t)$ but not in $\beta(d)$.
Hence, by merging the bags $\beta(r)$ and $\beta(d)$ and keeping $\alpha(d)$ as the apex set, we do not increase the adhesion of the resulting tree-decomposition and, moreover, the odd cycle packing number of $G^{\Delta}(\beta(r)\cup\beta(d))-\alpha(d)$ is the same as the one of $G^{\Delta}(\beta(d))-\alpha(d)$.
Hence, now we may choose $d$ to be the root and keep repeating this process until either $T$ has a single vertex, or we have found a root $r$ as desired where the new bag of $r$ fully contains the bag of the originally chosen root.

We now describe how to transform $(T,\beta,\alpha)$ into a tree-decomposition of $G$ of bounded width.

Consider the graph $G_r\coloneqq G[\beta(r)]-\alpha(r)$.
Suppose there exists an edge $xy\in E(G_r)\cap E(G)$, then, by \ref{ocp3}, we know that $v_{xy}\in \beta(r)$.
Let $G'_r \coloneqq G_r\cap G$.

We claim that a maximum matching in $G'_r$ has size at most $k$ and, in particular, if $M$ is a maximum matching of $G'_r$ then for all $xy\in M$ except for at most $k-\alpha(r)$ we have that $v_{xy}\in\alpha(r)$.
To see this let $M$ be a maximum matching in $G'_r$.
By our observation above we have that for all $xy\in M$ we must also have that $v_{xy}\in \beta(r)$.
Moreover, since $\ocp(G_r)\leq k-|\alpha(r)|$ there can be at most $k-|\alpha(r)|$ edges $xy$ of $M$ such that $v_{xy}\in V(G_r)$.
This is because $\{ x,y,v_{xy}\}$ induces an odd cycle in $G_r$ otherwise and all of these triangles are vertex disjoint since $M$ is a matching.
It follows that $|M|\leq k$ as claimed.
From here on let $A'\coloneqq \alpha(r)\cap V(G)$ and notice that $|A'|+|M|\leq k$ due to our claim.

Since $G_r'$ is bipartite, K\H{o}nig's Theorem implies the existence of a set $S\subseteq V(G'_r)$ with $|S|=|M|$ such that $S$ is a vertex cover of $G'_r$.
Hence, $G_r'-S$ is an independent set.
Let $A\coloneqq A'\cup S$.

We need one more observation before we can start constructing our desired tree-decomposition.
Let $P$ be a path in $G-\alpha(r)$ with both endpoints in $\beta(r)$.
Then, by \zcref{ocp3} we know that $V(P)\subseteq \beta(r)$.
Consequently, no component of $G-A$ can contain more than one vertex of $I\coloneqq V(G_r'-S)$.
For each $v\in I$ let $H_v'$ be the component of $G-A$ that contains $v$.
Let $H_1',\dots,H_{\ell}'$ be the components of the graph induced by $A$ and the vertex sets of all components of $G-A$ that are disjoint from $I$, and for each $i \in [\ell]$ let $H_i\coloneqq G[V(H_i')\cup A]$.
Moreover, for each $v\in I$, let $H_v$ be the subgraph of $G$ induced by $V(H_v')\cup A$.

To indicate that our refinement procedure for $(T,\beta,\alpha)$ takes only finitely many steps and can thereby be handled by an induction, we distinguish two cases.

Our induction establishes the following:
For graphs $H \subseteq G$ where $|\beta(r)\cap V(H)|\leq k+1$ we prove by induction on $|V(H)|$ that $H$ has a tree-decomposition $(T_H,\beta_H)$ of width at most $k$ such that there exists $r_H\in V(T_H)$ with $\beta(r)\cap V(H)\subseteq \beta(r_H)$ and if $|\beta(r)\cap V(H)|> k+1$ there exists a tree-decomposition $(T_H,\beta_H)$ of width at most $k$ such that there exists $r_H\in V(T_H)$ with $\alpha(r)\cap V(H)\subseteq \beta(r_H)$.
In the following we argue for the case in which $H = G$, with the remainder following by induction.
Let $\mathcal{H}\coloneqq \{ H_1,\dots,H_{\ell} \} \cup \{ H_v \colon v\in I \}$.

Notice that the base case is given by $|V(H)|\leq k+1$ in which case we tree-decomposition whose tree has a single node suffices completely.
\medskip

\textbf{Case 1:}
There are at least two graphs in $\mathcal{H}$.
\smallskip

In this case we have that $|V(H)| < |V(G)|$ for all $H\in \mathcal{H}$ and $|V(H)\cap\beta(r)|\leq k+1$.
Indeed, since $\mathsf{adhesion}(T,\beta)\leq k$ we obtain that either $|V(H)\cap\beta(r)|\leq k$, we have $V(H)\subseteq\beta(r)$, or there exists $a\in A$ such that $N_G(a)\cap V(H)\subseteq \beta(r)$.
As $|V(H)|<|V(G)|$ for all $H\in\mathcal{H}$, we may assume by induction that each $H\in\mathcal{H}$ has a tree-decomposition $(T_H,\beta_H)$ of width at most $k$ with root $r_H\in V(T_H)$ such that $ V(H) \cap \beta(r) \subseteq \beta_H(r_H)$.

We now construct $(T_G,\beta_G)$ as follows.
Let $T_G$ be obtained from the disjoint union of all $T_H$ by first introducing, for each $H\in\mathcal{H}$, a vertex $t_H$ adjacent to $r_H$, and then a vertex $r_G$ adjacent to all $t_H$.
Then we set $\beta_G(r_G)\coloneqq A \supseteq \alpha(r)\cap V(G)$, for each $i\in[\ell]$ we set $\beta_G(t_{H_i})\coloneqq A$ and for each $v\in I$ we set $\beta_G(t_{H_v})\coloneqq A\cup \{ v \}$.
Next, for each $H\in\mathcal{H}$ and $t\in V(T_H)$, we set $\beta_G(t)\coloneqq \beta_H(t)$.
Finally, in case $|\beta(r)\cap V(G)|\leq k+1$, we know that $|(\beta(r)\cap V(G))\setminus A|\leq k+1 - |A|$, since $A = A' \cup S \subseteq \beta(r)\cap V(G)$.
In this special case we add $\beta(r)\cap V(G)$ to the bag of $r_G$.

The result is a pair $(T_G,\beta_G)$ such that $|\beta_G(t)|\leq k+1$ for all $t\in V(T_G)$.
Moreover, $\beta_G(r_G) \setminus A \subseteq \beta_G(t_H)\subseteq \beta_G(r_H)$ and in case $\beta_G(r_G) \neq A$, the set $\beta_G(r_G) \setminus A$ is an independent set in $G$.
Since $v \in \beta_G(t_{H_v})$ for each $i \in I$, we therefore know that every edge of $G$ is contained in at least one bag of $(T_G,\beta_G)$ and each vertex of $G$ induces a connected subtree of $T_G$ via $\beta_G$.
Hence, $(T_G,\beta_G)$ is the desired tree-decomposition of width at most $k$ since either we have $|\beta(r)\cap V(G)|\leq k+1$ and $\beta(r)\cap V(G)\subseteq \beta(r_G)$ or, in all other cases, we have that $\alpha(r)\cap V(G)\subseteq \beta(r_G)$.
\medskip

\textbf{Case 2:}
There is exactly one graph in $\mathcal{H}$.
\smallskip

It follows that $r$ has a unique neighbour in $T$, say $d$.

By our assumption we have that $(\beta(r)\setminus \beta(d))\cap V(G)\neq\emptyset$.
Note that for each $v \in I \subseteq (\beta(r)\cap V(G))\setminus A$ there exists a graph in $\mathcal{H}$.
Therefore, in this case, $(\beta(r)\cap V(G))\setminus A$ contains at most one vertex and thus, $|\beta(r)\cap V(G)|\leq k+1$.

Now pick any vertex $x\in(\beta(r)\setminus\beta(d))\cap V(G)$ and consider the graph $(G-x)^{\Delta}$ together with the \ocp-tree-decomposition $(T,\beta,\alpha)$ obtained by restricting to $(G-x)^{\Delta}$.

Let now $d\neq t\in V(T)$ be two distinct nodes of $T$ such that $\beta(d)\cap V(G-x)=\beta(t)\cap V(G-x)$.
Notice that any cycle containing a vertex $v\in V(G^{\Delta})\setminus V(G)$ must also contain both neighbours of $v$.
Hence, without loss of generality, we may assume that $\alpha(d),\alpha(t)\subseteq V(G-x)$.
Next observe that, by the axioms of tree decompositions, we must have that there exists a node $h\in N_T(d)$ which lies on the unique $d$-$t$-path in $T$ and $\alpha(d)\cap V(G-x)\subseteq \beta(h)$ since $\alpha(d)\cap V(G-x)=\alpha(t)\cap V(G-x)$ by assumption.
It follows that $|\beta(d)\cap V(G-x)|\leq k$ since the adhesion of $(T,\beta)$ is at most $k$.
Hence, we may assume without loss of generality that $\beta(d)\cap V(G) \subseteq \alpha(d)=\alpha(t)\supseteq \beta(t)\cap V(G)$.

By merging adjacent bags that coincide in their vertices of $G-x$ we may therefore assume, as before, that no two adjacent bags intersect $V(G-x)$ in the same set.
Indeed, notice that this also does not increase the adhesion of $(T,\beta)$ above $k$ since any vertex of $V(G^{\Delta})\setminus V(G)$ that lives in exactly one of these two adjacent bags $\beta(d)$ and $\beta(t)$ must either live in both, or, by the axioms of tree-decompositions, can only occur in bags of nodes of one of the two components of $T-dt$.
Moreover, we may assume $r$ to still be a node of $T$.
So, by our induction hypothesis, $G-x$ has a tree-decomposition $(T_{G-x},\beta_{G-x})$ of width at most $k$ with a root $r_{G-x}$ such that $\beta(r)\cap V(G-x)\subseteq \beta_{G-x}(r_{G-x})$.
Since, by choice of $x$, we know that all neighbours of $x$ belong to $\beta(r)\cap V(G)$ we may now simply introduce a new node $r_G$ adjacent to $r_{G-x}$ to $T_{G-x}$, thereby creating the tree $T_G$.
If we then set $\beta_G(r_G)\coloneqq \beta(r)\cap V(G)$ and $\beta_G(t)\coloneqq \beta_{G-x}(t)$ for all $t\in V(T_{G-x})$ we have obtained our desired tree-decomposition $(T_G,\beta_G)$ for $G$.
\end{proof}

Notice that $G^{\Delta}$ can be constructed from any bipartite graph $G$ in time $\mathbf{O}(|E(G)|)$ and $|V(G^{\Delta})|=|V(G)|+|E(G)|$.
Hence, the $\mathsf{NP}$-completeness of computing the \ocp-treewidth of graphs is an immediate corollary of \zcref{lemma:GDelta} and \zcref{cor:biptw}.
With this, the proof of \zcref{thm:OCPtwNP} is complete.

\subsection{Parity grids have large \ocp-treewidth}
In this subsection, we show that both the parity handle and parity vortex have large \ocp-treewidth.
First, observe that if there is a separation of small order in a cylindrical grid, then one side of the separation contains most of the vertices.

\begin{observation}\label{lem:separationingrid}
    Let $(S_1,S_2)$ be a separation of order $s$ in the cylindrical grid $\mathcal{C}_{k,\ell}$ with $k,\ell \geq s+1$.
    Then there is a unique $i \in \{1,2\}$ with $|S_i| \leq s^2/4$.
\end{observation}

\begin{theorem}\label{lem:gridshavelargeocptw}
    For every positive integer $k$, we have $\ocptw(\mathscr{H}_{2k})\geq k$ and $\ocptw(\mathscr{V}_{2k})\geq k$.
\end{theorem}
\begin{proof}
    Suppose that the parity handle $\mathscr{H}_{2k}$ has a \ocptd $(T,\beta,\alpha)$ of \ocp-width $s$, where $s<k$.
    To reach a contradiction, we show that there is a bag containing an even subdivision of $\mathscr{H}_k$ as a subgraph.
    
    For $t_1t_2\in E(T)$, let $T_1$ and $T_2$ be components of $T-t_1t_2$ containing $t_1$ and $t_2$ respectively.
    As the adhesion of $(T,\beta)$ is at most $s$, $( \bigcup_{t\in V(T_1)}\beta(t), \bigcup_{t\in V(T_2)}\beta(t) )$ is a separation of order at most $s$.
    The separation $( \bigcup_{t\in V(T_1)}\beta(t), \bigcup_{t\in V(T_2)}\beta(t) )$ is also a separation of the cylindrical grid $\mathcal{C}_{2k,8k}$, since $V(\mathscr{H}_{2k}) = V(\mathcal{C}_{2k,8k})$.
    Thus, by \ref{lem:separationingrid}, there is a unique $i \in \{ 1,2 \}$ such that $|\bigcup_{t\in V(t_i)}\beta(t)| \leq s^2/4$.
    Therefore, we can define an orientation $\vec{T}$ of the edges of $T$ such that $t_1t_2 \in \vec{T}$ if and only if $|\bigcup_{t\in V(t_1)}\beta(t)| \leq s^2/4$.
    Since $T$ is a tree, there is a node $t$ of $\vec{T}$ that has outdegree $0$.
    
    For each parity breaking edge $e$, suppose that $e$ shares an endpoint with $i$-th and $j$-th radial paths.
    Let $Q_e$ be a vertex set that contains vertices in $i,i+1,j,j+1$-th radial paths.
    Since $|\alpha(t)| \leq k'$, there are at least $k$ concentric cycles that does not intersect with $\alpha(t)$, and there are at least $k$ sets $Q_e$ that do not intersect with $\alpha(t)$.
    By taking union of them, we can find an even subdivision of $\mathscr{H}_k$ as a subgraph, say $H$, in $\mathscr{H}_{2k} - \alpha(t)$.

    By the definition of $t$, we have $|\beta(t)\cap V(H)|\geq 2$.
    Since $H$ is $2$-connected, for each $v\in V(H)$, there is a path in $H$ whose endpoints are in $\beta(t)\cap V(H)$.
    By the definition of \ocptd, this implies $V(H)\subseteq \beta(t)$.
    However, this leads to a contradiction since $\ocp(H) = \ocp(\mathscr{H}_{k})=k>s$.
    Therefore, we conclude that $\ocptw(\mathscr{H}_{2k}) \geq k$.

    We omit the arguments for $\mathscr{V}_{2k}$ as they proceed along very similar lines.
\end{proof}

\subsection{Brambles}
Next, we define a bramble notion for \ocp-treewidth, which gives us a helpful conceptualisation of obstructions for having low \ocp-treewidth.
We note that, when comparing our concept to brambles for treewidth, our definition most closely matches that of a \textsl{strong} bramble.

\begin{definition}\label{def:ocpbramble}
    Let $G$ be a graph.
    An \emph{\ocp-bramble of order $k$} is a set $\mathcal{B}\subseteq 2^{V(G)}$ satisfying following properties:
    \begin{enumerate}
        \item for each $B_1,B_2 \in \mathcal{B}$, they intersect each other, i.e. $B_1 \cap B_2 \neq \emptyset$,
        \item for each $B \in \mathcal{B}$, $G[B]$ is a $2$-connected subgraph with $\ocp(G[B]) \geq k$, and
        \item if a set $S \subseteq V(G)$ satisfies $S \cap B \neq \emptyset$ for all $B \in \mathcal{B}$, we have $|S| \geq k$. 
    \end{enumerate}
\end{definition}

As expected of a bramble, there exists a duality theorem between \ocp-treewidth and \ocp-brambles.

\begin{theorem}\label{thm:brambleduality}
    Let $G$ be a graph.
    Then one of the following options holds for every positive integer $k$.
    \begin{enumerate}
        \item If $G$ has an \ocp-bramble of order $k$, then $\ocptw(G)\geq k$.
        \item If $\ocptw(G)> f_{\ref{thm:globalstructure}}(k^2)$, then $G$ has an \ocp-bramble of order at least $k$.
    \end{enumerate}
    In particular, given $G$ with $\ocptw(G)>f_{\ref{thm:globalstructure}}(k^2)$, we can find an \ocp-bramble of order at least $k$ in time $\max(f_{\ref{thm:ktminormodel}}(k^2),2^{\mathbf{poly}(k)})|V(G)|^6$.
\end{theorem}
\begin{proof}
    First, let $G$ be a graph that has an \ocp-bramble $\mathcal{B}$ of order $k$ and an \ocptd $\mathcal{T}=(T,\beta,\alpha)$ of \ocp-width $k'$ where $k'<k$.
    We define an orientation of $T$ as follows.
    
    For $t_1t_2\in E(T)$, let $T_1$ and $T_2$ be components of $T-t_1t_2$ containing $t_1$ and $t_2$ respectively. 
    As $|\beta(t_1)\cap \beta(t_2)|<k$, there is a bramble element $R\in \mathcal{B}$ such that $R\cap (\beta(t_1)\cap \beta(t_2))=\emptyset$.
    Since $R$ induces a connected subgraph and $\beta(t_1)\cap \beta(t_2)$ is a separator in $G$, there is a unique $i\in \{1,2\}$ such that $R\subseteq \bigcup_{s\in V(T_i)} \beta(s)\setminus\beta(t_{3-i})$.
    Furthermore, if $R'\in \mathcal{B}$ is another bramble element such that $R'\cap (\beta(t_1)\cap \beta(t_2))=\emptyset$, then $R\cap R'\neq \emptyset$ implies that $R'\subseteq \bigcup_{s\in V(T_i)} \beta(s)\setminus \beta(t_{3-i})$.
    Therefore, we can define an orientation $\vec{T}$ such that $(t_1,t_2)\in \vec{T}$ if and only if there is a bramble element $R\in \mathcal{B}$ satisfying $R\subseteq \bigcup_{s\in V(T_1)} \beta(s)-\beta(t_{2})$.

    Since $T$ is a tree, there is a node $t_3$ of $\vec{T}$ that has outdegree $0$.
    Because $|\alpha(t)|\leq k'<k$, there is a bramble element $R\in \mathcal{B}$ such that $R\cap \alpha(t)=\emptyset$.
    Then $\ocp(G[R])\geq k$ and $\ocp(G[\beta(t)-\alpha(t)])<k$ implies $R\not \subseteq \beta(t)$.
    Therefore, there is an edge $t_3t_4\in E(T)$ such that $R\cap (\bigcup_{s\in V(T_4)} \beta(s)\setminus \beta(t_3))\neq\emptyset$.
    By the definition of $t_3$, there is a bramble element $R'\in \mathcal{B}$ such that $R'\subseteq \bigcup_{s\in V(T_3)}\setminus \beta(t_4)$.
    Since $R\cap R'\neq \emptyset$ and $R$ is $2$-connected, $|R'\cap (\beta(t_3)\cap \beta(t_4))|\geq 2$.
    Again since $R$ is $2$-connected, for every vertex $r\in R$, there is a path whose endpoints are in $\beta(t_3)\cap \beta(t_4)$ that contains $r$.
    This gives $V(R)\subseteq \beta(t_3)$, which is a contradiction.
    Therefore, we conclude that $\ocptw(G)\geq k$.

    \medskip
    Now suppose that $G$ is a graph with $\ocptw(G)> f_{\ref{thm:globalstructure}}(k^2)$.
    Then, by \zcref{thm:globalstructure}, we can find either $\mathscr{H}_{k^2}$ or $\mathscr{V}_{k^2}$ as an odd minor in time $\max(f_{\ref{thm:ktminormodel}}(k^2),2^{\mathbf{poly}(k)})|V(G)|^6$.

    First suppose that $G$ contains $\mathscr{H}_{k^2}$ as an odd minor.
    Index the vertices of $\mathscr{H}_{k^2}$ by $\{v_i^j\}_{i\in [4k^2],\, j\in [k^2]}$ so that $\{v_i^j\}_{i\in [4k^2]}$ lies on the $j$-th concentric cycle in cyclic order, $\{v_i^j\}_{j\in [k^2]}$ lies on the $i$-th radial path, and for each $i\in [k^2]$, the edge $v_{2i}^{1}v_{4k^2+2-2i}^{1}$ is a parity breaking edge. 
    Let $W'$ be a graph obtained from $\mathscr{H}_{k^2}$ by deleting edges as following: (see \zcref{fig:paritywall})
    $$E(W')=E(\mathscr{H}_{k^2})\setminus \{v_i^jv_i^{j+1}\mid i\in [4k^2],\, j\in [k^2-1],\, i+j\mbox{ is odd} \}$$
    Then $W'$ has maximum degree $3$.
    
    For each $\ell\in [k^2]$, define a path $Q_\ell'$ by 
    \begin{align*}
        E(Q_\ell')=&\{ v^1_{2\ell} v^1_{4k^2+2-2\ell} \}\cup\{v_{2\ell-1}^jv_{2\ell}^j,v^j_{4k^2+1-2\ell}v^j_{4k^2+2-2\ell}\mid j\in [k^2]\}\\
        &\cup \{ v^j_iv^{j+1}_i \mid i\in \{2\ell-1,2\ell,4k^2+1-2\ell,4k^2+2-2\ell\},\, j\in [k^2-1],\,i+j\mbox{ is even} \}
    \end{align*}
    so that each $Q_\ell'$ is a path that contains one parity breaking edge (see \zcref{fig:paritywall}).

\begin{figure}
    \centering
    \begin{tikzpicture}[scale=1]
    \pgfdeclarelayer{background}
    \pgfdeclarelayer{foreground}
    \pgfsetlayers{background,main,foreground}
    \begin{pgfonlayer}{background}
        \pgftext{\includegraphics[width=6cm]{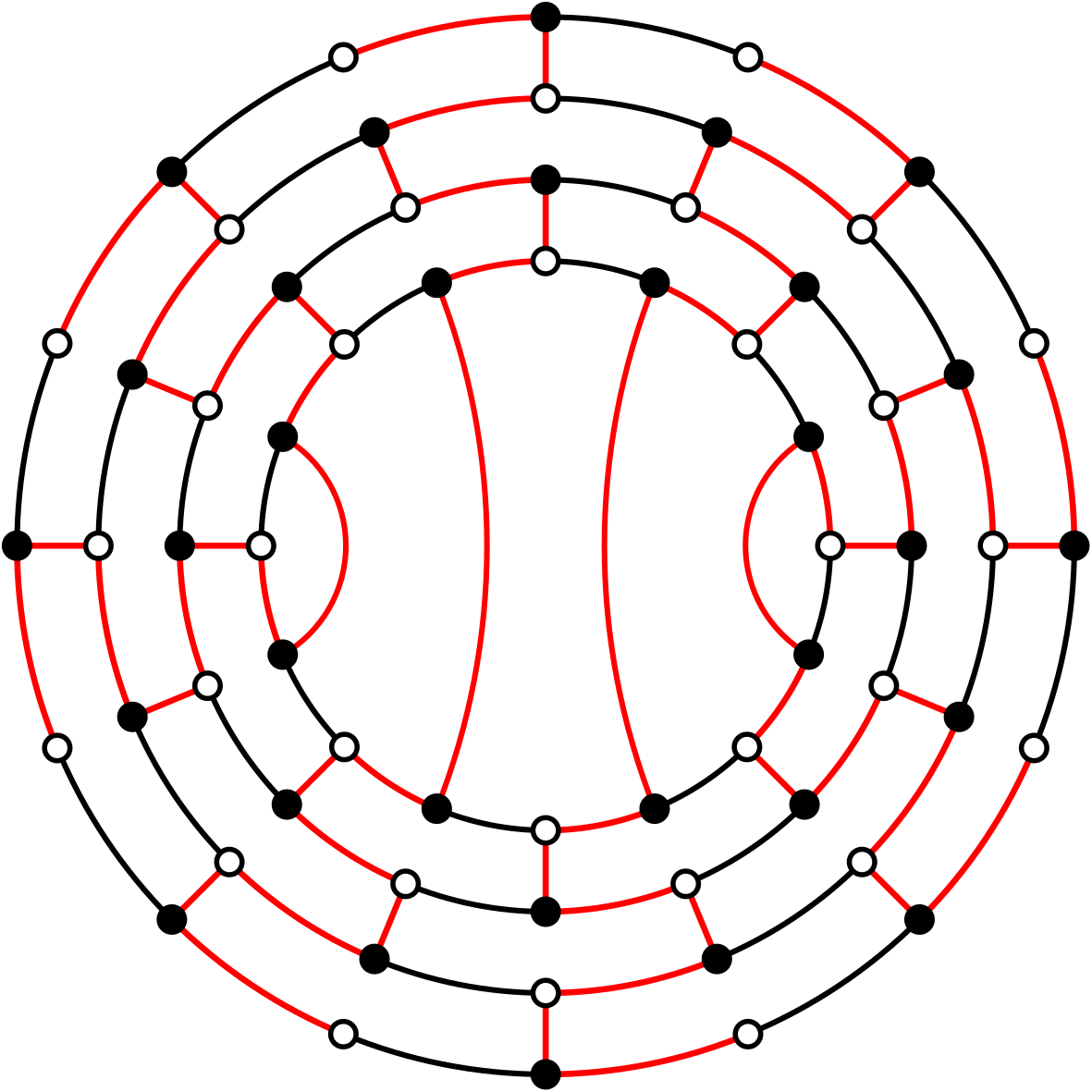}} at (C.center);
    \end{pgfonlayer}{background}
    \begin{pgfonlayer}{main}
    \end{pgfonlayer}{main}
    \begin{pgfonlayer}{foreground}
    \end{pgfonlayer}{foreground}
    \end{tikzpicture}
    \caption{We remove edges of $\mathscr{H}_{k^2}$ to obtain $W'$ whose maximum degree is $3$.
    The paths $Q_\ell'$ are coloured red.}
    \label{fig:paritywall}
\end{figure}

    $G$ contains $\mathscr{H}_{k^2}$ as an odd minor, so $G$ also contains $W'$ as an odd minor.
    Since the maximum degree of $W'$ is $3$, $G$ contains an even subdivision of $W'$ as a subgraph, say $W$.
    Then we can find cycles $C_1,\cdots C_{k^2}$ that corresponds to the subdivision of the first to $k^2$-th concentric cycles of $W'$, and paths $Q_1,\cdots,Q_{k^2}$ that corresponds to the subdivision of $Q_1',\cdots,Q_{k^2}'$.

    Let $\mathcal{B}=\{B_i^j\}_{i\in [k],\,j\in[k]}$ such that each $B_i^j$ is defined as the block of $G[\bigcup_{a\in [k]} V(C_{(i-1)k+a})\cup \bigcup_{b\in [k]}V(Q_{(j-1)k+b})]$ that contains $C_{ik}$.
    We claim that $\mathcal{B}$ is an \ocp-bramble of order $k$.

    First, as each $Q_\ell'$ meets with each concentric cycles, each $Q_\ell$ also meets with each $C_i$.
    Hence if $i\geq i'$, we have $B_{i}^j\cap B_{i'}^{j'}\supseteq V(C_{i'k}) \cap V(Q_{jk})\neq \emptyset $, so each bramble elements intersect each other.
    By the definition of $B_i^j$, they are $2$-connected.
    Lastly, let $H\subseteq V(G)$ be as set of vertices with $|H|<k$.
    Each $C_i$'s are disjoint from each other, so there is some $p\in [k]$ so that $H\cap \bigcup_{a\in [k]} V(C_{(p-1)k+a}) =\emptyset$.
    Also, each $Q_\ell$'s are disjoint from each other, so there is some $q\in [k]$ so that $H\cap \bigcup_{b\in [k]} V(Q_{(q-1)k+b}) =\emptyset$.
    Then $H\cap B_p^q=\emptyset$.
    Therefore, $G$ has an \ocp-bramble of order $k$.

    If $G$ contains $\mathscr{V}_{k^2}$ as an odd minor, we can also similarly obtain an \ocp-bramble of order $k$.
\end{proof}

\section{Dynamic programming on \ocp-tree-decompositions}\label{sec:dynamicprogramming}

\newcommand{\w}{\mathrm{W}}

We solve \mis problem on vertex-weighted graphs $(G,w)$, where $G$ is of bounded \ocp-treewidth.
A \emph{vertex-weighted graph} is a pair $(G,w)$ consisting of a graph $G$ and a function $w: V(G) \rightarrow \mathbb{N}$.  

\begin{theorem}\label{thm:fptalgoformis} (\cite{Fiorini2025Integer})
    There exists a computable function $f\colon \mathbb{N} \rightarrow \mathbb{N}$ such that \mis can be solved in time $|V(G)|^{f(\ocp(G))}$.
\end{theorem}

First, we note that \zcref{thm:globalstructure} allows us to approximate the \ocp-treewidth of a graph and in particular find a tame \ocptd, if the \ocp-treewidth of the graph is low.

\begin{theorem}\label{thm:approxocptw}
    Let $G$ be a graph.
    There exists an algorithm that either determines that $\ocptw(G) > k$ or computes a tame \ocptd $(T,\beta,\alpha)$ for $G$ of width at most $f_{\ref{thm:globalstructure}}(k)$ and with $|V(T)| \in \mathbf{O}(|V(G)|)$ in time $\max(f_{\ref{thm:ktminormodel}}(t),2^{\mathbf{poly}(t)})|V(G)|^6$.
\end{theorem}
\begin{proof}
    Suppose first that we have $\mathscr{H}_{2k}$ or $\mathscr{V}_{2k}$ as an odd minor in $G$.
    Then, according to \zcref{lemma:OCPtwIsOddMinorMonotone}, the \ocp-treewidth of $G$ is bounded from below by the \ocp-treewidth of these two graphs, which we know to be at least $k$, according to \zcref{lem:gridshavelargeocptw}.
    Thus, we have $\ocptw(G) > k$.

    Therefore, when applying \zcref{thm:globalstructure}, we either find a a witness that tells us that the \ocp-treewidth of $G$ is at least $k$, or we find a tame \ocptd $(T,\beta,\alpha)$ for $G$ with width at most $f_{\ref{thm:globalstructure}}(k)$ and $|V(T)| \in \mathbf{O}(|V(G)|)$ in $\max(f_{\ref{thm:ktminormodel}}(t),2^{\mathbf{poly}(t)})|V(G)|^6$-time.
\end{proof}

We can then immediately use following theorem to solve maximum weight independent set in weighted graphs, provided that their \ocp-treewidth is low.

\begin{theorem}\label{thm:miswithboundedocptw}
    Let $k$ be a positive integer.
    There exists a computable function $f$ and an algorithm that takes as input a vertex-weighted graph $(G, w)$ with $\sum_{v \in V(G)} w(v) \in \mathbf{O}(|V(G)|)$ and either correctly determines that $\ocptw(G)>k$ or computes a maximum weight independent set of $(G,w)$ in time $|V(G)|^{f(k)}$.
\end{theorem}

A $4$-tuple $\mathcal{T}=(T,r,\beta, \alpha)$ is a \emph{rooted \ocptd} for a graph $G$ if  $(T,\beta, \alpha)$ is an \ocptd, and $r \in V(T)$.
We say $r$ is the \emph{root} of $\mathcal{T}$.
For $t \in V(T)$, we denote by $\mathcal{T}_t$ the rooted \ocptd $(T_t, t, \beta_t, \alpha_t)$, where $T_t$ is the subtree of $T$ rooted at $t$, and $\beta_t$, $\alpha_t$ are the restriction of $\beta$ and $\alpha$ to $T_t$, respectively.
Furthermore, we denote by $(G_t, w)$ the vertex-weighted subgraph of $(G,w)$ induced by $\beta(t)$ and by $(G_{T_t}, w)$ the vertex-weighted subgraph of $(G,w)$ induced by the vertex set $\bigcup_{d \in V(T_t)} \beta(d)$. 

\begin{proof}[Proof of \zcref{thm:miswithboundedocptw}]
    Without loss of generality, we assume that $G$ is connected. If $\ocptw(G) > k$, then the algorithm in \zcref{thm:approxocptw} correctly determines this.
    Otherwise, the algorithm outputs a tame \ocptd $(T,\beta, \alpha)$ of $G$ of \ocp-width $k' \coloneqq f_{\ref{thm:globalstructure}}(k)$ and with $|V(T)| \in \mathbf{O}(|V(G)|)$ in time $\max(f_{\ref{thm:ktminormodel}}(t),2^{\mathbf{poly}(t)})|V(G)|^6$.
    We choose an arbitrary root $r \in V(T)$ and consider $(T,\beta, \alpha)$ as a rooted \ocptd $\mathcal{T}=(T,r,\beta, \alpha)$. For each edge $td \in E(T)$, we set $B_{td} \coloneqq \beta(t) \cap \beta(d)$.
    For the root $r$, we let $P_r \coloneqq \alpha(r)$.
    For each non-root node $t \in V(T) \setminus \{r\}$ and the unique neighbour $d \in V(T)$ of $t$ that is closer to the root in $T$, we let $P_t \coloneqq B_{td} \cup \alpha(t)$.
    Note that we have $|P_t| \leq k'+1$ for all $t \in V(T)$ due to our \ocptd being tame.
    For each subset $X \subseteq P_t$, we store the maximum weight $\w_{t,X}$ of an independent set $\mathcal{I}$ of $(G_{T_t}, w)$ such that $\mathcal{I} \cap P_t = X$ in the dynamic programming table on $\mathcal{T}$.
    If $X$ is not an independent set, we set $\w_{t,X} = -\infty$.

    We now build an auxiliary graph that will allow us to merge the tables.
    Assume that we are given some $t \in V(T)$, its children
    $d_1, d_2, ..., d_\ell$ in $T$, and some $X \subseteq P_t$.
    Furthermore, assume that we have already computed the maximum weight values $\w_{d_i, X'}$ for all $X' \subseteq P_{d_i}$ and for all $i \in [\ell]$.
    Let $(G_{t,X}^+, w_{t,X}^+)$ be the vertex-weighted graph obtained from $(G_t - P_t - N_{G_t}[X],w)$ as follows.
    If $|B_{td_i} \setminus \alpha(t)| = 1$, we let $v_{d_i}$ be the unique vertex that lies in $B_{td_i} \setminus \alpha(t)$.
    For each $i \in [\ell]$, we introduce a new vertex $x_i$.
    If there is $v_{d_i}$, and $v_{d_i} \notin N[X]$, we connect $x_i$ and $v_{d_i}$ in $G_{t,X}^+$ and set
    \[
        w_{t,X}^+(v_{d_i}) \coloneqq \max\left\{\w_{d_i, X'} - \sum_{v \in X \cap B_{td_i}} w(v) : X' \subseteq P_{d_i} \text{ and } X' \cap B_{td_i} = (X \cup \{v_{d_i}\}) \cap B_{td_i}\right\} \text{ and }
    \]
    \[
        w_{t,X}^+(x_i) \coloneqq \max\left\{\w_{d_i, X'} - \sum_{v \in X \cap B_{td_i}} w(v) : X' \subseteq P_{d_i} \text{ and } X' \cap B_{td_i} = X \cap B_{td_i}\right\} .
    \]
    Otherwise,
    $x_i$ is isolated in $G_{t,X}^+$, and we set
    \[
        w_{t,X}^+(x_i) \coloneqq \max\left\{\w_{d_i, X'} - \sum_{v \in X \cap B_{td_i}} w(v) : X' \subseteq P_{d_i} \text{ and } X' \cap B_{td_i} = X \cap B_{td_i}\right\} .
    \]
    
    For all the other vertices $v$ in $\beta(t) \setminus (P_t \cup N_{G_t}[X])$, we set $w_{t,X}^+(v) \coloneqq w(v)$.
    Note that if $\sum_{v \in V(G)} w(v) \in \mathbf{O}(|V(G)|)$, then we have $\sum_{v \in V(G_{t,X}^+)} w_{t,X}^+(v) \in \mathbf{O}(|V(G)|)$.
    Since we did not create any new cycles in $G_{t,X}^+$, we have \ocp$(G_{t,X}^+) \leq k'$.
    Note that we can compute $w_{t,X}^+(v_{d_i})$ and $w_{t,X}^+(x_i)$ for all $i \in [\ell]$ in $\mathbf{O}(\ell2^{k'+1})$ time.
    Hence this auxiliary graph can be built in time $\mathbf{O}(\ell2^{k'+1}|V(G)|)$

    We can find the maximum weight independent set in $G_{t,X}^+$ using the algorithm in \zcref{thm:fptalgoformis} in time $|V(G)|^{f(k')}$.
    This allows us to set $\w_{t,X}$ to be $\sum_{v \in X}w(v)$ plus the maximum weight of an independent set in $(G_{t,X}^+, w_{t,X}^+)$ if $X$ is an independent set, and otherwise we set $\w_{t,X} = -\infty$.
    The tuple $(G_{t,X}^+, w_{t,X}^+)$ can be built and these computation can be carried out for all $X \subseteq P_t$ in time $\ell2^{2(k'+1)}|V(G)|^{f(k') + 1}$.
    We then do this for all $t \in V(T)$, and note that the sum of all $\ell$ for each $t \in V(T)$ is equal to $|E(T)|$.
    Thus we have a total runtime of $2^{2(k'+1)}|V(G)|^{f(k')+2}$.
    We can then return the maximum weight of an independent set of $(G,w) = (G_{T_r},w)$ as $\max\{W_{r,X} : X \subseteq P_r\}$.
 
    We now prove correctness of the above algorithm.
    Fix $t \in V(T)$ and an independent set $X \subseteq P_t$. Let $\mathcal{I}^+$ be a maximum weight independent set of $(G_{t,X}^+, w_{t,X}^+)$.
    We now want to show that $w_{t,X}^+(\mathcal{I}^+) + w(X)$ is the same as the maximum weight of an independent set $\mathcal{I}$ of $(G_{T_t}, w)$ such that $\mathcal{I} \cap P_t = X$, which implies that $\w_{t,X}$ is set to the correct value.

    We first build an independent set in $(G_{T_t}, w)$ from $\mathcal{I}^+$. Let $S$ be the set obtained from $G$ by taking the union of the following.
    \begin{itemize}
        \item $X$,
        
        \item $\mathcal{I}^+ \cap \beta(t)$,

        \item for each $i \in [\ell]$, a maximum weight independent set $S_i$ of $(G_{T_{d_i}},w)$ such that $S_i \cap B_{td_i} = (\mathcal{I}^+ \cup X) \cap B_{td_i}$
    \end{itemize}
    Then $S$ is an independent set in $G_{T_t}$ with $S \cap P_t = X$. Hence, $w(S) \leq w(\mathcal{I})$. Note that the above sets are disjoint except at $(\mathcal{I}^+ \cup X) \cap B_{td_i}$, so
    $$w(S) = w(X) + w(\mathcal{I}^+ \cap \beta(t)) + \sum_{i \in [\ell]} (w(S_i) - w((\mathcal{I}^+ \cup X) \cap B_{td_i})).$$
    Note that for all $i \in [\ell]$, $\mathcal{I}^+ \cap B_{td_i}$ is either empty or equal to $\{v_{d_i}\}$ due to the tameness of our \ocptd.
    If $\mathcal{I}^+ \cap B_{td_i} = \{v_{d_i}\}$, then $w(S_i) - w((\mathcal{I}^+ \cup X) \cap B_{td_i}) = w_{t,X}^+(v_{d_i}) - w(v_{d_i})$.
    If $\mathcal{I}^+ \cap B_{td_i} = \varnothing$, then $w(S_i) - w((\mathcal{I}^+ \cup X) \cap B_{td_i}) = w_{t,X}^+(x_i)$.
    Hence $w(S) = w_{t,X}^+(\mathcal{I}^+) + w(X)$ and we obtain $w^+(\mathcal{I}^+) + w(X) \leq w(\mathcal{I})$.

    For the other direction, we build an independent set in $(G_{t,X}^+, w_{t,X}^+)$ from $\mathcal{I}$.
    Let $S_i^+ \coloneqq \mathcal{I} \cap V(G_{T_{d_i}} - P_t - N_{G_{T_t}}[X])$ and let $S^* \coloneqq V(G_t - P_t - N_{G_{T_t}}[X])$.
    We construct $S^+$ by taking the union of the set $S^*$ and $\{x_i\}$ for each $i\in [\ell]$ with $v_{d_i} \notin S^+_i$.
    Observe that $S^+$ is an independent set of $G_{t,X}^+$, and $w_{t,X}^+(S^+)$ is the same as $w(\mathcal{I}) - w(X)$.
    Then by the choice of $\mathcal{I}^+$, we have $w(\mathcal{I}) \leq w_{t,X}^+(S^+) + w(X) = w_{t,X}^+(\mathcal{I}^+) + w(X)$.
    Therefore, it holds that $w_{t,X}^+(\mathcal{I}^+) + w(X) = w(\mathcal{I})$.
\end{proof}

\section{A toolbox for graph minors}\label{sec:graphMinorPrelims}
Our methods are built upon the literature surrounding the graph minor series.
As such, we will need to introduce many of the key concepts relevant to this area.
Most of the definitions we present here are adaptations of concepts found in \cite{GorskySW2025polynomialboundsgraphminor}, which we sometimes modified to crop them to their most essential features.

The concepts we present here are largely derived through refinement of the definitions in \cite{KawarabayashiTW2018New,KawarabayashiTW2021Quickly} and the Graph Minors Series by Robertson and Seymour.
The interested reader can find more in-depth explanations for these concepts, including numerous illustrations, in \cite{GorskySW2025polynomialboundsgraphminor}.

\subsection{Meshes and Tangles}
Meshes are a generalisation of grids and walls, which enable us to state several theorems more concisely.

\paragraph{Meshes}
    Let $n,m$ be integers with $n,m\geq 2$.
    A \emph{$(n\times m)$-mesh} is a graph $M$ which is the union of paths $M=P_1\cup\cdots\cup P_n\cup Q_1\cup \cdots \cup Q_m$ where
    \begin{itemize}
        \item $P_1,\cdots,P_n$ are pairwise vertex-disjoint, and $Q_1,\cdots,Q_m$ are pairwise vertex-disjoint.
        \item for every $i\in [n]$ and $j\in [m]$, the intersection $P_i\cap Q_j$ induces a path,
        \item each $P_i$ is a $V(Q_1)$-$V(Q_m)$-path intersecting the paths $Q_1,\cdots Q_m$ in the given order, and each $Q_j$ is a $V(P_1)-V(P_m)$-path intersecting the paths $P_1,\cdots, P_h$ in the given order. 
    \end{itemize}
    We say that the paths $P_1,\cdots,P_n$ are the \emph{horizontal paths}, and the paths $Q_1,\cdots,Q_m$ are the \emph{vertical paths}.
    A mesh $M'$ is a \emph{submesh} of a mesh $M$ if every horizontal (vertical) paths of $M'$ is a subpath of a horizontal (vertical) paths $M$, respectively.
    We may write $n$-mesh instead of $(n\times n)$-mesh as a shorthand.
    A mesh $M$ is a \emph{bipartite mesh} if $M$ is a bipartite graph.
    
\begin{figure}[!hb]
    \centering
    \begin{tikzpicture}[scale=1.5]
    \pgfdeclarelayer{background}
    \pgfdeclarelayer{foreground}
    \pgfsetlayers{background,main,foreground}
    \begin{pgfonlayer}{background}
        \pgftext{\includegraphics[width=6cm]{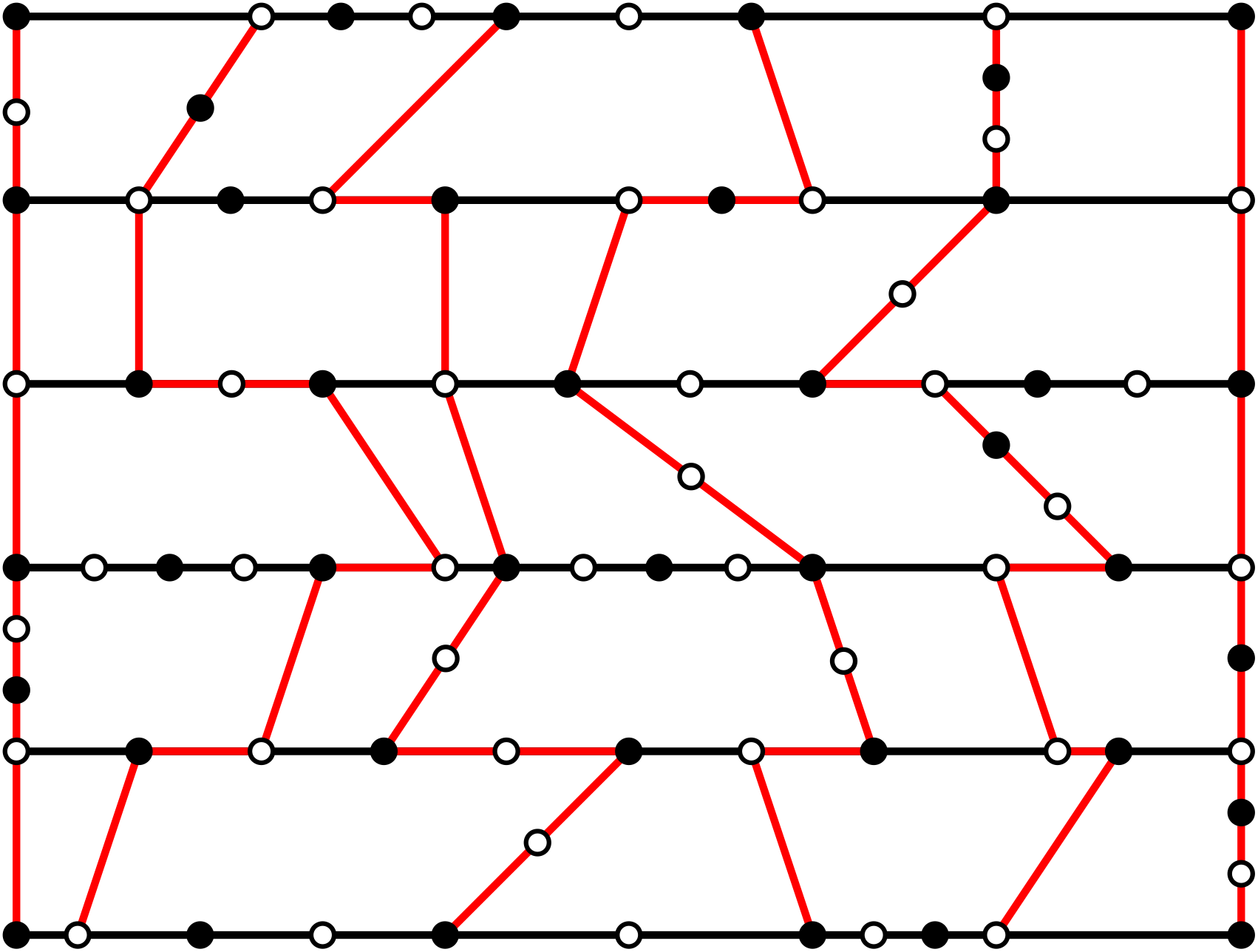}} at (C.center);
    \end{pgfonlayer}{background}
    \begin{pgfonlayer}{main}
        \node[]() at (-3.2,2.2) {$P_1$};
        \node[]() at (-3.2,1.32) {$P_2$};
        \node[]() at (-3.2,0.44) {$P_3$};
        \node[]() at (-3.2,-0.44) {$P_4$};
        \node[]() at (-3.2,-1.32) {$P_5$};
        \node[]() at (-3.2,-2.2) {$P_6$};
        \node[]() at (-2.85,2.5) {$\color{red}{Q_1}$};
        \node[]() at (-1.7,2.5) {$\color{red}{Q_2}$};
        \node[]() at (-0.55,2.5) {$\color{red}{Q_3}$};
        \node[]() at (0.6,2.5) {$\color{red}{Q_4}$};
        \node[]() at (1.75,2.5) {$\color{red}{Q_5}$};
        \node[]() at (2.9,2.5) {$\color{red}{Q_6}$};
    \end{pgfonlayer}{main}
    \begin{pgfonlayer}{foreground}
    \end{pgfonlayer}{foreground}
    \end{tikzpicture}
    \caption{A bipartite $(6\times 6)$-mesh.}
    \label{fig:bipartitemesh}
\end{figure}
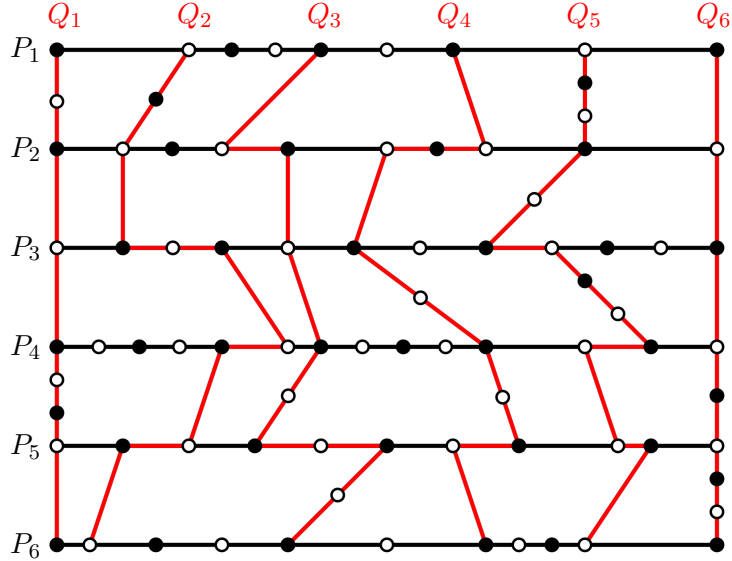

\paragraph{Minor models controlled by meshes.}
Let \(G\) be a graph, and let \(M\) be a \((w \times h)\)-mesh in \(G\) with vertical paths \(P_1, \ldots, P_w\) and horizontal paths \(Q_1, \ldots, Q_h\). In this setting, for any set \(Z \subseteq V(G)\) with \(|Z| < \min\{w, h\}\),
at least one vertical path is disjoint from \(Z\), and at least one horizontal path is disjoint from \(Z\).
Since every vertical path intersects every horizontal path, there exists a component of \(G - Z\) containing all vertical and horizontal paths of the mesh that are disjoint from \(Z\).

Let $H$ be a graph and let $\varphi$ be a minor model of $H$ in the graph $G$.
We define three cases where we say that $\varphi$ is \emph{controlled} by $M$ as follows:

First, let \(H\) be $K_t$ with \(t \le \min\{w, h\}\).
In this case, we say that \(\varphi\) is \emph{controlled} by \(M\) if for every set \(Z \subseteq V(G)\) with \(|Z| < t\), and every \(x \in V(H)\), if the branch set \(\varphi(x)\) is disjoint from \(Z\), then it is contained in the same component of \(G - Z\) as the vertical and horizontal paths of \(M\) which are disjoint from \(Z\).
By \hyperref[prop:mengersthm]{Menger's Theorem} this is equivalent to saying that for each \(x \in V(H)\) and each vertical or horizontal path \(P\) of the mesh, there exist \(t\) internally disjoint \(\varphi(x)\)--\(V(P)\) paths \(R_1, \ldots, R_t\) in \(G\).
(We allow \(R_1, \ldots, R_t\) to be the same path of length \(\le 1\), if the sets \(\varphi(x)\) and \(V(P)\) intersect or there exists an edge between them.)

The second case is when $H$ is a universal parity breaking grid $\mathscr{U}_t$ with $2t+1\leq \min\{w,h\}$.
For a vertical or horizontal path $P$ of $\mathcal{U}_t$, let $V_P$ be the union of the branch sets $\bigcup_{x\in V(P)}\varphi(x)$.
We say that $\varphi$ is \emph{controlled} by $M$ if for every set $Z\subseteq V(G)$ with $|Z| < t$, and every vertical or horizontal path $P$ of $\mathcal{U}_t$, if $V_P$ is disjoint from $Z$, then it is contained in the same component of $G-Z$ as the vertical and horizontal paths of $M$ which are disjoint from $Z$.

The last case is when $H$ is a parity handle $\mathcal{H}_t$ with $t\leq \min\{w,h\}$.
For a concentric cycle or radial path $X$ of $\mathcal{H}_t$, let $V_X$ be the union of the branch sets $\bigcup_{x\in V(X)}\varphi(x)$.
We say that $\varphi$ is \emph{controlled} by $M$ if for every set $Z\subseteq V(G)$ with $|Z|<t$, and every vertical or horizontal path $P$ of $\mathcal{U}_t$, if $V_P$ is disjoint from $Z$, then it is contained in the same component of $G-Z$ as the vertical and horizontal paths of $M$ which are disjoint from $Z$.

For the second and the third cases, we can find equivalent condition using \hyperref[prop:mengersthm]{Menger's Theorem} similarly as the first case. 

While walls represent amazing witnesses for large treewidth which impose a lot of additional structure onto a graph, they are also relatively local and concrete in nature and for this reason not always good candidates for working with treewidth in a more abstract way.
Two more abstract alternatives are given by so-called ``highly linked sets'' which are sets that cannot be spread out over several bags of any tree-decomposition of small width, and ``tangles'' which constitute possibly the most abstract and versatile obstruction for treewidth.

\paragraph{The set of separations.}
Let $G$ be a graph and $k$ be a positive integer.
We denote by $\mathcal{S}_k(G)$ the collection of all separations $(A,B)$ of order $< k$ in $G$.

\paragraph{Tangles.}
An \emph{orientation} of $\mathcal{S}_k(G)$ is a set $\mathcal{O}$ such that for all $(A,B)\in\mathcal{S}_k(G)$ exactly one of $(A,B)$ and $(B,A)$ belongs to $\mathcal{O}$. 
A \emph{tangle} of order $k$ in $G$ is an orientation $\mathcal{T}$ of $\mathcal{S}_k(G)$ such that for all $(A_1,B_1),(A_2,B_2),(A_3,B_3)\in\mathcal{T}$, it holds that $G[A_1]\cup G[A_2]\cup G[A_3]\neq G$.
If $\mathcal{T}$ is a tangle and $(A,B)\in\mathcal{T}$ we call $A$ the \emph{small side} and $B$ the \emph{big side} of $(A,B)$.

Let $G$ be a graph and $\mathcal{T}$ and $\mathcal{D}$ be tangles of $G$.
We say that $\mathcal{D}$ is a \emph{truncation} of $\mathcal{T}$ if $\mathcal{D}\subseteq\mathcal{T}$.
\medskip

Let $r \in \mathbb{N}$ with $r\geq 3$, let $G$ be a graph, and $M$ be an $r$-mesh in $G$.
Let $\mathcal{T}_M$ be the orientation of $\mathcal{S}_r$ such that for every $(A,B)\in\mathcal{T}_M$, the set $B\setminus A$ contains the vertex set of both a horizontal and a vertical path of $M$, we call $B$ the \emph{$M$-majority side} of $(A,B)$.
Then $\mathcal{T}_M$ is the tangle \emph{induced} by $M$.
If $\mathcal{T}$ is a tangle in $G$, we say that $\mathcal{T}$ \emph{controls} the mesh $M$ if $\mathcal{T}_M$ is a truncation of $\mathcal{T}$.

Let $G$ and $H$ be graphs as well as $\mathcal{T}$ be a tangle in $G$.
We say that a minor model $\mu$ of $H$ in $G$ is \emph{controlled} by $\mathcal{T}$ if there does not exist a separation $(A,B)\in\mathcal{T}$ of order less than $|V(H)|$ and an $x \in V(H)$ such that $\mu(x)\subseteq A\setminus B$.

Notice that, if a minor model $\mu$ of $H$ in $G$ is controlled by some mesh $M$ in $G$, then $\mu$ is also controlled by $\mathcal{T}_M$.

\paragraph{Large blocks.}
Given a block $B$ of a graph $G$, we note that $B$ naturally defines an orientation $\mathcal{T}_B$ of $\mathcal{S}_2$, by letting the big side of each separation always be the one that contains $V(B)$.
This orientation can be seen to be a tangle of order $1$.

Notice that, if $\mathcal{T}$ is a tangle of order $k\geq 2$ in a graph $G$ and $B$ and $B'$ are distinct blocks of $G$, then at most one of $\mathcal{T}_B$, $\mathcal{T}_{B'}$ can be a truncation of $\mathcal{T}$.
Moreover, there must exist at least one block $B$ of $G$ such that $\mathcal{T}_B\subseteq \mathcal{T}$. 
We say that $B$ is the \emph{large block} of a tangle $\mathcal{T}$ of order $k$ if $\mathcal{T}_B$ is a truncation of $\mathcal{T}$.
\smallskip

Given a tangle $\mathcal{T}$ of order $k+x$ in $G$ and a set $X$ with $|X| = x$, we define the tangle $\mathcal{T}-X$ of order $k$ in $G$ as follows
\[ \mathcal{T} - X \coloneqq \{ (A,B) \in \mathcal{S}_k(G) ~\!\colon\!~ (A \cup X, B \cup X) \in \mathcal{T} \} . \]
If $\varphi$ is a minor model of $K_t$ in $G$, we define an associated tangle $\mathcal{T}_\varphi$ of order $t$ by adding the separation $(A,B) \in \mathcal{S}_t(G)$ to $\mathcal{T}_\varphi$ if at least one branch set of $\varphi$ is contained entirely in $B$.
Note that this condition must hold for either $A$ or $B$ and thus this properly defines a tangle.

This allows us to state the following result on $K_t$-minor models which will be of great use to us.
\begin{proposition}[Geelen, Gerards, Reed, Seymour, and Vetta \cite{GeelenGRSV2009Oddminor,Yamaguchi2016Packing} (Theorem 13)]\label{thm:ktminormodel}
    There is a constant $\mathsf{c}_{\ref{thm:ktminormodel}}$ such that for all $l \geq 1$ if a graph $G$ contains a $K_t$-minor model $\varphi$, where $t = \lceil \mathsf{c}_{\ref{thm:ktminormodel}} l \sqrt{\log 12l} \rceil$, then either $G$ contains an odd $K_l$-minor model, or for some set $A$ of vertices with $|A| \leq 8l$, the large block of $\mathcal{T}_\varphi - A$ is bipartite.

    Moreover, there exists an algorithm that takes $G$ and $\varphi$ as above as input and finds one of the outcomes above in time $\mathbf{O}(f_{\ref{thm:ktminormodel}}(t)|V(G)|^{\omega_{\ref{def:matrixmultconstant}}})$ where $f_{\ref{thm:ktminormodel}}$ is a computable function (see \cite{Yamaguchi2016Packing}).
\end{proposition}

\subsection{Drawings, embeddings, and \texorpdfstring{$\Sigma$}{Sigma}-renditions}\label{sec:renditions}
Modified notions of embeddings in surfaces are at the core of graph minor structure theory.
One of our most notable departures from \cite{GorskySW2025polynomialboundsgraphminor} is the switch from \textsl{$\Sigma$-decompositions} to \textsl{$\Sigma$-renditions}.
The latter simplify the structure of ``cell'' via the idea of paintings as presented in \cite{KawarabayashiTW2018New}.

\paragraph{Surfaces.}
By a \emph{surface} we mean a compact $2$-dimensional manifold with or without boundary.

Given a pair $(\mathsf{h}, \mathsf{c}) \in \mathbb{N} \times [0,2]$ we define $\Sigma^{(\mathsf{h}, \mathsf{c})}$ to be the two-dimensional surface without boundary created from the sphere by adding $\mathsf{h}$ handles and $\mathsf{c}$ crosscaps (see \cite{MoharT2001Graphs} for more details).
If $\mathsf{c} = 0$ the surface $\Sigma^{(\mathsf{h}, \mathsf{c})}$ is an \emph{orientable} surface, otherwise it is called \emph{non-orientable}.
By Dyck's theorem \cite{Dyck1888Beitraege,FrancisW1999Conways}, two crosscaps are equivalent to a handle in the presence of a third crosscap.
Thus the notation $\Sigma^{(\mathsf{h}, \mathsf{c})}$ is sufficient to capture all two-dimensional surfaces without boundary.

If $\Sigma$ is a surface (with holes), we let $\overline{\Sigma}$ be the surface resulting from gluing a closed, respectively open disk onto each open, respectively closed hole of $\Sigma$.
We let the \emph{Euler-genus} of $\Sigma$ be $2\mathsf{h} + \mathsf{c}$, where $\Sigma^{\mathsf{(\mathsf{h},\mathsf{c})}}$ is the surface to which $\overline{\Sigma}$ is isomorphic.
According to a classic theorem by Dyck \cite{Dyck1888Beitraege,FrancisW1999Conways}, the Euler-genus is thus a concrete value for each given surface.

In a given surface $\Sigma$, the \emph{closure} of a set $X \subseteq \Sigma$ is denoted by $\overline{X}$ and we denote the \emph{boundary} of $X$ as $\mathsf{bd}(X)$.
A curve $C$ in a surface $\Sigma$ is called \emph{contractible} if it is simple, closed, and some component of $\Sigma - C$ is a disk whose closure has $C$ as its boundary.
We call a simple, closed curve $\gamma$ in an arc-wise connected surface $\Sigma$ \emph{non-contractible} if it is not contractible, and if $\Sigma - \gamma$ has exactly one more boundary component than $\Sigma$, we call $\gamma$ \emph{one-sided}.\footnote{Equivalently, the deletion of a one-sided curve decreases the Euler-genus of $\Sigma$ by exactly one.}
\medskip

In a first step, we introduce a framework to talk about ``paintings'' of graphs in a surface.
We then extend this to a notion of ``renditions'' of a graph in a surface.
These notions are refinements of notions found in \cite{KawarabayashiTW2018New,KawarabayashiTW2021Quickly,GorskySW2025polynomialboundsgraphminor}, whose origins are further found in the Graph Minors series.
In particular, the notion of a painting draws heavily from the ideas presented in \cite{KawarabayashiTW2018New}.

\paragraph{Paintings in surfaces.}
A \emph{painting} in a surface $\Sigma$ is a pair $\Gamma = (U,N)$, where $N \subseteq U \subseteq \Sigma$, $N$ is finite, $U \setminus N$ has a finite number of arcwise-connected components, called \emph{cells} of $\Gamma$, and for every cell $c$, the closure $\overline{c}$ is a closed disk where $N_\Gamma(c) \coloneqq \overline{c} \cap N \subseteq \mathsf{bd}(\overline{c})$.
If $|N_\Gamma(c)| \geq 4$, the cell $c$ is called a \emph{vortex}.
We further let $N(\Gamma) \coloneqq N$, $U(\Gamma) \coloneqq U$, and let $C(\Gamma)$ be the set of all cells of $\Gamma$.
\medskip

Any given painting $\Gamma = (U,N)$ defines a hypergraph with $N$ as its vertices and the set of closures of the cells of $\Gamma$ as its edges.
Accordingly, we call $N$ the \emph{nodes} of $\Gamma$.

\paragraph{$\Sigma$-renditions.}
Let $G$ be a graph and $\Sigma$ be a surface.
A \emph{$\Sigma$-rendition} of $G$ is a triple $\rho = (\Gamma, \sigma, \pi)$, where
\begin{itemize}
    \item $\Gamma$ is a painting in $\Sigma$,
    \item for each cell $c \in C(\Gamma)$, $\sigma(c)$ is a subgraph of $G$, and
    \item $\pi \colon N(\Gamma) \to V(G)$ is an injection,
\end{itemize}
such that
\begin{description}
    \item[R1] $G = \bigcup_{c \in C(\Gamma)}\sigma(c)$,
    \item[R2] for all distinct $c,c' \in C(\Gamma)$, $\sigma(c)$ and $\sigma(c')$ are edge-disjoint,
    \item[R3] $\pi(N_\Gamma(c)) \subseteq V(\sigma(c))$ for every cell $c \in C(\Gamma)$, and
    \item[R4] for every cell $c \in C(\Gamma)$, $V(\sigma(c) \cap \bigcup_{c' \in C(\Gamma) \setminus \{ c \}} (\sigma(c'))) \subseteq \pi(M_\Gamma(c))$.
\end{description}
We write $N(\rho)$ for the set $N(\Gamma)$, let $N_\rho(c) = N_\Gamma(c)$ for all $c \in C(\Gamma)$, and similarly, we lift the set of cells from $C(\Gamma)$ to $C(\rho)$.
If it is clear from the context which $\rho$ is meant, we will sometimes simply write $N(c)$ instead of $N_\rho(c)$, and if the $\Sigma$-rendition $\rho$ for $G$ is understood from the context, we usually identify the sets $\pi(N(\rho))$ and $N(\rho)$ along $\pi$ for ease of notation.
The number of vortex cells of a $\Sigma$-rendition $\rho$ is called the \emph{breadth} of $\rho$.
\medskip

Let $G$ be a graph, let $\Sigma$ be a surface, let $H \subseteq G$ be a subgraph, and let $\rho = (\Gamma,\sigma,\pi)$ be a $\Sigma$-rendition of $G$.
Then we define $\sigma'$ and $\Gamma'$ as follows.
Initialise $U'$ as $U$.
For every cell $c \in C(\rho)$, set $\sigma'(c) = H \cap \sigma(c)$, and if $\sigma_\rho(c) \cap H \subseteq N_\rho(c)$, then $c$ is discarded from $U'$.
Similarly, let $N' = N(\Gamma) \cap V(H)$.
This allows us to set $\Gamma' = (U',N')$.
We then call $\rho' = (\Gamma', \sigma', \pi\restrict{N(\Gamma')})$ the \emph{restriction of $\rho$ to $H$}.
If the restriction of two $\Sigma$-renditions to $H$ is isomorphic, we say that they \emph{agree on the rendition of $H$}.
Furthermore, for any vortex $c$ of $\rho$ we let $G - [\sigma(c)]$ be defined as the graph $(G - (V(\sigma(c)) \setminus N(c))) - E(\sigma(c))$, which informally is the graph in which we delete all of the insides of $c$ but leave its boundary vertices intact.

\paragraph{Societies.}
Let $\Omega$ be a cyclic ordering of the elements of some set which we denote by $V(\Omega)$.
A \emph{society} is a pair $(G,\Omega),$ where $G$ is a graph and $\Omega$ is a cyclic ordering with $V(\Omega)\subseteq V(G)$.
A \emph{cross} in a society $(G,\Omega)$ is a pair $(P_1,P_2)$ of vertex-disjoint paths in $G$ such that for both $i \in [2]$ the path $P_i$ has the endpoints $s_i,t_i\in V(\Omega)$ and is otherwise disjoint from $V(\Omega)$, and the vertices $s_1,s_2,t_1,t_2$ occur in $\Omega$ in the order listed.

Let $(G, \Omega)$ be a society.
For a given set $S \subseteq V(\Omega)$ a vertex $s \in S$ is an \emph{endpoint} of $S$ if there exists a vertex $t \in V(\Omega) \setminus S$ that immediately precedes or succeeds $s$ in $\Omega$.
We call $S$ a \emph{segment} of $\Omega$ if $S$ has two or less endpoints.
For $s,t \in V(\Omega)$ we denote by $s\Omega t$ the uniquely determined segment with the first vertex $s$ and the last vertex $t$ according to $\Omega$.

\paragraph{Renditions of societies.}
Let $(G,\Omega)$ be a society, and let $\Sigma$ be a surface with one boundary component $B$ homeomorphic to the unit circle.
A \emph{rendition} of $G$ in $\Sigma$ is a $\Sigma$-rendition $\rho$ of $G$ such that the image under $\pi_{\rho}$ of $N(\rho) \cap B$ is $V(\Omega)$ and $\Omega$ is one of the two cyclic orderings of $V(\Omega)$ defined by the way the points of $\pi_{\rho}(V(\Omega))$ are arranged in the boundary $B$.

\begin{proposition}[Two Paths Theorem \cite{Jung1970Verallgemeinerung,Seymour1980Disjoint,Shiloach1980Polynomial,Thomassen19802Linked}]\label{prop:TwoPaths}
A society $(G,\Omega)$ has no cross if and only if it has a vortex-free rendition in a disk.
\end{proposition}

In particular these outcomes can be found in linear time \cite{KawarabayashiLR2015Connectivity}.

\paragraph{Transactions and their types.}
Given a society $(G,\Omega)$, with $A,B$ being two disjoint segments of $\Omega$, a \emph{transaction} $\mathcal{P}$ in $(G,\Omega)$ is an $A$-$B$-linkage in $G$.
The inclusion-wise minimal segments $X$ and $Y$ of $\Omega$ for which $\mathcal{P}$ is an $X$-$Y$-linkage are called the \emph{end segments of $\mathcal{P}$ in $(G,\Omega)$}.
\medskip

If $(G, \Omega)$ is a society with a $\Sigma$-rendition $\rho$ and $\mathcal{P}$ is a transaction in $(G,\Omega)$, we call $\mathcal{P}$ \emph{exposed} if for every path $P \in \mathcal{P}$ there exists a vortex $c \in C(\rho)$ and an edge $e \in E(P) \cap \sigma(c)$.
On the other hand, if for every path $P \in \mathcal{P}$ we have $E(P) \cap E(\sigma(c)) = \emptyset$ for all vortices $c \in C(\rho)$, we call $\mathcal{P}$ \emph{unexposed}.
A transaction may be neither exposed nor unexposed.
But at least half of any transaction will fall into one of these two categories.
\medskip

Let \(\mathcal{P}\) be a transaction in a society \((G, \Omega)\) with the end segments \(X\) and \(Y\).
Suppose that the members of \(\mathcal{P}\) can be enumerated as \(P_1, \ldots, P_n\) so that if \(x_i \in X\) and \(y_i \in Y\) denote the endpoints of \(P_i\), then the vertices \(x_1, \ldots, x_n\) appear in the segment \(X\) in the listed order or the reverse one, and the vertices \(y_1, \ldots, y_n\) appear in \(Y\) in the listed order or the reverse one. Then we say that \(\mathcal{P}\) is \emph{monotone}, and if \(P_1, \ldots, P_n\) are ordered as above, they are \emph{indexed naturally}.

Should the vertices \(x_1, \ldots, x_n, y_n, \ldots, y_1\) appear in \(\Omega\) in the listed cyclic ordering or its reverse, we call \(\mathcal{P}\) a \emph{planar transaction}, and if the vertices \(x_1, \ldots, x_n, y_1, \ldots, y_n\) appear in \(\Omega\) in the listed cyclic ordering or its reverse, we call \(\mathcal{P}\) a \emph{crosscap transaction}.
The paths \(P_1\) and \(P_n\) are called the \emph{boundary paths} of \(\mathcal{P}\).

Let $H$ be a subgraph of a graph $G$.
An \emph{$H$-bridge} in $G$ is a connected subgraph $B$ of $G$ such that $E(B) \cap E(H) = \emptyset$ and either $E(B)$ consists of a unique edge with both ends in $H$, or 
$B$ is constructed from a component $C$ of $G - V(H)$ and the non-empty set of edges $F \subseteq E(G)$ with one end in $V(C)$ and the other in $V(H)$, by taking the union of $C$, the endpoints of the edges in $F$, and $F$ itself.
The vertices in $V(B) \cap V(H)$ are called the \emph{attachments} of $B$.

We let $H$ denote the subgraph of $G$ obtained from the union of elements of $\mathcal{P}$ by adding the elements of $V(\Omega)$ as isolated vertices.
Further, we define $H'$ as the subgraph of $H$ consisting of $\mathcal{P}$ and all vertices of $X \cup Y$.
Consider all $H$-bridges of $G$ with at least one attachment in $V(H') \setminus V(P_1 \cup P_n)$, and for each such $H$-bridge $B$, let $B'$ denote the graph obtained from $B$ by deleting all attachments that do not belong to $V(H')$.
We let $G_1$ denote the union of $H'$ and all graphs $B'$ as above and call $G_1$ the \emph{strip (of $\mathcal{P}$)}.

The \emph{\(\mathcal{P}\)-strip society in \((G, \Omega)\)} is defined as the society \((G_1, \Omega_1)\), where \(\Omega_1\) is the concatenation of the segment \(X\) ordered from \(x_1\) to \(x_n\), and the segment \(Y\) ordered from \(y_n\) to \(y_1\).
If the \(\mathcal{P}\)-strip society admits a vortex-free $\Delta$-rendition in a disk $\Delta$, we call \(\mathcal{P}\) a \emph{flat transaction}.
Further, if no edge of $G$ has an endpoint in $V(G_1) \setminus V(P_1 \cup P_n)$ and the other point $V(G) \setminus V(G_1)$, then we call $\mathcal{P}$ \emph{isolated}.
Let $X',Y'$ be the two distinct segments of $\Omega$ that have one endpoint in $X$ and the other in $Y$.
Note that $V(\Omega) = X \cup Y \cup X' \cup Y'$.
We say that \(\mathcal{P}\) is \emph{separating} if it is isolated and there exists no $X'$-$Y'$-path in $G - V(G_1)$.

Monotone transactions are easy to find using a classic theorem by Erd\H{o}s and Szekeres \cite{ErdosS1935Combinatorial}.

\begin{lemma}\label{lem:monotonetransaction}
    Let $p,q$ be a positive integers, let $(G,\Omega)$ be a society, and let $\mathcal{P}$ be a transaction of order $(p-1)(q-1) + 1$.
    Then there exists a planar transaction $\mathcal{P}' \subseteq \mathcal{P}$ of order $p$ or a crosscap transaction $\mathcal{Q} \subseteq \mathcal{P}$ of order $q$, which can be found in $\mathbf{O}(pq|V(G)|)$-time.
\end{lemma}

A transaction $\mathcal{P}$ of order $2n$ in a society $(G, \Omega)$, for a positive integer $n$, is called a \emph{handle transaction} if $\mathcal{P}$ can be partitioned into two transactions $\mathcal{R}, \mathcal{Q}$ each of order $n$, such that both are planar, and $S^\mathcal{R}_1, S^\mathcal{Q}_1, S^\mathcal{R}_2, S^\mathcal{Q}_2$ are segments partitioning $\Omega$, with $\mathcal{X}$ being a $S^\mathcal{X}_1$-$S^\mathcal{X}_2$-linkage for both $\mathcal{X} \in \{ \mathcal{R} , \mathcal{Q} \}$, and the segments are found on $\Omega$ in the order they were listed above.
We call a handle transaction $\mathcal{P} = \mathcal{R} \cup \mathcal{Q}$ \emph{isolated in $G$}, respectively \emph{flat in $G$}, if both the $\mathcal{R}$-strip society and the $\mathcal{Q}$-strip society of $(G, \Omega)$ are isolated and flat in $G$.

\subsection{Traces, grounded subgraphs, and \texorpdfstring{$c_0$}{c₀}-disks}\label{sec:traces}
Let $\rho$ be a $\Sigma$-rendition of a society $(G,\Omega)$.
For every cell $c \in C(\rho)$ with $|N_\rho(c)| = 2$, we select one of the components of $\mathsf{boundary}(c) - N_\rho(c)$.
This selection will be called a \emph{tie-breaker in $\rho$}, and we assume that every rendition comes equipped with a tie-breaker.

\paragraph{Traces of paths and cycles.}
Let $G$ be a graph and $\rho$ be a $\Sigma$-rendition for $G$.
Let $Q$ be a cycle or path in $G$ that uses no edge of $\sigma(c)$ for every vortex $c \in C(\rho)$.
We say that $Q$ is \emph{grounded} if it uses edges of $\sigma(c_1)$ and $\sigma(c_2)$ for two distinct cells $c_1, c_2 \in C(\rho)$, or $Q$ is a path with both endpoints in $N(\rho)$.
If $Q$ is grounded we define the \emph{trace} of $Q$ as follows.
Let $P_1,\dots,P_k$ be distinct maximal subpaths of $Q$ such that $P_i$ is a subgraph of $\sigma(c)$ for some cell $c$.
Fix $i \in [k]$.
The maximality of $P_i$ implies that its endpoints are $\pi(n_1)$ and $\pi(n_2)$ for distinct nodes $n_1,n_2 \in N(\rho)$.
If $|N_\rho(c)| = 2$, let $L_i$ be the component of $\mathsf{boundary}(c) - \{ n_1,n_2 \}$ selected by the tie-breaker, and if $|N_\rho(c)| = 3$, let $L_i$ be the component of $\mathsf{boundary}(c) - \{ n_1,n_2 \}$ that is disjoint from $N_\rho(c)$.
Finally, we define $L_i'$ by pushing $L_i$ slightly so that it is disjoint from all cells in $C(\rho)$.
We construct these curves while maintaining that they intersect only at a common endpoint.
The \emph{trace} of $Q$ is defined to be $\bigcup_{i\in[k]} L_i'$.
The trace of a cycle is thus the homeomorphic image of the unit circle, and the trace of a path is an arc in $\Sigma$ with both endpoints in $N(\rho)$.

\paragraph{Aligned disks and grounded subgraphs.}
Let $G$ be a graph and let $\rho = (\Gamma, \sigma, \pi)$ be a $\Sigma$-rendition of $G$.
Furthermore, a disk in $\Sigma$ is called \emph{$\rho$-aligned} if its boundary only intersects $\Gamma$ in nodes.
For any $\rho$-aligned disk $\Delta$, we call the subgraph of $G$ that is drawn by $\Gamma$ and $\sigma$ onto $\Delta$ the \emph{crop of $H$ by $\Delta$ (in $\rho$)}.
Furthermore, the \emph{restriction $\rho'$ of $\rho$ to $\Delta$} is defined as the $\Delta$-rendition that consists of the restriction of both $\Gamma$, $\sigma$, and $\pi$ to $\Delta$.
Let $V(\Omega_{\Delta}) = \mathsf{bd}(\Delta) \cap N(\rho)$ and let $\Omega_{\Delta}$ be the cyclic ordering of $V(\Omega_{\Delta})$ obtained by traversing along $\mathsf{bd}(\Delta)$ in the anticlockwise direction.
Further, let $G_{\Delta}$ be the crop of $G$ by $\Delta$.
Then $(G_{\Delta}, \Omega_{\Delta})$ is the \emph{$\Delta$-society (in $\rho$)}.
Given a $\Delta$-society in $\rho$, we also call $\rho'$ the \emph{restriction of $\rho$ to $(G_\Delta, \Omega_\Delta)$} and vice versa $(G_\Delta, \Omega_\Delta)$ is the society associated with $\rho'$.

Let $C \subseteq G$ be a grounded cycle with the trace $T$ in $\rho$.
If $T$ is the boundary of a unique disk $\Delta_C$ in $\Sigma$, we note that $\Delta_C$ is $\rho$-aligned and we call the $\Delta_C$-society $(G_{\Delta_C}, \Omega_{\Delta_C})$ the \emph{$C$-society (in $\rho$)}.
Importantly, not all parts of $C$ have to be contained in $G_{\Delta_C}$, as the subpaths of $C$ that are drawn within cells of $\rho$ could be ``pushed out'' of $\Delta_C$ by the tie-breakers associated with the trace of $C$.

We say that a 2-connected subgraph $H$ of $G$ is \emph{grounded (in $\rho$)} if every cycle in $H$ is grounded and no vertex of $H$ is drawn by $\Gamma$ into the interior of a vortex of $\rho$.
If $H$ is also planar, we say that it is \emph{flat in $\rho$} if there exists a $\rho$-aligned disk $\Delta \subseteq \Sigma$ which contains all cells $c \in C(\Gamma)$ with $E(\sigma(c)) \cap E(H) \neq \emptyset$ and $\Delta$ does not contain any vortices of $\Gamma$.

\paragraph{$c_0$-disks.}
Let $\rho$ be a $\Sigma$-rendition of a graph $G$ with a vortex $c_0$ and let $T$ be the trace of a grounded cycle $C \subseteq G$ that bounds a disk $\Delta_0 \subseteq \Sigma$ containing $c_0$.
We call $\Delta_0$ the \emph{$c_0$-disk of $C$}.
Given a 2-connected, grounded graph $H \subseteq G$ that contains a cycle that has a $c_0$-disk, we call the inclusion-wise minimal $c_0$-disk of a cycle in $H$ be the \emph{$c_0$-disk of $H$}.

\subsection{Cylindrical renditions and nests}\label{sec:nest}
Many parts of our proof take place in a society that has a rendition in a disk containing a single vortex $c_0$.
Let $(G,\Omega)$ be a society.
If $\rho$ is a $\Delta$-rendition of $(G,\Omega)$ in a disk $\Delta$ with a unique vortex $c_0 \in C(\rho)$, we say that $\rho$ is a \emph{cylindrical rendition of $(G,\Omega)$ around $c_0$ (in $\Delta$)}.

\paragraph{Inner and outer graphs of a cycle.}
Let $(G, \Omega)$ be a society with a $\Sigma$-rendition $\rho$.
Further, let $C$ be a grounded cycle whose trace bounds a disk $\Delta_C$ and the $\Delta_C$-society $(G', \Omega')$.
We call $G' \cup C$ the \emph{inner graph of $C$ (in $\rho$)} and call $G'$ itself the \emph{proper inner graph of $C$ (in $\rho$)}.
Let $B = \pi(N(\rho) \cap \mathsf{bd}(\Delta_C))$.
We define the \emph{proper outer graph of $C$ (in $\rho$)} as $G'' \coloneqq G[B \cup (V(G) \setminus V(G'))]$ and call $G'' \cup C$ the \emph{outer graph of $C$ (in $\rho$)}.

\paragraph{Nests.}
Let $\rho$ be a $\Sigma$-rendition of a society $(G, \Omega)$ with the vortices $c_1, \ldots , c_b$, then a \emph{nest (in $\rho$)} is a set of disjoint cycles $\mathcal{C} = \{ C_1, \ldots , C_s \}$ in $G$ such that each of them is grounded in $\rho$ and the trace of $C_i$ bounds a surface $\Sigma_i$ in such a way that $\bigcup_{i=1}^b c_i \subseteq \Sigma_1 \subseteq \cdots \subseteq \Sigma_s \subseteq \Sigma$.
If $c_1, \ldots , c_b$ are clear from the context, we call $\mathcal{C}$ a nest in $\rho$ and drop the reference to the vortices.
Furthermore, when introducing a nest $\mathcal{C} = \{ C_1, \ldots , C_s \}$, we will not specify that it has order $s$ if this is clear from the context.
If $\rho$ is a cylindrical rendition around the vortex $c_0$, we will say that the nest lies \emph{around $c_0$}.

Let $(G,\Omega)$ be a society, let $\rho$ be a cylindrical rendition of $(G,\Omega)$ in a disk $\Delta$ around the vortex $c_0$, and let $C$ be a grounded cycle or $V(\Omega)$-path with a $c_0$-disk.
Given a grounded $C$-path $P$ of length at least one, we say that $P$ \emph{sticks out towards $c_0$ (in $\rho$)} if the $c_0$-disk of $P \cup C$ is not the $c_0$-disk of $C$ and otherwise we say that $P$ \emph{sticks out away from $c_0$ (in $\rho$)}.

Let $\mathcal{C} = \{ C_1, \dots , C_s \}$ be a nest in a $\Sigma$-rendition $\rho$.
We let $c_1$ be a potentially artificial disk bounded by the trace $T_1$ of $C_1$ in $\rho$ and replace the surface $\Sigma_1$, which is bounded by $T_1$ and does not contain the boundary of $\Sigma$, with $c_1$.
This implicitly creates a cylindrical rendition $\rho'$ of $G$ around $c_0$.
We say that $\mathcal{C}$ is \emph{cozy} if for every $i \in [s]$ and every grounded $C_i$-path $P$ that sticks out away from $c_0$ in $\rho'$ we have $V(P) \cap V(\Omega) \neq \emptyset$ or there exists a $j \in [s] \setminus \{ i \}$ such that $V(P) \cap V(C_j) \neq \emptyset$.

\paragraph{Linkages in relation to a nest.}
If $(G, \Omega)$ is a society with a nest $\mathcal{C} = \{ C_1, \ldots , C_s \}$ in a  $\Sigma$-rendition $\rho$ of $(G, \Omega)$, we call a $V(\Omega)$-$V(C_1)$-linkage $\mathcal{R}$ a \emph{radial linkage (in $\rho$) for $\mathcal{C}$}.
We say that $\mathcal{R}$ is \emph{orthogonal to $\mathcal{C}$} if for all $C \in \mathcal{C}$ and all $R \in \mathcal{R}$ the graph $C \cap R$ is a path.
In general, letting $H$ be the inner graph of $C_1$ in $\rho$, we call a linkage $\mathcal{R}'$ in which each path has an endpoint in $V(\Omega)$ and the other in $V(H)$ \emph{orthogonal} to $\mathcal{C}$ if for all $C \in \mathcal{C}$ and all $R' \in \mathcal{R}'$ the graph $C \cap R'$ is a path.
Similarly, we say that a transaction $\mathcal{P}$ in $(G,\Omega)$ is \emph{orthogonal to $\mathcal{C}$} if for all $C\in\mathcal{C}$ and all $P\in\mathcal{P}$ the graph $C\cap P$ consists of exactly two paths.

\subsection{Parity-preservation and non-orientable, even-faced \texorpdfstring{$\Sigma$}{Σ}-renditions}
We now introduce the core modification of $\Sigma$-renditions that our structural results will be centred on.
Our intention is to define a type of rendition in which we have additional control over the odd cycles in the graph.
We are most concerned with the behaviour of grounded odd cycles and want to exclude types of grounded odd cycles that imply the existence of an ``odd face'' in our rendition.
To get a feeling for this, recall that in any planar graph, if we consider an odd cycle then either it is itself a face, or it divides the graph into two subgraphs which are separated at the drawing of the odd cycles, each of which must have two odd faces according to the Handshake lemma, which implies that there exists an odd face in each.
Due to this intuition, which extends beyond planar graphs, graph embeddings which guarantee that all faces are even imply that the behaviour of the odd cycles in the graph is particularly restricted.
We will want to extend these notions, amongst others, to renditions in the following.

Let $c$ be a cell in a $\Sigma$-rendition, and let $u,v \in N(c)$ be distinct nodes of $c$.
We call $c$ \emph{parity-preserving between $u$ and $v$}, if all $u$-$v$-paths in $\sigma(c)$ have the same parity.
Let $H$ be a subgraph of $G$, let $c$ be a cell in $\Sigma$-rendition, and let $u,v \in N(c)$ be distinct nodes of $c$.
We say $c$ is \emph{$2$-connected to $H$ via $u$ and $v$} if there is a path whose endpoints are in $H$ and contains a $u$-$v$-path in $\sigma(c)$ as a subpath.
More simply, we say that $c$ is \emph{$2$-connected to $H$} if the nodes $u$ and $v$ we use are not relevant.

A $\Sigma$-rendition $\rho$ of $G$ is called \emph{even-faced around $H$} if every non-vortex cell of $\rho$ that is $2$-connected to $H$ via $u$ and $v$ is parity-preserving between $u$ and $v$, and the trace of each grounded odd cycle that contains an edge in a cell that is $2$-connected to $H$ is non-contractible in $\Sigma$.
An even-faced $\Sigma$-rendition $\rho$ of $G$ around $H$ is called \emph{non-orientable} if the trace of each grounded odd cycle that contains an edge in a cell that is $2$-connected to $H$ is one-sided in $\Sigma$.
If $G$ is itself 2-connected, we write that $\rho$ is an even-faced (and possibly non-orientable) $\Sigma$-rendition of $G$ to imply that $\rho$ is even-faced (and possibly non-orientable) around $G$.

Note that for a cell $c$, we can check whether it is parity-preserving or not in $\mathbf{O}(|E(c)|)$ time.
Thus for a given graph $G$ with a $\Delta$-rendition $\rho$ in a disk $\Delta$ and a subgraph $H$, we can check whether $\rho$ is even-faced or not in $\mathbf{O}(|E(G)|)$-time.

In some special types of societies, we will want to make the $H$ around which the rendition is even-faced implicit.
A transaction in a society $(G,\Omega)$ is called \emph{even-faced} if its strip has a vortex-free $\Delta$-rendition in a disk $\Delta$ that is even-faced around the union of paths in the transaction.
We extend this notion to handle transactions in the natural way.
A cylindrical rendition $\rho$ of a society $(G,\Omega)$ (in a disk $\Delta$) that has a nest $\mathcal{C}$ is called \emph{even-faced} if $\rho$ is even-faced around $\bigcup \mathcal{C}$.

We close with a few remarks on these definitions.
First, we want to emphasise that $\rho$ being non-orientable does not imply that the surface in which $\rho$ embeds the graph is itself non-orientable.
Instead we wish to imply that all of the interesting (i.e.\ grounded) odd cycles in the graph behave in a way that implies they are in a non-orientable surface.
Of course, if our surface does not feature any crosscaps, this implies that there are no grounded odd cycles.
Furthermore, we note that if $\rho$ is even-faced this does not mean that each odd cycle in $G$ contains an edge in $E(\sigma_\rho(c_0))$, since there may be a cell $c \in C(\rho)$ such that $\sigma_\rho(c)$ contains an odd cycle $C$ that is not grounded.
However, due to $c$ being parity-preserving, no path in $\sigma_\rho(c)$ that connects two of the vertices corresponding to the nodes on the boundary of $c$ can use an edge in $E(C)$.

\subsection{Depth of vortices}\label{subsec:vortex}
Finally, we need a notion of ``depth'' for our vortices.
So far, the definitions allow us to hide almost anything within a vortex, as long as its interface with the rest of the graphs agrees on the cylindrical ordering of some set of vertices.
In our proof, such vortices will be refined further until the graphs drawn in their interiors stop providing enough infrastructure to continue this process.

Let $G$ be a graph and $\rho$ be a $\Sigma$-rendition of $G$ with a vortex cell $c_0$.
Notice that $c_0$ defines a society $(\sigma(c_0),\Omega_{c_0})$ where $V(\Omega_{c_0})$ is the set of vertices of $G$ corresponding $N_\rho(c_0)$.
The ordering $\Omega$ is obtained by traversing along the boundary of the closure of $c_0$ in anti-clockwise direction.
We call $(\sigma(c_0),\Omega_{c_0})$ as obtained above the \emph{vortex society} of $c_0$.

We define the \emph{depth} of a society $(G,\Omega)$ as the maximum cardinality of a transaction in $(G,\Omega)$.
The \emph{depth} of the vortex $c_0$ is thereby defined as the depth of its vortex society.
Given a $\Sigma$-rendition $\rho$ with vortices, we define the \emph{depth of $\rho$} as the maximum depth of its vortex societies.

\paragraph{Linear decompositions of vortices.}
Let $(G,\Omega)$ be a society.
A \emph{linear decomposition} of $(G,\Omega)$ is a labelling $v_1,v_2,\dots,v_n$ of $V(\Omega)$ such that $v_1,v_2,\dots,v_n$ appear in $\Omega$ in the order listed, together with sets $(X_1,X_2,\dots,X_n)$ such that
\begin{enumerate}
    \item $X_i\subseteq V(G)$ and $v_i\in X_i$ for all $i\in[n]$,
    \item $\bigcup_{i\in[n]}X_i=V(G)$ and for every $uv\in E(G)$ there exists $i\in[n]$ such that $u,v\in X_i$, and
    \item for every $x\in V(G)$ the set $\{ i\in[n] ~\!\colon\!~ x\in X_i \}$ forms an interval in $[n]$.
\end{enumerate}
The \emph{adhesion} of a linear decomposition is $\max \{ |X_i\cap X_{i+1}| ~\!\colon\!~ i\in[n-1] \}$.
The \emph{width} of a linear decomposition is $\max \{ |X_i| ~\!\colon\!~ i\in[n] \}$.
\medskip

It is easy to see that every society with a linear decomposition of adhesion at most $k$ has depth at most $2k$, where depth here is used in our sense of the definition.
The (partial) reverse of this observation was shown by Robertson and Seymour in \cite{RobertsonS1990Graph} (see also \cite{GorskySW2025polynomialboundsgraphminor}).

\begin{proposition}[Robertson and Seymour \cite{RobertsonS1990Graph}]\label{prop:depth_to_lin_decomp}
Let $k$ be a non-negative integer and $(G,\Omega)$ be a society of depth at most $k$.
Then $(G,\Omega)$ has a linear decomposition of adhesion at most $k$.
\end{proposition}

\subsection{Reconciliation}
There is a somewhat subtle issue in regards to combining a cylindrical rendition of a society with the rendition of a flat transaction within it.
One might expect that one can simply add these two renditions together to get a rendition that fits both what originally lies outside of the vortex and whatever is contained in the strip of the flat transaction.
But this ignores the fact that we cannot a priori know how exactly the structure of the cells in the overlap of these two graphs looks in the two renditions.
For this purpose we make use of a helpful result from \cite{KawarabayashiTW2021Quickly} (see Lemma 5.15) that allows us to marry the renditions of a strip society and a cylindrical rendition at the loss of a little bit of the nest.

Before we state this result, we must first define a slightly technical construction.
Let $(G,\Omega)$ be a society with a cylindrical rendition $\rho$ around a vortex $c_0$, with a nest $\mathcal{C} = \{ C_1, \ldots , C_s \}$, and let $\mathcal{P} = \{ P_1, \ldots , P_p \}$ be a transaction that is orthogonal to $\mathcal{C}$.
Further, let $C_i \in \mathcal{C}$ and let $(G',\Omega')$ be the $C_i$-society in $\rho$, then the \emph{restriction of $\mathcal{P}$ to $(G',\Omega')$} is the unique transaction $\mathcal{P}' = \{ P_1', \ldots , P_p' \}$ of order $p$ in $(G',\Omega')$ such that $P_j' \subseteq P_j$ for all $j \in [p]$.
Note that the paths in $\mathcal{P}'$ are uniquely determined due to $\mathcal{P}$ being orthogonal to $\mathcal{C}$.

\begin{proposition}[Kawarabayashi, Thomas, and Wollan \cite{KawarabayashiTW2021Quickly}]\label{lem:reconciliation}
    Let $s \geq 7$, let \((G, \Omega)\) be a society with a cylindrical rendition $\rho$ in a disk $\Delta$ around a vortex \(c_0\), and let \(\mathcal{C} = \{ C_1, \ldots, C_s \}\) be a nest in \(\rho\) and let \(\mathcal{Q}\) be an exposed, monotone transaction of order at least three in \((G, \Omega)\) which is orthogonal to \(\mathcal{C}\).
    Let \(Y_1\) and \(Y_2\) be the two segments obtained by deleting the end segments of \(\mathcal{Q}\) from \(\Omega\).
    Assume that there exists a linkage \(\mathcal{P} = \{P_1, P_2\}\) such that \(P_i\) links \(Y_i\) and \(V(\sigma(c_0))\) for \(i \in [2]\), \(\mathcal{P}\) is disjoint from \(\mathcal{Q}\), and \(\mathcal{P}\) is orthogonal to \(\mathcal{C}\).

    Let \(i \ge 7\), let \((G', \Omega')\) be the \(C_i\)-society in \(\rho\), let \(\mathcal{Q}'\) be the restriction of \(\mathcal{Q}\) to \((G', \Omega')\), and let \(\rho' \) be the restriction of \(\rho\) to \((G', \Omega')\).
    Then \(\mathcal{Q}'\) is exposed and monotone in \(\rho'\), \(\mathcal{Q}'\) is a crosscap transaction if and only if \(\mathcal{Q}\) is a crosscap transaction, and if the \(\mathcal{Q}\)-strip society is flat and isolated in \((G, \Omega)\), then the \(\mathcal{Q}'\)-strip society is flat and isolated in \((G', \Omega')\).
\end{proposition}

The following corollary illustrates why we say that this result allows us to reconcile the rendition $\rho$ with the rendition of the $\mathcal{Q}$-strip society.

\begin{corollary}
    Let $s \geq 7$, let \((G, \Omega)\) be a society with a cylindrical rendition $\rho$ in a disk $\Delta$ around a vortex \(c_0\), and let \(\mathcal{C} = \{ C_1, \ldots, C_s \}\) be a nest in \(\rho\) and let \(\mathcal{Q}\) be an exposed, monotone transaction of order at least three in \((G, \Omega)\) which is orthogonal to \(\mathcal{C}\).
    Further, let $(G',\Omega')$ be the $\mathcal{Q}$-strip society, let $H$ be the outer graph of $C_7$, and let \(Y_1\) and \(Y_2\) be the two segments obtained by deleting the end segments of \(\mathcal{Q}\) from \(\Omega\).
    Assume that there exists a linkage \(\mathcal{P} = \{P_1, P_2\}\) such that \(P_i\) links \(Y_i\) and \(V(\sigma(c_0))\) for \(i \in [2]\), \(\mathcal{P}\) is disjoint from \(\mathcal{Q}\), and \(\mathcal{P}\) is orthogonal to \(\mathcal{C}\).

    Then there exists a vortex-free $\Delta$-rendition of $(G',\Omega')$ and a vortex-free $\Sigma$-rendition $\rho''$ of $(G \cup G',\Omega)$, where $\Sigma$ is $\Delta$ if $\mathcal{Q}$ is a planar transaction and $\Sigma$ is the surface resulting from adding a crosscap into $\Delta$ if $\mathcal{Q}$ is a crosscap transaction, such that $\rho$ and $\rho''$ agree on $H$ and $\rho'$ and $\rho''$ agree on $G''$.

    In particular, the renditions $\rho'$ and $\rho''$ can be found in $\mathbf{O}(m)$-time.
\end{corollary}

We will need this idea in two contexts and in particular, we want to ensure that it interacts nicely with our notion of even-faced renditions.

\begin{lemma}\label{lem:reconciliationforevenfaced}
    Let $(G,\Omega)$ be a society with an even-faced, cylindrical rendition $\rho$ in a disk $\Delta$ around the vortex $c_0$ with a nest $\mathcal{C} = \{ C_1, \ldots , C_s \}$ and let $\mathcal{Q}$ be an even-faced, monotone transaction of order at least three with the $\mathcal{Q}$-society $(G',\Omega')$ that is orthogonal to $\mathcal{C}$.
    Further, let $(G',\Omega')$ be the $\mathcal{Q}$-strip society and suppose that there exists a vortex-free $\Sigma$-rendition $\rho'$ of $( G'' \coloneqq G - [\sigma_\rho(c_0)] \cup G', \Omega)$, where $\Sigma$ is $\Delta$, if $\mathcal{Q}$ is a planar transaction and $\Sigma$ is the surface resulting from adding into $\Delta$ a crosscap placed in $c_0$, if $\mathcal{Q}$ is a crosscap transaction, such that $\rho$ agrees with $\rho'$ on $G - [\sigma_\rho(c_0)]$ and the restriction of $\rho'$ to $G'$ is even-faced.

    Then every cell of $\rho'$ is parity-preserving and, if $\mathcal{Q}$ is a crosscap transaction, $\rho'$ is even-faced.
    If $\mathcal{Q}$ is instead a planar transaction, we let $P_1,P_2$ be the two boundary paths of $\mathcal{P}$ and let $\mathsf{C}_1,\mathsf{C}_2$ be the two distinct cycles in $C_1 \cup P_1 \cup P_2$ such that $\mathsf{C}_i$ contains an edge of $E(P_i) \setminus E(C_1)$ but no vertex of $P_{3-i}$ for both $i \in [2]$ and let $\Delta_i \subseteq \Delta$ be the disk bounded by the trace of $\mathsf{C}_i$.
    Then the trace of each grounded odd cycle in $\rho'$ that is found in $G''$ separates the interior of $\Delta_1$ from the interior of $\Delta_2$.
\end{lemma}
\begin{proof}
    We note that the fact that each cell of $\rho'$ is parity-preserving follows directly from the construction of $\rho'$.
    For the proof of remainder of our statement we translate $\rho'$ and $G''$ into a proper even-faced\footnote{An embedding of a graph in a surface is even-faced if all closed walks bounding faces are of even length.} embedding $\phi$ of a graph $G^\star$ in $\Sigma$ and then make observations about the non-contractible cycles of $G^\star$ in $\phi$.

    First, for every cell $c \in C(\rho')$ with $|N_{\rho'}(c)| = |\{ v_1,v_2,v_3 \}| \leq 3$ such that there exist $i,j \in [|N_{\rho'}(c)|]$ for which there does not exist a $v_i$-$v_j$-path in $\sigma_{\rho'}(c)$, we let $H_i$ be the component of $\sigma_{\rho'}(c)$ that contains $v_i$ and split $c$ into disjoint subdisks of $c$ each containing one of the components $H_j$ (if $H_i = H_j$ these two will occupy the same subdisk).
    Secondly, after eliminating all cells from the first step, if there exists a cell $c \in C(\rho')$ with $|N_{\rho'}(c)| = |\{ v_1,v_2,v_3 \}| = 3$ and there exists a vertex $v \in V(\sigma_{\rho'}(c))$ such that $\sigma_{\rho'}(c) - v$ contains no $v_i$-$v_j$-path for some $i \in [3]$ and both $j \in [3] \setminus \{ i \}$, we split $c$ into a disk having $v_i$ and $v$ on its boundary and a second disk having $v$ and the two vertices $v_j$ with $j \in [3] \setminus \{ i \}$ in their boundary in the straightforward way.

    We let $\rho''$ be the resulting $\Sigma$-decomposition which now notably has the following properties:
    For every cell $c \in C(\rho'')$, each pair of vertices in $N_{\rho''}(c)$ are connected via some path in $\sigma_{\rho''}(c)$ and in particular, if $|N_{\rho''}(c)| = |\{ v_1,v_2,v_3 \}| = 3$, there exist pairwise internally disjoint paths $P_{1,2},P_{2,3},P_{3,1}$ such that $P_{i,j}$ connects $v_i$-$v_j$.

    To build $G^\star$, we start with the vertices $N(\rho'')$, which are the \emph{nodes} of $G^\star$, and add an edge $uv$ to $G^\star$ if there exists a cell $c \in C(\rho'')$ such that $u,v \in N_{\rho''}(c)$ and there exists a $u$-$v$-path of odd length in $\sigma_{\rho''}$.
    Furthermore, we add a $u$-$v$-path of length 2 to $G^\star$ if there exists a cell $c \in C(\rho'')$ such that $u,v \in N_{\rho''}(c)$ and there exists a $u$-$v$-path of even length in $\sigma_{\rho''}$.
    By drawing within the cell corresponding to the paths and edges we added, we arrive at an embedding $\phi$ of $G^\star$ into $\Sigma$.
    We observe that by construction $\rho'$ is even-faced if and only if $\phi$ is even-faced.

    Suppose first that $\Sigma$ contains a crosscap.
    Let $C$ be some odd cycle in $G^\star$ whose embedding in $\phi$ describes a contractible cycle $\mathsf{C}$ in $\Sigma$.
    Then, since $\mathsf{C}$ is contractible, it bounds a disk $d$ which contains a planar drawing of a subgraph $H$ of $G^\star$ of which $C$ bounds a face.
    Since $C$ is odd, there exists another face in $H$ that is bounded by an odd cycle $C'$ in $G^\star$.
    We now note that, due to the fact that $C'$ bounds a face in $G^\star$, $C'$ corresponds to an odd cycle with a contractible trace in $\rho'$ that lies entirely in $G'$ or $G - [\sigma_\rho(c_0)]$.
    The first option contradicts the fact that the restriction of $\rho'$ to $G'$ is even-faced and the second contradicts the fact that $\rho$ is even-faced.
    This confirms our statement in the case in which $\mathcal{Q}$ is a crosscap transaction.

    Thus we may instead suppose that $\Sigma$ is $\Delta$ and $\mathcal{Q}$ is a planar transaction.
    We observe that there are exactly two faces of $G^\star$ in $\phi$ that are not found in subgraph of $G^\star$ corresponding to $G'$ or in the subgraph of $G^\star$ corresponding to $G - [\sigma_\rho(c_0)]$.
    The boundary of these two faces correspond to $\mathsf{C}_1,\mathsf{C}_2$ and accordingly, we call them $\mathsf{C}_1^\star,\mathsf{C}_2^\star$.
    Since all other faces of $G^\star$, including the ``outer face'' of $G^\star$ whose boundary contains $V(\Omega)$, are even, we note that $\mathsf{C}_1^\star$ is odd if and only if $\mathsf{C}_2^\star$ is odd.
    If these cycles are even, then by analogous arguments to the previous case, $\phi$ and thus $\rho'$ is even-faced.

    We may therefore further suppose that $\mathsf{C}_1^\star$ and $\mathsf{C}_2^\star$ are odd.
    Let $\mathsf{C}$ be an odd cycle in $G^\star$, then $\mathsf{C}$ describes a disk $\Delta' \subseteq \Delta$ and this disk must contain one odd face of $G^\star$, meaning that it contains either $\mathsf{C}_1^\star$ or $\mathsf{C}_2^\star$ in its entirety.
    In particular, the other cycle must lie outside of the interior of $\Delta'$, whilst possibly intersecting $\mathsf{C}$ and thus its boundary.
    We can therefore conclude that the trace of any odd cycle in $G''$ found in $\rho'$ that corresponds to $\mathsf{C}$ indeed separates the interior of $\Delta_1$ and $\Delta_2$, as they are defined in our statement.
    This concludes our proof.
\end{proof}

Despite the complexity of the above statement, it turns out that the more annoying option is immediately helpful in the context of our specific project, as it can easily be seen to yield $\mathscr{H}_k$ as an odd minor controlled by a mesh produced by the nest.
We believe that the type of transaction featured in the above statement plays a significant role in generalisations of the statements we present in this article.
For this reason, we give it a name and present several of our results in such a way that these transactions are a valid outcome.
For the purpose of our proofs, we will later make use of the fact that these transactions allow us to find a large parity handle (see \zcref{lem:get-parity-handle}).

Let $(G,\Omega)$ be a society with an even-faced, cylindrical rendition $\rho_0$ in the disk $\Delta$ with a nest $\mathcal{C} = \{ C_1, \ldots , C_s \}$ around the vortex $c_0$, and let $\mathcal{P}$ be an even-faced, planar transaction of order $p$ in $(G, \Omega)$ that is orthogonal to $\mathcal{C}$.
We call $\mathcal{P}$ an \emph{odd handle} if the following holds.

Letting $(G',\Omega')$ be the strip society of $\mathcal{P}$, there exists a vortex-free $\Delta$-rendition $\rho$ of $(G'' \coloneqq G' \cup (G - [\sigma_{\rho_0}(c_0)]), \Omega)$, for which we let $P_1,P_2$ be the two boundary paths of $\mathcal{P}$ and let $\mathsf{C}_1,\mathsf{C}_2$ be the two distinct cycles in $C_1 \cup P_1 \cup P_2$ such that $\mathsf{C}_i$ contains an edge of $E(P_i) \setminus E(C_1)$ but no vertex of $P_{3-i}$ for both $i \in [2]$ and let $\Delta_i \subseteq \Delta$ be the disk bounded by the trace of $\mathsf{C}_i$ in $\rho$.
Further, $G''$ contains an odd cycle that is grounded in $\rho$, every cell of $\rho$ is parity-preserving, and the trace of each grounded odd cycle in $\rho$ that is found in $G''$ separates the interior of $\Delta_1$ from the interior of $\Delta_2$.

We choose the term odd handle in reference to the fact that we now observe that $\mathcal{P}$ allows us to find a parity handle as an odd minor.

\begin{lemma}\label{lem:get-parity-handle}
    Let $s,k$ be positive integers with $s \geq 2k+4$.
    Let $(G,\Omega)$ be a society with an even-faced, cylindrical rendition $\rho_0$ in the disk $\Delta$ with a nest $\mathcal{C} = \{ C_1, \ldots , C_s \}$ around the vortex $c_0$, and let $\mathcal{P}$ be an odd handle of order at least $2k+2$.

    Then $G$ contains $\mathscr{H}_k$ as an odd minor controlled by a mesh whose horizontal paths are subpaths of distinct cycles from $\mathcal{C}$.
\end{lemma}

\subsection{Surface walls}\label{sec:surfacwalls}
In our local structure theorem we want to ensure that any part of the surface we add is sufficiently represented by a lot of grid-like infrastructure.
This part is inspired by the work of Thilikos and Wiederrecht on excluding graphs of bounded genus \cite{ThilikosW2024Excluding} and essentially the same definitions are featured in \cite{GorskySW2025polynomialboundsgraphminor}.

\paragraph{Annulus walls.}
Let $m,n$ be positive integers.
The \emph{$(n\times m)$-annulus grid} is the graph obtained from the $(n\times m)$-grid by adding the edges $\{ \{(i,1),(i,n)\} ~\!\colon\!~\ i\in[n] \}$.
The \emph{elementary $(n\times m)$-annulus wall} is the graph obtained from the $(n\times 2m)$-annulus grid by deleting all edges in the following set
\begin{align*}
    \big\{  \{(i,j),(i+1,j) \}  ~\!\colon\!~ i\in[n-1],\text{ }j\in[2m]\text{, and }i\not\equiv j\mod 2 \big\}.
\end{align*}
An \emph{$(n \times n)$-annulus wall} is a subdivision of the elementary $(n \times n)$-annulus wall.
We also write \emph{$n$-annulus wall or grid} as an abbreviation for an $(n \times n)$-annulus wall or grid.
\medskip

One can also see an annulus $(n\times m)$-wall as the graph obtained by completing the horizontal paths of a wall to cycles instead of discarding the vertices of degree one.
This viewpoint will be very helpful in the following constructions.
An $n$-annulus wall contains $n$ cycles $C_1,\dots,C_n$, such that $C_i$ consists exactly of the vertices of the $i$th row of the original wall.
We refer to these cycles as the \emph{base cycles} of the $n$-annulus wall.

\paragraph{Cylindrical meshes.}
In some settings it will be easier to work with a version of meshes that is cylindrical, since it is easier to argue for their existence than for annulus walls or grids.

Let $m,n$ be positive integers, let $M$ be a graph, and let $C_1, \ldots, C_m$ be cycles and $P_1, \ldots , P_n$ be paths in $M$ such that the following holds for all $i \in [m]$ and $j \in [n]$:
\begin{itemize}
    \item $C_1, \ldots , C_m$ are pairwise vertex-disjoint, $P_1, \ldots, P_n$ are pairwise vertex-disjoint, and $M = C_1 \cup \cdots \cup C_m \cup P_1 \cup \cdots \cup P_n$.

    \item $C_i \cap P_j$ is a path, and if $i \in \{ 1, m \}$ or $j \in \{ 1, n \}$, then $C_i \cap P_j$ has exactly one vertex,

    \item when traversing $C_i$ starting from an endpoint of $P_1\cap C_i$, then either the paths $P_1, \ldots , P_n$ are encountered in the order listed or the next $P_j$ one encounters is $P_n$ and from here the paths are encountered in the order $P_n,\dots,P_1$, and

    \item $P_j$ has one end in $C_1$ and the other in $C_m$, and when traversing $P_j$ starting from its endpoint on $C_1$, the cycles $C_1, \ldots , C_m$ are encountered in the order listed.
\end{itemize}
If the above conditions hold for $M$, we call $M$ an \emph{$(n \times m)$-cylindrical mesh}.
The cycles $C_1, \ldots, C_m$ are called the \emph{concentric cycles}, or \emph{cycles}, of $M$ and the paths $P_1, \ldots , P_n$ are called the \emph{radial paths}, or \emph{rails}, of $M$.
According to this definition, annulus walls are cylindrical meshes.
We also call $(n \times n)$-cylindrical meshes \emph{$n$-cylindrical meshes}.

\paragraph{Wall segments.}
Let $n$ be a positive integer.
An \emph{elementary $n$-wall-segment} is the graph $W_0$ obtained from the $(n\times 8n)$-grid by deleting all edges in the following set
\begin{align*}
    \big\{  \{(i,j),(i+1,j) \}  ~\!\colon\!~ i\in[n-1],\text{ }j\in[8n]\text{, and }i\not\equiv j\mod 2 \big\}.
\end{align*}
The vertices in $\{ (i,1)  ~\!\colon\!~ i\in[n] \}$ are said to be the \emph{left boundary} of the segment, while the vertices in $\{ (i,n)  ~\!\colon\!~ i\in[n] \}$ form the \emph{right boundary of the segment}.
Finally, we refer to the vertices in $\{ (1,i) ~\!\colon\!~ i=2j\text{, }j\in[1,8n] \}$ as the \emph{top boundary}.
\smallskip

An \emph{elementary $n$-handle-segment} is the graph obtained from the elementary $n$-wall-segment $W_0$ by adding the following edges, which we call \emph{handle edges},
\begin{align*}
    &\big\{ \{(1,2i),(1,6n+2-2i)\}  ~\!\colon\!~  i\in[1,n]  \big\}\\
    \cup \ &\big\{ \{(1,2i),(1,8n+2-2i)\}  ~\!\colon\!~ i\in[n+1,2n]  \big\}.
\end{align*}

An \emph{elementary $n$-crosscap-segment} is the graph obtained from the elementary $n$-wall-segment $W_0$ by adding the following edges, which we call \emph{crosscap edges},
\begin{align*}
    \big\{ \{(1,2i),(1,4n+2i)\}  ~\!\colon\!~  i\in[1,2n]  \big\}.
\end{align*}

An \emph{elementary $n$-vortex-segment} is the graph obtained from the disjoint union of two elementary $n$-wall segments $W_0$ and $W_1$ by making the $i$th top boundary vertex of $W_0$ adjacent to the $i$th top boundary vertex of $W_1$ for each $i \in [4n]$, and by making the $j$th left boundary vertex of $W_1$ adjacent to the $j$th right boundary vertex of $W_1$ for each $j \in [n]$.

We denote the $(n\times 4n)$-annulus wall defined on the vertex set of $W_1$ as above by $W$.
The base cycle $C_1,\dots C_n$ of $W$ are assumed to be ordered such that $C_n$ contains all vertices adjacent to $W_0$ and we call $\{ C_1,\dots, C_s\}$ the \emph{nest} of the elementary $n$-vortex segment.
We refer to $C_n$ as the \emph{outer cycle} and to $C_1$ as the \emph{inner cycle} of the elementary $n$-vortex segment.
Finally notice that there exist $4n$ pairwise disjoint ``vertical'' paths which are orthogonal to both the horizontal paths of $W_0$ and the cycles $C_1,\dots, C_n$.
The family of these paths is called the \emph{rails} of the elementary $n$-vortex-segment.
\smallskip

In all four types of segments we refer to the elementary wall segment $W_0$ as the \emph{base}.
If we do not want to specify the \textit{type} of an elementary wall-, handle-, crosscap-, or vortex-segment, we simply refer to the graph as a \emph{elementary $n$-segment}, or \emph{elementary segment} if $n$ is not specified.

Let $n$ and $\ell$ be positive integers and let $S_1,\dots,S_\ell$ be elementary $n$-segments.
The \emph{cylindrical closure} of $S_1,\dots,S_{\ell}$ is the graph obtained by introducing, for every $i\in[\ell-1]$ and every $j\in[n]$ and edge between the $j$th vertex of the right boundary of $S_i$ and the $j$th vertex of the left boundary of $S_{i+1}$ together with edges between the $j$th vertex of the right boundary of $S_{\ell}$ and the $j$th vertex of the left boundary of $S_1$.

\paragraph{Surface-walls.}
The \emph{elementary extended $n$-surface wall} with \emph{$h$ handles, $c$ crosscaps, and $b$ vortices} is the graph obtained from the cylindrical closure of $n$-segments $S_1,\dots,S_{h+c+b+1}$ such that there are exactly one elementary $n$-wall-segment, $h$ elementary $n$-handle-segments, $c$ elementary $n$-crosscap-segments, and $b$ elementary $n$-vortex-segments among the $S_i$.
An \emph{extended $n$-surface-wall} is a subdivision of an elementary $n$-surface-wall.
We refer to the tuple $(h,c,b)$ as the \emph{signature} of the extended $n$-surface-wall with $h$ handles, $c$ crosscaps, and $b$ vortices.
An extended surface wall without vortices is also simply called a \emph{surface wall}.

Notice that every extended $n$-surface-wall with signature $(h,c,b)$ contains an $(n \times (h+c+b+1))$-annulus wall consisting of $n$ cycles.
We refer to this wall as the \emph{base wall} of the extended $n$-surface-wall.
Let $C_1,\dots,C_n$ be the base cycles of the base wall.
We will usually assume that $C_1$ is the cycle that contains all top boundary vertices of all segments involved, while $C_n$ can be seen as the ``outermost'' cycle.
We refer to $C_n$ as the \emph{simple cycle} of the extended $n$-surface-wall.

\paragraph{Parity walls.}
Let $n$ be a non-negative integer.
Let $W$ be an extended $n$-surface-wall with a base wall $W'$ that is bipartite and with the following additional properties.
Furthermore, for each handle or crosscap segment $S$ of $W$, either $S$ has a 2-colouring that extends the 2-colouring of $W'$ restricted to $S \cap W'$, in which case we call $S$ \emph{even}, or, for each of the paths $P$ corresponding to the handle, respectively the crosscap segment of $S$ the graph $(S \cap W') \cup P$ is non-bipartite, in which case we call $S$ \emph{odd}.
For each vortex segments $S'$ of $W$ either $S'$ has a 2-colouring that extends the 2-colouring of $W'$ restricted to $S \cap W'$, in which case we call $S'$ \emph{even}, or each cycle of the nest of $S'$ is odd and each odd cycle in $S'$ separates the vertices of the inner cycle of $S'$ from the vertices of the base cycle of $W$, in which case we call $S'$ odd.
A wall with all of these properties is called a \emph{parity $n$-surface-wall}.
Let $h^e$ be the number of even handle segments in $W$, $h^o$ be the number of odd handle segments, $c^e$ be the number of even crosscap segments, $c^o$ be the number of odd crosscap segments, let $b^o$ be the number of odd vortex segments, and let $b^e$ be the number of even vortex segments.
We refer to the tuple $(h^o,h^e,c^o,c^e,b^o,b^e)$ as the \emph{signature} of the parity $n$-surface-wall with $h^o$ odd handles, $h^e$ even handles, $c^o$ odd crosscaps, $c^e$ even crosscaps, $b^o$ odd vortices, and $b^e$ even vortices.

We note that in our situation we will only ever consider parity $n$-surface-walls without odd handles and without odd vortices.
The terms \emph{base wall} and \emph{simple cycle} are then adapted to parity $n$-surface-walls in the natural fashion.
\section{Bipartite flat wall theorem}\label{sec:bipartiteflatwall}
In this section, we prove how to find an even-faced mesh or transaction in a given flat mesh or flat transaction, respectively.
In \cite{RobertsonS1995Grapha} and \cite{GorskySW2025polynomialboundsgraphminor}, a flat mesh is the starting point to find a $\Sigma$-decomposition of a graph almost embedded on a surface.
Similarly, for our structure theorem, an even-faced flat mesh is a starting point, and we repeatedly find even-faced flat transactions to construct an almost embedding on a surface.

\begin{definition}[Even-faced flat mesh]\label{def:flatmesh}
    Let $n \geq 2$ be an integer.
    Let $G$ be a graph, let $M \subseteq G$ be an $n$-mesh.
    We say that $M$ is \emph{flat} in $G$ if there exits a $\Sigma$-rendition $\rho$ of $G$ in a sphere $\Sigma$ with a single vortex $c_0$, such that $M$ is flat in $\rho$, and the trace of the perimeter of $M$ separates $c_0$ and the vertices in $N(\rho) \cap V(M)$ in $\Sigma$.
    We additionally say that $M$ is \emph{even-faced flat} in $G$ if $\rho$ is even-faced around $M$.
\end{definition}

\begin{theorem}[Gorsky, Seweryn, and Wiederrecht \cite{GorskySW2025polynomialboundsgraphminor}]\label{thm:flatmesh}
    Let $t,n'$ be integers with $t\geq 5$, $n'\geq 2$, let $n=100t^3(n'+2t+2)$, and let $G$ be a graph with an $n$-mesh $M$. Then either
    \begin{itemize}
        \item there exists a model $\mu$ of $K_t$ in $G$ which is controlled by $M$, or
        \item there exist a set $Z$ with $|Z|<16t^3$ and an $(n'\times n')$-submesh $M'$ of $M$ which is disjoint from $Z$ and flat in $G-Z$.
    \end{itemize}
    Furthermore, there exists $\mathbf{O}(|E(G)|)$-time algorithm which finds either the model $\mu$ or the set $Z$, the submesh $M'$, and a $\Sigma$-rendition witnessing that $M'$ is flat in $G-Z$.
\end{theorem}

We start with a basic lemma on the parity of paths connecting two vertices in non-bipartite graphs.

\begin{lemma}\label{lem:paritypaths}
    Let $G$ be a $2$-connected non-bipartite graph and let $u,v \in V(G)$ be distinct.
    Then there exist two $u$-$v$-paths $P_1$ and $P_2$ such that $P_1$ has an odd length and $P_2$ has an even length.
    Furthermore, there exists an $\mathbf{O}(|E(G)|)$-time algorithm that finds such a pair $P_1$ and $P_2$.
\end{lemma}
\begin{proof}
    Since $G$ is non-bipartite, $G$ contains an odd cycle $C$.
    By Menger's theorem, there are two disjoint $\{ u, v \}$-$V(C)$-paths $Q_1$ and $Q_2$.
    The endpoints $u_1,u_2$ of $Q_1$ and $Q_2$ on $C$, with $u_i \in V(Q_i)$ for $i \in [2]$, divide $C$ into two paths $R_1$ and $R_2$ with the same endpoints.
    Since the length of $C$ is odd, the parity of $R_1$ and $R_2$ must be different.
    Then the paths $Q_1u_1R_1u_2Q_2$ and $Q_1u_1R_2u_2Q_2$ are two $u$-$v$-paths with different parity.
\end{proof}

Next, we prove that a sufficiently large flat mesh contains an even-faced flat mesh of given size.

\begin{theorem}\label{thm:bipartitemesh}
    Let $k,h\geq 2$ be integers and let $G$ be a graph that contains a flat $5kh$-mesh $M$.
    Then one of the following holds:
    \begin{enumerate}
        \item $G$ contains an odd minor model $\mu$ of $\mathscr{U}_h$ that is controlled by $M$, or
        \item $G$ contains an even-faced flat $k$-submesh $M'$ of $M$.
    \end{enumerate}
    Furthermore, there exists an $\mathbf{O}(|E(G)|)$-time algorithm that finds either the model $\mu$ or the submesh $M'$, and a $\Sigma$-rendition witnessing that $M'$ is even-faced and flat.
\end{theorem}
\begin{proof}
    Let $\rho$ be a $\Sigma$-rendition of $G$ in which $M$ is flat, let $P_0, P_1, \cdots ,P_{5kh-1}$ be the horizontal paths of $M$, and let $Q_0, Q_1, \cdots ,Q_{5kh-1}$ be the vertical paths of $M$.    
    For $a \in [4h+1]$ and $b \in [4h+1]$, let $M_{a}^b$ be the unique block of $\bigcup_{(a-1)k+1 \leq i \leq ak} P_i \cup \bigcup_{(b-1)k+1 \leq j \leq bk} Q_j$.
    Then each $M_{a}^b$ is a $(k\times k)$-mesh.
    For each $a$ and $b$, the trace of the perimeter of $M_{a}^b$ bounds a closed disk $\Delta_a^b$.
    Since $M$ is a flat mesh, we may assume that $\Delta_a^b$ contains each edge of $M^b_a$.
    Let $\rho_a^b$ be the $\Sigma$-rendition obtained by merging the cells not in $\Delta_a^b$ in to a vortex $c_a^b$.
    Then $M_a^b$ is flat in $\rho_a^b$.
    If there is some $a,b$ such that $\rho_{a}^b$ is even-faced, then we found a desired flat even-faced $k$-mesh.
    Note that checking whether $\rho_a^b$ is even-faced or not costs us $\mathbf{O}(|E(G_a^b)|)$-time, where $G_a^b \coloneqq G - [\sigma_{\rho_a^b}(c_a^b)]$.
    So in total the whole process can be done in $\mathbf{O}(|E(G)|)$-time.
    
    Now suppose that no $\rho_{a}^b$ is even-faced.
    In this case, we will show that $G$ contains an even subdivision of the universal parity breaking grid of order $h$.
    Let $B_a^b$ be the block of $G_a^b$ containing the perimeter of $M_a^b$.
    Since $\rho_a^b$ is not even-faced, $B_a^b$ is non-bipartite.
    For each odd $a$ and odd $b$, fix a vertex $v_a^b \in P_{ak} \cap Q_{bk}$.
    
    For each odd $a$ and even $b$, let $u_a^b$ and $w_a^b$ be vertices such that $B \cap P_{ak} = u_a^bP_{ak}w_a^b$, where on $P_{ak}$, the four vertices $v_{a}^{b-1}$, $u_a^b$, $w_a^b$ and $v_{a}^{b+1}$ appear in the listed order.
    Then, according to \zcref{lem:paritypaths}, since $B_a^b$ is a $2$-connected non-bipartite graph, there are two $u_a^b$-$w_a^b$-paths that have lengths of differing parities.
    By concatenating them with $v_a^{b-1}P_{ak}u_a^b$ and $w_a^bP_{ak}v_a^{b+1}$, we obtain two $v_a^{b-1}$-$v_a^{b+1}$-paths $R_a^b$ and $\overline{R}_a^b$, where $R_a^b$ has even length and $\overline{R}_a^b$ has odd length.
    We may assume that that $R_a^b\cap G_a^{b-1}=\overline{R}_a^b\cap G_a^{b-1}$ and $R_a^b\cap G_a^{b+1}=\overline{R}_a^b\cap G_a^{b+1}$ are subpaths of $P_{ak}$.
    
    Similarly, for each even $a$ and odd $b$, we can obtain two $v_{a-1}^b$-$v_{a+1}^b$-paths $S_a^b$ and $\overline{S}_a^b$, where $S_a^b$ has even length and $\overline{S}_a^b$ has odd length, such that $S_a^b\cap G_{a-1}^{b} = \overline{S}_a^b\cap G_{a-1}^{b}$ and $S_a^b\cap G_{a+1}^{b} = \overline{S}_a^b\cap G_{a+1}^{b}$ are subpaths of $Q_{ak}$.
    Then for each $a$ and $b$ with $a+b$ odd (orange area in \zcref{fig:evenfacedflatmesh}), we can choose one of $R^a_b$ or $\overline{R}^a_b$ and $S^a_b$ or $\overline{S}_b^a$, so that the union of all chosen paths forms a graph that contains a universal parity breaking grid of order $k$.
    
    Note that since $M$ is a mesh, for each odd $a$ and odd $b$ (blue area in \zcref{fig:evenfacedflatmesh}), the intersection between $R_a^{b-1}$ and $S_{a-1}^b$ is a subpath of $P_{bk}\cap Q_{ak}$ who has $v_a^b$ as an endpoint.
    Similarly, the intersection between one of $R_a^{b-1},\overline{R}_a^{b-1},R_a^{b+1},\overline{R}_a^{b+1}$ and one of $S_{a-1}^{b},\overline{S}_{a-1}^{b},S_{a+1}^{b},\overline{S}_{a+1}^{b}$ is always a subpath of $P_{bk}\cap Q_{ak}$ who has $v_a^b$ as an endpoint.
    Hence, when we take the union of all chosen paths, there will be no cycle contained in one $M_a^b$.
    
\end{proof}

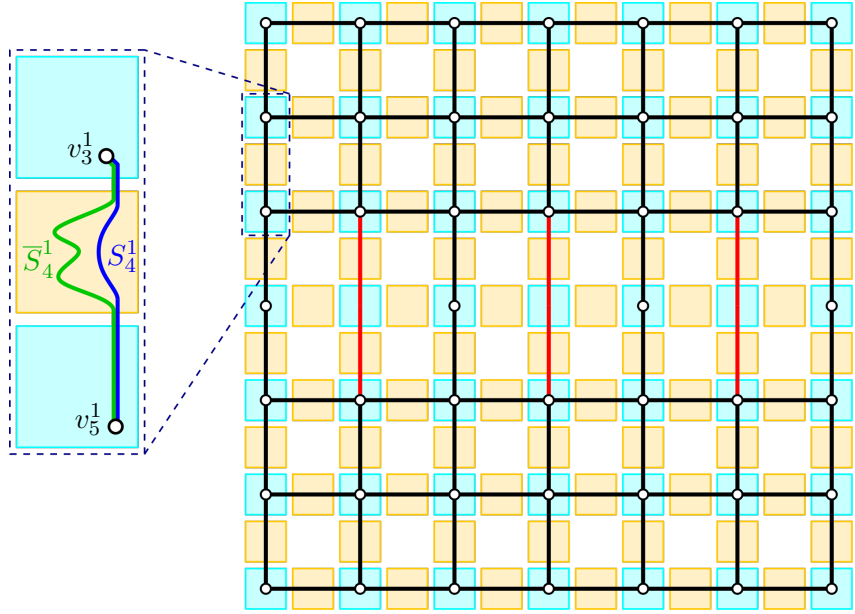
\begin{figure}[ht]
    \centering
    \begin{tikzpicture}[scale=0.7]
    \pgfdeclarelayer{background}
    \pgfdeclarelayer{foreground}
    \pgfsetlayers{background,main,foreground}
    \begin{pgfonlayer}{background}
        \pgftext{\includegraphics[width=16cm]{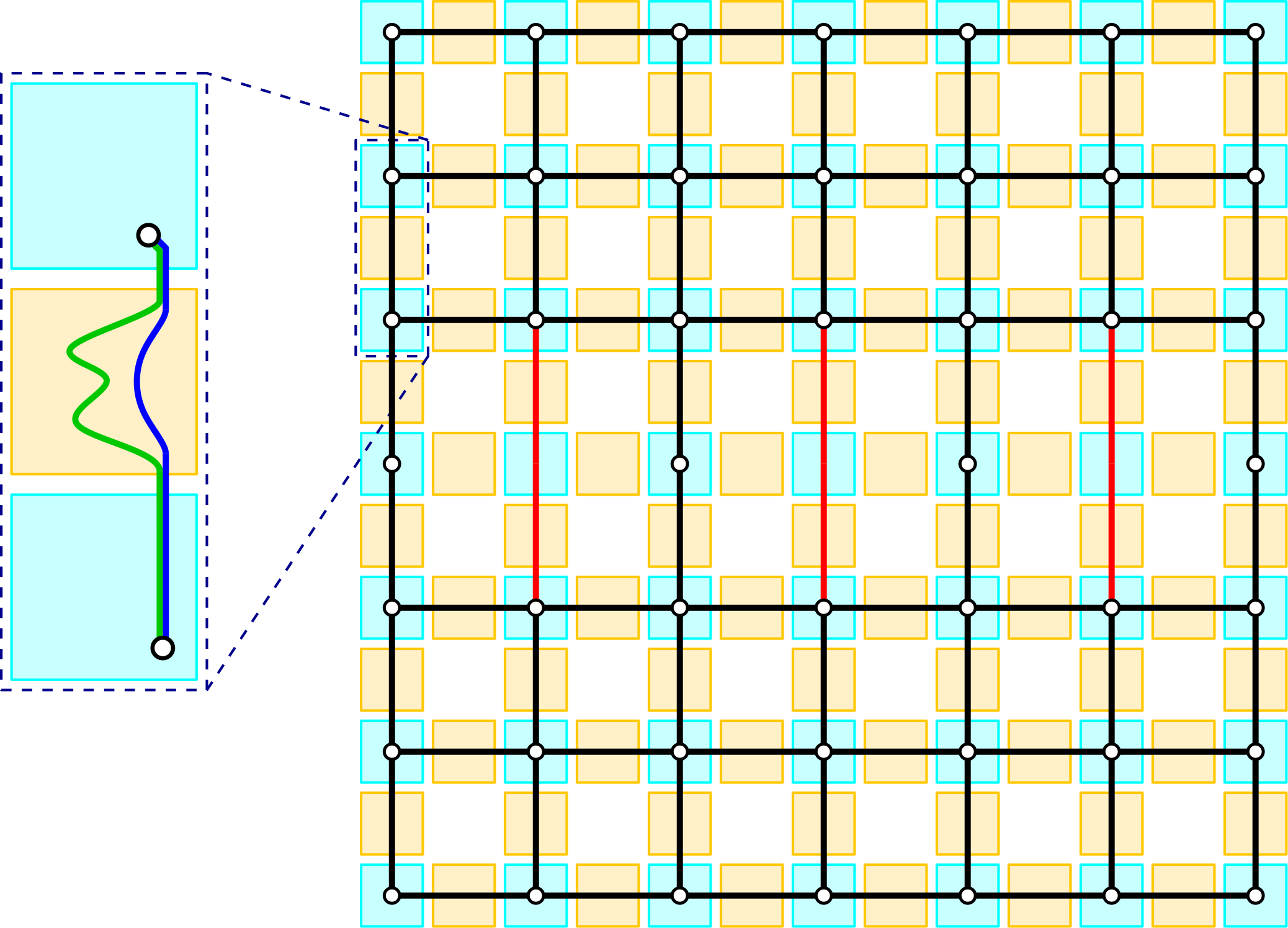}} at (C.center);
    \end{pgfonlayer}{background}
    \begin{pgfonlayer}{main}
    \end{pgfonlayer}{main}
        \node[]() at (-6.65,3) {$v_{3}^{1}$};
        \node[]() at (-6.5,-2.15) {$v_{5}^{1}$};
        \node[]() at (-7.45,0.9) {$\color{green!70!black}{\overline{S}_{4}^{1}}$};
        \node[]() at (-5.85,0.95) {$\color{blue}{S_{4}^{1}}$};
    \begin{pgfonlayer}{foreground}
    \end{pgfonlayer}{foreground}
    \end{tikzpicture}
    \caption{In each orange area, we can obtain two paths with different parity between two adjacent blue areas.
    We can choose the parity of each path between two blue area so that we can find a subgraph that has the universal parity breaking grid as an odd minor:}
    \label{fig:evenfacedflatmesh}
\end{figure}

By combining \zcref{thm:ktminormodel}, \zcref{thm:flatmesh} and \zcref{thm:bipartitemesh}, we can deduce the first main result of this section.

\begin{theorem}[Even-faced Flat Mesh Theorem]\label{thm:evenfacedflatmesh}
    Let $t,k,h$ be integers with $t\geq 5$, $k,h\geq 2$, let $t_0=\lceil c_{\ref{thm:ktminormodel}}t\sqrt{\log 12t} \rceil$, let $n=100t_0^3(5kh+2t_0+2)$, and let $G$ be a graph with an $n$-mesh $M$.
    Then either
    \begin{itemize}
        \item there exists an odd-minor model $\mu$ of $K_t$ in $G$ which is controlled by $M$,
        \item for some set $A$ of vertices with $|A|\leq 8t$, the large block of $\mathcal{T}_M-A$ is bipartite,
        \item there exists an odd-minor model $\mu'$ of a universal parity breaking grid or order $h$ in $G$ which is controlled by $M$, or
        \item there exist a set $Z$ with $|Z|<16t_0^3$ and an even-faced $k$-submesh $M'$ of $M$ that is disjoint from $Z$ and flat in $G-Z$.
    \end{itemize}
    Furthermore, there exists $\mathbf{O}(f_{\ref{thm:ktminormodel}}(t)|V(G)|^{\omega_{\ref{def:matrixmultconstant}}})$-time algorithm which finds either the model $\mu$, the set $A$, the model $\mu'$, or the set $Z$, the submesh $M'$, and a $\Sigma$-decomposition witnessing that $M'$ is even-faced and flat in $G-Z$.
\end{theorem}

Now we find an even-faced transaction in a strip society that is flat and isolated.
The following theorem is what we want to make an even-faced version.

\begin{lemma}[Lemma 7.1. from \cite{GorskySW2025polynomialboundsgraphminor}]\label{lem:flatteningtransactions}
    Let $t\geq 5,w'\geq 2$ be integers, let $w=97t^3(w'+40)+1$, let $(G,\Omega)$ be a society with a cylindrical rendition $\rho$ around a vortex $c_0$, and let $\mathcal{C}$ be a nest in $\rho$ with $|\mathcal{C}|\geq t+16$.
    Let $\mathcal{P}$ be a transaction of order $w$ in $(G,\Omega)$ that is monotone and orthogonal to $\mathcal{C}$.
    
    Then, either there exists a model of $K_t$ controlled by a mesh whose horizontal paths are subpaths of distinct cycles from $\mathcal{C}$, or there exist a set $A \subseteq V(\sigma(c_0))$ with $|A|\leq 8t^3$, and a subtransaction $\mathcal{P}' \subseteq \mathcal{P}$ of order $w'$ that consists of consecutive paths from $\mathcal{P}$, is disjoint from $A$ and isolated and flat in $(G-A,\Omega)$.
    
    Furthermore, there exists an $\mathbf{O}(|E(G)|)$-time algorithm, which finds $\mu$ or $A$ and $\mathcal{P'}$ as above.
\end{lemma}

The following lemma enables us, whenever we find an unexposed transaction that is orthogonal to the nest, to either find a $\mathscr{U}_h$ or refine the transaction to be even-faced.

\begin{lemma}\label{lem:bipartitetransaction}
    Let $s,h,p$ be positive integers with $s\geq h+1$ and let $b$ be a non-negative integer.
    Let $(G, \Omega)$ be a society with an even-faced cylindrical rendition $\rho_0$ in the disk $\Delta$ around the vortex $c_0$ with a cozy nest $\mathcal{C} = \{ C_1, \ldots , C_s \}$ and let $\mathcal{R}$ be a radial linkage of order $h$ where $Y$ is the set of endpoints of $\mathcal{R}$ in $\Omega$.
    Let $H \coloneqq G - [\sigma_{\rho_0}(c_0)]$.
    Further, let $\rho$ be a $\Delta$-rendition of $(G, \Omega)$ with $b$ vortices $c_1, \ldots , c_b$ such that $\rho$ and $\rho_0$ agree on the rendition of $H$.
    Additionally, let $\mathcal{P}$ be an unexposed transaction of order $2ph+bp+p$ in $(G, \Omega)$ that is orthogonal to $\mathcal{C}$.
    
    Then either there exists
    \begin{itemize}
        \item an odd-minor model $\mu$ of $\mathscr{U}_h$ controlled by a mesh whose horizontal paths are subpaths of distinct cycles from $\mathcal{C}$ such that for each horizontal or vertical path $X$ of $\mathscr{U}_h$, there exist $h$ internally disjoint $Y$-$\bigcup_{x\in V(X)}\mu(x)$ paths,
        \item  or there exists an separating even-faced transaction $\mathcal{P}' \subseteq \mathcal{P}$ of order $p$, with $(G',\Omega')$ being the $\mathcal{P}'$-strip society, such that $V(\sigma(c_i)) \cap V(G') = \emptyset$ for all $i \in [b]$.
    \end{itemize}
    In particular, there exists an algorithm that, when given $\rho$, $(G,\Omega)$, and $\mathcal{P}$ as input, finds one of these outcomes in time at most $\mathbf{O}(|E(G)|)$.
\end{lemma}
\begin{proof}
    Let $P_1,\ldots,P_{2ph+bp+p}$ be the paths in $\mathcal{P}$, naturally indexed.
    For each $i\in [2h+b+1]$, let $\mathcal{P}_i=\{P_{(i-1)p+1},P_{(i-1)p+2},\ldots,P_{ip}\}$.
    Let $(G_i,\Omega_i)$ be the $\mathcal{P}_i$-strip society.
    Since $\mathcal{P}$ is unexposed in $\rho$, each of the vortices $c_1, \ldots , c_b$ can be contained in at most one $G_i$.
    Thus, we can find $S \subseteq [2h+b+1]$ with $|S|=2h+1$ so that for each $i \in S$, the strip society $(G_i,\Omega_i)$ does not contain a vortex.
    Furthermore, for each $i\in S$, $\mathcal{P}_i$ is an separating flat transaction.

    If there is some $i \in S$ such that $\mathcal{P}_i$ is even-faced, we are done.
    So suppose that for each $i \in S$, $\mathcal{P}_i$ is not even-faced.
    Note that checking whether the strip society of $\mathcal{P}_i$ is even-faced or not can be done in $\mathbf{O}(|E(G_i)|)$-time, so in sum, the whole process can be carried out in $\mathbf{O}(|E(G)|)$-time.
    
    Now we will find a universal parity breaking grid $\mathscr{U}_t$ as an odd minor.
    Let $X$ and $Y$ denote the two unique $P_1$-$P_{2ph+bp+p}$-paths in $C_1$ of positive length.
    Let $X_i$ and $Y_i$ denote the two unique $P_{(i-1)p+1}$-$P_{ip}$-paths in $C_1$ of positive length so that $X_i$ is contained in $X$.
    The uniqueness of these paths stems from the fact that $\mathcal{P}$ is orthogonal to $\mathcal{C}$, which further implies that $X_i$ and $Y_i$ are disjoint.
    Let $\Delta_i$ be the closure of the disk in $\Delta$ minus the traces of $P_{(i-1)p+1}$, $P_{ip}$, $X_i$ and $Y_i$ whose interior does not contain any node belonging to a cycle in $\mathcal{C}$.
    Further, let $u_i \in P_{(i-1)p+1} \cap X_i$ be an endpoint of $X_i$ and let $v_i \in P_{ip} \cap Y_i$ be an endpoint of $Y_i$.
    
    We will show that for each $i \in S$, there are two $u_i$-$v_i$-paths $Q_i$ and $\overline{Q}_i$ whose traces are contained in $\Delta_i$, such that $Q_i$ has the same parity as the $u_i$-$v_i$-paths contained in $C_1$,\footnote{Since $\rho_0$ is even-faced all cycles in $\mathcal{C}$ have even length.} while $\overline{Q}_i$ has the opposite parity.
    Note that finding such paths can be done in time at most $\mathbf{O}(|E(G_i)|)$, so in total, the whole process can be done in time at most $\mathbf{O}(|E(G)|)$.

    Let $c_i \in C(\rho)$ be a non-parity-preserving cell in the restriction of $\rho$ to $G_i$.
    Since $\rho_0$ and $\rho$ agree on $H$ and $\rho_0$ is even-faced, each such $c_i$ can be found in $\Delta_i$.
    Let $n_i$ and $m_i$ be the nodes of $c$ such that there are two $n_i$-$m_i$ paths that have different parity.
    Then by concatenating them with two disjoint paths between $\{n_i,m_i\}$ and $\{u_i,v_i\}$, we obtain the two $u_i$-$v_i$-paths $Q_i$ and $\overline{Q}_i$ with the desired parities.
    
    Let $M_1$ be a mesh whose horizontal paths are subpaths of $C_1, \ldots , C_{h+1}$ and whose vertical paths are subpaths of $\{ P_{(i-1)p+1} \mid i \in S \}$ that lie outside of the disk bounded by $C_1$ and intersect with $X$.
    In turn, let $M_2$ be a mesh whose horizontal paths are subpaths of $C_1, \ldots , C_{h+1}$ and whose vertical paths are subpaths of $\{ P_{ip} \mid i \in S \}$ that lie outside of the disk bounded by $C_1$ and intersect with $Y$.
    For each $i$, we can choose $Q_i$ or $\overline{Q}_i$ and add it between $M_1$ and $M_2$, such that the resulting graph contains the universal parity breaking grid of order $t$ as an odd minor (see \zcref{fig:evenfacedtransaction}).

    Let $X$ be a horizontal or vertical path of $\mathscr{U}_h$.
    By the Menger's theorem, it is easy to see that there exists $h$ internally disjoint $Y$-$\bigcup_{x\in V(P)}\mu(x)$ paths.
\end{proof}    
\begin{figure}[ht]
    \centering
    \begin{tikzpicture}[scale=0.67]
    \pgfdeclarelayer{background}
    \pgfdeclarelayer{foreground}
    \pgfsetlayers{background,main,foreground}
    \begin{pgfonlayer}{background}
        \pgftext{\includegraphics[width=16cm]{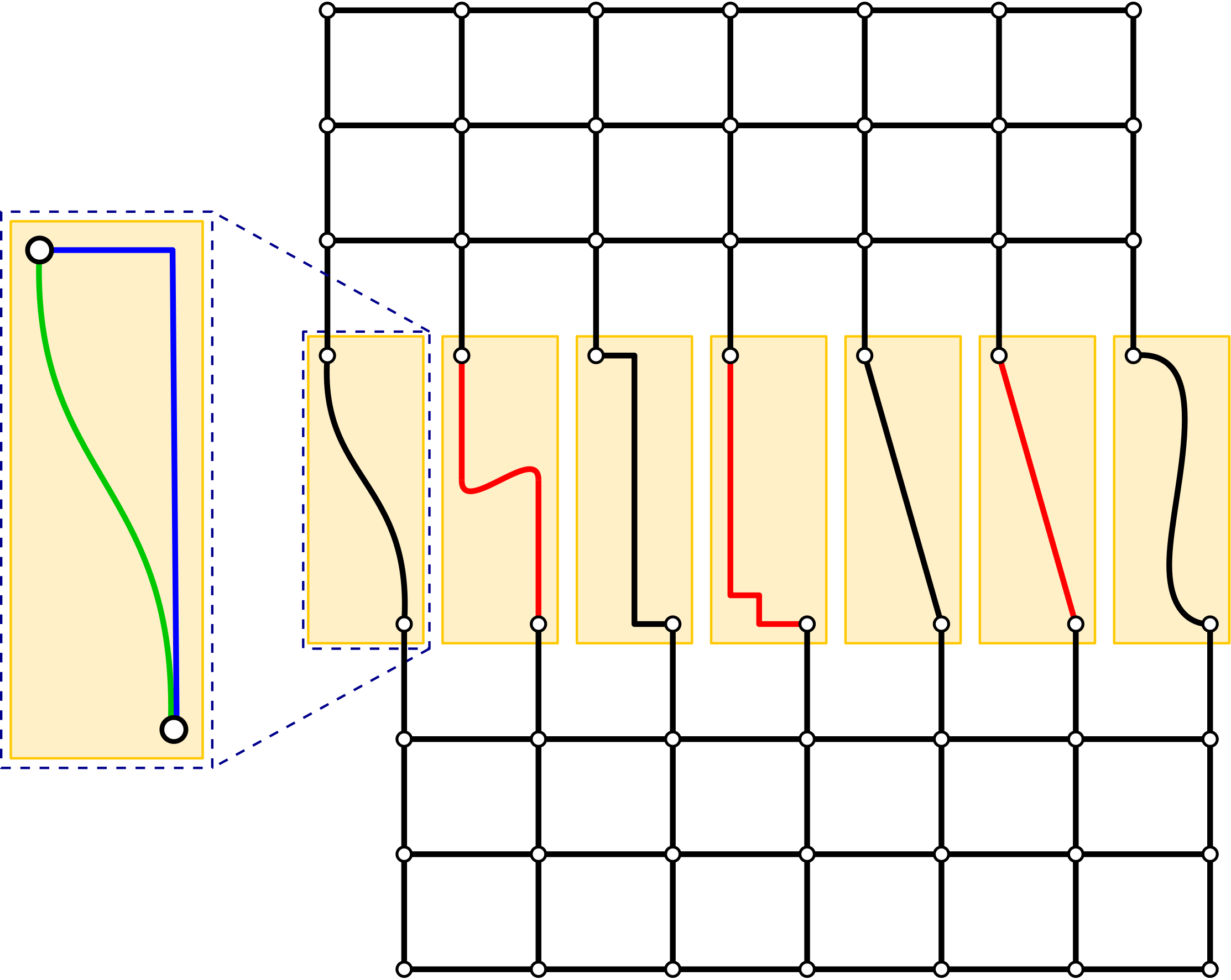}} at (C.center);
    \end{pgfonlayer}{background}
    \begin{pgfonlayer}{main}
    \end{pgfonlayer}{main}
        \node[]() at (-7,2.7) {$u_{1}$};
        \node[]() at (-6.25,-3) {$v_{1}$};
        \node[]() at (-7,0) {$\color{green!70!black}{Q_1}$};
        \node[]() at (-6.15,1.1) {$\color{blue}{\overline{Q}_1}$};
    \begin{pgfonlayer}{foreground}
    \end{pgfonlayer}{foreground}
    \end{tikzpicture}
    \caption{In each orange area, we can obtain two paths with different parity between two adjacent blue areas.
    We can choose the parity of each path between two meshes so that we can find a subgraph that has the universal parity breaking grid as an odd minor:}
    \label{fig:evenfacedtransaction}
\end{figure}
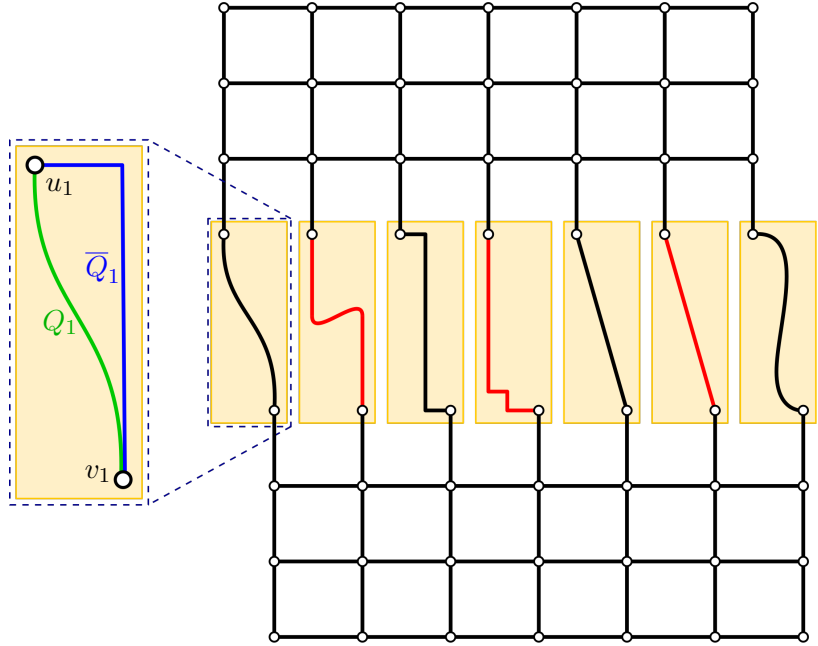

Here is one additional observation we need later:
Suppose that in \zcref{lem:bipartitetransaction}, we found an odd minor model $\mu$ of $\mathscr{U}_h$.

In addition, if $\mathcal{P}$ is planar, then we can further split the first case of \zcref{lem:bipartitetransaction} according to whether the cylindrical rendition of $(G,\Omega)$ remains even-faced after we add the $\mathcal{P}'$-strip society to it or not.

\begin{lemma}\label{lem:planarbipartitetransaction}
    Let $s,h,p$ be positive integers with $s\geq h+1$ and let $b$ be a non-negative integer.
    Let $(G, \Omega)$ be a society with an even-faced cylindrical rendition $\rho_0$ in the disk $\Delta$ around the vortex $c_0$ with a cozy nest $\mathcal{C} = \{ C_1, \ldots , C_s \}$ and let $\mathcal{R}$ be a radial linkage of order $h$ where $Y$ is the set of endpoints of $\mathcal{R}$ in $\Omega$.
    Let $H \coloneqq G - [\sigma_{\rho_0}(c_0)]$.
    Further, let $\rho$ be a $\Delta$-rendition of $(G, \Omega)$ with $b$ vortices $c_1, \ldots , c_b$ such that $\rho$ and $\rho_0$ agree on the rendition of $H$.
    Additionally, let $\mathcal{P}$ be a planar transaction of order $2ph+bp+p$ in $(G, \Omega)$ that is orthogonal to $\mathcal{C}$ and unexposed in $\rho$.
    
    Then either there exists
    \begin{itemize}
        \item an odd-minor model $\mu$ of $\mathscr{U}_h$ controlled by a mesh whose horizontal paths are subpaths of distinct cycles from $\mathcal{C}$ such that for each horizontal or vertical path $X$ of $\mathscr{U}_h$, there exist $h$ internally disjoint $Y$-$\bigcup_{x\in V(X)}\mu(x)$ paths,
        \item a separating even-faced planar transaction $\mathcal{P}' \subseteq \mathcal{P}$ of order $p$, with $(G',\Omega')$ being the $\mathcal{P}'$-strip society, such that $V(\sigma(c_i)) \cap V(G') = \emptyset$ for all $i \in [b]$, and the restriction of $\rho$ to $H \cup G'$ is even-faced, or
        \item there exists a separating odd handle $\mathcal{P}' \subseteq \mathcal{P}$ of order $p$.
    \end{itemize}
    In particular, there exists an algorithm that, when given $\rho$, $(G,\Omega)$, and $\mathcal{P}$ as input, finds one of these outcomes in time at most $\mathbf{O}(|E(G)|)$.
\end{lemma}
\begin{proof}
    By \zcref{lem:bipartitetransaction}, we can either conclude the first result or find an even-faced transaction $\mathcal{P}'\subseteq \mathcal{P}$ or order $p$, with $(G',\Omega')$ being the $\mathcal{P}'$-strip society, such that $V(\sigma(c_i))\cap V(G') = \emptyset$ for all $i\in [b]$.
    Suppose that we have the later option.
    
    If the restriction of $\rho$ to $H\cup G'$ is even-faced, then we conclude the second result.
    If not, by using \zcref{lem:reconciliationforevenfaced}, we can deduce the last result.
\end{proof}

Lastly, we also need a version of the lemmas for an exposed but flat transactions.
The proof is same as before, except that we do not have to care about vortices between paths if the input transaction is already flat.

\begin{lemma}\label{lem:exposedbipartitetransaction}
    Let $s,h,p$ be positive integers with $s\geq h+1$ and let $b$ be a non-negative integer.
    Let $(G, \Omega)$ be a society with an even-faced cylindrical rendition $\rho_0$ in the disk $\Delta$ around the vortex $c_0$ with a cozy nest $\mathcal{C} = \{ C_1, \ldots , C_s \}$ and let $\mathcal{R}$ be a radial linkage of order $h$ where $Y$ is the set of endpoints of $\mathcal{R}$ in $\Omega$. 
    Let $H \coloneqq G - [\sigma_{\rho_0}(c_0)]$.
    Additionally, let $\mathcal{P}$ be a flat transaction of order $2ph+p$ in $(G, \Omega)$ that is orthogonal to $\mathcal{C}$.
    
    Then either there exists an odd-minor model $\mu$ of $\mathscr{U}_h$ controlled by a mesh whose horizontal paths are subpaths of distinct cycles from $\mathcal{C}$ such that for each horizontal or vertical path $X$ of $\mathscr{U}_h$, there exist $h$ internally disjoint $Y$-$\bigcup_{x\in V(X)}\mu(x)$ paths, or there exists a flat even-faced transaction $\mathcal{P}' \subseteq \mathcal{P}$ of order $p$ which is orthogonal.

    In particular, there exists an algorithm that, when given $\rho$, $(G,\Omega)$, and $\mathcal{P}$ as input, finds one of these outcomes in time at most $\mathbf{O}(|E(G)|)$.
\end{lemma}

\section{Society classification}\label{sec:societyclassification}
Core to the approach of the graph minor structure theorem in \cite{GorskySW2025polynomialboundsgraphminor} is the Society Classification Theorem, a tool which originates from \cite{KawarabayashiTW2021Quickly}.
This theorem takes a society with a cylindrical rendition and a nest, and yields a set of outcomes which each somehow contribute positively towards the proof of the graph minor structure theorem.
Since we have similar plans for our later sections and future extensions of our results, we will be proving our own variant of this theorem.

The options in our society classification variant (see below) are roughly as follows:
\begin{enumerate}[label=\textbf{(S\arabic*)}]
    \item an odd-minor model of $K_t$,
    \item an odd-minor model of $\mathscr{U}_h$,
    \item a bipartite minor model of $K_t$ which identifies a large bipartite part of our graph after the deletion of a few apices,
    \item a crosscap transaction that is even-faced after the removal of some apices,
    \item a handle transaction that is even-faced after the removal of some apices,
    \item we find what is essentially an even-faced rendition of the society with a few vortices of bounded depth (supported by a surface wall), or
    \item an odd-minor model of $\mathcal{H}_h$.
\end{enumerate}
The last option here is a sort of escape valve for the earlier strange situation which we called a ``separating odd handle''.
This object can in fact be carried through to replace the last option of our result, but this requires significantly more effort, related to what is called ``padding'' in \cite{GorskySW2025polynomialboundsgraphminor}.
In the interest of not overcomplicating this already rather technical part of our proof, we leave this extension of the theorem to future work.

\begin{theorem}\label{thm:evensocietyclassification}
    There exist polynomial functions $\mathsf{apex}^\mathsf{bip},\mathsf{apex}^\mathsf{genus}_{\ref{thm:evensocietyclassification}}, \mathsf{loss}_{\ref{thm:evensocietyclassification}} \colon \mathbb{N} \rightarrow \mathbb{N}$, $\mathsf{cost}_{\ref{thm:evensocietyclassification}}\colon \mathbb{N}^2 \rightarrow \mathbb{N}$, $\mathsf{nest}_{\ref{thm:evensocietyclassification}}, \mathsf{radial}_{\ref{thm:evensocietyclassification}}\colon \mathbb{N}^3\to\mathbb{N}$, and $\mathsf{apex}^\mathsf{fin}_{\ref{thm:evensocietyclassification}},\mathsf{depth}_{\ref{thm:evensocietyclassification}} \colon \mathbb{N}^4\to\mathbb{N}$, such that the following holds.
    
    Let $t,h,s,k,p,r$ be positive integers with $s \geq \mathsf{nest}_{\ref{thm:evensocietyclassification}}(t,h,k)$.
    Let $(G,\Omega)$ be a society with an even cylindrical rendition $\rho$ in a disk $\Delta$, a cozy nest $\mathcal{C} = \{ C_1, \ldots , C_s \}$ in $\rho$ around the vortex $c_0$ with a radial linkage $\mathcal{R}$ of order $r \geq \mathsf{radial}_{\ref{thm:evensocietyclassification}}(t,h,k)$.
    Further, let $(G', \Omega')$ be the $C_{s - \mathsf{cost}_{\ref{thm:evensocietyclassification}}(t,k)}$-society in $\rho$.

    Then $G'$ contains a set $A \subseteq V(G')$ such that one of the following exists:
    \begin{enumerate}[label=\textbf{(S\arabic*)},ref={(S\arabic*)}]
        \item\label{item_s1} An odd $K_t$-minor model controlled by a mesh whose horizontal paths are subpaths of distinct cycles from $\mathcal{C}$. 

        \item\label{item_s2} An odd minor model for the universal parity breaking grid $\mathscr{U}_h$ of order $h$ controlled by a mesh whose horizontal paths are subpaths of distinct cycles from $\mathcal{C}$.

        \item\label{item_s3} A bipartite $K_t$-minor model $\varphi$ in $G - A$, with $|A| \leq \mathsf{apex}^\mathsf{bip}(t)$, such that the large block of $\mathcal{T}_\varphi - A$ is bipartite and the minor model is controlled by a mesh whose horizontal paths are subpaths of distinct cycles from $\mathcal{C}$.
        
        \item\label{item_s4} An even-faced crosscap transaction $\mathcal{P}$ of order $p$ in $(G'-A,\Omega')$, with $|A| \leq \mathsf{apex}^\mathsf{genus}_{\ref{thm:evensocietyclassification}}(t)$, and a nest $\mathcal{C}'$ in $\rho$ of order $s - (\mathsf{loss}_{\ref{thm:evensocietyclassification}}(t)+\mathsf{cost}_{\ref{thm:evensocietyclassification}}(t,k))$ around $c_0$ to which $\mathcal{P}$ is orthogonal.

        \item\label{item_s5} An even-faced handle transaction $\mathcal{P}$ of order $p$ in $(G'-A,\Omega')$, with $|A| \leq \mathsf{apex}^\mathsf{genus}_{\ref{thm:evensocietyclassification}}(t)$, and a nest $\mathcal{C}'$ in $\rho$ of order $s - (\mathsf{loss}_{\ref{thm:evensocietyclassification}}(t)+\mathsf{cost}_{\ref{thm:evensocietyclassification}}(t,k))$ around $c_0$ to which $\mathcal{P}$ is orthogonal.
        
        \item\label{item_s6} An even-face rendition $\rho'$ of $(G - A, \Omega)$ in $\Delta$ with breadth $b \in [\nicefrac{1}{2}(t-3)(t-4)+h-1]$ and depth at most $\mathsf{depth}_{\ref{thm:evensocietyclassification}}(t,k,p)$, $|A| \leq \mathsf{apex}^\mathsf{fin}_{\ref{thm:evensocietyclassification}}(t,k,p)$, and a parity $k$-surface-wall $D$ with signature $(0,0,0,0,0,b)$, such that $D$ is grounded in $\rho'$, the base cycles of $D$ are the cycles $C_{s-\mathsf{cost}_{\ref{thm:evensocietyclassification}}(t,k)-1-k},\dots,C_{s-\mathsf{cost}_{\ref{thm:evensocietyclassification}}(t,k)-1}$, and there exists a bijection between the vortices $v$ of $\rho'$ and the even vortex segments $S_v$ of $D$, where $v$ is the unique vortex contained in the disk $\Delta_{S_v}$ defined by the trace of the inner cycle of the nest of $S_v$, and $\Delta_{S_v}$ is chosen to avoid the trace of the simple cycle of $D$.
        
        \item\label{item_s7} An odd minor model for the parity handle $\mathcal{H}_h$ of order $h$ controlled by a mesh whose horizontal paths are subpaths of distinct cycles from $\mathcal{C}$.
    \end{enumerate}
    
    In particular, the set $A$, the odd or bipartite $K_t$-minor model, the odd $\mathscr{U}_h$-minor model, the odd $\mathscr{H}_h$-minor model, the transaction $\mathcal{P}$, the rendition $\rho'$, and the extended surface-wall $D$ can each be found in time $(f_{\ref{thm:ktminormodel}}(t)+\mathbf{poly}( h + s + p + k )) |E(G)||V(G)|^2$.
\end{theorem}
We will discuss explicit bounds for the above functions at the end of this section.

\subsection{A toolbox for society classification}
As part of our proof of our society classification variant, we will be using several tools from \cite{GorskySW2025polynomialboundsgraphminor}.
The simplest of which is an algorithm which finds a transaction in a society relatively efficiently.

\begin{proposition}\label{prop:findtransactionorlineardecomp}
    Let $p$ be a positive integer and let $(G,\Omega)$ be a society.
    Then it is possible to find either a transaction of order $p$ in $(G,\Omega)$ or a linear decomposition of \((G, \Omega)\) with adhesion less than \(p\) in time $\mathbf{O}(p|E(G)| \log |V(G)|)$.
\end{proposition}

Another simple lemma concerns the construction of cozy nests.

\begin{proposition}[Gorsky, Seweryn, and Wiederrecht \cite{GorskySW2025polynomialboundsgraphminor}]\label{prop:makenestcozy}
Let $s \geq 1$ be an integer, $(G,\Omega)$ be a society, $\rho$ be a cylindrical rendition of $(G,\Omega)$ in a disk with a nest $\mathcal{C} = \{ C_1, \dots , C_s \}$ around the vortex $c_0$.
Then there exists a cozy nest $\mathcal{C}'$ of order $s$ in $(G,\Omega)$ around $c_0$, such that outer graph of $C_1$ in $\rho$ contains $\bigcup \mathcal{C}'$, and an algorithm that finds $\mathcal{C}'$ in time $\mathbf{O}(s|E(G)|^2)$.
\end{proposition}

Of course in our proof of \zcref{thm:evensocietyclassification}, we make use of the version of the Society Classification theorem found in \cite{GorskySW2025polynomialboundsgraphminor}.
\begin{proposition}\label{thm:societyclassification}
    There exist polynomial functions $\mathsf{apex}^\mathsf{genus}_{\ref{thm:societyclassification}}, \mathsf{loss}_{\ref{thm:societyclassification}} \colon \mathbb{N} \rightarrow \mathbb{N}$, $\mathsf{nest}_{\ref{thm:societyclassification}}, \mathsf{cost}_{\ref{thm:societyclassification}} \colon \mathbb{N}^2 \rightarrow \mathbb{N}$, and $ \mathsf{apex}^\mathsf{fin}_{\ref{thm:societyclassification}}$,  $\mathsf{depth}_{\ref{thm:societyclassification}} \colon \mathbb{N}^3\to\mathbb{N}$, such that the following holds.
    
    Let $t,s,k,p$ be positive integers with $s \geq \mathsf{nest}_{\ref{thm:societyclassification}}(t,k)$.
    Let $(G,\Omega)$ be a society with a cylindrical rendition $\rho$ in a disk $\Delta$, a cozy nest $\mathcal{C} = \{ C_1, \ldots , C_s \}$ in $\rho$ around the vortex $c_0$ with a radial linkage $\mathcal{R}$ of order $\max(4k,8t+96)\nicefrac{1}{2}(t-3)(t-4)$.
    Further, let $(G', \Omega')$ be the $C_{s - \mathsf{cost}_{\ref{thm:societyclassification}}(t,k)}$-society in $\rho$.

    Then $G'$ contains a set $A \subseteq V(G')$ such that one of the following exists:
    \begin{enumerate}
        \item A $K_t$-minor model in $G$ controlled by a mesh whose horizontal paths are subpaths of distinct cycles from $\mathcal{C}$.

        \item A flat crosscap transaction $\mathcal{P}$ of order $p$ in $(G'-A,\Omega')$, with $|A| \leq \mathsf{apex}^\mathsf{genus}_{\ref{thm:societyclassification}}(t)$, and a nest $\mathcal{C}'$ in $\rho$ of order $s - (\mathsf{loss}_{\ref{thm:societyclassification}}(t)+\mathsf{cost}_{\ref{thm:societyclassification}}(t,k))$ around $c_0$ to which $\mathcal{P}$ is orthogonal.

        \item A flat handle transaction $\mathcal{P}$ of order $p$ in $(G'-A,\Omega')$, with $|A| \leq \mathsf{apex}^\mathsf{genus}_{\ref{thm:societyclassification}}(t)$, and a nest $\mathcal{C}'$ in $\rho$ of order $s - (\mathsf{loss}_{\ref{thm:societyclassification}}(t)+\mathsf{cost}_{\ref{thm:societyclassification}}(t,k))$ around $c_0$ to which $\mathcal{P}$ is orthogonal.

        \item A rendition $\rho'$ of $(G - A, \Omega)$ in $\Delta$ with breadth $b \in [\nicefrac{1}{2}(t-3)(t-4)-1]$ and depth at most $\mathsf{depth}_{\ref{thm:societyclassification}}(t,k,p)$, $|A| \leq \mathsf{apex}^\mathsf{fin}_{\ref{thm:societyclassification}}(t,k,p)$, and an extended $k$-surface-wall $D$ with signature $(0,0,b)$, such that $D$ is grounded in $\rho'$, the base cycles of $D$ are the cycles $C_{s-\mathsf{cost}_{\ref{thm:societyclassification}}(t,k)-1-k},\dots,C_{s-\mathsf{cost}_{\ref{thm:societyclassification}}(t,k)-1}$, and there exists a bijection between the vortices $v$ of $\rho'$ and the vortex segments $S_v$ of $D$, where $v$ is the unique vortex contained in the disk $\Delta_{S_v}$ defined by the trace of the inner cycle of the nest of $S_v$, and $\Delta_{S_v}$ is chosen to avoid the trace of the simple cycle of $D$.
    \end{enumerate}
    
    In particular, the set $A$, the $K_t$-minor model, the transaction $\mathcal{P}$, the rendition $\rho'$, and the extended surface-wall $D$ can each be found in time $\mathbf{poly}( t + s + p + k ) |E(G)||V(G)|^2$.
\end{proposition}

Using our existing tools, we can almost get to \zcref{thm:evensocietyclassification} with ease.
In particular, the first three options of \zcref{thm:societyclassification} can immediately be refined using our existing tools to yield one of the options in our theorem.
To deal with the last option, we will have to do a bit more work and import, as well as modify several results from \cite{GorskySW2025polynomialboundsgraphminor}.

\paragraph{Orthogonalisation.}
We first prove a variant of the following central lemma within \cite{GorskySW2025polynomialboundsgraphminor}, which allows us to find a transaction that is orthogonal to all cycles except the innermost one of the nest.

\begin{proposition}\label{lemma:orthogonal_transaction}
Let $s,p$ be positive integers.
Let $(G, \Omega)$ be a society with a cylindrical rendition $\rho$ around a vortex $c_0$ in the disk $\Delta$ with a cozy nest $\mathcal{C} = \{ C_1, \ldots , C_s \}$.
Further, let $\mathcal{P}$ be an exposed transaction of order $s(p-1)+1$ in $(G,\Omega)$ with the end segments $X_1, X_2$.

Then there exists an exposed transaction $\mathcal{P}'$ of order $p$ such that
\begin{enumerate}
    \item $\mathcal{P}'$ is orthogonal to $\{ C_2, C_3, \dots , C_s \}$,
    
    \item $\mathcal{P}'$ connects vertices of $X_1 \cap V(\mathcal{P})$ to vertices of $X_2 \cap V(\mathcal{P})$, and

    \item the intersection of $\bigcup\mathcal{P}'$ with the inner graph of $C_1$ in $\rho$ is fully contained in $C_1\cup\bigcup\mathcal{P}$.
\end{enumerate}
Moreover, there exists an algorithm that finds $\mathcal{P}'$ in time $\mathbf{O}( p |E(G)| )$.
\end{proposition}

Whilst this lemma is quite useful, for our application we cannot afford to lose a cycle of the nest each time we apply it.
Luckily, we will be using it in a context in which we have more structure.
In our setting the transaction together with the nest has a vortex-free rendition.
Here we show that we can find a transaction that is orthogonal to the entirety of the nest.
We state this lemma without reference to our special type of decompositions, as it may find applications outside of this work.
A similar result is presented in \cite{PaulPTW2024Obstructionsa} (see Lemma 4.16).

\begin{lemma}\label{lem:planar_orthogonal_transaction}
Let $s,p$ be positive integers.
Let $(G, \Omega)$ be a society with a cylindrical rendition $\rho$ around a vortex $c_0$ in the disk $\Delta$ with a cozy nest $\mathcal{C} = \{ C_1, \ldots , C_s \}$, and let $H$ be the outer graph of $C_1$ in $\rho$.
Further, let $\mathcal{P}$ be an unexposed transaction of order $s(p+1)+1$ in $(G,\Omega)$ with the end segments $X_1, X_2$, such that there exists a vortex-free rendition $\rho'$ of $(H \cup \bigcup \mathcal{P}, \Omega)$ that agrees with $\rho$ on $H$.

Then there exists a transaction $\mathcal{P}'$ of order $p$ such that
\begin{enumerate}
    \item $\mathcal{P}'$ is orthogonal to $\mathcal{C}$ and unexposed in $\rho$,
    
    \item $\mathcal{P}'$ connects vertices of $X_1 \cap V(\mathcal{P})$ to vertices of $X_2 \cap V(\mathcal{P})$, and

    \item the intersection of $\bigcup\mathcal{P}'$ with the inner graph of $C_1$ in $\rho$ is fully contained in $C_1\cup\bigcup\mathcal{P}$.
\end{enumerate}
Moreover, there exists an algorithm that finds $\mathcal{P}'$ in time $\mathbf{O}( p |E(G)| )$.
\end{lemma}
\begin{proof}
    Unsurprisingly, we start by applying \zcref{lemma:orthogonal_transaction} to $\mathcal{P}$, which yields a transaction $\mathcal{Q}$ of order $p' \coloneqq p+2$.
    We note that due to the existence of $\rho'$ both $\mathcal{P}$ and $\mathcal{Q}$ must be planar.
    Thus, we may assume that $\mathcal{Q} = \{ P_1, \ldots , P_{p'} \}$ is indexed naturally.

    According to \zcref{lemma:orthogonal_transaction}, $\mathcal{Q}$ is already orthogonal to all cycles in $\mathcal{C}$ except for $C_1$.
    Thus our goal will be the rectify this last non-orthogonal part of $\mathcal{Q}$.
    Furthermore, since the intersection of $\bigcup \mathcal{P}$ with the inner graph $H'$ of $C_1$ is found entirely in $C_1 \cup \bigcup \mathcal{P}$, the rendition $\rho'$ in particular tells us that there is a vortex-free rendition of $(H \cup \bigcup \mathcal{Q}, \Omega)$, which can be derived from $\rho'$.
    For the sake of simplicity we will directly work with $\rho'$.
    Since $\mathcal{C}$ is cozy and $\mathcal{Q}$ is orthogonal to $\mathcal{C} \setminus \{ C_1 \}$, we also note that there cannot exist a $C_1$-path in $H \cap \bigcup \mathcal{Q}$, as this path would have to be disjoint from $C_2$ and stick out away from $c_0$ in $\rho$.

    We let $a_i \in X_1$ and $b_i \in X_2$ be the endpoints of $P_i \in \mathcal{Q}$ for all $i \in [p']$.
    For both $j \in [2]$, let $\mathcal{R}_j$ be the set of $X_j$-$C_1$-paths in $\bigcup \mathcal{Q}$, which contains exactly $p'$ paths.
    For all $i \in [p']$, we let $a_i'$ be the second endpoint of the path in $\mathcal{R}_1$ that has $a_i$ as an endpoint and define $b_i'$ analogously.

    Let $i \in [p']$ and let $P$ be a non-trivial $P_i$-path in $C_1$.
    Suppose that $a_j'$ is found in $V(P)$ for some $j \in [p'] \setminus \{ i \}$, but $b_j'$ is not.
    Then $a_j'$ cannot be an endpoint of $P$ and is thus an internal vertex of $P$.
    Due to the existence of the vortex-free rendition $\rho'$ and the coziness of $\mathcal{C}$, there also cannot exist any non-trivial $C_1$-paths in $P_j \cap H$.
    Thus $P_j$ must somehow reach $b_j'$ starting from $a_j'$ via a subpath of $P_j$ found entirely in $H'$.
    This is clearly impossible if $b_j'$ is not found in $V(P)$.

    As a consequence, for all distinct $i,j \in [p']$, there exists a unique $P_i$-path in $C_1$ that contains both $a_j'$ and $b_j'$.
    Let $i \in [2,p'-1]$.
    Due to the fact that $\rho'$ is a vortex-free rendition, the trace of $P_i$ separates the traces of $P_{i-1}$ and $P_{i+1}$ on $\Delta$.
    Thus we further know that $a_{i-1}'$ and $a_{i+1}'$ are found in distinct $P_i$-paths of $C_1$.
    This in particular implies that for each $i \in [2,p'-1]$ there exist exactly two $P_i$-paths in $C_1$ that contain vertices of other paths from $\mathcal{Q}$, one which contains vertices of paths with a lower index than $i$ and the second containing vertices of paths with a higher index than $i$.
    Furthermore, each of these two paths contains vertices from both $X_1$ and $X_2$.

    Note that for $P_1$ and $P_{p'}$ there may not exist a non-trivial $C_1$-path in $G$.
    Let $S_1,S_2$ be the two unique $P_1$-$P_{p'}$-paths in $C_1$.
    We choose these indices such that $a_i' \in S_1$ and $b_i' \in S_2$ for all $i \in [2,p']$.
    For all paths $P_i \in \mathcal{Q}$, with $i \in [2,p'-1]$, there exists an $S_1$-$S_2$-path $Q_i$ in $P_i \cap H'$.
    Let $a_i''$ be the endpoint of $Q_i$ in $S_1$ and let $b_i''$ be the endpoint of $Q_i$ in $S_2$.

    For each $i \in [2,p'-1]$ we let $S_1^i$ be the unique $a_i'$-$a_i''$-path in $S_1$ and let $S_2^i$ be the unique $b_i'$-$b_i''$-path in $S_2$.
    We claim that for both $j \in [2]$, the path $S_j^i$ is disjoint from $\bigcup (\mathcal{Q} \setminus \{ P_i \})$.
    Suppose this is not the case, then $S_j^i$ contains at least one of the two $P_i$-paths, say $P'$, in $C_1$ that contain vertices of other paths from $\mathcal{Q}$.
    Note that $S_j^i$ is disjoint from $X_{3-j}$ and thus $P'$ is disjoint from $X_{3-j}$, which contradicts that $P'$ must contain vertices of both $X_1$ and $X_2$.
    Thus our claim holds.
    
    We now construct $p' - 2 = p$ paths $Q_2', \ldots , Q_{p'-1}'$ as follows.
    For each $i \in [2,p'-1]$ start at $a_i$ and follow $P_i$ until $a_i'$.
    From $a_i'$ walk towards $a_i''$ using $S_1^i$.
    Then walk towards $b_i''$ using $Q_i$, reach $b_i'$ via $S_2^i$ and reach $b_i$ via $P_i$.

    The paths $Q_2', \ldots , Q_{p'-1}'$ we can construct this way are pairwise disjoint, orthogonal to $\mathcal{C}$, exposed in $\rho$, and the intersection of $\bigcup \{ Q_2', \ldots , Q_{p'-1}' \}$ with $H'$ is fully contained in $C_1 \cup \bigcup \mathcal{P}$ by construction.
    This completes our proof.
\end{proof}

We will also require the ability to manipulate radial linkages.
In particular, we will need to make a radial linkage orthogonal in certain contexts.

Let $(G,\Omega)$ be a society, let $\rho$ be a $\Sigma$-rendition of $(G,\Omega)$ in a disk $\Delta$, and let $\mathcal{P}$ and $\mathcal{Q}$ be two linkages of the same order.
If for each path $Q \in \mathcal{Q}$ there exists a path $P \in \mathcal{P}$ with the same endpoints as $Q$, we say that $\mathcal{P}$ and $\mathcal{Q}$ are \emph{end-identical}.

\begin{proposition}[Gorsky, Seweryn, and Wiederrecht \cite{GorskySW2025polynomialboundsgraphminor}]\label{prop:radialtoorthogonal}
    Let $s,r$ with $r \leq s$.
    Let $(G,\Omega)$ be a society with a $\Sigma$-rendition $\rho$ with a cozy nest $\mathcal{C}$ of order $s$ and a radial linkage $\mathcal{R}$ of order $r$ for $\mathcal{C}$.

    Then there exists a radial linkage $\mathcal{R}'$ of order $r$ for $\mathcal{C}$ that is orthogonal to $\mathcal{C}$ and end-identical to $\mathcal{R}$.

    Moreover, there exists an algorithm running in $\mathbf{O}(r|E(G)|)$-time that finds $\mathcal{R}'$.
\end{proposition}

\paragraph{Finding specific (odd) minors in extended surface walls.}
Our ultimate goal is to prove by induction that we will at some point have to arrive in one of the options of \zcref{thm:evensocietyclassification} if we try to refine the kind of society mentioned in the setting of the theorem.
For this purpose we need one last tool which allows us to say that the iterative process that we use will have to terminate after finding too many disjoint vortices.

For this purpose we first import a lemma that builds a large $K_t$-minor model controlled by the base wall of an extended surface wall $W$ if $W$ has enough vortices hosting crosses.
\begin{proposition}[Gorsky, Seweryn, and Wiederrecht \cite{GorskySW2025polynomialboundsgraphminor}]\label{prop:cliques_in_extended_walls}
Let $t$ be a positive integer, $\Sigma$ be the sphere, and let $W$ be an extended $(2t+8)$-surface-wall with $\nicefrac{1}{2}(t-3)(t-4)$ vortices.
Moreover, let $G$ be the union of $W$ and $(t-3)(t-4)$ vertex-disjoint $V(W)$-paths $L_i,R_i$, with $i \in [\nicefrac{1}{2}(t-3)(t-4)]$, such that for every $i \in [\nicefrac{1}{2}(t-3)(t-4)]$, the paths $L_i$ and $R_i$ form a cross on the innermost cycle of the nest of the $i$th vortex-segment of $W$, both $L_i$ and $R_i$ are internally disjoint from $W$, and the endpoints of $L_i$ and $R_i$ are endpoints of rails of the $i$th vortex-segment of $W$.
Then $G$ contains a $K_t$-minor controlled by $W$.

Moreover, such a minor model of $K_t$ can be found in time $\mathbf{O}(|V(G)|)$.
\end{proposition}

We will need a modified version of this that allows us to find $\mathscr{U}_t$ as an odd minor instead.
Luckily this is not too difficult to prove.
Given an even cycle $C$ and a path $P$, that is internally disjoint from $C$ but has both endpoints in $V(C)$, we call $P$ \emph{parity breaking} if $C \cup P$ does not have a proper 2-colouring.

\begin{lemma}\label{lemma:parity_grid_in_extended_walls}
Let $h$ be a positive integer, $\Sigma$ be the sphere, and let $W$ be an parity $(2h+1)$-surface-wall with signature $(h^o,h^e,c^o,c^e,b^o,b^e)$ such that $b^e \geq h$.
Moreover, let $G$ be the union of $W$ and $h$ vertex-disjoint $V(W)$-paths $P_i$, with $i \in [h]$, such that for every $i \in [h]$, the endpoints of $P_i$ are on the innermost cycle $C_1^i$ of the nest of the $i$th even vortex-segment of $W$ and $P_i$ is parity breaking for $C_1^i$.
Then $G$ contains an odd $\mathscr{U}_h$-minor model controlled by $W$.

Moreover, an odd minor model of $\mathscr{U}_h$ can be found in time $\mathbf{O}(|V(G)|)$.
\end{lemma}
\begin{proof}
    Let $S_0,S_1,\cdots,S_{h}$ be the segments in $W$, where $S_0$ is a $(2h+1)$-wall segment while $S_1,\cdots,S_h$ are even $(2h+1)$-vortex segments.
    We find the odd minor model of a universal parity breaking grid of order $h$ as follows:\footnote{In our construction we will consider the $\mathscr{U}_h$ as the $90$-degree rotation of what is depicted in \zcref{fig:universalparitybreakinggrid} so that the red edges are in horizontal paths.}
    
    Let $C_1,\cdots,C_{2h+1}$ be the base cycles of $W$ in order such that $C_{2h+1}$ is the simple cycle of $D$.
    We will find vertex-disjoint paths from the $i$-th left boundary to the $i$-th right boundary of $S_0$, each disjoint from $S_0$, such that odd-numbered paths have the same parity with the paths with the same endpoints in $S_0$, while the even-numbered paths have the opposite parity.
    Then the resulting graph contains $\mathscr{U}_h$ as an odd minor.

    Now we explain how to find paths with the desired parity.
    We will find $2h+1$ vertex-disjoint paths $Q_1^i,Q_2^i,\cdots,Q_{2h+1}^i$ in $S_i$, where the endpoints of $Q_j^i$ are in the $j$-th left and right boundary of $S_i$, that have the same parity as path with same endpoints in the base wall, except for $Q_{2i}^i$, which has the opposite parity.

    For each $Q_j^i$, with $j\in [2i-1]$, we start from $j$-th left boundary of $S_i$ and follow along $C_i$ until we have met the $j$-th path of the rails of $S_i$.
    We then follow along this rail until we have met the outermost $j$-th cycle of the nest of $S_j$.
    From here we travel ``around'' the vortex following the nest of $S_j$, by choosing the direction that lets us encounter the last rail before we would encounter the $(2h)$-th rail.
    However, before that happens, when first meeting the $(8h+5-j)$-th rail, we go back through that rail to $C_i$ and follow it to arrive at the $j$-th right boundary of $S_i$.

    For each $Q_j^i$, with $j \in [2i+1,2h+1]$, we simply follow along $C_i$ from the left to the right boundary of $S_i$.
    Since $S_i$ is an even-vortex segment, each $P_j^i$ with $j\neq 2i$ has desired parity.
    
    Now we define $Q_{2i}^i$.
    We start from $j$-th left boundary of $S_i$ and follow along $C_i$ until we have met the $(2h+1)$-th path of the rails of $S_i$.
    Next we follow along this rail to the end until we have met the innermost cycle $C_1^i$.
    Let $x_1,x_2$ be the endpoints of $P_i$, and for $k \in [2]$, let $y_k$ be the endpoint of $(2h+k)$-th path of the rail.
    (We may assume that each rail meets with $C_1^i$ in unique vertex.)
    Since $C_1^i$ is $2$-connected, there exists an $\{x_1,x_2\}$-$\{y_1,y_2\}$-linkage in $C_1^i$.
    We move our growing path $Q_{2i^i}$ through one of the paths in this linkage, then follow $P_i$, and come back to the endpoint of $(2h+2)$-th path of the rails.
    Next we run down through this rail to go back to $C_i$, and then follow it to arrive at the $(2i)$-th right boundary of $S_i$.

    Note that since $P_i$ is a parity breaking path for $C_1^i$, the path $Q_{2i}^i$ has the opposite parity compared with a path that has same endpoints in the base wall.

    Now for each $j\in [2h+1]$, by concatenating the $Q_{j}^i$ using the segments of $C_j$, we obtain the desired path between $j$-th left and right boundary of $S_0$.\qedhere

    \begin{figure}[ht]
    \centering
    \begin{tikzpicture}[scale=0.7]
    \pgfdeclarelayer{background}
    \pgfdeclarelayer{foreground}
    \pgfsetlayers{background,main,foreground}
    \begin{pgfonlayer}{background}
        \pgftext{\includegraphics[width=16cm]{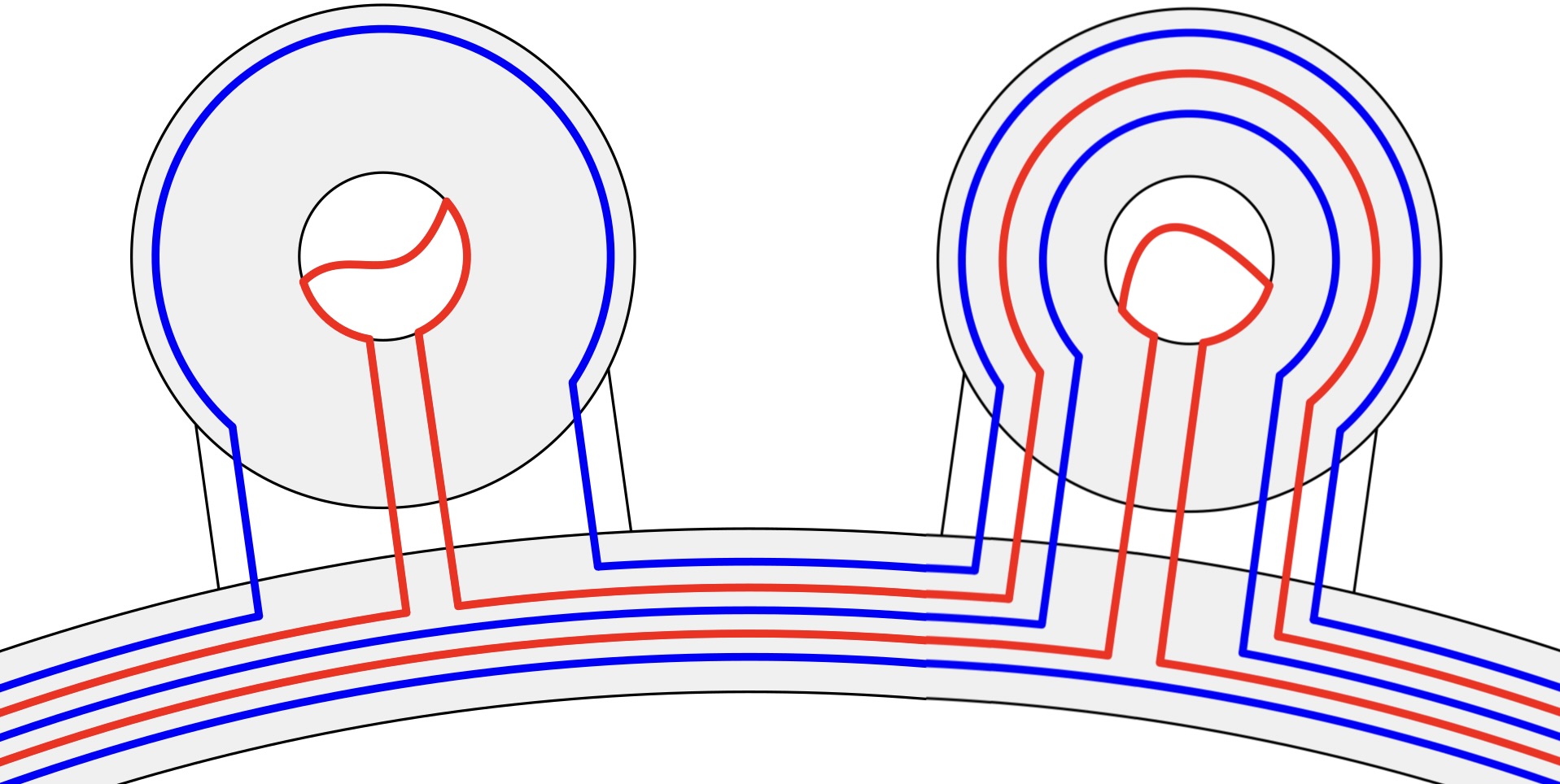}} at (C.center);
    \end{pgfonlayer}{background}
    \begin{pgfonlayer}{main}
    \end{pgfonlayer}{main}
    \begin{pgfonlayer}{foreground}
    \end{pgfonlayer}{foreground}
    \end{tikzpicture}
    \caption{A sketch of how to find $\mathscr{U}_h$ as an odd minor in an extended surface-wall by finding paths with desired parity through even vortex segments in the proof of \zcref{lemma:parity_grid_in_extended_walls}.}
    \label{fig:evenvortexsegments}
\end{figure}
\end{proof}

\paragraph{Finding an unexposed even-faced planar transaction.}
Our first task is to find an unexposed even-faced planar transaction to refine the given even-faced rendition. This can be done easily thanks to \zcref{lem:planar_orthogonal_transaction} and \zcref{lem:planarbipartitetransaction} after we find an unexposed transaction in the following lemma.

\begin{lemma}\label{lem:finding_unexposed_transaction}
    Let $s,p,d$ be positive integers and $b$ be a non-negative integer. Let $(G, \Omega)$ be a society with an even-faced rendition $\rho$ in the disk $\Delta$ with $b$ vortices $c_1, \ldots, c_b$ of depth at most $d$ and a cozy nest $\mathcal{C} = \{ C_1, \ldots , C_s \}$ and let $H$ be the outer graph of $C_1$ in $\rho$. Furthermore, let $\mathcal{P}$ be a transaction of order $p' \coloneqq (b+1)p + bd$ in $(G, \Omega)$ with the end segments $X_1$ and $X_2$. Then there exists an unexposed transaction $\mathcal{P}'$ of order $p$ such that there is an even-faced vortex-free rendition $\rho'$ of $(H \cup \bigcup \mathcal{P}', \Omega)$ that agrees with $\rho$ on $H$.
    
    In particular, there exists an algorithm that, when given $\rho$, $(G,\Omega)$, and $\mathcal{P}$ as input, finds $\mathcal{P}'$ in time $\mathbf{O}(|E(G)|)$. 
\end{lemma}
\begin{proof}
    We first show that for each vortex $c$, the number of paths in $\mathcal{P}$ that meets $c$ as the first vortex is bounded.
    Let $\mathcal{Q}=\{Q_1,\cdots,Q_{d'}\} \subseteq \mathcal{P}$ be the collection of such paths, and let $X_1'$ be the end segment of $\mathcal{Q}$ in $X_1$.
    Also let $(\sigma(c),\Omega_c)$ be the vortex society of $c$.
    Suppose that $d' \coloneqq |\mathcal{Q}|>d$.
    For each $i\in [d']$, let $v_i$ be the first vertex in $Q_i$ that meets $c$, and $Q_i'$ be the subpath of $Q_i$ between $v_i$ and $X_1'$.
    Let $Y_1$ be the segment of $\Omega_c$ between $v_1$ and $v_{d'}$ such that the disk $\Delta'\subseteq \Delta$ bounded by $Q_1', Q_{d'}', X_1'$, and $Y_1$ does not contain the vortex $c$.
    Note that $Y_1$ contains all $v_i$'s.
    Let $Y_2=\Omega_c\setminus Y_1$.
    Since $Q_1'$ and $Q_{d'}'$ meet no other vortex and cannot be crossed by other paths, each path in $\mathcal{Q}$ must pass through $Y_2$ to arrive at $X_2$.
    Let $w_i$ be the first vertex in $Q_i$ that meets $Y_2$, and let $u_i$ be the last vertex in $Q_i$ that meets with $Y_1$ before $w_i$.
    Then this gives $d'$ disjoint $Y_1$-$Y_2$-paths in $c$.
    Since $c$ has depth at most $d$, we have $d' \leq d$, a contradiction to the bound on the depth of $c$.

    Therefore, at most $bd$ paths in $\mathcal{P}$ can meet a vortex.
    Let $\mathcal{P}''\subseteq \mathcal{P}$ be the paths that do not meet a vortex.
    Now let $\mathcal{P}'_1, \ldots , \mathcal{P}_{b+1}'$ be disjoint collections of $p$ consecutive paths in $\mathcal{P}''$.
    We have $b$ vortices and each vortex can be contained in at most one strip society of $\mathcal{P}_i'$, so there is an $i \in [b+1]$ such that the strip society of $\mathcal{P}_i'$ does not contain a vortex.
    Then $\mathcal{P}_i'$ is the desired unexposed planar transaction of order at least $p$.
    Furthermore, by construction, there exists an even-faced vortex-free rendition $\rho'$ of $(H \cup \bigcup \mathcal{P}', \Omega)$ that agrees with $\rho$ on $H$.  
\end{proof}

For the next lemma, we define a function $\mathsf{depth}_{\ref{lem:separating_evenfaced_planar_transaction}} : \mathbb{N}^5 \to \mathbb{N}$ such that
\[ \mathsf{depth}_{\ref{lem:separating_evenfaced_planar_transaction}}(s,b,p,h,d) \coloneqq 2((b+1)s(2ph + bp + p + 1) + bd + b + 1) . \]
We now further refine the unexposed transaction found in \zcref{lem:finding_unexposed_transaction}, yielding one of our desired odd minors, a separating even-faced planar transaction, or a separating odd handle.

\begin{lemma}\label{lem:separating_evenfaced_planar_transaction}
    Let $s,h,p,d$ be positive integers with $s\geq h+1$ and let $b$ be a non-negative integer.
    Let $(G, \Omega)$ be a society with an even-faced cylindrical rendition $\rho_0$ in the disk $\Delta$ around the vortex $c_0$ with a cozy nest $\mathcal{C} = \{ C_1, \ldots , C_s \}$ and let $\mathcal{R}$ be a radial linkage of order $h$ where $Y$ is the set of endpoints of $\mathcal{R}$ in $\Omega$.
    Let $H \coloneqq G - [\sigma_{\rho_0}(c_0)]$.
    Further, let $\rho$ be a $\Delta$-rendition of $(G, \Omega)$ with $b$ vortices $c_1, \ldots , c_b$ of depth at most $d$ such that $\rho$ and $\rho_0$ agree on the rendition of $H$.

    Then one of the following holds:
    \begin{enumerate}
        \item The society $(G, \Omega)$ has depth at most $\mathsf{depth}_{\ref{lem:separating_evenfaced_planar_transaction}}(s,b,p,h,d)$.
                
        \item There is an odd minor model for the universal parity breaking grid $\mathscr{U}_h$ controlled by a mesh whose horizontal paths are subpaths of distinct cycles from $\mathcal{C}$ such that for each horizontal or vertical path $X$ of $\mathscr{U}_h$, there exist $h$ internally disjoint $Y$-$\bigcup_{x\in V(X)}\mu(x)$ paths.

        \item There is a separating even-faced planar transaction $\mathcal{P}$ of order $p$ that is orthogonal to $\mathcal{C}$ with $(G',\Omega')$ being the $\mathcal{P}$-strip society, such that $V(\sigma(c_i)) \cap V(G') = \emptyset$ for all $i \in [b]$, and the restriction of $\rho$ to $H \cup G'$ is even-faced.

        \item There is a separating odd handle $\mathcal{P}$ of order $p$ in $(G, \Omega)$.
    \end{enumerate}

    In particular, there exists an algorithm that, when given $\rho$, $(G,\Omega)$, and $\mathcal{P}$ as input, finds one of these outcomes in time $\mathbf{O}((ph + bp)|E(G)|)$. 
\end{lemma}
\begin{proof}
    We may assume that $(G, \Omega)$ has depth at least $\mathsf{depth}_{\ref{lem:separating_evenfaced_planar_transaction}}(s,b,p,h,d)+1$ and we can thus find a transaction $\mathcal{P}$ of order $p' \coloneqq \nicefrac{1}{2}\mathsf{depth}_{\ref{lem:separating_evenfaced_planar_transaction}}(s,b,p,h,d)$ in $\Delta$ by using \zcref{prop:findtransactionorlineardecomp}.
    Let $X, Y$ be the end segments of $\mathcal{P}$ in $(G, \Omega)$.
    By \zcref{lem:finding_unexposed_transaction}, we can find an unexposed transaction $\mathcal{P}'$ of order $s(2ph + bp + p + 1) + 1$ in time $\mathbf{O}(|E(G)|)$ such that there is a vortex-free rendition $\rho'$ of $(H \cup \bigcup \mathcal{P}', \Omega)$ that agrees with $\rho$ on $H$.
    Then due to \zcref{lem:planar_orthogonal_transaction}, we obtain an unexposed transaction of order $2ph + bp + p$ that is orthogonal to $\mathcal{C}$ in time $\mathbf{O}((ph + bp)|E(G)|)$.
    Moreover, thanks to \zcref{lem:planarbipartitetransaction} we obtain one of the second to the fourth options in time $\mathbf{O}(|E(G)|)$. 
\end{proof}

Finally, we will need the following simple rerouting argument that allows us to attach a grounded linkage to a radial linkage.

\begin{lemma}\label{lem:finding_linkages}
    Let $s,p,d,k$ be positive integers with $s \geq k$. Let $(G, \Omega)$ be a society with a $\Delta$-rendition $\rho$ in the disk $\Delta$, a nest $\mathcal{C} = \{ C_1, \ldots , C_s \}$, and $\mathcal{R}$ be a radial linkage for $\mathcal{C}$ of order $k$ with $B \subseteq V(\Omega)$ being the set of endpoints of the paths of $\mathcal{R}$ in $V(\Omega)$.
    Further, let $H$ be the inner graph of $C_1$ in $\rho$, and let $A \subseteq V(H)$ be a set of size $k' \geq k$.
    
    If there is an $A$-$C_l$-linkage $\mathcal{L}$ with $k \leq l \leq s$ of order $k'$ in $G$, then for every $A' \subseteq A$ of size $k$, there exists an $A'$-$B$-linkage $\mathcal{P}$ of order $k$ in $G$ such that $V(\mathcal{P})\cap V(H)\subseteq V(\mathcal{L})$.
    Moreover, there is an algorithm that finds $\mathcal{P}$ in time $\mathbf{O}(k |E(G)|)$.
\end{lemma}
\begin{proof}
    Let $A'\subseteq A$ be a set of size $k$.
    Suppose that there is no $A'$-$B$-linkage of order $k$ in $G$.
    We may assume that each edge of $G$ is contained in one of $V(\mathcal{C})$, $V(\mathcal{L})$, or $V(\mathcal{R})$.
    Then by \zcref{prop:mengersthm}, there exists a separator $S$ of size at most $k-1$ between $A'$ and $B$.
    Since $|S|\leq k-1$, there is some $i\in [l]$ such that $V(C_i)\cap S=\emptyset$.
    There are $k$ paths between $A'$ and $C_i$, which are subpaths of $\mathcal{L}$, so $C_i$ and some path in $\mathcal{L}$ intersect in $G-S$.
    As there are also $k$ paths between $B$ and $C_i$, which are subpaths of $\mathcal{R}$, we can also connect $C_i$ to $B$, which is a contradiction to $S$ being an $A'$-$B$-separator.
    Therefore, there is an $A'$-$B$-linkage $\mathcal{P}$ of order $k$ in $G$ such that $V(\mathcal{P})\cap V(H)\subseteq V(\mathcal{L})$.
    Furthermore, there is an algorithm that finds $\mathcal{P}$ in time $\mathbf{O}(k|E(G)|)$ (see the discussion below \zcref{prop:mengersthm}).
\end{proof}

We next introduce several notions that we will use to prove our main theorem.

\paragraph{Society configurations.}
Let $s,k$ be positive integers.
We call the tuple $\mathfrak{S} = ((G, \Omega), \rho, c, \Delta, \mathcal{C},\mathcal{L},X,Y)$ an \emph{$(s, k)$-society configuration} if
\begin{enumerate}
    \item $(G,\Omega)$ be a society with a cylindrical rendition $\rho$ around a vortex $c$ in the disk $\Delta$ with a nest $\mathcal{C}=\{C_1,\ldots,C_s\}$,
    \item $\mathcal{L}$ is a radial linkage of order $k$,
    \item $X$ is the set of endpoints of $\mathcal{L}$ in $C_1$, and
    \item $Y$ is the set of endpoints of $\mathcal{L}$ in $\Omega$.
\end{enumerate}

\paragraph{Moved in society configurations.}
Let $\mathfrak{S}=((G, \Omega), \rho, c, \Delta, \mathcal{C}=\{C_1,\ldots,C_s\}, \mathcal{L}, X,Y)$ be an $(s, k)$-society configuration.
We define an $(s-m-1, k)$-society configuration $\mathfrak{S}^{-m}=((G', \Omega'), \rho', c, \Delta', \mathcal{C}', \mathcal{L}', X', Y')$ that is obtained by \emph{moving in by $m$} by removing the $m+1$ outermost cycles of the nest $\mathcal{C}$ as follows:
Let $(G', \Omega')$ be the $C_{s-m}$-society in $\rho$ and $\Delta'$ be the disk bounded by the trace of $C_{s-m}$.
Furthermore, we let $\rho'$ be the restriction of $\rho$ to $G'$, let $\mathcal{C}' \coloneqq \{C_1,\ldots,C_{s-m-1}\}$, and let $\mathcal{L}'$ be the set of $\Omega'$-$V(C_1)$ subpaths of the paths in $\mathcal{L}$.
Also, let $X'$ and $Y'$ be the set of endpoints of $\mathcal{L}'$ in $C_1$ and $\Omega'$, respectively.
Note that there are $m$ cycles $C_{s-m+1},\ldots,C_s$ that are disjoint with $G'$.

\paragraph{Dividing a society configuration into two parts}
In the course of the proof of the society classification theorem, we repeatedly divide a society configuration into two parts.
Let $\mathfrak{S}=((G,\Omega),\rho,c,\Delta,\mathcal{C}=\{C_1,\ldots,C_s\},\mathcal{L},X,Y)$ be an $(s,k)$-society configuration.
Let $\mathcal{P} = \{P_1, \ldots, P_p\}$ be a separating even-faced planar transaction of order $p$ that is orthogonal to $\mathcal{C}$ with $(G_{\mathcal{P}},\Omega_{\mathcal{P}})$ being the $\mathcal{P}$-strip society, such that the restriction of $\rho$ to $(G-[\sigma_{\rho}(c)]) \cup G_{\mathcal{P}}$ is even-faced.
We \emph{divide the $(s,k)$-society configuration $\mathfrak{S}$ by the transaction $\mathcal{P}$ of order $p = 2k+2s+2$} into two smaller $(s, k)$-society configurations $\mathfrak{C}^L$ and $\mathfrak{C}^R$ as follows.

\begin{figure}[!h]
    \centering
    \begin{tikzpicture}[scale=2.3]
    \pgfdeclarelayer{background}
    \pgfdeclarelayer{foreground}
    \pgfsetlayers{background,main,foreground}
    \begin{pgfonlayer}{background}
        \pgftext{\includegraphics[width=6cm]{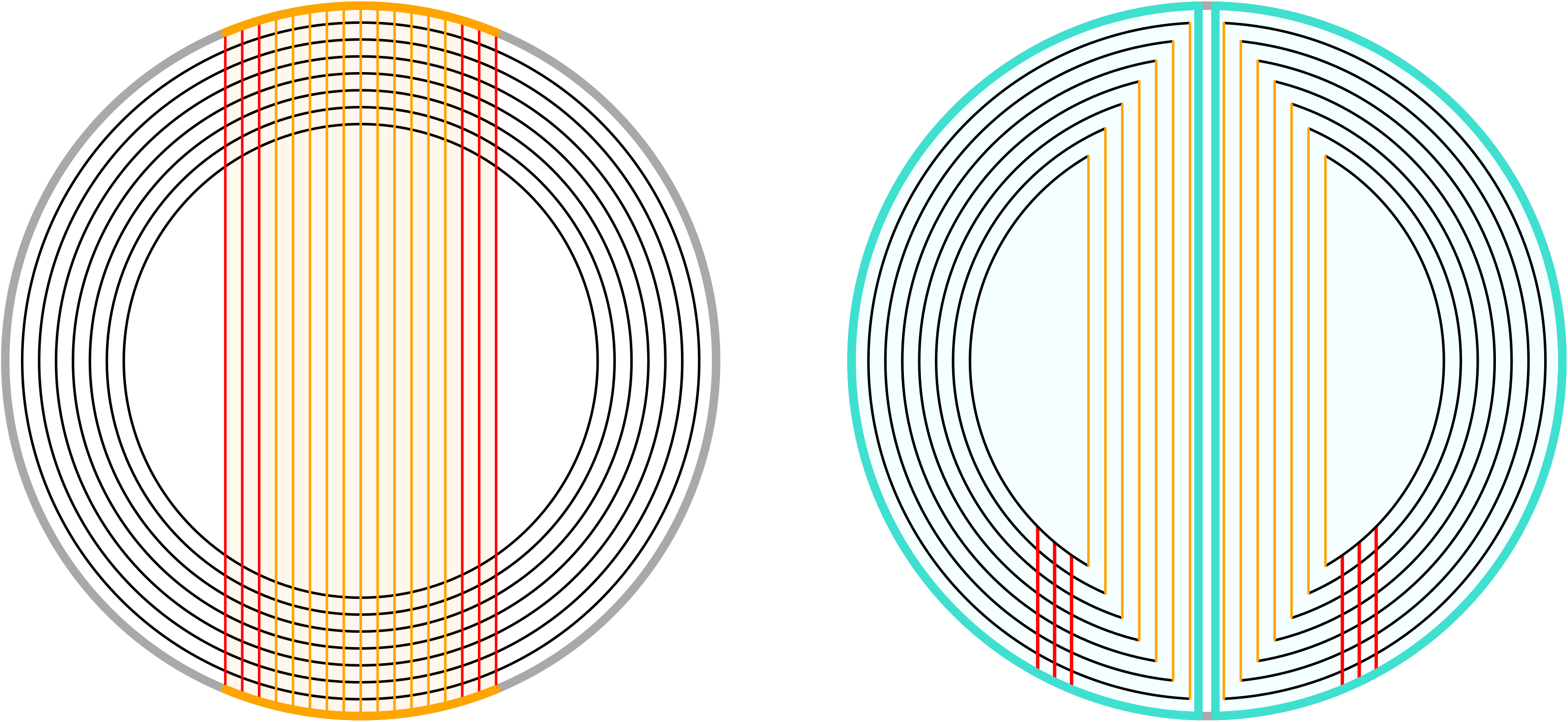}} at (C.center);
    \end{pgfonlayer}{background}
    \begin{pgfonlayer}{main}
        \node[]() at (-2.75,1.2) {$(G,\Omega)$};
        \node[]() at (-2.27,0) {$\color{orange}{\mathcal{P}}$};
        \node[]() at (0.4,1.2) {$\color{cyan!90!black}{(G^L,\Omega^L)}$};
        \node[]() at (2.83,1.2) {$\color{cyan!90!black}{(G^R,\Omega^R)}$};
        \node[]() at (1.05,-0.5) {$\color{red}{\mathcal{L}^L}$};
        \node[]() at (2.22,-0.5) {$\color{red}{\mathcal{L}^R}$};
    \end{pgfonlayer}{main}
    \begin{pgfonlayer}{foreground}
    \end{pgfonlayer}{foreground}
    \end{tikzpicture}
    \caption{An illustration of dividing a society configuration into two parts: 
    We start from a society $(G,\Omega)$ with a separating even-faced planar transaction $\color{orange}{\mathcal{P}}$ (Left).
    We define a two new societies $\color{cyan!90!black}{(G^L,\Omega^L)}$ and $\color{cyan!90!black}{(G^R,\Omega^R)}$ whose nests and radial linkages $\color{red}{\mathcal{L}^L}$ and $\color{red}{\mathcal{L}^R}$ are obtained by using the cycles and paths from previous nests and transactions (Right).}
    \label{fig:vortexrefinementtree}
\end{figure}

\begin{itemize}
    \item First we define four subtransactions $\mathcal{P}^L$, $\mathcal{P}^R$, $\mathcal{P}^{L,M}$, $\mathcal{P}^{R,M}$ of $\mathcal{P}_i$ in such a way that there are $k$ paths in $\mathcal{P}^L$ and $\mathcal{P}^R$ taken from the left-most $k$ and the right-most $k$ paths of $\mathcal{P}_i$. Moreover, the remaining paths in $\mathcal{P}_i$ are split into two so that there are $s+1$ paths in $\mathcal{P}^{L,M} = \{P_1^L, \ldots, P_{s+1}^L\}$ and in $\mathcal{P}^{R,M} = \{P_1^R, \ldots, P_{s+1}^R\}$, the former indexed naturally and the latter indexed backwards with respect to the indices of paths in $\mathcal{P}_i$.
    \begin{align*}
        \mathcal{P}^L =& \{P_1, \ldots, P_k\},\\
        \mathcal{P}^{L,M} =& \{P_{k+1}, \ldots, P_{k+s+1}\}, \text{ where we let } P_a^L\coloneqq P_{k+a} \text{ for } a \in [s+1],\\ 
        \mathcal{P}^{R,M} =& \{P_{k+s+2}, \ldots, P_{k+2s+2}\}, \text{ where we let } P_{a}^R\coloneqq P_{k+2s+3-a}  \text{ for } a \in [s+1], \text{ and }\\
        \mathcal{P}^R =& \{P_{k+2s+3}, \ldots, P_{2k+2s+2}\}.
    \end{align*}

    \item Let $m \in \{L, R\}$.
    For $l \in [s]$, let $u_l^m, v_l^m$ be the vertices of $P_{l}^m$ in $C_l$ such that the $(u_l^m, v_l^m)$-path does not contain any vertices in $C_l$ other than its endpoints. Let $Q_l^{m}$ be the path in $C_l$ with $u_l^m$ and $v_l^m$ being its endpoints such that it does not intersect $P_{l+1}^m$.
    Further, let $C_l^{m}$ be the cycle defined by the union of two paths $u_l^mP_{l}^mv_l^m$ and $Q_l^{m}$.
    Now we define the elements of $\mathfrak{C}^m$ as follows:

    \begin{enumerate}
        \item The disk $\Delta^m \subseteq \Delta_{i-1}$ is the closure of the component of $\Delta_{i-1} - T(P_{s+1}^m)$ which contains $T(P_s^m)$, where $T(P_{l}^m)$ is the trace of $P_{l}^m$ in $\rho$ for $l \in [s+1]$.

        \item The society $(G^m, \Omega^m)$ is the $\Delta^m$-society in $\rho$.
        Note that we have $V(\Omega^m)\subseteq V(\Omega)\cup V(P^m_{s+1})$.
        
        \item The even-faced cylindrical rendition $\rho^m$ of $G^m$ around $c^m$ is a restriction of $\rho$ to $\Delta^m$.

        \item The vortex $c^m$ is $C_1^m$-disk in $\rho^m$.
        
        \item The nest is $\mathcal{C}^m = \{ C_1^{m}, \ldots , C_{s}^{m} \}$.

        \item\label{item:dividing_society_config_6} To define the radial linkage $\mathcal{L}^m$, choose an end segment of $\mathcal{P}^m$ in $(G,\Omega)$.
        Let $Y^m$ be the set of endpoints of $\mathcal{P}^m$ contained in the chosen end segment, and let $\mathcal{L}^m$ be the set of $Y^m$-$V(C_1^m)$-subpaths of paths in $\mathcal{P}^m$.
        Subsequently, let $X^m$ be the set of endpoints of $\mathcal{L}^m$ other than $Y^m$.
        Observe that $\mathcal{L}^m$ is also a $C_1$-$\Omega$-linkage.
    \end{enumerate}
    Finally, we define new society configurations to be $\mathfrak{C}^m\coloneqq ((G^m,\Omega^m),\rho^m,c^m,\Delta^m,\mathcal{C}^m,\mathcal{L}^m,X^m,Y^m)$.
    Note that $\mathcal{L}^L$ and $\mathcal{L}^R$ are disjoint.
\end{itemize}

\medskip
\paragraph{Vortex refinement tree}
Let $k$ be a positive integer.
A \emph{vortex refinement tree} is a tuple $\mathfrak{T} = (T, \mathfrak{B},\Lambda)$, where
\begin{enumerate}
    \item $T$ is a rooted tree such that each node has at most $2$ children, whose nodes are marked as either ``incomplete'' or ``completed''.
    In addition, each node is marked as either ``moving'' or ``staying'', where any node marked as moving has a parent that is either the root or has two children,
    \item $\mathfrak{B}$ is a mapping assigning each node $v\in V(T)$ to a society configuration, and
    \item $\Lambda$ is a closed set in $\Delta$ whose boundary intersects $\rho$ only at the nodes.
\end{enumerate}

Here, $\Lambda$ will be the area where our linkages are allowed to pass through.

\subsection{Proof of our Society Classification Theorem}
We are now ready to prove the main theorem of this section.

\begin{proof} [Proof of \zcref{thm:evensocietyclassification}]
Let
\begin{align*}
    t' \coloneqq~& \lceil c_{\ref{thm:ktminormodel}}t\sqrt{\log 12t} \rceil\\
    \mathsf{apex}^\mathsf{bip}(t) \coloneqq~ & 8t \in  \mathbf{O}(t),\\
    \mathsf{apex}^\mathsf{genus}_{\ref{thm:evensocietyclassification}}(t) \coloneqq~ &\mathsf{apex}^\mathsf{genus}_{\ref{thm:societyclassification}}(t') \in  \mathbf{O}(t^9),\\
    \mathsf{loss}_{\ref{thm:evensocietyclassification}}(t) \coloneqq~ &\mathsf{loss}_{\ref{thm:societyclassification}}(t') \in  \mathbf{O}(t^4),\\
    \mathsf{cost}_{\ref{thm:evensocietyclassification}}(t,k) \coloneqq~ & \mathsf{cost}_{\ref{thm:societyclassification}}(t',k) \in  \mathbf{O}(t^2+k),\\
    \mathsf{radial}_{\ref{thm:evensocietyclassification}}(t,h,k) \coloneqq~ & 4k(\nicefrac{1}{2}(t'-3)(t'-4)+h-1)^2 \in \mathbf{O}(k(t^3+h)^2), \\
    \mathsf{nest}'(t,h,k) \coloneqq~ & (k+1)(\nicefrac{1}{2}(t'-3)(t'-4)+h-1)+2k+\mathsf{radial}_{\ref{thm:evensocietyclassification}}(t,h,k) \in \mathbf{O}(k(t^3+h)^2),\\
    \mathsf{nest}_{\ref{thm:evensocietyclassification}}(t,h,k) \coloneqq~ & \mathsf{nest}_{\ref{thm:societyclassification}}(t',\mathsf{nest}'(t,h,k))\in \mathbf{O}(t^{14}k+t^{10}h^2k),\\
    \mathsf{apex}^\mathsf{fin}_{\ref{thm:evensocietyclassification}}(t,h,k,p) \coloneqq~ & \mathsf{apex}^\mathsf{fin}_{\ref{thm:societyclassification}}(t',k,2ph+p)\in \mathbf{O}(t^{36}hk^2p^3),\mbox{ and }\\
    \mathsf{depth}_{\ref{thm:evensocietyclassification}}(t,h,k,p) \coloneqq~ & \mathsf{depth}_{\ref{lem:separating_evenfaced_planar_transaction}}(\mathsf{nest}'(t,h,k),\nicefrac{1}{2}(t'-3)(t'-4)-1, \\
    &\ 2\mathsf{radial}_{\ref{thm:evensocietyclassification}}(t,h,k)+2\mathsf{nest}'(t,h,k)+2,h,\mathsf{depth}_{\ref{thm:societyclassification}}(t',k,2ph+p))\in \mathbf{O}(t^{68}h^5k^4p^2).
\end{align*}

Let $b'' \coloneqq \nicefrac{1}{2}(t'-3)(t'-4)+h-1$.
Since $s \geq \mathsf{nest}_{\ref{thm:evensocietyclassification}}(t,h,k) = \mathsf{nest}_{\ref{thm:societyclassification}}(t', \mathsf{nest}'(t,h,k))$, we may apply \zcref{thm:societyclassification} to $(G,\Omega)$, which takes $\mathbf{poly}( t + s + h + p + k ) |E(G)||V(G)|^2$-time. 
Then we can find a set $A\subseteq V(G')$ and one of the four options.
If we have the first option of \zcref{thm:societyclassification} and find a $K_{t'}$-minor model, we can apply \zcref{thm:ktminormodel}. 
Then we find either an odd $K_t$-minor model as stated in \ref{item_s1} or a bipartite $K_{t'-8t}$-minor model, which contains a bipartite $K_t$-minor model, in $G-A$ as in \ref{item_s3} with an apex set $|A| \leq \mathsf{apex}^\mathsf{bip}(t)$ in time $\mathbf{O}(f_{\ref{thm:ktminormodel}}(t)|V(G)|^{\omega_{\ref{def:matrixmultconstant}}})$.
Both the second and the third option of \zcref{thm:societyclassification} yield a flat transaction of order $2ph + p$ in $(G' - A, \Omega')$ with $|A| \leq \mathsf{apex}^\mathsf{genus}_{\ref{thm:societyclassification}}(t') = \mathsf{apex}^\mathsf{genus}_{\ref{thm:evensocietyclassification}}(t)$ and allow us to use \zcref{lem:exposedbipartitetransaction} to obtain \ref{item_s2}, \ref{item_s4}, or \ref{item_s5} in time $\mathbf{O}(|E(G)|)$.

Suppose that we have the fourth option of \zcref{thm:societyclassification}.
We obtain a rendition $\rho'$ of $(G-A,\Omega)$ in $\Delta$ with breadth $b' \leq \nicefrac{1}{2}(t'-3)(t'-4)-1$, depth at most $d = \mathsf{depth}_{\ref{thm:societyclassification}}(t',k,2ph+p)$, $|A|\leq \mathsf{apex}^\mathsf{fin}_{\ref{thm:societyclassification}}(t',k,2ph+p)$, and an extended $\mathsf{nest}'(t,h,k)$-surface wall $D$ as described in the fourth option of \zcref{thm:societyclassification}.

Let $c'_1,\ldots,c'_{b'}$ be the vortices in $\rho'$.
Let $\rho_0$ be a $\Delta$-rendition of $(G-A,\Omega)$ that agrees with $\rho'$ on $G-A-[\sigma_{\rho'}(c_0)]$.
Let $\mathcal{C}_0$ be the set of base cycles of $D$ of order $\mathsf{nest}'(t,h,k)$.
Let $\mathcal{L}_0$ be the radial linkage of $\mathcal{C}_0$ obtained by taking $\mathsf{radial}_{\ref{thm:evensocietyclassification}}(t,h,k)$ paths from the rails of the vortex segments in $D$, and let $X_0$ and $Y_0$ be the endpoints of $\mathcal{L}_0$ that are in $C_1$ and $\Omega$, respectively.
Finally, let $\mathfrak{S}_0=((G-A,\Omega),\rho_0,c_0,\Delta,\mathcal{C}_0,\mathcal{L}_0,X_0,Y_0)$.

\medskip
\textbf{Building a vortex refinement tree: }
We now proceed by inductively building a vortex refinement tree $\mathfrak{T}=(T,\mathfrak{B},\Lambda)$ until either there is no active leaf node remaining, or $T$ has $b''+1$ leaf nodes.
We start from a tree with a single node $t_0$, where $\mathfrak{B}(t_0)=\mathfrak{S}_0$.
Add a node $t_1$ as a unique child of $t_0$, and let $\mathfrak{B}(t_1)=\mathfrak{S}_0^{-k} = ((G_1,\Omega_1),\rho_1,c_1,\Delta_1,\mathcal{C}_1,\mathcal{L}_1,X_1,Y_1)$ and let $\Lambda=\Delta_1-c_0$, where $\Delta_1$ is the disk of $\mathfrak{B}(t_1)$.
Furthermore, we add an incomplete, moving node $t_2$ as a unique child of $t_1$ and let $\mathfrak{B}(t_2)=\mathfrak{S}_0^{-3k}$, where $\Lambda$ does not change.
Note that we add the two new nodes $t_1$ and $t_2$ to ensure that there are $k$ base cycles and $k$ cycles for a nest when constructing a parity $k$-surface wall later, where we use $k$ cycles in the middle to ensure there is a linkage between each nest and the base cycles. We also define two mappings $\mathfrak{C}$ and $\mathfrak{D}$ assigning each moving node to a society configuration in order to build a parity $k$-surface wall and to find linkages easily when required. We let $\mathfrak{C}(t_2)=((G_1,\Omega_1),\rho_1,c_1,\Delta_1,\widetilde{\mathcal{C}}_1,\widetilde{\mathcal{L}}_1,\widetilde{X}_1,Y_1)$, where $\widetilde{\mathcal{C}}_1=\{\widetilde{C}^1_1,\ldots,\widetilde{C}^1_k\}$ is the set of the outermost $k$ cycles of $\mathcal{C}_1$, $\widetilde{\mathcal{L}}_1$ is the $V(\Omega_1)$-$V(\widetilde{C}^i_1)$-subpaths of the paths in $\mathcal{L}_1$, and $\widetilde{X}_1\subseteq V(\widetilde{C}^i_1)$ is the endpoints of the paths in $\mathcal{L}_1$.
Moreover, let $\mathfrak{D}(t_{2})=(\mathfrak{S}_0)^{-2k}$. We further define $\mathcal{L}^{\mathsf{base}}$ to be the $V(\Omega)$-$V(C_{\mathsf{nest}'(t,h,k)-k+1})$-subpaths of the paths in $\mathcal{L}_0$, $\mathcal{C}^{\mathsf{base}}$ to be the outhermost $k$ cycles of $\mathcal{C}_0$, and $Y^{\mathsf{base}} = Y_0$, which will be used when building a parity $k$-surface wall.

We iteratively apply \zcref{lem:separating_evenfaced_planar_transaction} to each incomplete leaf node.
In each step, one of the following happens: i) we stop the iteration concluding \ref{item_s2} or \ref{item_s7}, ii) we add a new node as a child and mark it as completed and moving, iii) we add two new incomplete moving leaf nodes as children, or iv) we add one new incomplete staying leaf node as a unique child.

Now, suppose that $T$ has at most $b''$ leaf nodes, and there is an incomplete leaf node $t_i \in V(T)$.
Note that since each moving node has at most two children, $t_i$ has $d_i \leq b''$ ancestors marked as moving.
Let $\mathfrak{B}(t_i) = \mathfrak{S}_i \coloneqq ((G_i, \Omega_i), \rho_i, c_i, \Delta_i, \mathcal{C}_i = \{ C_1^i, \ldots , C_{\mathsf{nest}'(t,h,k)-(k+1)(d_i+2)}^i \}, \mathcal{L}_i, X_i, Y_i)$.
Suppose that the restriction of $\rho'$ to $G_i$ contains at most $b'$ vortices.
We apply \zcref{lem:separating_evenfaced_planar_transaction} to $(G_i,\Omega_i)$ to obtain one of the following results in time $\mathbf{O}(k(b'+h)(t^2+h)^2 |E(G)|)$.

\begin{enumerate}
    \item The society $(G_i,\Omega_i)$ has depth at most $\mathsf{depth}_{\ref{lem:separating_evenfaced_planar_transaction}}(\mathsf{nest}'(t,h,k)-(k+1)d_i,b',2\mathsf{radial}_{\ref{thm:evensocietyclassification}}(t,h,k)+2\mathsf{nest}'(t,h,k)+2,h,\mathsf{depth}_{\ref{thm:societyclassification}}(t',k,2ph+p))$
    \item There is an odd minor model for the universal parity breaking grid $\mathscr{U}_h$ controlled by a mesh whose horizontal paths are subpaths of distinct cycles from $\mathcal{C}$.
    \item There is a separating even-faced planar transaction $\mathcal{P}_i = \{P_1^i, \ldots, P_{p'}^i\}$ of order $p' = 2\mathsf{radial}_{\ref{thm:evensocietyclassification}}(t,h,k)+2\mathsf{nest}'(t,h,k)+2$ that is orthogonal to $\mathcal{C}_i$ with $(G_{\mathcal{P}_i},\Omega_{\mathcal{P}_i})$ being the $\mathcal{P}_i$-strip society, such that $V(\sigma(c_j')) \cap V(G_{\mathcal{P}_i}) = \emptyset$ for all $j \in [b]$, and the restriction of $\rho_i$ to $(G-[\sigma_{\rho_i}(c_j)]) \cup G'$ is even-faced.
    \item There is a separating odd handle $\mathcal{P}$ of order $2\mathsf{radial}_{\ref{thm:evensocietyclassification}}(t,h,k)+2\mathsf{nest}'(t,h,k)+2$ in $(G_i,\Omega_i)$.
\end{enumerate}

\smallskip
We start by discussing the first and the third options listed above.
The second and the fourth option are addressed later on.
If we obtain the first option, we simply mark $t_i$ as completed.

\smallskip
Now, suppose that we obtain the third option.
In this case, we divide the $(\mathsf{nest}'(t,h,k)-(k+1)d_i, \mathsf{radial}_{\ref{thm:evensocietyclassification}}(t,h,k))$-society configuration $\mathfrak{S}_i$ by the transaction $\mathcal{P}_i$ of order $2\mathsf{radial}_{\ref{thm:evensocietyclassification}}(t,h,k)+2\mathsf{nest}'(t,h,k)+2$ to obtain two new society configurations $\mathfrak{S}^L_i=((G^L_i,\Omega^L_i),\rho^L_i,c^L_i,\Delta^L_i,\mathcal{C}_i^L,\mathcal{L}_i^L,X^L_i,Y_i^L)$ and $\mathfrak{S}^R_i=((G^R_i,\Omega^R_i),\rho^R_i,c^R_i,\Delta^R_i,\mathcal{C}^R_i,\mathcal{L}^R_i, X^R_i,Y^R_i)$.
As there is a linkage $\mathcal{L}_i^m$ with $m \in \{L, R\}$ for $\mathcal{C}_i^m$, there is a block $B_i^m$ in $G_i^m$ that contains all cycles in $\mathcal{C}_i^m$. If $B_i^m$ is bipartite, and the restriction of $\rho'$ to $G_i^m$ does not contain any of the vortices $c_1',\ldots,c_{b'}'$, we say that $\mathfrak{S}_i^m$ is \emph{quasi-bipartite}.

Suppose that both $\mathfrak{S}_i^L$ and $\mathfrak{S}_i^R$ are not quasi-bipartite.
In this case, we add two incomplete, moving nodes $t^L_{i+1}$ and $t^R_{i+1}$ as children of $t_i$ where $\mathfrak{B}(t^L_{i+1})=(\mathfrak{S}^L_i)^{-2k}$ and $\mathfrak{B}(t^R_{i+1})=(\mathfrak{S}^R_i)^{-2k}$.
Here, we move in by $2k$ to make sure that there are enough cycles when constructing a parity $k$-surface wall, and there is a linkage between nests and the base cycles.
We update $\Lambda$ by taking the union with the disk bounded by the strip society of $\mathcal{P}_i$.
Furthermore, for each $m\in \{L,R\}$, we additionally define $\mathfrak{C}(t^m_{i+1})=((G^m_i,\Omega^m_i),\rho^m_i,c^m_i,\Delta^m_i,\widetilde{\mathcal{C}}_i^m,\widetilde{\mathcal{L}}_i^m,\widetilde{X}^m_i,Y_i^m)$, where $\widetilde{\mathcal{C}}_i^m=\{\widetilde{C}^i_1,\ldots,\widetilde{C}^i_k\}$ is the set of the outermost $k$ cycles of $\mathcal{C}_i^m$, $\widetilde{\mathcal{L}}_i^m$ is the $V(\Omega^m_i)$-$V(\widetilde{C}^i_1)$-subpaths of paths in $\mathcal{L}_i^m$, and $\widetilde{X}^m_i\subseteq V(\widetilde{C}^i_1)$ is the endpoints of paths in $\mathcal{L}_i^m$.
Lastly, let $\mathfrak{D}(t^m_{i+1})=(\mathfrak{S}^m_i)^{-k}$.

Now consider the other case where $\mathfrak{S}_i^L$ or $\mathfrak{S}_i^R$ is quasi-bipartite.
Note that only one of $\mathfrak{S}_i^L$ or $\mathfrak{S}_i^R$ can satisfy this condition.
Let $m\in \{L,R\}$ so that $\mathfrak{S}_i^m$ is not quasi-bipartite.
Then we add an incomplete staying node $t_{i+1}$ as a unique child of $t_i$ where $\mathfrak{B}(t_{i+1})=\mathfrak{S}^m_i$.

Lastly, notice that whenever we obtain the third option to grow the vortex refinement tree, the number of vertices in the union of $V(G_u)$ for every leaf node $u$ strictly decreases.
Therefore, the whole process will end before $|V(G)|$ iteration, so we have $|V(T)|\leq 2|V(G)|$.

\medskip
\textbf{Building a linkage:}
Before dealing with the second and fourth option, we explain how to build a linkage that connects the society configurations we found in each step to the original society configuration.
Let $\mathfrak{T}=(T,\mathfrak{B},\Lambda)$ be a vortex refinement tree and let $W\subseteq V(T)$ be a set of leaf nodes.
For each node $u\in V(T)$, let $\mathfrak{B}(u)=((G_u,\Omega_u),\rho_u,c_u,\Delta_u,\mathcal{C}_u=\{C_1^u,\ldots,C_s^u\},\mathcal{L}_u,X_u,Y_u)$.
We traverse the vortex refinement tree starting from $W$ toward the root to find a $(\bigcup_{t\in W}X_{t})$-$V(\Omega)$-linkage $\mathcal{R}$ of order $4k(b'')^2$ such that the paths in $\mathcal{R}$ are contained in $\Lambda$.

Let $T_u$ be the subtree of $T$ rooted at $u$.
We inductively define a $( \bigcup_{t_i\in V(T_u)\cap W} X_{t_i})$-$Y_u$-linkage $\mathcal{R}_u$ such that
\begin{itemize}
    \item the paths in $\mathcal{R}_u$ are contained in $\Lambda \cap \Delta_u$,
    \item if $u$ is not the root node and its parent $w$ has a unique child, then $\mathcal{R}_u$ is also a $( \bigcup_{t_i\in V(T_u)} X_{t_i})$-$\Omega_{w}$-linkage, and
    \item if $u$ is not the root node and its parent $w$ has two children, then $\mathcal{R}_u$ is also a $( \bigcup_{t_i\in V(T_u)} X_{t_i})$-$C^w_{k+1}$-linkage.
\end{itemize}

We find $\mathcal{R}_u$ as follows starting from the nodes in $W$:
\begin{enumerate}
    \item If $u$ is a leaf, we define $\mathcal{R}_u = \mathcal{L}_u$.
    
    \item If $u$ is a node with a unique child $v$ such that $\mathcal{R}_v$ is defined, then since $\mathcal{R}_v$ is also a $( \bigcup_{t_i\in V(T_u)} X_{t_i} )$-$\Omega_u$-linkage, we can apply \zcref{lem:finding_linkages} to $(G_u,\Omega_u)$, $\mathcal{R}_v$, and $\mathcal{L}_u$ to find $\mathcal{R}_u$.

    \item Suppose that $u$ is a node with two children $v_1$ and $v_2$ such that $\mathcal{R}_{v_1}$ and $\mathcal{R}_{v_2}$ are defined.
    Since $\Delta_{v_1}$ and $\Delta_{v_2}$ are disjoint, $\mathcal{R}_{v_1}$ and $\mathcal{R}_{v_2}$ are also disjoint.
    Thus, $\mathcal{R}_{v_1}\cup \mathcal{R}_{v_2}$ is a $( \bigcup_{t_i\in V(T_u)} X_{t_i})$-$C_{s-m}^u$-linkage.
    Therefore, we can apply \zcref{lem:finding_linkages} to $(G_u,\Omega_u)$, $\mathcal{R}_{v_1}\cup \mathcal{R}_{v_2}$, and $\mathcal{L}_u$ to find $\mathcal{R}_u$.
\end{enumerate}

Observe that the second and the third condition of $\mathcal{R}_u$ follow from the fact that $\mathcal{L}_u$ results from \zcref{item:dividing_society_config_6} of the definition of dividing a society configuration.

We repeat the process all the way up to the root of $T$, to find a desired $(\bigcup_{i\in [\ell]}X_{t_i})$-$Y_0$-linkage.
Furthermore, since $|V(T)|\leq 2|V(G)|$, we can find such a linkage in time $\mathbf{O}(k|E(G)|\cdot |V(G)|)$.

\smallskip

Now, suppose that we obtain the second option of \zcref{lem:separating_evenfaced_planar_transaction}.
Note that there is a $X_i$-$Y_0$-linkage of order $4kb'' \geq (2h+1)^2$ constructed as above.
So, the odd minor model for $\mathscr{U}_h$ is also controlled by a mesh whose horizontal paths are subpaths of distinct cycles from $\mathcal{C}_0$, and we have \ref{item_s2}.

If we obtain the fourth option of \zcref{lem:separating_evenfaced_planar_transaction}, note that $2\mathsf{radial}_{\ref{thm:evensocietyclassification}}(t,h,k)+2\mathsf{nest}'(t,h,k)+2\geq 2h+2$.
Therefore, by \zcref{lem:get-parity-handle}, $G_i$ contains $\mathcal{H}_h$ as an odd minor controlled by a mesh whose horizontal paths are subpaths of distinct cycles from $\mathcal{C}_i$.
Similar to the second option, by the $X_i$-$Y_0$-linkage of order $4kb'' \geq 4h^2$ constructed as above, it is easy to see that this odd minor model is controlled by a mesh whose horizontal paths are subpaths of distinct cycles from $\mathcal{C}_0$ as in \ref{item_s7}.

\medskip
\textbf{Building a surface wall:}
Suppose that at some point in the construction of the vortex refinement tree, we have $\mathfrak{T}=(T,\mathfrak{B},\Lambda)$ such that $T$ has $b''+1$ leaf nodes, say $t_1,\ldots,t_{b''+1}$.
Let $\mathfrak{B}(t_i)=((G_i,\Omega_i),\rho_i,c_i,\Delta_i,\mathcal{C}_i,\mathcal{L}_i,X_i,Y_i)$.
Then either $c_i$ contains one of the vortices $c_1',\ldots,c_{b'}'$, or $\mathfrak{B}(t_i)$ is not quasi-bipartite.
Since it holds that $b' \leq \nicefrac{1}{2}(t'-3)(t'-4)-1$, there are at least $h$ leaf nodes whose $\mathfrak{B}(t_i)$ is not quasi-bipartite.
Without loss of generality, assume that $\mathfrak{B}(t_1),\ldots,\mathfrak{B}(t_h)$ are not quasi-bipartite.

Then we build a parity-$(2h+1)$-surface-wall with signature $(0,0,0,0,0,h)$ as follows:
For each $i$, let $s_i$ be the nearest ancestor of $t_i$ that is marked as moving, and let $\mathfrak{C}(s_i)=(G^\mathfrak{C}_i,\Omega^\mathfrak{C}_i,\rho^\mathfrak{C}_i,c^\mathfrak{C}_i,\Delta^\mathfrak{C}_i,\mathcal{C}^\mathfrak{C}_i,\mathcal{L}^\mathfrak{C}_i,X^\mathfrak{C}_i,Y^\mathfrak{C}_i)$ and $\mathfrak{D}(s_i)=(G^\mathfrak{D}_i,\Omega^\mathfrak{D}_i,\rho^\mathfrak{D}_i,c^\mathfrak{D}_i,\Delta^\mathfrak{D}_i,\mathcal{D}^\mathfrak{D}_i,\mathcal{L}^\mathfrak{D}_i,X^\mathfrak{D}_i,Y^\mathfrak{D}_i)$.

For each $i$, choose a set $\widetilde{X}_i'\subseteq X^{\mathfrak{D}}_i$ with $|X_i'|= 4h(2h+1)$.
Let $X_i'\subseteq X^{\mathfrak{C}}_i$ be the set of vertices corresponds to $\widetilde{X}_i'$ via radial linkage which has $\mathcal{L}_i^\mathfrak{C}$ and $\mathcal{L}_i^{\mathfrak{D}}$.

Then similarly as before, we can find a $\bigcup_{i\in [h]} X_i'$-$Y_1$-linkage $\mathcal{R}$ of order $4h^2(2h+1)$.
Then for each $i$, we have $X_i'$-$Y_1$-linkage of order $4h(2h+1)$ that is subset of $\mathcal{R}$, and naturally index the paths in this linkage as $P^i_1,\ldots,P^i_{4h(2h+1)}$.
For each $j\in [h]$, let $Z_j^i$ be the minimal segment of $\Omega$ that contains the endpoints of $P^i_{4(j-1)(2h+1)+1},\ldots,P^i_{4j(2h+1)}$, and let $D^i_j$ be the area bounded by the trace of $P^i_{4(j-1)(2h+1)+1}$, $P^i_{4j(2h+1)}$, $Z_j^i$, and the arc of the innermost cycle of $\mathcal{C}_i$ that contains $X_i'$.
Then by the pigeonhole principle, for at least one of $j$, $D^i_j$ does not contain all other $c_1,\ldots,c_{i-1},c_{i+1},\ldots,c_h$.
Fix such $j_i$ for each $i$ and let $\mathcal{R}'=\bigcup_{i\in [h]}\{P^i_{4(j_i-1)(2h+1)+1},\ldots,P^i_{4j_i(2h+1)}\}$.
We can extend these paths by concatenating the paths in the $\mathcal{L}^{\mathfrak{D}}_i$ and the paths in the $\mathcal{L}^{\mathsf{base}}$, we can find $\bigcup_{i\in [h]} X_i^{\mathfrak{D}}$-$Y^{\mathsf{base}}$-linkage such that for each $i\in [h]$, it contains $X_i^{\mathfrak{D}}$-$Y^{\mathsf{base}}$-linkage of order $4(2h+1)$, which is orthogonal to both $\mathcal{C}^{\mathsf{base}}$ and $\mathcal{C}_i^\mathfrak{D}$.

With $\mathfrak{D}(s_i)$ for $i\in [h]$, we can build a parity-$(2h+1)$-surface-wall with signature $(0,0,0,0,0,h)$ by using $\mathcal{C}^\mathsf{base}$ and the linkages found above.
Since the $\mathfrak{B}(t_i)$ we chose are not quasi-bipartite, we can add a parity breaking paths whose endpoints are on the innermost cycle of $\mathcal{C}_i$, and whose internal vertices are contained in $c_i$.
This allows us to apply \zcref{lemma:parity_grid_in_extended_walls} to find an odd minor model of $\mathscr{U}_h$ as in \ref{item_s7}.

\smallskip
If there is no active leaf node while we have $b\leq b''$ leaf nodes, we build a parity $k$-surface-wall with signature $(0,0,0,0,0,b)$ for \ref{item_s6} similarly as above.
Furthermore, the whole process can be done in time $(f_{\ref{thm:ktminormodel}}(t)+\mathbf{poly}( h + s + p + k )) |E(G)||V(G)|^2$.
\end{proof}
\section{Local structure}\label{sec:localstructure}
Now that we have gone through the arduous task of adapting the Society Classification Theorem for our purposes, we need to also reprove the Local Structure Theorem in our setting.
This theorem first appears in \cite{RobertsonS2003Grapha} as Theorem 3.1, with Theorem 15.1 being the corresponding statement in \cite{GorskySW2025polynomialboundsgraphminor}.
We will follow the proof in \cite{GorskySW2025polynomialboundsgraphminor} fairly closely and thus not elaborate too much on its details, since those can be found in \cite{GorskySW2025polynomialboundsgraphminor}.
Instead we will focus on proving that graphs which admit the type of local structure that can be found in our theorem indeed do have bounded odd cycle packing number.

Before we can state our first local structure theorem, we will need to introduce another concept relating to the way we represent our graph on a surface.
This definition is specific to the local structure theorem and we can forget about it again once we have determined that the resulting structure does not permit graphs with arbitrary odd cycle packing number.

These definitions differ slightly from those found in \cite{GorskySW2025polynomialboundsgraphminor} as we desire some slightly different restrictions related to even-faced renditions and the fact that some of our vortices are allowed to not have a cross as long as they are non-bipartite in exchange.

\paragraph{Non-orientable, even landscapes.}
Let $k, w \geq 4$ be integers, let $G$ be a graph, and let $\Sigma$ be a surface of Euler-genus $g$.
Let $h$, $c^o$, $c^e$, and $b$ be non-negative integers where $g=2h+c$ and $c = c^o + c^e \neq 0$ if and only if $\Sigma$ is non-orientable.
Moreover, let $D \subseteq G$ be a parity $k$-surface-wall with signature $(0,h,c^o,c^e,0,b)$, let $W \subseteq G-A$ be a $w$-mesh in $G$, and let $\mathcal{T}_D$ and $\mathcal{T}_W$ be the tangles they respectively define.
Finally, let $A \subseteq V(G) \setminus V(D)$.

The tuple $\Lambda = (A,W,D,\rho)$ is called a \emph{non-orientable, even $\Sigma$-landscape} of \emph{detail $k$} and \emph{oddness $c^o$} if
\begin{description}
    \item[L1~~] $\rho$ is a non-orientable, even-faced $\Sigma$-rendition of the large block of $\mathcal{T}_W$,
    \item[L2~~] $D$ and $W$ are grounded in $\rho$,
    \item[L3~~] $W$ is flat in $\rho$,
    \item[L4~~] the disk bounded by the trace of the simple cycle of $D$ in $\rho$ avoids the traces of the other base cycles of $D$,
    \item[L5~~] the tangle $\mathcal{T}_D$ is a truncation of the tangle $\mathcal{T}_W$,
    \item[L6~~] if $C$ is a cycle from the nest of some vortex-segment of $D$, then the trace of $C$ is a contractible closed curve in $\Sigma$,
    \item[L7~~] $\rho$ has exactly $b$ vortices and there exists a bijection between the vortices $v$ of $\rho$ and the vortex segments $S_v$ of $D$ such that $v$ is the unique vortex of $\rho$ that is contained in the $v$-disk $\Delta_{C_1}$ of the inner cycle of $S_v$, where $\Delta_{C_1}$ avoids the trace of the simple cycle of $D$, and
    \item[L8~~] for every vortex $v$ of $\rho$, the society induced by the outer cycle from the nest of the corresponding vortex segment is non-bipartite or has a cross.
\end{description}
We refer to the vortex segments as the \emph{vortices} of $\Lambda$ and call $A$ the \emph{apex set}.
The integer $b$ is called the \emph{breadth} of $\Lambda$ and the \emph{depth} of $\Lambda$ is the depth of $\rho$.
We say that $\Lambda$ is \emph{centred} at the mesh $W$.

\paragraph{Non-orientable, even layouts.}
Let $k \geq 4$, $l$, $d$, $b$, $r$, and $a$ be non-negative integers and $\Sigma$ be a surface.
We say that a graph $G$ with a mesh $M$ has a \emph{non-orientable, even $(k,l)$-$(a,b,d,r)$-$\Sigma$-layout} \emph{centred at $M$} if there exists a set $A \subseteq V(G)$ of size at most $a$ and a submesh $M'\subseteq M$ such that there exists a non-orientable, even $\Sigma$-landscape $(A,M',D,\rho)$ of detail $k$, oddness $l$, breadth $b$, and depth $d$ for $G$ where every vortex of $\rho$ has a linear decomposition of adhesion at most $d$, and $M'$ is a $w$-mesh with $w \geq a+b(2d+1)+6+r$.

\subsection{A Local Structure Theorem}
We will now present a theorem in the style of Theorem 15.1 in \cite{GorskySW2025polynomialboundsgraphminor}.
Note that this is not the ultimate goal of this section, since we will proceed to show that this result further implies that the graphs we are interested in have bounded \ocp.
We present this result separately due to its conceptual closeness to the aforementioned result from \cite{GorskySW2025polynomialboundsgraphminor} and due to the sheer length and technical complexity of the proof of this result.

Recall that, given some $r$-mesh $M$ and a separation $(A,B)$ of order at most $r-1$ we said that $X \in \{ A,B \}$ is the \textsl{majority side} of $(A,B)$ if $X \setminus Y$, where $Y \in \{ A,B \} \setminus \{ X \}$, contains the vertices of both a vertical and a horizontal path of $M$.
Moreover, we observed that the orientation $\mathcal{T}_M$ of $\mathcal{S}_r$ is a tangle.

Let $\Sigma$ be a surface and let $\rho$ be a $\Sigma$-rendition of a graph $G$ containing an $r$-mesh $M$.
We say that $\rho$ is \emph{$M$-central} if there is no cell $c \in C(\rho)$ such that $V(\sigma(c))$ contains the majority side of a separation from $\mathcal{T}_M$.
Similarly, let $A \subseteq V(G)$, $|A| \leq r-1$, let $\Sigma'$ be a surface and $\rho'$ be a $\Sigma'$-rendition of $G-A$.
Then we say that $\rho'$ is \emph{$(M-A)$-central} for $G$ if no cell of $\rho'$ contains the majority side of a separation from $\mathcal{T}_M \cap \mathcal{S}_{r-|A|}$.

\begin{theorem}\label{thm:strongest_localstructure}
There exist functions $\mathsf{apex}^\mathsf{bip}_{\ref{thm:strongest_localstructure}}, \mathsf{clique}_{\ref{thm:strongest_localstructure}} \colon \mathbb{N}\to\mathbb{N}$, $\mathsf{apex}_{\ref{thm:strongest_localstructure}},\mathsf{depth}_{\ref{thm:strongest_localstructure}}\colon\mathbb{N}^2\to\mathbb{N}$, and $\mathsf{mesh}_{\ref{thm:strongest_localstructure}}\colon\mathbb{N}^3\to\mathbb{N}$ such that for all positive integers $k$, $t$, $r$, and $t' \coloneqq \mathsf{clique}_{\ref{thm:strongest_localstructure}}(t)$, every graph $G$, and every $\mathsf{mesh}_{\ref{thm:strongest_localstructure}}(t,r,k)$-mesh $M \subseteq G$ one of the following holds.
\begin{enumerate}
    \item\label{itm:localstructurenearlybipartite} There exists a set $X \subseteq G$, with $|X| \leq \mathsf{apex}^\mathsf{bip}_{\ref{thm:strongest_localstructure}}(t')$, the large block of $\mathcal{T}_M - X$ is bipartite, where $\mathcal{T}_M$ is the tangle associated with $M$,
    \item\label{itm:localstructureparityhandle} $G$ has $\mathscr{H}_t$ as an odd minor controlled by $M$, or
    \item\label{itm:localstructureevenfacedrendition} $G$ has a non-orientable, even $(k,l)$-$(\mathsf{apex}_{\ref{thm:strongest_localstructure}}(t,k),\nicefrac{1}{2}(t'-3)(t'-4)+t-1,\mathsf{depth}_{\ref{thm:strongest_localstructure}}(t,k),r)$-$\Sigma$-layout $\Lambda$ centred at $M$ where $l < \nicefrac{1}{2}(t'-3)(t'-4)+t-1$, the surface $\Sigma$ has genus less than ${t'}^2$, and the $\Sigma$-rendition $\rho$ of $\Lambda$ is $(M - A)$-central where $A$ is the apex set of $\Lambda$.
\end{enumerate}
Moreover, it holds that

{\centering
  $ \displaystyle
    \begin{aligned}
        \mathsf{clique}_{\ref{thm:strongest_localstructure}}(t) \in                                                             & \ \mathbf{O}(t^4), \\
        \mathsf{apex}^\mathsf{bip}_{\ref{thm:strongest_localstructure}}(t) \in                                                  & \ \mathbf{O}(t^{12}), \\
        \mathsf{apex}_{\ref{thm:strongest_localstructure}}(t,k),~ \mathsf{depth}_{\ref{thm:strongest_localstructure}}(t,k) \in  & \ \mathbf{O}\big((t+k)^{150}\big), \text{ and} \\
        \mathsf{mesh}_{\ref{thm:strongest_localstructure}}(t,r,k) \in                                                           & \ \mathbf{O}\big((t+k)^{161} + rt^{12} \big) .
    \end{aligned}
  $
\par}

There also exists an algorithm that, given $t$, $k$, $r$, a graph $G$, and a mesh $M$ as above as input finds one of these outcomes in time $(f_{\ref{thm:ktminormodel}}(t) + \mathsf{poly}(k))|E(G)||V(G)|^2$.
\end{theorem}

\paragraph{Extending the surface.}
A major part of the proof of the Local Structure Theorem consists of using the opportunity provided by the options in \zcref{thm:evensocietyclassification} that gives us crosscap and handle transactions to grow the surface we are embedding our graph in.
To accomplish this we extract two results from \cite{KawarabayashiTW2021Quickly} (see Lemma 10.2 and Lemma 10.5), one for the crosscap case and the other for handles.
A more thorough, illustrated explanation of the intuition behind these two statements is given in \cite{GorskySW2025polynomialboundsgraphminor}.
We adopt the minor changes from \cite{GorskySW2025polynomialboundsgraphminor} to these statements which involve only considering a radial linkage instead of a linkage from $\Omega$ to the vortex.
This comes at a very minor cost in the number of cycles we lose.
More importantly, we note that by applying \zcref{lem:reconciliationforevenfaced} into the proof of Lemma 10.2 and Lemma 10.5 whenever \zcref{lem:reconciliation} is applied (see Lemma 5.15 from \cite{KawarabayashiTW2021Quickly}), we can also ensure additional properties for the resulting rendition.
For the crosscap case, this tells us that the new rendition is even-faced and non-orientable.

\begin{lemma}\label{lemma:integrate-crosscap}
Let $s$ and $p$ be non-negative integers.
Let $(G,\Omega)$ be a society with an even-faced, cylindrical rendition $\rho_0$ in the disk $\Delta$ with a nest $\mathcal{C} = \{ C_1, \ldots , C_{s+9} \}$ around the vortex $c_0$.
Let $X_1,X_2$ be disjoint segments of $\Omega$ such that there exist
\begin{itemize}
    \item a radial linkage $\mathcal{R}$ orthogonal to $\mathcal{C}$ starting in $X_1$, and
    \item an even-faced, flat crosscap transaction $\mathcal{P}$ of order at least $p+2s+7$ with all endpoints in $X_2$ and disjoint from $\mathcal{R}$.
\end{itemize}
Let $\Sigma^*$ be a surface, homeomorphic to the projective plane minus an open disk, which is obtained from $\Delta$ by adding a crosscap to the interior of $c_0$.
\smallskip
Then there exists a crosscap transaction $\mathcal{P}' \subseteq \mathcal{P}$ of order $p$, consisting of the middle $p$ paths of $\mathcal{P}$, and a non-orientable, even-faced rendition $\rho_1$ of $(G,\Omega)$ in $\Sigma^*$ (around the $\bigcup \mathcal{C}$) with a unique vortex $c_1$ and the following hold:
\begin{enumerate}
    \item $\mathcal{P}'$ is disjoint from $\sigma(c_1)$,

    \item the vortex society of $c_1$ in $\rho_1$ has an even-faced, cylindrical rendition $\rho_1'$ with a nest $\mathcal{C}'=\{ C_1',\dots,C_s'\}$ around the unique vortex $c_1'$,

    \item every element of $\mathcal{R}$ has an endpoint in $V(\sigma_{\rho_1'}(c_1'))$,

    \item $\mathcal{R}$ is orthogonal to $\mathcal{C}'$, and for every $i \in [s]$ and every $R \in \mathcal{R}$, $C_i' \cap R = C_{i+8} \cap R$. Moreover,

    \item let $\mathcal{R} = \{ R_1, \ldots , R_{\ell} \}$.
    For each $i \in [\ell]$ let $x_i$ be the endpoint of $R_i$ in $X_1$, and let $y_i$ be the last vertex of $R_i$ on $c_1$ when traversing along $R_i$ starting from $x_i$; if $x_1,x_2,\dots,x_{\ell}$ appear in $\Omega$ in the order listed, then $y_1,y_2,\dots,y_{\ell}$ appear on $\mathsf{bd}(c_1)$ in the order listed.

    \item Finally, let $\Delta'$ be the open disk bounded by the trace of $C_{s+8}$ in $\rho_0$.
    Then $\rho_0$ restricted to $\Delta\setminus\Delta'$ is equal to $\rho_1$ restricted to $\Delta\setminus \Delta'$.
\end{enumerate}
Moreover, there exists an algorithm that computes this outcome in time $\mathbf{poly}(sp|\mathcal{R}|)|V(G)|$.
\end{lemma}

The case in which we try to integrate a handle is more complex.
As we can see in \zcref{lem:reconciliationforevenfaced}, either of the two planar transactions making up the handle transaction may cause odd cycles to appear.
For us, if this happens for at least one of the two planar transactions, this is a way to build one of our obstructions.
Thus we are satisfied with the following version of the statement.

\begin{lemma}\label{lemma:integrate-handle}
Let $s$ and $p$ be non-negative integers.
Let $(G,\Omega)$ be a society with an even-faced, cylindrical rendition $\rho_0$ in the disk $\Delta$ with a nest $\mathcal{C} = \{ C_1, \ldots , C_{s+9} \}$ around the vortex $c_0$.
Let $X_1,X_2$ be disjoint segments of $\Omega$ such that there exist
\begin{itemize}
    \item a radial linkage $\mathcal{R}$ orthogonal to $\mathcal{C}$ starting in $X_1$, and
    \item an even-faced, flat handle transaction $\mathcal{P}$ of order at least $2p+4s+12$ with all endpoints in $X_2$ and disjoint from $\mathcal{R}$.
\end{itemize}
Let $\mathcal{P}_1$ and $\mathcal{P}_2$ be the two planar transactions such that $\mathcal{P} = \mathcal{P}_1 \cup \mathcal{P}_2$.
Let $\Sigma^+$ be a surface, homeomorphic to the torus minus an open disk, which is obtained from $\Delta$ by adding a handle to the interior of $c_0$.
\smallskip
Then there exist transactions $\mathcal{P}_1' \subseteq \mathcal{P}_1$ and $\mathcal{P}_2' \subseteq \mathcal{P}_2$ of order $p$, each $\mathcal{P}_i'$ consisting of the middle $p$ paths of $\mathcal{P}_i$, with $i \in [2]$, such that $\mathcal{P}'= \mathcal{P}_1' \cup \mathcal{P}_2'$ is a handle transaction, and a rendition $\rho_1$ of $(G,\Omega)$ in $\Sigma^+$ with a unique vortex $c_1$ and the following hold:
\begin{enumerate}
    \item all cells of $\rho_1$ are parity-preserving,

    \item $\mathcal{P}'$ is disjoint from $\sigma(c_1)$,

    \item either
    \begin{itemize}
        \item for at least one $i \in [2]$, $\mathcal{P}_i$ is an odd handle in $(G,\Omega)$, or 

        \item $\rho_1$ is non-orientable and even-faced around $\bigcup \mathcal{C}$, and the vortex society of $c_1$ in $\rho_1$ has an even-faced, cylindrical rendition $\rho_1'$ with a nest $\mathcal{C}'=\{ C_1', \ldots , C_s' \}$ around the unique vortex $c_1'$, and
        \begin{enumerate}
            \item every element of $\mathcal{R}$ has an endpoint in $V(\sigma_{\rho_1'}(c_1'))$,

            \item $\mathcal{R}$ is orthogonal to $\mathcal{C}'$, and for every $i \in [s]$ and every $R \in \mathcal{R}$, $C_i' \cap R = C_{i+8} \cap R$. Moreover,

            \item let $\mathcal{R}=\{ R_1, \ldots , R_\ell \}$.
            For each $i \in [\ell]$ let $x_i$ be the endpoint of $R_i$ in $X_1$, and let $y_i$ be the last vertex of $R_i$ on $c_1$ when traversing along $R_i$ starting from $x_i$; then if $x_1, x_2, \ldots , x_\ell$ appear in $\Omega$ in the order listed, then $y_1, y_2, \ldots , y_\ell$ appear on $\mathsf{bd}(c_1)$ in the order listed.

            \item Finally, let $\Delta'$ be the open disk bounded by the trace of $C_{s+8}$ in $\rho_0$.
            Then $\rho_0$ restricted to $\Delta \setminus \Delta'$ is equal to $\rho_1$ restricted to $\Delta \setminus \Delta'$.
        \end{enumerate}
    \end{itemize}
\end{enumerate}
Moreover, there exists an algorithm that computes this outcome in time $\mathbf{poly}(sp|\mathcal{R}|)|V(G)|$.
\end{lemma}

Luckily, via \zcref{lem:get-parity-handle}, the first option yields a large parity handle.
This allows us to simplify \zcref{lemma:integrate-handle} by additionally excluding $\mathscr{H}_k$.

\begin{lemma}\label{lemma:integrate-handle-clean}
Let $s,k,p$ be positive integers with $s \geq 2k+4$.
Let $(G,\Omega)$ be a society with an even-faced, cylindrical rendition $\rho_0$ in the disk $\Delta$ with a nest $\mathcal{C} = \{ C_1, \ldots , C_{s+9} \}$ around the vortex $c_0$.
Let $X_1,X_2$ be disjoint segments of $\Omega$ such that there exist
\begin{itemize}
    \item a radial linkage $\mathcal{R}$ orthogonal to $\mathcal{C}$ starting in $X_1$, and
    \item an even-faced, flat handle transaction $\mathcal{P}$ of order at least $\max(2p,2k)+4s+12$ with all endpoints in $X_2$ and disjoint from $\mathcal{R}$.
\end{itemize}
Let $\mathcal{P}_1$ and $\mathcal{P}_2$ be the two planar transactions such that $\mathcal{P} = \mathcal{P}_1 \cup \mathcal{P}_2$.
Let $\Sigma^+$ be a surface, homeomorphic to the torus minus an open disk, which is obtained from $\Delta$ by adding a handle to the interior of $c_0$.
\smallskip
Then either $G$ contains $\mathscr{H}_k$ as an odd minor, or there exist transactions $\mathcal{P}_1' \subseteq \mathcal{P}_1$ and $\mathcal{P}_2' \subseteq \mathcal{P}_2$ of order $p$, each $\mathcal{P}_i'$ consisting of the middle $p$ paths of $\mathcal{P}_i$, with $i \in [2]$, such that $\mathcal{P}'= \mathcal{P}_1' \cup \mathcal{P}_2'$ is a handle transaction, and a non-orientable, even-faced rendition $\rho_1$ of $(G,\Omega)$ in $\Sigma^+$ (around $\bigcup \mathcal{C}$) with a unique vortex $c_1$ and the following hold:
\begin{enumerate}
    \item $\mathcal{P}'$ is disjoint from $\sigma(c_1)$,

    \item the vortex society of $c_1$ in $\rho_1$ has an even-faced, cylindrical rendition $\rho_1'$ with a nest $\mathcal{C}'=\{ C_1', \ldots , C_s' \}$ around the unique vortex $c_1'$, and

    \item every element of $\mathcal{R}$ has an endpoint in $V(\sigma_{\rho_1'}(c_1'))$,

    \item $\mathcal{R}$ is orthogonal to $\mathcal{C}'$, and for every $i \in [s]$ and every $R \in \mathcal{R}$, $C_i' \cap R = C_{i+8} \cap R$. Moreover,

    \item let $\mathcal{R}=\{ R_1, \ldots , R_\ell \}$.
    For each $i \in [\ell]$ let $x_i$ be the endpoint of $R_i$ in $X_1$, and let $y_i$ be the last vertex of $R_i$ on $c_1$ when traversing along $R_i$ starting from $x_i$; then if $x_1, x_2, \ldots , x_\ell$ appear in $\Omega$ in the order listed, then $y_1, y_2, \ldots , y_\ell$ appear on $\mathsf{bd}(c_1)$ in the order listed.

    \item Finally, let $\Delta'$ be the open disk bounded by the trace of $C_{s+8}$ in $\rho_0$.
    Then $\rho_0$ restricted to $\Delta \setminus \Delta'$ is equal to $\rho_1$ restricted to $\Delta \setminus \Delta'$.
\end{enumerate}
Moreover, there exists an algorithm that computes this outcome in time $\mathbf{poly}(sp|\mathcal{R}|)|V(G)|$.
\end{lemma}

\paragraph{Helpful tools for the proof.}
Another smattering of tools is distributed throughout different parts of \cite{GorskySW2025polynomialboundsgraphminor}.
First, we will need to connect two orthogonal radial linkages to build a new orthogonal radial linkage.

\begin{proposition}[Gorsky, Seweryn, and Wiederrecht \cite{GorskySW2025polynomialboundsgraphminor}]\label{prop:connected_linkages}
Let $s,r,k,\ell$ be positive integers with $s \geq r+3$.
Let $(G,\Omega)$ be a society with a $\Sigma$-rendition $\rho$ and a nest $\mathcal{C} = \{ C_1, \ldots , C_s \}$.
Moreover, let $\mathcal{L}$ and $\mathcal{R}$ each be radial linkages of order $r$ in $(G,\Omega)$ such that both are orthogonal to $\mathcal{C}$ and let $I = [\ell,k] \subseteq [2,s]$ be an interval with $|I| = r+2$.

Then there exists a radial linkage $\mathcal{P}$ of order $r$ in $(G,\Omega)$ such that
\begin{enumerate}
    \item $\mathcal{P}$ is orthogonal to $\{ C_i ~\!\colon\!~ i \in [s] \setminus I \}$ with endpoints on $C_1$,
    
    \item $H_{\ell}\cap \bigcup\mathcal{P}$ is a subgraph of $H_{\ell}\cap \mathcal{L}$, where $H_{\ell}$ is the inner graph of $C_{\ell}$ in $\rho$.
    In particular, the endpoints of $\mathcal{P}$ on $V(C_1)$ coincide with the endpoints of $\mathcal{L}$ on $V(C_1)$, and
    
    \item $H_k \cap \bigcup \mathcal{P}$ is a subgraph of $H_k \cap \mathcal{R}$, where $H_k$ is the outer graph of $C_k$ in $\rho$.
    In particular, the endpoints of $\mathcal{P}$ on $V(\Omega)$ coincide with the endpoints of $\mathcal{R}$ on $V(\Omega)$.
\end{enumerate}
Moreover, there exists an algorithm that finds $\mathcal{P}$ in time $\mathbf{O}(r|E(G)|)$.
\end{proposition}

Another important ingredient in our proof will be the ability to build layouts via finding surface walls and $K_t$-minors within them.
For this purpose we first introduce so-called ``surface configurations''.

Let $(G,\Omega)$ be a society with a cylindrical rendition and a nest $\mathcal{C}$ around the vortex $c_0$.
Further let $\mathcal{P}_1,\dots,\mathcal{P}_{\ell}$ be a set of transactions on $(G,\Omega)$ as well as $\mathcal{R}$ be a radial linkage such that
\begin{itemize}
    \item $V(\mathcal{P}_i)\cap V(\mathcal{P}_j)=\emptyset$ for all $i\neq j\in[\ell]$ as well as $V(\mathcal{R})\cap V(\mathcal{P}_i)=\emptyset$ for all $i\in[\ell]$,
    \item for each $i\in[\ell]$, $\mathcal{P}_i$ is orthogonal to $\mathcal{C}$ and $\mathcal{R}$ is orthogonal to $\mathcal{C}$, and
    \item there exist pairwise disjoint segments $I_1,J_1,I_2,J_2,\dots,I_{\ell},J_{\ell},R$ of $\Omega$ such that these segments appear on $\Omega$ in the order listed, for each $i\in[\ell]$, $\mathcal{P}_i$ is a $V(I_i)$-$V(J_i)$-linkage, and each path from $\mathcal{R}$ has one endpoint in $V(R)$.
\end{itemize}
We call the tuple $(G,\Omega,\mathcal{C},\mathcal{R},\mathfrak{P}=\{ \mathcal{P}_1,\dots,\mathcal{P}_{\ell}\})$ a \emph{surface configuration} of $(G,\Omega)$ and the sequence $(I_1,J_1,I_2,J_2,\dots,I_{\ell},J_{\ell},R)$ the \emph{signature} of $(G,\Omega,\mathcal{C},\mathcal{R},\mathfrak{P})$.
Finally we say that $(s,r,p_1,\dots,p_{\ell})$ is the \emph{strength} of $(G,\Omega,\mathcal{C},\mathcal{R},\mathfrak{P})$ if $|\mathcal{C}|=s$, $|\mathcal{R}|=r$ and $|\mathcal{P}_i|=p_i$ for all $i\in[\ell]$.

The following core observation on surface configurations stems from \cite{GorskySW2025polynomialboundsgraphminor} as well.

\begin{observation}\label{obs:surface-configs-to-walls}
Let $k \geq 3$, $\ell$, $h$, and $c$ be non-negative integers with $\ell = h + c$.
Moreover, let $s \geq k$, $r \geq 4k$, $p_i \geq 4k$ for all $i \in [\ell]$.

Then, for every surface configuration $(G,\Omega,\mathcal{C},\mathcal{R},\mathfrak{P})$ of strength $(s,r,p_1,\dots,p_{\ell})$ with $h$ handle-transactions and $c$ crosscap-transactions, $G$ contains a $k$-surface-wall $W$ with $h$ handles and $k$ crosscaps as a subgraph such that the base cycles of $W$ are cycles from $\mathcal{C}$ and the vertical paths of the $k$-wall-segment of $W$ are subpaths of the paths from $\mathcal{R}$.
\end{observation}

We will also need the following result that allows us to extract minors from surface walls representing surfaces of sufficient genus.

\begin{proposition}[Gorsky, Seweryn, and Wiederrecht \cite{GorskySW2025polynomialboundsgraphminor}]\label{prop:universal-surface-walls}
There exists a universal constant $c_{\ref{prop:universal-surface-walls}}$ such that for all non-negative integers $c,h,t$, every graph $H$ that embeds in a surface homeomorphic to the surface obtained from the sphere by adding $h$ handles and $c$ crosscaps is a minor of the extended $\big( c_{\ref{prop:universal-surface-walls}}g^4(|V(H)|+g)^2\big)$-surface wall $W$ with $h$ handles, $c$ crosscaps, and 0 vortices, where $g=2h+c$.
Moreover, if $H$ is a complete graph, then there exists an $H$-minor model controlled by $W$.

In particular, $K_t$ is a minor for every extended $k$-surface-wall with $h$ handles and $c$ crosscaps where $2h+c=t^2$ and $k\in \mathbf{\Omega}(t^{12})$.
\end{proposition}

\paragraph{The proof of the Local Structure Theorem.}
We begin by providing more accurate estimates for the functions we will use throughout our proofs.
This also serves to justify the bounds we later claim in the statement in \zcref{thm:strongest_localstructure}.

We begin by establishing a few auxiliary functions that track certain constructions in our proof.
The parameter $g$ here is used to remember how many times we have increased the genus of our surface so far.
\begin{align*}
    \mathsf{clique}_{\ref{thm:strongest_localstructure}}(t) \coloneqq~ & \lceil \mathsf{c}_{\ref{thm:ktminormodel}} (12t^2+6t) \sqrt{\log 144t^2+72t} \rceil \\
    \mathsf{radial}(g,t,k) \coloneqq~& (g+2)(8k+8c_{\ref{prop:universal-surface-walls}}t^{12}+4+ \mathsf{radial}_{\ref{thm:evensocietyclassification}}(\mathsf{clique}_{\ref{thm:strongest_localstructure}}(t),t,k) + 1) \\
    \mathsf{nest}(g,t,k)\coloneqq~ &(g+2)\big(8k+8c_{\ref{prop:universal-surface-walls}}t^{12} + \mathsf{radial}(g-1,t,k) +14+ \mathsf{cost}_{\ref{thm:evensocietyclassification}}(t,4(k+c_{\ref{prop:universal-surface-walls}}t^{12}+1))\\
    &+\mathsf{loss}_{\ref{thm:evensocietyclassification}}(t) \big) + \mathsf{nest}_{\ref{thm:evensocietyclassification}}(\mathsf{clique}_{\ref{thm:strongest_localstructure}}(t),t,4(k+c_{\ref{prop:universal-surface-walls}}t^{12}+1)) + \nicefrac{4k}{2}(t-3)(t-4)+1\\
    \mathsf{transaction}(g,t,k)\coloneqq~ & 4k+4c_{\ref{prop:universal-surface-walls}}t^{12}+24 + 4\mathsf{nest}(g-1,t,k)+2\mathsf{radial}(g-1,t,k)
\end{align*}
For these functions, based on the bounds in \cite{GorskySW2025polynomialboundsgraphminor}, we can give the following estimates:
$\mathsf{clique}_{\ref{thm:strongest_localstructure}}(t) \in \mathbf{O} \in (t^4)$, $\mathsf{radial}(g,t,k) \in \mathbf{O}(gt^{12} + gkt^8)$, $\mathsf{nest}(g,t,k),\mathsf{transaction}(g,t,k) \in \mathbf{O}(g^2t^{12} + kg^2t^8 + t^{36} + kt^{32})$.
With these values fixed, the remaining functions of \zcref{thm:strongest_localstructure} are as follows.
\begin{align*}
    \mathsf{apex}^\mathsf{bip}_{\ref{thm:strongest_localstructure}}(t) \coloneqq~& \max(2\mathsf{apex}^\mathsf{bip}_{\ref{thm:evensocietyclassification}}(t), 8\mathsf{clique}_{\ref{thm:strongest_localstructure}}(t)^3) \\
    \mathsf{apex}_{\ref{thm:strongest_localstructure}}(t,k) \coloneqq~& t^2\cdot \mathsf{apex}^\mathsf{genus}_{\ref{thm:evensocietyclassification}}(t) + \mathsf{apex}^\mathsf{fin}_{\ref{thm:evensocietyclassification}}(t,t,4(k+c_{\ref{prop:universal-surface-walls}}t^{12}+1),\mathsf{transaction}(t^2,t,k)) + 16t^3\\
    \mathsf{depth}_{\ref{thm:strongest_localstructure}}(t,k) \coloneqq~& \mathsf{depth}_{\ref{thm:evensocietyclassification}}(\mathsf{clique}_{\ref{thm:strongest_localstructure}}(t),t,4(k+c_{\ref{prop:universal-surface-walls}}t^{12}+1),\mathsf{transaction}(t^2,t,k))\\
    \mathsf{mesh}_{\ref{thm:strongest_localstructure}}(t,r,k) \coloneqq ~& 100 \mathsf{clique}_{\ref{thm:strongest_localstructure}}(t)^3\Big(\mathsf{apex}_{\ref{thm:strongest_localstructure}}(t,k)+2\mathsf{nest}(t^2,t,k)+2\mathsf{clique}_{\ref{thm:strongest_localstructure}}(t)+11+r\\&
    +\nicefrac{1}{2}(t-3)(t-4)(2\mathsf{depth}_{\ref{thm:evensocietyclassification}}(t,4(k+c_{\ref{prop:universal-surface-walls}}t^{12}+1),\mathsf{transaction}(t^2,t,k))+1)\Big)
\end{align*}
This means we can bound the above functions as follows:
$\mathsf{apex}^\mathsf{bip}_{\ref{thm:strongest_localstructure}}(t) \in \mathbf{O}(t^{12})$, $\mathsf{apex}_{\ref{thm:strongest_localstructure}}(t,k) \in \mathbf{O}((t+k)^{150})$, $\mathsf{depth}_{\ref{thm:strongest_localstructure}}(t,k) \in \mathbf{O}((t+k)^{141})$, and $\mathsf{mesh}_{\ref{thm:strongest_localstructure}}(t,r,k) \in \mathbf{O}((t+k)^{161} + rt^{12} )$.

We note that in many aspects our proof closely follows the core proof of the local structure theorem in \cite{GorskySW2025polynomialboundsgraphminor} (see the proof of Theorem 15.4).
Our main adjustments are further case distinctions when applying \zcref{thm:evenfacedflatmesh}, \zcref{thm:evensocietyclassification}, and other results.

\begin{proof}[Proof of \zcref{thm:strongest_localstructure}]
    We let $t' \coloneqq \mathsf{clique}_{\ref{thm:strongest_localstructure}}(t)$ and let $M$ be a $\mathsf{mesh}_{\ref{thm:strongest_localstructure}}(t,r,k)$-mesh in $G$.
    We first apply \zcref{thm:evenfacedflatmesh} to $M$ using the values $t'$ and $r$.
    In $\mathbf{O}(f_{\ref{thm:ktminormodel}}(t')|V(G)|^{\omega_{\ref{def:matrixmultconstant}}})$-time, the first three options yields either a universal parity breaking grid of order $3t$ controlled by $M$, an odd minor model $\varphi$ of $K_{12t^2+6t}$ controlled by $M$, a set $X \subseteq V(G)$ with $|X| \leq \mathsf{apex}^\mathsf{bip}_{\ref{thm:strongest_localstructure}}(t)$, such that the large block of $\mathcal{T}_M - A$ is bipartite, or a set $A_0$ with $|A_0| < 16{t'}^3$ and an even-faced $r'$-submesh $M'$ of $M$ that is disjoint from $A_0$ and flat in $G - A_0$.
    In the first and second case, we can simply construct a model of $\mathscr{H}_t$ that is controlled by $M$ by reducing both cases first to the universal parity breaking grid of order $3t$ and then applying \zcref{universalcontainsparityhandleandvortex} to find our desired parity handle of order $t$.
    The third case in turn immediately satisfies our statement.

    Thus we may assume that we arrive in the last option of \zcref{thm:evenfacedflatmesh}, yielding a set $A_0$ and an even-faced $r'$-submesh $M'$ of $M$ that is disjoint from $A_0$ and flat in $G - A_0$, with
    \begin{align*}
        r'  \geq    ~& \frac{1}{100{t'}^3}(\mathsf{mesh}(t',k,r) - 2t' - 2) \\
            \geq    ~& \mathsf{apex}_{\ref{thm:strongest_localstructure}}(t',k) + 9 + r + 2\mathsf{nest}({t'}^2,t',k) \\
                    ~& + \nicefrac{1}{2}(t'-3)(t'-4)(2\mathsf{depth}_{\ref{thm:evensocietyclassification}}(t',4(k + c_{\ref{prop:universal-surface-walls}}{t'}^{12} + 1),\mathsf{transaction}({t'}^2,t',k)) + 1) .
    \end{align*}
    We use $M'$ to define a set $\mathcal{C}_0' = \{ {C^0_1}', \ldots , {C^0}'_{\mathsf{nest}({t'}^2,t',k)} \}$ of $\mathsf{nest}({t'}^2,t',k)$ pairwise disjoint cycles by iteratively peeling off the perimeters ${C_i^0}'$ of the $(r' - 2(i-1))$-submeshes of $M'$ for $i \in [\mathsf{nest}({t'}^2,t',k)]$.
    Let $M_0$ be the $( \mathsf{apex}_{\ref{thm:strongest_localstructure}}(t',k) + \nicefrac{1}{2}(t'-3)(t'-4)(2\mathsf{depth}_{\ref{thm:evensocietyclassification}}(t',4(k + c_{\ref{prop:universal-surface-walls}}{t'}^{12} + 1),\mathsf{transaction}({t'}^2,t',k)) + 1) + 9 + r )$-submesh of $M'$ that is left after removing the cycles in $\mathcal{C}_0'$ from $M'$ and let $U_0$ be the perimeter of $M_0$.
    Our setup allows us to make the following observation.

    \begin{observation}
        Let $I \subseteq [\mathsf{nest}({t'}^2,t',k)]$ with $|I| = k$ and let $O$ be a cylindrical $k$-mesh with the cycles ${C^0_i}'$, where $i \in I$, and the rails taken from subpaths of the paths in $M$ that connect these cycles.
        Then the tangle $\mathcal{T}_O$ is a truncation of the tangle $\mathcal{T}_{M_0}$.
    \end{observation}

    Let $\Sigma_0$ be the sphere.
    As $M'$ is flat in $G$, there exists a non-orientable, even-faced $\Sigma_0$-rendition $\rho_0'$ of $G$ with exactly one vortex $c_0$ such that $M'$ is grounded in $\rho_0'$ and the trace of the perimeter of $M'$ in $\rho_0'$ bounds a closed disk $\Delta_0' \subseteq \Sigma_0$ that is disjoint from $c_0$ and every vertex in $V(M') \cap N(\rho_0')$ is drawn in $\Delta_0'$.

    Let $X_0 = V(U_0) \cap N(\rho_0')$ and let $\Delta_0''$ be the open disk in $\Sigma_0$ that is bounded by the trace of $U_0$ in $\rho_0'$ and disjoint from $c_0$.
    Further, let $\Delta_0$ be the closed disk $\Sigma_0 \setminus \Delta_0''$.
    By construction, the $\Delta_0$-society $(G_0,\Omega_0)$ has a non-orientable, even-faced, cylindrical rendition $\rho_0$ in the disk $\Delta_0$ with the nest $\mathcal{C}_0'$ around the vortex $c_0$.
    In fact, we may select $\mathsf{radial}({t'}^2,t',k)$ pairwise disjoint paths obtained from the mesh $M'$ to construct a radial linkage $\mathcal{R}_0'$ that links vertices from distinct rows of $M'$ on $U_0$ to the cycle ${C_1^0}'$.

    We find a cozy nest $\mathcal{C}_0 = \{ C^0_1, \ldots , C^0_{\mathsf{nest}({t'}^2,t',k)} \}$ around $c_0$ in $\rho_0$ in $\mathbf{O}( \mathsf{nest}({t'}^2,t',k)|E(G)|^2 )$-time using \zcref{prop:makenestcozy}.
    Applying \zcref{prop:radialtoorthogonal} then yields a radial linkage $\mathcal{R}_0$ of order $\mathsf{radial}({t'}^2,t',k)$ that is end-identical with $\mathcal{R}_0'$ and orthogonal to $\mathcal{C}_0$ in $\mathbf{O}(\mathsf{radial}({t'}^2,t',k)\mathsf{nest}({t'}^2,t',k)|E(G)|)$-time.
    We observe that making our nest cozy did not affect the way the tangles interact.

    \begin{observation}
        Let $I \subseteq [\mathsf{nest}({t'}^2,t',k)]$ with $|I| = k$ and let $O$ be a cylindrical $k$-mesh with the cycles $C^0_i$, where $i \in I$, and the rails taken from subpaths of the paths in $M$ that connect these cycles.
        Then the tangle $\mathcal{T}_O$ is a truncation of the tangle $\mathcal{T}_{M_0}$.
    \end{observation}

    If $\overline{G_0}$ is the graph $G - (V(G_0) \setminus V(\Omega_0))$, then $(\overline{G_0},\Omega_0)$ has a vortex-free, even-faced rendition derived from $\rho_0'$ in the surface obtained from $\Sigma_0$ by removing an open disk.
    Finally, let $\mathcal{C}^\star = \{ C^0_{\mathsf{nest}({t'}^2,t',k) - (4(k + c_{\ref{prop:universal-surface-walls}}{t'}^{12} + 1) - 1)}, \ldots , C^0_{\mathsf{nest}({t'}^2,t',k)} \}$ be the outermost $4(k + c_{\ref{prop:universal-surface-walls}}{t'}^{12} + 1)$ cycles of $\mathcal{C}^0$.
    This makes $(G_0,\Omega_0,\mathcal{C}^\star,\mathcal{R}_0,\emptyset)$ a surface configuration of $(G_0,\Omega_0)$ for the surface $\Sigma_0$ of strength $(4(k + c_{\ref{prop:universal-surface-walls}}{t'}^{12} + 1), \mathsf{radial}({t'}^2,t',k))$.

    \paragraph{Inductively refining a layout: The setup.}
    As a consequence of \zcref{obs:surface-configs-to-walls}, the tuple $\Lambda_0 = (A_0,M_0,D_0,\rho_0')$ is in fact a $\Sigma_0$-landscape of detail $4(k + c_{\ref{prop:universal-surface-walls}}{t'}^{12} + 1)$, where $D_0$ is a $4(k + c_{\ref{prop:universal-surface-walls}}{t'}^{12} + 1)$-surface-wall without handles or crosscaps whose base cycles coincide with the cycles in $\mathcal{C}^\star$.

    Given the following collection of objects for some $i \in [0,{t'}^2-1]$, our goal is to show that by applying \zcref{thm:evensocietyclassification} we can either find the small set $X$ whose removal turns the large block of $\mathcal{T}_M - X$ bipartite, find the parity handle of order $t$ as an odd minor controlled by $D_i$, or through \ref{item_s6} in \zcref{thm:evensocietyclassification}, create a non-orientable, even $\Sigma_i$-layout centred at $M$.
    We have already shown how to accomplish this for $i = 0$ and thus, our efforts will imply that we eventually find the $\Sigma_i$-layout centred at $M$ that we desire.

    Our list of objects is as follows.
    Let $i \in [0, {t'}^2-1]$ be given together with
    \begin{itemize}
        \item an apex set $A_i \subseteq V(G)$ with $A_{i-1} \subseteq A_i$, $A_{-1} \coloneqq \emptyset$, and $|A_i| \leq i \cdot \mathsf{apex}^\mathsf{genus}_{\ref{thm:evensocietyclassification}}(t') + 16{t'}^3$,
        
        \item a surface $\Sigma_i$ obtained from the sphere by adding a total of $i$ handles and crosscaps in some combination,
        
        \item a non-orientable, even-faced $\Sigma_i$-rendition $\rho_i$ with a unique vortex $c_i$,
        
        \item a $4(k + c_{\ref{prop:universal-surface-walls}}{t'}^{12} + 1)$-surface wall $D_i$ with the same amount of handles and crosscaps used to obtain $\Sigma_i$ from the sphere, such that the base cycles of $D_i$ coincide with the cycles of $\mathcal{C}^\star$,
        
        \item a society $(G_i,\Omega_i)$ such that there exists a $\rho_i$-aligned disk $\Delta_i$ whose boundary intersects $\rho_i$ exactly in $V(\Omega_i)$ and the restriction $\rho_i'$ of $\rho_i$ to $\Delta_i$ is a cylindrical rendition of $(G_i,\Omega_i)$ with $c_i$ being its unique vortex,
        
        \item a cozy nest $\mathcal{C}_i = \{ C^i_1, \dots , C^i_{\mathsf{nest}({t'}^2-i,t',k)} \}$ of order $\mathsf{nest}({t'}^2-i,t',k)$ in $\rho_i'$ around $c_i$,
        
        \item a family of transactions $\mathfrak{P}_i = \{ \mathcal{P}_1, \dots , \mathcal{P}_i \}$ on $(G_0,\Omega_0)$ such that $\mathcal{P}_j$ is a handle or crosscap transaction,
        
        \item a radial linkage $\mathcal{R}_i$ of order $\mathsf{radial}({t'}^2-i,t',k)$ orthogonal to $\mathcal{C}_i \cup \mathcal{C}^\star$ whose endpoints on $(G_0,\Omega_0)$ coincide with some of the endpoints of $\mathcal{R}_0$ and which is disjoint from the paths in $\bigcup \mathfrak{P}_i$, and
        
        \item all objects above are chosen such that $(G_0 - A_i, \Omega_0, \mathcal{C}^\star, \mathcal{R}_i, \mathfrak{P}_i)$ is a $\Sigma_i$-configuration of strength $(4(k + c_{\ref{prop:universal-surface-walls}}{t'}^{12} + 1), \mathsf{radial}({t'}^2-i,t',k),p_1, \dots , p_i)$ with $p_j = 4(k + c_{\ref{prop:universal-surface-walls}}{t'}^{12})$ for all $j \in [i]$.
    \end{itemize}
    Using \zcref{obs:surface-configs-to-walls} again, we can see that the last point ensures the existence of the $4(k+c_{\ref{prop:universal-surface-walls}}{t'}^{12}+1)$-surface-wall $D_i$.
    Should we ever reach $i = {t'}^2$, we can find a $K_{t'}$-minor controlled by $D_i$ via \zcref{prop:universal-surface-walls}.
    This in turn lets us apply \zcref{thm:ktminormodel} to either find the set $X$ such that the large block of $\mathcal{T}_{D_i} - X$ is bipartite or an odd $K_{12t^2+6t}$-minor model controlled by $D_i$ that allows us to construct $\mathscr{U}_{3t}$ and thus $\mathcal{H}_t$, which means that we are done as well.
    We therefore know that this process will terminate.
    Both of these properties transfer to $M$ by construction of our $\Sigma_i$-configuration.
    Thus, if the large block of $\mathcal{T}_{D_i} - X$ is bipartite, then so is the large block of $\mathcal{T}_M - X$ is bipartite, and analogously, if we find a minor model controlled by $D_i$, it is also controlled by $M$.
    Either outcome therefore satisfies our desires.

    \paragraph{Inductively refining a layout: Construction.}
    Suppose for some $i \in [0, {t'}^2-1]$ our list of objects above has already been constructed.
    We begin by applying \zcref{thm:evensocietyclassification} to $(G_i, \Omega_i)$ (using $t'$ as the desired size of the odd or bipartite clique minor, $\mathsf{nest}({t'}^2-i,t',k)$ as the size of the nest, $4(k + c_{\ref{prop:universal-surface-walls}}{t'}^{12})$ as the size of the handle or crosscap transaction we want, $3t$ as the size of the universal parity grid, respectively the parity handle, we wish to find, and $4(k + c_{\ref{prop:universal-surface-walls}}{t'}^{12}) + 1$ as the strength of the extended surface wall we want to find) and split our remaining arguments into a case distinction depending on the outcome of this application.

    \textbf{Case 1:}
    We first consider the outcomes \ref{item_s1}, \ref{item_s2}, and \ref{item_s7} of \zcref{thm:evensocietyclassification}./
    Clearly, \ref{item_s7} immediately yields a model of the parity handle of order $t$ that is controlled by $M$.
    Similarly, the odd $K_{t'}$-minor model controlled by $M$ from \ref{item_s1} also immediately helps us, as it allows us to find the parity handle and case \ref{item_s2} also allows us to proceed as laid out previously to find the desired parity handle.
    
    \textbf{Case 2:}
    Suppose we instead land in the outcome \ref{item_s3} of \zcref{thm:evensocietyclassification}.
    Then there exists a set $A$ with $|A| \leq \mathsf{apex}^\mathsf{bip}(t')$ and a bipartite $K_{t'}$-minor model $\varphi$ in $G - A$ such that the large block of $\mathcal{T}_\varphi - A$ is bipartite and $\varphi$ is controlled by a (cylindrical) mesh $M_{\mathcal{C}_i}$ containing all of the cycles of $\mathcal{C}_i$.
    Note that the tangle associated with $M_{\mathcal{C}_i}$ is a truncation of $\mathcal{T}_{M_0}$.
    Thus in particular the large block of $\mathcal{T}_M - A$ is bipartite and we have again reached the first option of our statement.

    \textbf{Case 3:}
    Let $(G',\Omega')$ be the $C^i_{\mathsf{nest}(t'^2-i,t',k) - \mathsf{cost}_{\ref{thm:evensocietyclassification}}(t',4(k + c_{\ref{prop:universal-surface-walls}}{t'}^{12} + 1) - 1}$-society in $\rho_i$.
    This case contains two subcases which can be treated mostly analogously.
    These correspond to outcomes \ref{item_s4} and \ref{item_s5} of \zcref{thm:evensocietyclassification}, where both outcomes yield a set $A \subseteq V(G)$ with $|A| \leq \mathsf{apex}^\mathsf{genus}_{\ref{thm:evensocietyclassification}}(t')$.

    In the first of the two outcomes, we also find an even-faced crosscap transaction of order $\mathsf{transaction}(t'^2-i,t',k)$ in $(G' - A, \Omega')$ together with a nest $\mathcal{C}'$ in $\rho_i$ of order $\mathsf{nest}(t'^2-i,t',k) - \mathsf{cost}_{\ref{thm:evensocietyclassification}}(t',4(k + c_{\ref{prop:universal-surface-walls}}{t'}^{12} + 1) - \mathsf{loss}_{\ref{thm:evensocietyclassification}}(t')$ around $c_i$ to which $\mathcal{Q}$ is orthogonal.
    In the second outcome, we instead find even-faced handle transaction of order $\mathsf{transaction}(t'^2-i,t',k)$ in $(G' - A, \Omega')$ together with a nest $\mathcal{C}'$ in $\rho_i$ of order $\mathsf{nest}(t'^2-i,t',k) - \mathsf{cost}_{\ref{thm:evensocietyclassification}}(t',4(k + c_{\ref{prop:universal-surface-walls}}{t'}^{12} + 1) - \mathsf{loss}_{\ref{thm:evensocietyclassification}}(t')$ around $c_i$ to which the two planar transactions $\mathcal{Q}_1$ and $\mathcal{Q}_2$ that make up $\mathcal{Q}$ are orthogonal.

    We then shave off a little off of the sides of the transaction in both cases as follows.
    In the first outcome, we reduce the order of $\mathcal{Q}$ by $2\mathsf{radial}(t'^2-i-1,t',k)$ through shedding off the last $2\mathsf{radial}(t'^2-i+1,t',k)$ paths of $\mathcal{Q}$.
    Alternatively, in the second outcome, we reduce the order of $\mathcal{Q}$ by $2\mathsf{radial}(t'^2-i-1,t',k)$ by removing from each $\mathcal{Q}_j$, $j\in[2]$, the last $\mathsf{radial}(t'^2-i-1,t',k)$ paths.
    Let $\mathcal{Q}$, $\mathcal{Q}_1'$, and $\mathcal{Q}_2'$ be the resulting transactions.
    Thanks to our choices, we may select a radial linkage $\mathcal{L}$ of order $\mathsf{radial}(t'^2-i-1,t',k)$ in $(G'-A,\Omega')$ that is orthogonal to $\mathcal{C}'$.
    In fact, we may select $\mathcal{Q}'$ and $\mathcal{L}$ such that there are disjoint segments $I_1$ and $I_2$ of $\Omega'$ where $\mathcal{Q}'$ has all endpoints in $I_1$ and the endpoints of $\mathcal{L}$ in $V(\Omega')$ are found in $I_2$.
    From this construction, we know that $|\mathcal{Q}'| = 4k + 4c_{\ref{prop:universal-surface-walls}}{t'}^{12}+14 + 4\mathsf{nest}(t'^2-i-1,t',k)$ and $|\mathcal{Q}_1'| = |\mathcal{Q}_2'| = 2k + 2c_{\ref{prop:universal-surface-walls}}t^{12}+7 + 2\mathsf{nest}(t'^2-i-1,t',k)$ in case $\mathcal{Q}'$ is a handle transaction.

    For the remainder of our arguments in \textbf{Case 3} we focus on the case in which $\mathcal{Q}'$ is a crosscap transaction.
    The other case can be handled analogously, except that instead of applying \zcref{lemma:integrate-crosscap} one needs to use \zcref{lemma:integrate-handle-clean} and some choices need to be made for $\mathcal{Q}_1'$ and $\mathcal{Q}_2'$ separately instead of only for $\mathcal{Q}'$.
    We chose the size of $\mathcal{Q}'$ explicitly to allow for this, as we will elaborate on later.

    We assume that $\mathcal{Q}' = \{ Q_1, \ldots , Q_{4k + 4c_{\ref{prop:universal-surface-walls}}{t'}^{12} + 14 + 4\mathsf{nest}(t'^2-i-1,t',k)} \}$ are indexed naturally.
    Let $a \coloneqq 2\mathsf{nest}(t'^2-i-1,t',k) + 7$ and $b \coloneqq 2\mathsf{nest}(t'^2-i-1,t',k)+ 4k + 4c_{\ref{prop:universal-surface-walls}}{t'}^{12} + 8$.
    Moreover, let $\mathcal{A} \coloneqq \{ Q_1, \ldots , Q_{a-1} \}$, $\mathcal{B} = \{ Q_{b+1}, \ldots , Q_{4k + 4c_{\ref{prop:universal-surface-walls}}{t'}^{12} + 14 + 4\mathsf{nest}(t'^2-i-1,t',k)} \}$, and $\mathcal{I} \coloneqq \{ Q_a, \ldots , Q_b \}$.
    Thus $|\mathcal{I}| = 4k + 4c_{\ref{prop:universal-surface-walls}}{t'}^{12} + 8 + 2\mathsf{nest}(t'^2-i-1,t',k)$ and $|\mathcal{A}| = |\mathcal{B}| = 2\mathsf{nest}(t'^2-i-1,t',k) + 6$.

    As a result we have split $\mathcal{Q}'$ into three parts.
    The paths in $\mathcal{A} \cup \mathcal{B}$ will be used to build a new nest for the next society, with the paths in the interior, thus called $\mathcal{I}$, serving to extend our growing $\Sigma_i$-configuration.
    In the following we will also be rerouting radial linkages to achieve two goals.
    First, we will want a radial linkage that feeds into the $\Sigma_i$-configuration and will ultimately allow us to use the paths in $\mathcal{I}$ to add a crosscap or handle segment to a surface wall.
    Secondly, we will want a separate radial linkage, later called $\mathcal{J}$, that will be extended to form $\mathcal{R}_{i+1}$ and allow us to continue our iterative process.

    We let the cycles in $\mathcal{C}'$ be numbered $C_1', \ldots , C_{\mathsf{nest}(t'^2-i,t',k) - \mathsf{cost}_{\ref{thm:evensocietyclassification}}(t',4(k+c_{\ref{prop:universal-surface-walls}}{t'}^{12}+1)) - \mathsf{loss}_{\ref{thm:evensocietyclassification}}(t')}'$ from innermost to outermost.
    Further, we let
    \begin{align*}
         I  =~& [\mathsf{nest}(t'^2-i,t',k) - \mathsf{cost}_{\ref{thm:evensocietyclassification}}(t',4(k+c_{\ref{prop:universal-surface-walls}}{t'}^{12}+1)) - \mathsf{loss}_{\ref{thm:evensocietyclassification}}(t') - 3, \ \mathsf{nest}(t'^2-i,t',k) \ - \\
            & \quad \mathsf{cost}_{\ref{thm:evensocietyclassification}}(t',4(k+c_{\ref{prop:universal-surface-walls}}{t'}^{12}+1)) - \mathsf{loss}_{\ref{thm:evensocietyclassification}}(t') - 8k-8c_{\ref{prop:universal-surface-walls}}{t'}^{12} - \mathsf{radial}(t'^2-i-1,t',k) - 9] \\
            \subseteq~& [\mathsf{nest}(t'^2-i,t',k) - \mathsf{cost}_{\ref{thm:evensocietyclassification}}(t',4(k+c_{\ref{prop:universal-surface-walls}}{t'}^{12}+1)) - \mathsf{loss}_{\ref{thm:evensocietyclassification}}(t')] .
    \end{align*}
    Then $|I| = 8k + 8c_{\ref{prop:universal-surface-walls}}{t'}^{12} + \mathsf{radial}(t'^2-i-1,t',k) + 6 = 2|\mathcal{I}| + |\mathcal{L}| + 2$.
    This in particular means
    \begin{align*}
        |I| - 2 = 2|\mathcal{I}| + |\mathcal{L}| =~& 8k+8c_{\ref{prop:universal-surface-walls}}{t'}^{12}+4 + \mathsf{radial}(t'^2-i-1,t',k) \leq \mathsf{radial}(t'^2-i,t',k) = |\mathcal{R}_i|.
    \end{align*}
    Let us now select $\mathcal{I}'$ to be all $V(\Omega')$-$V(C_1')$-subpaths of the paths in $\mathcal{I}$.
    It follows that $\mathcal{I}'$ is a radial linkage of order $2|\mathcal{I}|$.
    Moreover, $\mathcal{I}' \cup \mathcal{L}$ is a radial linkage of order $8k+8c_{\ref{prop:universal-surface-walls}}{t'}^{12}+\mathsf{radial}(t'^2-i-1,t',k)+4 = |I|-2$ which is orthogonal to $\mathcal{C}'$.
    This allows us to call \zcref{prop:connected_linkages} for $\mathcal{R}_i$, $\mathcal{I}' \cup \mathcal{L}$, $\mathcal{C}'$, and $I$.
    As a result we obtain, in time $\mathbf{poly}(t+k)|E(G)|$, a radial linkage $\mathcal{R}'$ of order $|I|-2$ which shares its endpoints on $V(\Omega')$ with the endpoints of $\mathcal{R}_i$ and its endpoints on $C_1'$ with $\mathcal{I}' \cup \mathcal{L}$.
    Moreover, $\mathcal{R}'$ is orthogonal to $\mathcal{C}' \setminus \{ C_i ~\colon~ i \in I \}$ and within the inner graph $G''$ of $C_{\mathsf{nest}(t'^2-i,t',k)-\mathsf{cost}_{\ref{thm:evensocietyclassification}}(t',4(k+c_{\ref{prop:universal-surface-walls}}{t'}^{12}+1))-\mathsf{loss}_{\ref{thm:evensocietyclassification}}(t')-8k-8c_{\ref{prop:universal-surface-walls}}{t'}^{12}-\mathsf{radial}(t'^2-i-1,t',k)-9}'$ it is disjoint from $\mathcal{Q}' \setminus \mathcal{I}$.
    Recall that $\mathcal{L}$ is a radial linkage that we constructed such that it is found outside of the strip of $\mathcal{Q}'$.
    This allows us to let $\mathcal{J} \subseteq \mathcal{R}'$ be all those paths that do not intersect the paths in $\mathcal{I}$ within $G''$.
    Then $|\mathcal{J}| = |\mathcal{R}'| - 8k-8c_{\ref{prop:universal-surface-walls}}{t'}^{12}-4 = \mathsf{radial}(t'^2-i-1,t',k)$.

    Next we adjust the remaining nest.
    Let $\mathcal{C}'' \subseteq \mathcal{C}$ be the set of the innermost
    \[ \mathsf{nest}(t'^2-i,t',k) - \mathsf{cost}_{\ref{thm:evensocietyclassification}}(t',4(k+c_{\ref{prop:universal-surface-walls}}{t'}^{12}+1)) - \mathsf{loss}_{\ref{thm:evensocietyclassification}}(t') - 8k-8c_{\ref{prop:universal-surface-walls}}{t'}^{12} - \mathsf{radial}(t'^2-i-1,t',k)-10 \]
    cycles of the nest $\mathcal{C}'$.
    Further, let $\Omega''$ be the cyclic ordering of the ground vertices of the cycle $C'_{\mathsf{nest}(t'^2-i,t',k) - \mathsf{cost}_{\ref{thm:evensocietyclassification}}(t',4(k+c_{\ref{prop:universal-surface-walls}}{t'}^{12}+1)) - \mathsf{loss}_{\ref{thm:evensocietyclassification}}(t') - 8k-8c_{\ref{prop:universal-surface-walls}}{t'}^{12}-\mathsf{radial}(t'^2-i-1,t',k)-9}$ and let $(G'',\Omega'')$ be the resulting society with cylindrical rendition $\rho''$, which is the restriction of $\rho_i$ to the $c_i$-disk $\Delta'$ of $C'_{\mathsf{nest}(t'^2-i,t',k) - \mathsf{cost}_{\ref{thm:evensocietyclassification}}(t',4(k+c_{\ref{prop:universal-surface-walls}}{t'}^{12}+1)) - \mathsf{loss}_{\ref{thm:evensocietyclassification}}(t') - 8k-8c_{\ref{prop:universal-surface-walls}}{t'}^{12} - \mathsf{radial}(t'^2-i-1,t',k) - 9}$ in $\rho_i$.
    Of course $\rho''$ remains non-orientable and even-faced.
    Moreover, let $\mathcal{J}'$ and $\mathcal{Q}''$ be the restrictions of $\mathcal{J}$ and $\mathcal{Q}'$ to $G''$.
    Note that $\mathcal{C}''$ is a nest in $\rho''$ and we have
    \begin{align*}
        |\mathcal{C}''| =~&\mathsf{nest}(t'^2-i,t',k) - \mathsf{cost}_{\ref{thm:evensocietyclassification}}(t',4(k+c_{\ref{prop:universal-surface-walls}}{t'}^{12}+1)) - \mathsf{loss}_{\ref{thm:evensocietyclassification}}(t') - 8k-8c_{\ref{prop:universal-surface-walls}}{t'}^{12} \\
                        & - \mathsf{radial}(t'^2-i-1,t',k) - 10 \\
                        =~& (t'^2-i-1)\big(8k+8c_{\ref{prop:universal-surface-walls}}{t'}^{12} + \mathsf{radial}(t'^2-i,t',k) + 24 + \mathsf{cost}_{\ref{thm:evensocietyclassification}}(t',4(k+c_{\ref{prop:universal-surface-walls}}{t'}^{12}+1)) \ + \\
                        & \mathsf{loss}_{\ref{thm:evensocietyclassification}}(t') \big) + \mathsf{nest}_{\ref{thm:evensocietyclassification}}(t',4(k+c_{\ref{prop:universal-surface-walls}}{t'}^{12}+1)) + \nicefrac{4k}{2}(t'-3)(t'-4) + 15 \\
                        \geq~& (t'^2-i-1)\big(8k+8c_{\ref{prop:universal-surface-walls}}{t'}^{12} + \mathsf{radial}(t'^2-i-1,t',k)+24 \ + \\
                        & \mathsf{cost}_{\ref{thm:evensocietyclassification}}(t',4(k+c_{\ref{prop:universal-surface-walls}}{t'}^{12}+1)) + \mathsf{loss}_{\ref{thm:evensocietyclassification}}(t') \big) \ + \\
                        & \mathsf{nest}_{\ref{thm:evensocietyclassification}}(t',4(k+c_{\ref{prop:universal-surface-walls}}{t'}^{12}+1)) + \nicefrac{4k}{2}(t'-3)(t'-4) + 15 \\
                        =~& \mathsf{nest}(t'^2-i-1,t',k)+14 .
    \end{align*}
    Furthermore, for $\mathcal{Q}''$ we have 
    \begin{align*}
        |\mathcal{Q}''| =~& 4k+4c_{\ref{prop:universal-surface-walls}}{t'}^{12} + 14 + 4\mathsf{nest}(t'^2-i-1,t',k) \\
                        =~& 2 (2k+2c_{\ref{prop:universal-surface-walls}}{t'}^{12}) + 4\mathsf{nest}(t'^2-i-1,t',k) + 14 .
    \end{align*}
    According to the calculations above we can now apply \zcref{lemma:integrate-crosscap}.
    Indeed, we exceed the necessary bounds for the lemma by more than $2\mathsf{nest}(t'^2-i-1,t',k)$.
    If instead we were dealing with a handle transaction, we would next want to apply \zcref{lemma:integrate-handle-clean}.
    Here we would have constructed two transactions, namely $\mathcal{Q}''_1$ and $\mathcal{Q}''_2$, each of order at least $2k+2c_{\ref{prop:universal-surface-walls}}{t'}^{12} + 2\mathsf{nest}(t'^2-i-1,t',k) + 7$ and thus, together, they form a transaction $\mathcal{Q}''$ of the order above.
    This explains our choice for the size of $\mathcal{Q}''$.

    We apply \zcref{lemma:integrate-crosscap} to $(G'',\Omega'')$, $\mathcal{C}''$, $\mathcal{J}'$, and $\mathcal{Q}''$ with $p = 4k + 4c_{\ref{prop:universal-surface-walls}}{t'}^{12} + 2$, rendition $\rho''$, and the disk $\Delta'$.\footnote{We note that we will need to argue a bit more about the case in which we apply \zcref{lemma:integrate-handle-clean}, since the result will not necessarily be a non-orientable, even-faced rendition. This will be deferred until we have finished our arguments for the crosscap case.}
    As a result we obtain the following list of objects.
    Let $\Sigma^*$ be a surface homeomorphic to the projective plane minus an open disk which is obtained from $\Delta'$ by adding a crosscap to the interior of $c_i$.
    Then there exists $\mathcal{I}'\subseteq \mathcal{Q}''$, which coincides with the restriction of $\mathcal{I}$ to $G''$, and a non-orientable, even-faced rendition $\rho'''$ of $(G'',\Omega'')$ in $\Sigma^*$ with a unique vortex $c_{i+1}'$ and the following properties hold:
    \begin{itemize}
        \item $\mathcal{I}'$ is disjoint from $\sigma(c_{i+1}')$,
    
        \item the vortex society $(G_{i+1},\Omega_{i+1})$ of $c_{i+1}'$ in $\rho'''$ has an even-faced, cylindrical rendition $\rho_{i+1}'$ with nest $\mathcal{C}_{i+1} = \{ C_1^{i+1}, \ldots , C^{i+1}_{\mathsf{nest}(t'^2-i-1,t',k)} \}$ around the unique vortex $c_{i+1}$,
        
        \item every element of $\mathcal{J}'$ has an endpoint in $V(\sigma_{\rho_{i+1}'}(c_{i+1}))$, and
    
        \item $\mathcal{J}'$ is orthogonal to $\mathcal{C}_{i+1}$. Moreover,
    
        \item let $\mathcal{J}' = \{ J_1, \ldots , J_{\mathsf{radial}(t'^2-i-1,t',k)} \}$.
        For each $j \in [\mathsf{radial}(t'^2-i-1,t',k)]$ let $x_j$ be the endpoint of $J_j$ in $V(\Omega'')$ and let $y_j$ be the endpoint of $J_j$ on $c_{i+1}'$; then if $x_1, x_2, \ldots , x_{\mathsf{radial}(t'^2-i-1,t',k)}$ appear on $\Omega''$ in the order listed, then $y_1, y_2, \ldots , y_{\mathsf{radial}(t'^2-i-1,t',k)}$ appear on $N_{\rho'''}(c_{i+1}')$ in the order listed.
    
        \item Finally, let $\Delta''$ be the open disk bounded by the trace of the outermost cycle of $\mathcal{C}''$ in $\rho''$.
        Then $\rho''$ restricted to $\Delta' \setminus \Delta''$ is equal to $\rho'''$ restricted to $\Delta' \setminus \Delta''$.
    \end{itemize}
    Now let $\rho_{i+1}$ be obtained by first unifying the renditions $\rho'''$ and $\rho_{i+1}'$ along the vortex society of $c_{i+1}'$, then combining the resulting rendition of $(G'',\Omega'')$ in the disk $\Delta'$ with the rendition $\rho''$ along to boundary of $\Delta''$, and then reintegrating $\rho''$ into $\rho_i$.
    Moreover, let $\Sigma_{i+1}$ be obtained from $\Sigma_i$ by removing the interior of $\Delta'$ and replacing it with $\Sigma^*$.
    Notice that this means that $\Sigma_{i+1}$ is obtained from $\Sigma_i$ by adding a single crosscap.
    We also set $A_{i+1}\coloneqq A_i\cup A$ and obtain 
    \begin{align*}
        |A_{i+1}|   \leq~& i\cdot\mathsf{apex}^\mathsf{genus}_{\ref{thm:societyclassification}}(t')+16t'^3+\mathsf{apex}^\mathsf{genus}_{\ref{thm:societyclassification}}(t')\\
                    \leq~& (i+1)\mathsf{apex}^\mathsf{genus}_{\ref{thm:societyclassification}}(t')+16t'^3.
    \end{align*}
    So the first three points of our invariant are maintained.

    By construction, the radial linkage $\mathcal{R}'$ can be extended onto the restriction of $\mathcal{R}_i$ to the proper outer graph of the closed curve obtained by following along the vertices on $\Omega'$.
    This allows us to extend the crosscap transaction $\mathcal{I}$ of order $4k + 4c_{\ref{prop:universal-surface-walls}}t'^{12} +2$ along $\mathcal{R}_i$ to obtain a crosscap transaction $\mathcal{P}_{i+1}$ on $(G_0,\Omega_0)$ whose paths are disjoint from all paths in $\mathcal{P}_j$ for all $j \in [i]$.
    Further, we can extend $\mathcal{R}'$ along $\mathcal{R}_i$ to form the radial linkage $\mathcal{R}_{i+1}$ such that it is orthogonal to both $\mathcal{C}^*$ and $\mathcal{C}_{i+1}$ and satisfies
    \begin{align*}
        |\mathcal{R}_{i+1}| = \mathsf{radial}(t'^2-i-1,t',k),
    \end{align*}
    its endpoints coincide with some of the endpoints of $\mathcal{R}_0$, and it is disjoint from the paths in $\mathfrak{P}_{i+1} = \mathfrak{P}_i \cup \{ \mathcal{P}_{i+1} \}$.
    Indeed, it is straightforward to see that $(G_0 - A_{i+1}, \Omega_0, \mathcal{C}^*, \mathcal{R}_{i+1}, \mathfrak{P}_{i+1})$ is a $\Sigma_{i+1}$-configuration of strength $(4(k+c_{\ref{prop:universal-surface-walls}}{t'}^{12}+1), \mathsf{radial}(t'^2-i,t',k), p_1, \ldots , p_{i+1})$ with $p_j = 4k+4c_{\ref{prop:universal-surface-walls}}{t'}^{12}$.
    So the last three points of our invariant are also satisfied.

    By \zcref{obs:surface-configs-to-walls} the $\Sigma_{i+1}$-configuration found in the previous paragraph yields the existence of an extended $(4(k+c_{\ref{prop:universal-surface-walls}}{t'}^{12}+1))$-surface wall $D_{i+1}$ with the amount of crosscaps and handles used to obtain $\Sigma_{i+1}$ from the sphere such that the base cycles of $D_{i+1}$ coincide with the cycles of $\mathcal{C}^*$ and $D_{i+1}$ has no vortex segments.
    This, together with the society $(G_{i+1},\Omega_{i+1})$, the cylindrical rendition $\rho_{i+1}$, and the nest $\mathcal{C}_{i+1}$ around the unique vortex $c_{i+1}$ of both $\rho_{i+1}$ and $\rho'_{i+1}$, satisfies the remaining three points of our invariant and thus this step is complete.

    We now remark upon the differences that occur if we are instead dealing with the case in which $\mathcal{Q}''$ and we applied \zcref{lemma:integrate-handle-clean} instead.
    The crucial difference here is that while $\rho''$ is even-faced, it is not guaranteed to be non-orientable.
    In case $\rho''$ happens to be non-orientable, we can proceed as above.
    Otherwise, following the discussion around \zcref{lem:get-parity-handle} and \zcref{lemma:integrate-handle-clean}, we receive $\mathscr{H}_t$ as an odd minor that is controlled by $M$.
  
    Observe that, in the case where $i=t'^2-1$ we have reached a situation where $K_{t'}$ embeds in $\Sigma_{i+1}$.
    This means that \zcref{prop:universal-surface-walls} implies the existence of a $K_{t'}$-model controlled by $D_{i+1}$.
    So in this case we would be done immediately.
    Hence, we may assume that $i+1\in[t'^2-1]$.

    \textbf{Case 4:} The remaining case to discuss is \ref{item_s6} in \zcref{thm:evensocietyclassification}, which provides us with a set $A \subseteq V(G_i)$ with $|A| \leq \mathsf{apex}^\mathsf{fin}_{\ref{thm:evensocietyclassification}}(t',t',4(k+c_{\ref{prop:universal-surface-walls}}{t'}^{12}+1),\mathsf{transaction}(t'^2,t',k))$, an even-face rendition $\rho'$ of $(G_i - A, \Omega_i)$ in $\Delta_i$ with breadth $b \in [\nicefrac{1}{2}(t'-3)(t'-4)+t-1]$ and depth at most $\mathsf{depth}_{\ref{thm:evensocietyclassification}}(t',4(k+c_{\ref{prop:universal-surface-walls}}{t'}^{12}+1),\mathsf{transaction}(t'^2-i,t',k))$, and a parity $(4(k+c_{\ref{prop:universal-surface-walls}}{t'}^{12}+1))$-surface-wall $D$ with signature $(0,0,0,0,0,b)$, such that $D$ is grounded in $\rho'$, the base cycles of $D$ are the cycles $C_{\mathsf{nest}(t'^2-i,t',k)-\mathsf{cost}_{\ref{thm:evensocietyclassification}}(t',4(k+c_{\ref{prop:universal-surface-walls}}{t'}^{12}+1))-1-k-c_{\ref{prop:universal-surface-walls}}{t'}^{12}}^i, \ldots , C_{\mathsf{nest}(t'^2-i,t',k) - \mathsf{cost}_{\ref{thm:evensocietyclassification}}(t',4(k+c_{\ref{prop:universal-surface-walls}}{t'}^{12}+1))-1}^i$, and there exists a bijection between the vortices $v$ of $\rho'$ and the even vortex segments $S_v$ of $D$, where $v$ is the unique vortex contained in the disk $\Delta_{S_v}$ defined by the trace of the inner cycle of the nest of $S_v$ where $\Delta_{S_v}$ is chosen to avoid the trace of the simple cycle of $D$.

    At this point the additional infrastructure on the order of $c_{\ref{prop:universal-surface-walls}}{t'}^{12}$ has served its purpose and we can concentrate on just building our extended $k$-surface wall.
    For this purpose, we must now combine the radial linkage $\mathcal{R}_i$ we found in our induction with the rails we have for the vortex segments in $D$.
    Notice that the union of all rails of $D$ form a radial linkage $\mathcal{R}''$ of order $16b(k+c_{\ref{prop:universal-surface-walls}}{t'}^{12})$ such that every vortex segment provides exactly $16(k+c_{\ref{prop:universal-surface-walls}}{t'}^{12})$ of these paths.
    Moreover, there are $4bk$ cycles among $C_{\mathsf{nest}(t'^2-i,t',k)-\mathsf{cost}_{\ref{thm:evensocietyclassification}}(t',4(k+c_{\ref{prop:universal-surface-walls}}{t'}^{12}+1))-1-4bk}^i, \ldots , C_{\mathsf{nest}(t'^2-i,t',k)-\mathsf{cost}_{\ref{thm:evensocietyclassification}}(t',4(k+c_{\ref{prop:universal-surface-walls}}{t'}^{12}+1))-1}^i$.
    Now let $\mathcal{R}'$ be a radial linkage of order $4bk$ formed by selecting $4k$ rails from each vortex segment of $D$.

    By definition of $D$, for every $j \in [\mathsf{nest}(t'^2-i,t',k)-\mathsf{cost}_{\ref{thm:evensocietyclassification}}(t',4(k+c_{\ref{prop:universal-surface-walls}}{t'}^{12}+1))-1-4bk, \mathsf{nest}(t'^2-i,t',k)-\mathsf{cost}_{\ref{thm:societyclassification}}(t',4(k+c_{\ref{prop:universal-surface-walls}}{t'}^{12}+1))-1]$ and every vortex $v$ of $\rho$, both the nest of $S_v$ and $v$ itself are disjoint from $C^i_j$.

    Let $I = [\mathsf{nest}(t'^2-i,t',k)-\mathsf{cost}_{\ref{thm:evensocietyclassification}}(t',4(k+c_{\ref{prop:universal-surface-walls}}{t'}^{12}+1))-1-4bk+1, \mathsf{nest}(t'^2-i,t',k)-\mathsf{cost}_{\ref{thm:evensocietyclassification}}(t',4(k+c_{\ref{prop:universal-surface-walls}}{t'}^{12}+1))-1]$.
    Then we have $|I| = 4bk+2$.
    Therefore, we may call upon \zcref{prop:connected_linkages} for $\mathcal{C}'$, $I$, $\mathcal{R}'$ and $\mathcal{R}_i$ to obtain a radial linkage $\mathcal{R}_{i+1}$ of order $4bk$ whose endpoints on the outermost cycles of the nests of the $S_v$ coincide with the vertices of the rails of $S_v$ for every vortex segment $S_v$ of $\rho'$ and whose other endpoints are a subset of the endpoints of $\mathcal{R}_i$ on $V(\Omega_0)$.
    Moreover, the endpoints of the paths among $\mathcal{R}_i$ that lead to $S_v$ appear consecutively on $\Omega_0$.
    Hence, we may obtain an extended $k$-surface wall $D_{i+1}$ from $D_i$ by discarding some of the cycles and paths in each of the handle and crosscap segments and integrating the nests of the $S_v$ along the paths in $\mathcal{R}_{i+1}$.
    We may also obtain a $\Sigma_i$-rendition $\rho_{i+1}$ of breadth at most $b$ and depth at most
    \begin{align*}
        \mathsf{depth}_{\ref{thm:evensocietyclassification}}(t',4(k+c_{\ref{prop:universal-surface-walls}}{t'}^{12}+1),&~\mathsf{transaction}(t'^2-i,t',k)) \leq \\
                    & \quad \mathsf{depth}_{\ref{thm:evensocietyclassification}}(t',4(k+c_{\ref{prop:universal-surface-walls}}{t'}^{12}+1),\mathsf{transaction}(t'^2,t',k))
    \end{align*}
    for $G-A_{i+1}$, where $A_{i+1\coloneqq A_i\cup A}$ with
    \begin{align*}
        |A_{i+1}|   \leq~& |A_i| + \mathsf{apex}^\mathsf{fin}_{\ref{thm:evensocietyclassification}}(t',t',4(k+c_{\ref{prop:universal-surface-walls}}{t'}^{12}+1),\mathsf{transaction}(t'^2,t',k)) \\
                    \leq~& i \cdot \mathsf{apex}^\mathsf{genus}_{\ref{thm:evensocietyclassification}}(t') + \mathsf{apex}^\mathsf{fin}_{\ref{thm:evensocietyclassification}}(t',t',4(k+c_{\ref{prop:universal-surface-walls}}{t'}^{12}+1),\mathsf{transaction}(t'^2,t',k)) + 16t'^3.
    \end{align*}
    It follows, in particular from the choice of the size of $M_0$, that $\Lambda = (A_{i+1}, M_0, D_{i+1}, \rho_{i+1})$ is a $k$-$(\mathsf{apex}_{\ref{thm:strongest_localstructure}}(t'^2,t',k),b,\mathsf{depth}_{\ref{thm:strongest_localstructure}}(t',k),r)$-$\Sigma_i$-layout centred at $M$ as desired.
    This completes our proof.
\end{proof}

\subsection{Local structure implies bounded odd cycle packing number}

We now show that if \zcref{itm:localstructurenearlybipartite} or \zcref{itm:localstructureevenfacedrendition} hold in \zcref{thm:strongest_localstructure}, then then there is a bounded size apex set $X \subseteq V(G)$ such that the large block of $\mathcal{T}_M - X$ has bounded odd cycle packing number.
Clearly, if \zcref{itm:localstructurenearlybipartite} holds, then we can choose our $X$ correspondingly, with $|X| \leq \mathsf{apex}^\mathsf{bip}_{\ref{thm:strongest_localstructure}}(t)$, and the odd cycle packing number of the large block of $\mathcal{T}_M - X$ is 0.
So we may assume that \zcref{itm:localstructureevenfacedrendition} holds.
Let $X$ contain the apex set $A$ with $|A| \leq \mathsf{apex}_{\ref{thm:strongest_localstructure}}(t,k)$.
It now suffices to prove the result for the large block of $\mathcal{T}_M - A$.
Note then that there is a non-orientable, even-faced $\Sigma$-rendition of the large block of $\mathcal{T}_M - A$ such that there are $b = \nicefrac{1}{2}(t-3)(t-4)$ vortices each of depth $d = \mathsf{depth}_{\ref{thm:strongest_localstructure}}(t,k)$ and $\Sigma$ has genus $g < \lceil \mathsf{c}_{\ref{thm:ktminormodel}} (2t^2+2t) \sqrt{\log 24t^2+24t} \rceil^3$.
Furthermore, around each vortex we have pairwise disjoint $k$-cylindrical meshes arising from the vortex segments of the surface wall that \zcref{thm:strongest_localstructure} provides.

We first state some helpful lemmas before proceeding with the main result.
The following lemma about odd $A$-paths will help in some later constructions.

\begin{lemma}[Lemma 11~\cite{GeelenGRSV2009Oddminor}]\label{lem:oddApaths}
    Let $G$ be a graph and $A \subseteq V(G)$. Then one of the following holds for all positive integers $k$.
    \begin{enumerate}
        \item there exist $k$ pairwise disjoint odd $A$-paths, or
        \item there exists a set $X \subseteq V(G)$ of size at most $2k-2$ such that $G- X$ does not contain an odd $A$-path.
    \end{enumerate}
    Furthermore, there exists an algorithm that given $G, A,$ and $k$, finds one of the outcomes in time $\mathbf{O}(|V(G)||E(G)|)$.
\end{lemma}

Note to achieve this runtime we need to apply the $\mathbf{O}(\sqrt{|V(G)|}|E(G)|)$ maximum matching algorithm of Micali and Vazirani~\cite{MicaliV1980algorithmfor}. See Section 4 of \cite{GeelenGRSV2009Oddminor} for a discussion of the algorithm.

Given a society $(G, \Omega)$, a set of paths $P_1, \ldots, P_k$ such that $P_i$ has ends $s_i,t_i$ for all $i \in [k]$ is called \textit{sequential} if the ends occur in the order $s_1,t_1, \ldots ,s_k, t_k$ in $\Omega$. Given a two colouring of $V(\Omega)$, a path $P$ with both ends in $V(\Omega)$ is called \textit{parity breaking} if the colouring of its endpoints does not extend to a proper two colouring of $P$.

Given these definitions, the following is a simple corollary of \zcref{lem:oddApaths}.
We note that here we use \emph{two colouring} to denote an arbitrary assignment of two colours to a set of vertices and call a two colouring that does not produce monochromatic edges \emph{proper}. 

\begin{corollary}\label{lem:paritybreakingApaths}
    Let $(G, \Omega)$ be a society, let $\tau : V(\Omega) \rightarrow [2]$ be a two colouring, let $A \subseteq V(\Omega)$, and let $k$ be a positive integer.
    Then one of the following holds:
    \begin{enumerate}
        \item There exists $k$ pairwise disjoint parity breaking $A$-paths, or
        \item there exists a set $X \subseteq V(G)$ of size at most $2k-2$ such that $G - X$ does not contain a parity breaking $A$-path.
    \end{enumerate}
    Furthermore, there exists an algorithm that given $G,A, \tau,$ and $k$, finds one of the outcomes in time $\mathbf{O}(|V(G)||E(G)|)$.
\end{corollary}
\begin{proof}
    Construct $G'$ from $G$ by adding the degree one vertices $B = \{ v_u : u \in A, \tau(u) = 1\}$ such that $v_u$ is adjacent only to $u$.
    Let $A' = \{u \in A : \tau(u) = 2\} \cup B$, and note that parity breaking $A$-paths are in bijection with odd $A'$-paths.
    Also note that in any hitting set $X \subseteq V(G')$, one may replace $v_u$ with $u$ without creating an odd $A$-path.
    Thus the result clearly follows from \zcref{lem:oddApaths}.
\end{proof}

We now prove a lemma that helps us use parity breaking paths to construct $\mathscr{V}_t$ as an odd-minor.

\begin{lemma}\label{lem:moveparitybreakingtobranchvtcs}
    Let $\rho$ be an even-faced cylindrical rendition of $(G, \Omega)$ in a disk $\Delta$ around a vortex $c_0$ with a nest $\{ C_1, \ldots, C_{2t} \}$, such that $C_1, \ldots, C_{2t}$ are the concentric cycles of a $(4t \times 2t)$-cylindrical mesh $M$ with radial paths $P_1, \ldots, P_{4t}$, and the trace of $C_{2t}$ is $\mathsf{bd}(\Delta)$.
    Let $\tau : V(M) \rightarrow [2]$ be a proper two colouring, and suppose there exists $2t$ parity breaking paths.
    Then there exist $t$ parity breaking paths whose endpoints are contained in the set of endpoints of $P_2, P_4, \ldots, P_{4t}$ in $C_{2t}$.
\end{lemma}
\begin{proof}
    Let $A$ be the set of endpoints of $P_2, P_4, \ldots, P_{4t}$ in $C_{2t}$.
    By \zcref{lem:paritybreakingApaths}, it suffices to show that there is no set $X \subseteq V(G)$ of size at most $2t-2$ such that $G-X$ does not contain a parity breaking $A$-path.
    Suppose towards a contradiction there was such an $X$, and note that $G - X$ still contains two paths $P_i, P_j$ whose endpoints are in $A$.
    Note also that at least one concentric cycle $C_i$ survives in $G - X$.
    There is also a parity breaking path $P$ that survives in $G - X$, and note that, because $\rho$ is even-faced, $P$ must use an edge of the vortex $c_0$.
    In particular this means that $P$ intersects $C_i$, and so creates an odd cycle in $P \cup C_i$ because $C_i$ is properly two coloured by $\tau$ while there is no extension of $\tau$ to properly two colour $P$.
    This odd cycle is in the same block as $C_i$, so there is a parity breaking $A$-path between the endpoints of $P_i, P_j$, a contradiction.
\end{proof}

We then need a lemma about the type of parity breaking paths we can find. We note that the original proof of the lemma in \cite{HuynhJW2019Unified} argues for the case in which $t_1=t_2$, but clearly the same arguments prove the more general statement that we derive here. We also extend their result to include an algorithm for finding such a subcollection.

\begin{lemma}[Lemma 25~\cite{HuynhJW2019Unified}]\label{lem:seqplanarcrosscap}
    Let $(G, \Omega)$ be a society and let $\mathcal{P}$ be a collection of $t_1t_2$ paths in $G$ with both ends in $V(\Omega)$. Then one of the following holds.
    \begin{enumerate}
        \item there exists a subcollection of $\mathcal{P}$ of size $t_1$ that is sequential, or
        \item there exists a subcollection of $\mathcal{P}$ of size $t_2$ that is a transaction.
    \end{enumerate}
    Moreover, there exists an algorithm that, given $G, \Omega, \mathcal{P}, t_1,$ and $t_2$, finds one of the outcomes in time $\mathbf{O}(|V(G)| + |E(G)| + (t_1t_2)^2)$.
\end{lemma}
\begin{proof}
    Consider arbitrarily cutting $\Omega$ into a linear ordering $\Omega'$. Then the ends of $\mathcal{P}$ determine intervals in $\Omega'$. Let $H$ be the corresponding interval graph. Because interval graphs are perfect, $\alpha(H)\omega(H) = \alpha(H)\chi(H) \geq |V(H)|$. Note this is a characterization of perfect graphs \cite{Lovasz1972Characterization}, but we only need the trivial direction. Thus there exists $t_1$ intervals which are pairwise intersecting or $t_2$ intervals which are pairwise disjoint. Note that pairwise disjoint intervals correspond to a subcollection of $\mathcal{P}$ that is sequential, while pairwise intersecting intervals correspond to a subcollection that is a transaction.

    Consider ordering the intervals in $V(H)$ by their left endpoint. Then for every vertex $v \in V(H)$, its pre-neighbourhood is a clique. We can then easily find a largest clique via finding a vertex with the largest pre-neighbourhood in time $\mathbf{O}(|V(H)| + |E(H)|)$. We can find the largest independent set by a similar greedy algorithm in time $\mathbf{O}(|V(H)|)$ after sorting the intervals by their right endpoint \cite{GuptaLL1982Efficient}. Note that we can easily order the intervals by either endpoint via reading off $\Omega'$ in time $\mathbf{O}(|V(\Omega')|)$. We can construct the intervals of $\Omega'$ given by $\mathcal{P}$ in time $\mathbf{O}(|V(G)| + |E(G)|)$. Because $|E(H)| \leq |\mathcal{P}|^2$, this gives a total runtime of $\mathbf{O}(|V(G)| + |E(G)| + (t_1t_2)^2)$. Note that if we are guaranteed that the first option will occur, we can find this option in time $\mathbf{O}(|V(G)| + |E(G)| + |\mathcal{P}|) = \mathbf{O}(|V(G)| + |E(G)|)$.
\end{proof}

Let $G$ be a graph with a non-orientable, even-faced $\Sigma$-rendition $\rho$ with no vortices and no cells $c$ with $|N(c)| = 1$. Then the \textit{natural embedding} of $G$ is a graph $G'$ embedded in $\Sigma$ formed from $G$ by the following operation.
For every cell $c$ with $N(c) = \{u,v\}$, with $u \neq v$, replace $\sigma(c)$ with a $u$-$v$-path of length 1 or 2 matching the unique parity of a $u$-$v$-path in $\sigma(c)$.
For every cell $c$ with $|N(c)| = 3$, replace $\sigma(c)$ with a cycle $C$ of length 4 or 6 containing $N(c)$ such that for each $u,v \in N(c)$, there is a $u$-$v$-path of length 1 or 2 in $C$ matching the unique parity of such a path in $\sigma(c)$.
This map of cells to paths or cycles induces a map of any grounded path or cycle in $G$ to a path or cycle respectively in $G'$ of the same parity.
Note that we can embed $G'$ in $\Sigma$ such that the trace of every grounded path or cycle in $G$ is the trace of its image in $G'$.

If we further assume that $G$ is 2-connected, then any odd cycle contained entirely in a cell $c$ of $\rho$ gives two paths of different parities between the vertices in $N(c)$, so every odd cycle is grounded. Then if $G'$ contains an odd cycle, it has a preimage which is a grounded cycle in $G$. In particular $G'$ has the same odd cycle packing number as $G$.

If we allow $\rho$ to have vortices, then we can extend the natural embedding of $G$ to a non-orientable, even-face $\Sigma$-rendition with the same set of vortices such that after removing the vortices, the remaining graph is embedded in the surface.
We call this augmented graph with its $\Sigma$-rendition the \textit{extended natural embedding} of $G$.
As before if we assume $G$ is 2-connected, then the extended natural embedding of $G$ also has the same odd cycle packing as $G$ because the grounded paths between vortices are preserved.

For a graph $G$ with a labelling $\gamma: E(G) \rightarrow \Gamma$ over an abelian group $\Gamma$, we say the \textit{label} of a cycle or path $H$ is the value $\sum_{e \in H} \gamma(e)$.
\begin{lemma}\label{lem:alllabelsingouplabeledgraphs}
    Let $G$ be a 2-connected graph and let $\gamma: E(G) \rightarrow \mathbb{Z}_2 \times \mathbb{Z}_2$. If $G$ has a cycle with label $(1,0)$ and a cycle with label $(0,1)$, then $G$ has a cycle with label $(1,1)$.
\end{lemma}
\begin{proof}
    Let $C_1$ be a cycle with label $(1,0)$ and let $C_2$ be a cycle with label $(0,1)$, chosen such that $|E(C_1)| + |E(C_2)|$ is minimum. Note that the symmetric difference of $C_1$ and $C_2$ is an edge disjoint collection of cycles with labels summing to $(1,1)$. We may assume no cycles in the collection have label $(1,1)$, so there are cycles of label $(1,0)$ and $(0,1)$. Thus we may assume $C_1$ and $C_2$ are edge disjoint by the minimality of $|E(C_1)| + |E(C_2)|$. Furthermore we may assume $C_1$ and $C_2$ share at most two vertices, otherwise we claim we can decompose the symmetric difference $C_1 \triangle C_2$ into at least 3 cycles, for which we may assume two of the cycles have labels $(1,0)$ and $(0,1)$. Indeed it suffices to show that there is some cycle in $C_1 \triangle C_2$ disjoint from a vertex of degree 4 because the decomposition of the remaining Eulerian graph into cycles must include at least 2 cycles.
    
    Consider suppressing all degree 2 vertices in $C_1 \triangle C_2$ to obtain the 4-regular multigraph $H$ on at least 3 vertices. We may assume $H$ is a simple graph, otherwise a cycle of length 2 does not contain some other vertex. Then $H$ is the union of two edge disjoint cycles $C_1', C_2'$ on the same vertex set. Taking a chord in $C_1'$ and some path between its endpoints in $C_1'$, we obtain a cycle which does not contain some vertex.
    Thus we may assume the cycles $C_1$ and $C_2$ share at most 2 vertices.
    
    Because $G$ is 2-connected, there are two vertex disjoint, possibly trivial, paths between $V(C_1)$ and $V(C_2)$ containing $V(C_1) \cap V(C_2)$. Let $u,v$ be the ends of these paths in $V(C_1)$. Note that through $C_1$ there is a $u$-$v$-path of label $(0,y)$ and $(1,y)$ for some $y \in \mathbb{Z}_2$. Similarly by taking the path through $C_2$ there is a $u$-$v$-path of label $(x,0)$ and $(x,1)$ for some $x \in \mathbb{Z}_2$. By choosing the paths of length $(a,y)$ and $(x,b)$ for $(a,b)$ such that $(a,b) + (x,y) = (1,1)$, we obtain a cycle of label $(1,1)$ through $u,v$.
\end{proof}

This allows us to give a helpful characterisation of cycles with one-sided traces in non-orientable, even-faced renditions.

\begin{lemma}\label{lem:oddcycleiffonesided}
    Let $G$ be a 2-connected graph with a non-orientable, even-faced $\Sigma$-rendition $\rho$ with no vortices and $|N(\rho)| \geq 2$. Suppose that $G$ is not bipartite. Then a grounded cycle in $G$ is odd if and only if its trace is one-sided.
\end{lemma}
\begin{proof}
    The forward direction follows from the definition of a non-orientable, even-faced $\Sigma$-rendition. Note that because $G$ is 2-connected with $|N(\rho)| \geq 2$, every odd cycle in $G$ is grounded as noted above. Thus it suffices to assume that $G$ contains some grounded odd cycle $C_1$ with a one-sided trace, and suppose towards contradiction that $G$ contains a grounded even cycle $C_2$ with a one-sided trace. Note that the trace of a cycle is one-sided if and only if it goes through a crosscap an odd number of times.
    
    Let $G'$ be the natural embedding of $G$. Let $\gamma: E(G') \rightarrow \mathbb{Z}_2 \times \mathbb{Z}_2$ be such that the first component of every edge is 1, and the second component is 1 if and only if the edge goes through a crosscap. Then by assumption, $G'$ contains a cycle of label $(1,1)$ and a cycle of label $(0,1)$. This implies that $G'$ contains a cycle of label $(1,0)$ by \zcref{lem:alllabelsingouplabeledgraphs} after the automorphism swapping $(1,1)$ and $(1,0)$. Thus $G'$ contains an odd cycle with a two-sided trace and hence so does $G$, a contradiction.
\end{proof}

With all of our tools gathered, we move on to the proof of the main theorem of this section, which implies the refinement of \zcref{thm:strongest_localstructure} that we are ultimately interested in using. In the below proof we first find a parity vortex as an odd-minor, or we remove all parity breaking paths from vortices. In the later case, we can then only obtain odd cycles from crosscaps and not via vortices. We then refine our rendition such that a cycle is odd if and only if its trace is one-sided. We use this to bound the odd-cycle-packing in terms of the genus.

\begin{theorem}\label{thm:localstructureboundsocp}
    Let $t$ be a positive integer and let $G$ be the large block of $\mathcal{T}_M$ for some mesh $M$.
    Suppose $G$ has a non-orientable, even-faced $\Sigma$-rendition $\rho_G$ of breadth $b$, depth $d$, and where $\Sigma$ has Euler-genus $g$.
    Furthermore, for each vortex $c_i$ with $i \in [b]$, there exists a grounded $k$-cylindrical mesh $M_i$ for $k = 4t$ with concentric cycles each bounding a disk containing vortex $c_i$, and these cylindrical meshes are pairwise disjoint and each controlled by $\mathcal{T}_M$. Then one of the following holds.
    \begin{enumerate}
        \item $G$ has $\mathcal{V}_t$ as an odd-minor controlled by $\mathcal{T}_M$, or
        \item There exists a set $A \subseteq V(G)$ of size at most $4bt(d + 1)$ such that the large block of $\mathcal{T}_M - A$ has odd cycle packing number at most $gb(d+1)$.
    \end{enumerate}
    Furthermore there exists an algorithm that, given $G, M_1, \ldots, M_b,$ and $t$ as above as input, finds one of these outcomes in time $\mathbf{O}(|V(G)||E(G)|)$.
\end{theorem}
\begin{proof}
    Let $\tau_1, \ldots, \tau_b$ be such that $\tau_i: V(M_i) \rightarrow [2]$ is a proper two colouring for each $i \in [b]$.
    For each vortex $c_i$, let $\{P^i_1, \ldots, P^i_k\}$ be the radial paths of $M_i$.
    Let $\{C^i_1, \ldots, C^i_k\}$ be the concentric cycles of $M_i$ which give a nest around $c_i$, with the trace of each $C^i_j$ bounding the closed disk $\Delta^i_j$ with $c_i \subseteq \Delta^i_1 \subseteq \ldots \subseteq \Delta^i_k$.

    For each vortex $c_i$, let $(G_i, \Omega_i)$ be the $\Delta^i_{2t}$ society. Suppose that $(G_i, \Omega_i)$ has $2t(d+1)$ parity breaking paths. Note that because $\rho_G$ is an non-orientable, even-faced rendition, each path must use an edge in the vortex $c_i$. By \zcref{lem:moveparitybreakingtobranchvtcs}, there are $t(d+1)$ parity breaking paths whose endpoints are contained in $(P^i_2 \cup P^i_4 \cup \ldots \cup P^i_{4t}) \cap C^i_{2t}$ such that no two endpoints are in the same radial path. Then by \zcref{lem:seqplanarcrosscap}, it must have $d + 1$ in a transaction or $t$ in series.

    Suppose there are $d+1$ paths in a transaction, say $Q_1, \ldots, Q_{d+1}$ indexed naturally. Let $X$ be the segment of $\Omega_i$ containing the ends of $Q_1, \ldots, Q_{d+1}$ and let $Y$ be the segment of $\Omega_{c_i}$ for the vortex society $(\sigma(c_i), \Omega_{c_i})$ lying between $Q_1$ and $Q_{d+1}$ where the paths enter $c_i$. Note that $X,Y,$ the trace of $Q_1$, and the trace of $Q_{d+1}$ bound a disk. Every path in $Q_2, \ldots, Q_{d}$ begins on $X$ and cannot cross the trace of $Q_1$ or $Q_{d+1}$, so they must pass through $Y$ and reach the segment $\Omega_{c_i} \setminus Y$. Thus $Q_1, \ldots, Q_{d+1}$ give a transaction of order $d+1$ across $c_i$, contradicting the depth of $c_i$.
    
    Then we must have $t$ which are sequential, and the union of these parity breaking paths and the $(4t \times t)$-cylindrical mesh contained in $C^i_{2t+1} \cup \ldots \cup C^i_{3t} \cup P^i_1 \cup \ldots \cup P^i_{4t}$ then clearly contains the parity vortex of order $t$. Note that the aforementioned $(4t \times t)$-cylindrical mesh contains a minor model $\varphi_W$ for the $(4t \times t)$-cylindrical grid $W$, and the minor model $\varphi_{\mathcal{V}_t}$ of the parity vortex can be chosen such that $V(\mathcal{V}_t)$ is in bijection with $V(W)$ and for every vertex $x \in V(\mathcal{V}_t)$ corresponding to $y \in V(W)$, $\varphi_{\mathcal{V}_t}(x) \supseteq \varphi_{W}(y)$. Because $W$ is a minor of $M_i$ and $M_i$ is controlled by $\mathcal{T}_M$, $W$ is controlled by $\mathcal{T}_M$. Hence $\mathcal{V}_t$ is controlled by $\mathcal{T}_M$.
    
    We may then assume $(G_i, \Omega_i)$ does not have $2t(d+1)$ parity breaking paths. Therefore by \zcref{lem:paritybreakingApaths}, there exists a set $A \subseteq \bigcup_{i \in [b]}V(G_i)$ of size at most $4bt(d + 1)$ such that after deleting it, $(G_i, \Omega_i)$ has no parity breaking paths for all $i \in [b]$.

    Note that we can apply \zcref{lem:paritybreakingApaths} to each $G_i$ to either find the set $A$ or find $2t(d+1)$ parity breaking paths. For each $G_i$ this takes time $\mathbf{O}(|V(G_i)||E(G_i)|)$, so the total runtime is $\mathbf{O}(|V(G)||E(G)|)$. If we obtain the parity breaking paths, we can then apply \zcref{lem:paritybreakingApaths} to each $G_i$, but now on a set containing one vertex from each of $P^i_2 \cap C^i_{2t}, P^i_4 \cap C^i_{2t}, \ldots, P^i_{4t} \cap C^i_{2t}$. We are then guaranteed to find $t(d+1)$ parity breaking paths, and we can find them in time $\mathbf{O}(|V(G)||E(G)|)$ over all $G_i$. We can then find a subcollection which is sequential in time $\mathbf{O}(|V(G)| + |E(G)|)$ by \zcref{lem:seqplanarcrosscap}. Note that we need not incur the $(t_1t_2)^2$ runtime in \zcref{lem:seqplanarcrosscap} because we need only to find the first option. We then easily obtain our parity vortex as an odd-minor.
    
     It remains to show the large block $B_G$ of $\mathcal{T}_M - A$ has odd-cycle-packing at most $gb(d+1)$. Let $H$ denote the extended natural embedding of $B_G$, and let $\rho$ be its non-orientable, even-faced $\Sigma$-rendition. Note that such an extended natural embedding exists because no cell $c \in C(\rho_G)$ with $(V(\sigma_{\rho_G}(c)) \setminus N_{\rho_G}(c)) \cap V(B_G) \neq \emptyset$ or $E(\sigma_{\rho_G}(c)) \cap E(B_G) \neq \emptyset$ can have $|N(c)| = 1$, because $B_G$ is 2-connected. We may denote subgraphs of $H$ by there preimage in $B_G$ when there can be no confusion.
    
    We replace the large block of $\mathcal{T}_M - A$ with its extended natural embedding so that we can precisely consider the edges that go through a crosscap, which we can think of as a small circular hole in the surface with anti-nodal points identified. Note that by 2-connectivity $(G_i \cap B_G, \Omega_i)$ has a parity breaking path if and only if the vortex society of $c_i$ in $\rho$ does, so every vortex society of $\rho$ has no parity breaking paths.

    Let $\overline{H} = (H - \bigcup_{i \in [b]} V(\sigma(c_i))\setminus N(c_i)) - \bigcup_{i \in [b]} E(c_i)$ and let $\overline{\rho}$ be the restriction of $\rho$ to $\overline{H}$.
    We note that $\overline{\rho}$ is an embedding with no vortices. Note that every block of $\overline{H}$ is either the large block of $\mathcal{T}_M - A - \bigcup_{i \in [b]}V(\sigma(c_i))\setminus N(c_i)$, which contains each $C^i_j$ for $i \in [b]$ and $j \in [2t+2,4t]$, or is embedded inside $\Delta_{2t+1}^i$ for some $i \in [b]$. Because $\overline{\rho}$ is a non-orientable, even-faced rendition, every block inside $\Delta_{2t+1}^i$ for some $i$ is bipartite and contains no cycles with a one-sided trace. If the large block of $\mathcal{T}_M - A - \bigcup_{i \in [b]}V(\sigma(c_i))\setminus N(c_i)$ is bipartite, then $\overline{H}$ is bipartite. We could then extend a two colouring of $\overline{H}$ to a two colouring of $H$ because each vortex society has no parity breaking paths. Thus we may assume by \zcref{lem:oddcycleiffonesided} that every grounded cycle in $\overline{\rho}$ is odd if and only if its trace is one-sided.
    
    Consider the graph formed from $\overline{H}$ by subdividing every edge by the number of times it goes through a crosscap. The resulting graph is bipartite and can be coloured with a 2-colouring $\tau$ which agrees with $\tau_i$ for $i \in [b]$ after possibly inverting the colouring of some of the $\tau_i$. Then $\tau$ can be extended to a 2-colouring of the vortices because each vortex does not contain a parity breaking path. This can be restricted to a 2-colouring of $H$ such that an edge is monochromatic if and only if it goes through a crosscap an odd number of times.

    Suppose for the sake of contradiction that there exists a collection $\mathcal{C}$ of $gb(d+1) + 1$ vertex disjoint odd cycles in $H$. Let $H'$ be the graph formed from $H$ by doing the following for each vortex $c_i$. Delete $V(\sigma(c_i)) \setminus N(c_i)$ and $E(c_i)$, then add a new vertex $v_i$ adjacent to all vertices in $N(c_i)$. Let $\rho'$ be the $\Sigma$-rendition of $H'$ derived from $\rho$ by placing the new vertex $v_i$ in the cell $c_i$, removing $c_i$, and making $v_i$ adjacent to all vertices in $N(c_i)$, whilst adding appropriate non-vortex cells each with $v_i$ and a vertex of $N(c_i)$ as their nodes.
    
    We think of creating $v_i$ through the process of subdividing every edge in $\sigma(c_i)$ and identifying $V(\sigma(c_i)) \setminus N(c_i)$ at a single vertex $v_i$. Then every cycle $C$ in $H$ is mapped to a circuit in $H'$ where every path in $C$ whose edges are found in $E(\sigma(c_i))$ with the endpoints $u \in N(c_i)$ and $w \in N(c_i)$ is replaced by the path of length 2 with edges $uv_i$ and $v_iw$. Note that each cycle is mapped to a circuit which may hit $v_i$ multiple times, but will not hit $v_i$ if it does not contain an edge in $\sigma(c_i)$. Similarly the trace of each cycle in $\mathcal{C}$ under $\rho$ is mapped to a (not necessarily simple) closed curve in $\Sigma$ under $\rho'$. We say two curves in $\rho'$ \textit{do not cross} if there is a way to continuously deform the curve in a small disc around each vortex point $v_i$ such that the resulting curves do not intersect, and otherwise we say that the curves \textit{cross}. That is, curves cross if they intersect at a point other than a vortex point $v_i$, or if the first curve ``enters'' on one side of a vortex point and ``exits'' on the other side, where the sides are relative to some small section of the second curve.

    Let $H_{\mathcal{C}}$ be the subgraph of $H'$ with vertex set equal to the image of $V(\bigcup \mathcal{C})$ under the described procedure. Note that $H_{\mathcal{C}}$ has at most $b$ vertices of degree larger than 3, those being $v_1, \ldots, v_b$. Let $\rho_{\mathcal{C}}$ be the restriction of $\rho'$ to $H_{\mathcal{C}}$.

    Consider a maximal collection of cycles $\mathcal{S}$ in $H_{\mathcal{C}}$ each with a trace that is one-sided and does not cross the trace of any other cycle in the collection. Note that the size of any such collection is at most $g$ because we can perturb the trace of each cycle to obtain a set of pairwise disjoint one-sided curves. For each cycle $C \in \mathcal{S}$, consider the minimal subgraph $C'$ of $H$ which is mapped to $C$. Note that $C'$ intersects at most $b$ of the cycles in $\mathcal{C}$. Consider a vortex point $v_i$ which is hit by $C$, and let $u_1, u_2  \in N_\rho(c_i)$ be the boundary nodes neighbouring $v_i$ in $C$. Consider the vortex society $(\sigma(c_i), \Omega_{c_i})$. Let $X,Y$ be the two disjoint segments of $\Omega_i$ formed by deleting $u_1, u_2$, and let $A'_i \subseteq V(\sigma(c_i))$ be the set of size at most $d$ whose deletion separates $X$ and $Y$ in $\sigma(c_i)$. Note that if we delete $V(C') \cup \bigcup_{i \in [b]} A'_i$, then every remaining path in $H_{\mathcal{C}}$ has a trace under $\rho_{\mathcal{C}}$ which does not cross the trace of $C$. Note that $V(C')$ hits at most $b$ cycles in $\mathcal{C}$ while $|\bigcup_{i \in [b]} A'_i| \leq bd$, so together they hit at most $b(d+1)$ cycles in $\mathcal{C}$. Now consider deleting this set for every cycle in $\mathcal{S}$, and recall $|\mathcal{S}| \leq g$. Call this deleted set $A'$, and note that there is at least one cycle $\tilde{C}$ remaining in $\mathcal{C}$ disjoint from $A'$.
    
    $\tilde{C}$ has a trace which doesn't cross any cycle in $\mathcal{S}$. Suppose towards contradiction that $\tilde{C}$ used a crosscap an odd number of times. Then there must be some cycle in the Eulerian subgraph of $H_{\mathcal{C}}$ induced by $\tilde{C}$ which uses a crosscap an odd number of times, contradicting the maximality of $\mathcal{S}$. Thus $\tilde{C}$ can be 2-coloured by $\tau$ such that there are an even number of monochromatic edges, contradicting that $\tilde{C}$ is odd.
\end{proof}

\begin{corollary}\label{cor:corecorollary}
    There exist functions $\mathsf{mesh}_{\ref{cor:corecorollary}}, \mathsf{apex}_{\ref{cor:corecorollary}}, \mathsf{cycles}_{\ref{cor:corecorollary}} \colon \mathbb{N} \to \mathbb{N}$ such that for all non-negative integers $t$, for every graph $G$, and every $\mathsf{mesh}_{\ref{cor:corecorollary}}(t)$-mesh $M \subseteq G$, one of the following holds.
    \begin{enumerate}
        \item There exists a set $A \subseteq V(G)$ of size at most $\mathsf{apex}_{\ref{cor:corecorollary}}(t)$ such that the large block of $\mathcal{T}_M - A$ is non-empty and has odd cycle packing number at most $\mathsf{cycles}_{\ref{cor:corecorollary}}(t)$,
        \item $G$ has the parity handle $\mathscr{H}_t$ of order $t$ as an odd-minor controlled by $\mathcal{T}_M$, or
        \item $G$ has the parity vortex $\mathscr{V}_t$ of order $t$ as an odd-minor controlled by $\mathcal{T}_M$.
    \end{enumerate}
    In particular, we have that

    {\centering
    $ \displaystyle
        \begin{aligned}
            \mathsf{apex}_{\ref{cor:corecorollary}}(t) \in   & \ \mathbf{O}\big(t^{150}\big) \text{ and} \\
            \mathsf{cycles}_{\ref{cor:corecorollary}}(t), \mathsf{mesh}_{\ref{cor:corecorollary}}(t) \in   & \ \mathbf{O}(t^{161}) .
        \end{aligned}
    $
    \par}
    
    Furthermore, we have $\mathsf{mesh}_{\ref{cor:corecorollary}}(t) > \mathsf{apex}_{\ref{cor:corecorollary}}(t)$, and there exists an algorithm that, given $G$, $M$, and $t$ as above as input, finds one of these outcomes in time $(f_{\ref{thm:ktminormodel}}(t) + \mathsf{poly}(k))|E(G)||V(G)|^2$.
\end{corollary}
\begin{proof}
    We set
    \begin{align*}
        t' &= \mathsf{clique}_{\ref{thm:strongest_localstructure}}(t) , \\
        k &= 4t , \\
        b &= \nicefrac{1}{2}(t'-3)(t'-4) + t , \\
        g &= {t'}^2 , \\
        \mathsf{apex}_{\ref{cor:corecorollary}}(t) &= \mathsf{apex}_{\ref{thm:strongest_localstructure}}(t,k) + 4bt(\mathsf{depth}_{\ref{thm:strongest_localstructure}}(t,k)+1) , \\
        \mathsf{cycles}_{\ref{cor:corecorollary}}(t) &= gb(\mathsf{depth}_{\ref{thm:strongest_localstructure}}(t,k) + 1) , \text{ and}\\
        \mathsf{mesh}_{\ref{cor:corecorollary}}(t) &= \mathsf{mesh}_{\ref{thm:strongest_localstructure}}(t,0,k) .
    \end{align*}
    We first apply \zcref{thm:strongest_localstructure}.
    If \zcref{itm:localstructureparityhandle} holds, then we are done.
    If \zcref{itm:localstructurenearlybipartite} holds, then we set $A$ to the corresponding apex set of size at most $\mathsf{apex}^\mathsf{bip}_{\ref{thm:strongest_localstructure}}(t')$ and again we are done. Note $\mathsf{apex}^\mathsf{bip}_{\ref{thm:strongest_localstructure}}(t') \leq \mathsf{apex}_{\ref{thm:strongest_localstructure}}(t',k)$.
    If \zcref{itm:localstructureevenfacedrendition} holds, then we apply \zcref{thm:localstructureboundsocp} to the large block  of $\mathcal{T}_M - A'$ for $A'$ the apex set of $\Lambda$ of size at most $\mathsf{apex}_{\ref{thm:strongest_localstructure}}(t',k)$ given in \zcref{itm:localstructureevenfacedrendition} of \zcref{thm:strongest_localstructure}.
    If we then obtain $\mathcal{V}_t$ as an odd-minor controlled by $\mathcal{T}_M$, we are done.
    Otherwise we set $A$ to contain $A'$ and the apex set given in \zcref{thm:localstructureboundsocp} of size at most $4bt(\mathsf{depth}_{\ref{thm:strongest_localstructure}}(t',k)+1)$ to obtain the desired result.
    Note here that according to the definition of the functions involved in \zcref{thm:strongest_localstructure}, we have 
    \begin{align*}
        \mathsf{apex}_{\ref{thm:strongest_localstructure}}(t',k) + 4bt(\mathsf{depth}_{\ref{thm:strongest_localstructure}}(t',k)+1) &~= \mathsf{apex}_{\ref{thm:strongest_localstructure}}(t',4t) + 2(t'-3)(t'-4) t( \mathsf{depth}_{\ref{thm:strongest_localstructure}}(t',4t) +1) \\
                                                                            &~< \mathsf{mesh}_{\ref{thm:strongest_localstructure}}(t',0,4t) \\
                                                                            &~= \mathsf{mesh}_{\ref{cor:corecorollary}}(t) .
    \end{align*}
    Thus $\mathsf{mesh}_{\ref{cor:corecorollary}}(t) > \mathsf{apex}_{\ref{cor:corecorollary}}(t)$ is guaranteed, which also implies that the large block of $\mathcal{T}_M - A$ is non-empty.
    Finally, we note that the runtime of the procedure emerging from the proof of this result is dominated by the runtime of \zcref{thm:strongest_localstructure}.
\end{proof}
\section{Global structure}\label{sec:localtoglobal}
As in the previous sections, we will need some additional definitions before we can finally prove \zcref{thm:globalstructure}.

\paragraph{Highly linked sets.}
Let $\alpha \in [2/3, 1)_{\mathbb{R}}$.
Moreover, let $G$ be a graph and $X \subseteq V(G)$ be a vertex set. 
A set $S \subseteq V(G)$ is said to be an \emph{$\alpha$-balanced separator} for $X$ if for every component $C$ of $G - S$ it holds that $|V(C) \cap X| \leq \alpha|X|$. 
Let $k$ be a non-negative integer.
We say that $X$ is a \emph{$(k, \alpha)$-linked set} of $G$ if there is no $\alpha$-balanced separator of size at most $k$ for $X$ in $G$.

Given a $(3k, \alpha)$-linked set $X$ of $G$ we define $$\mathcal{T}_{X} \coloneqq \{ (A, B) \in \mathcal{S}_{k+1}(G) ~\!\colon\!~ |X \cap B| > \alpha|X| \}.$$ 
It is not hard to see that $\mathcal{T}_{S}$ is a tangle of order $k+1$ in $G$.

We need an algorithmic way to find, given a highly linked set, a large wall whose tangle is a truncation of the tangle induced by the highly linked set.
This is done in \cite{ThilikosW2024Excluding} by algorithmatising a proof of Kawarabayashi et al.\ from \cite{KawarabayashiTW2021Quickly}. 

\begin{proposition}[Thilikos and Wiederrecht \cite{ThilikosW2024Excluding} (see Theorem 4.2.)]\label{thm:algogrid}
Let $k\geq 3$ be an integer and $\alpha\in [2/3,1)$.
There exist universal constants $c_1, c_2\in\mathbb{N}\setminus\{ 0\}$, and an algorithm that, given a graph $G$ and a $(c_1k^{20},\alpha)$-linked set $X\subseteq V(G)$ computes in time $2^{\mathbf{O}(k^{c_2})}|V(G)|^2|E(G)|\log(|V(G)|)$ a $k$-wall $W\subseteq G$ such that $\mathcal{T}_W$ is a truncation of $\mathcal{T}_X$.
\end{proposition}

Additionally, we need to be able to find balanced separators efficiently.

\begin{proposition}[Reed \cite{Reed1992Finding}]\label{prop_findsep}
    There exists an algorithm that takes as input an integer $k$, a graph $G$, and a set $X \subseteq V(G)$ of size at most $3k+1$, and finds, in time $2^{\mathbf{O}(k)}m$, either a $\nicefrac{2}{3}$-balanced separator of size at most $k$ for $X$ or correctly determines that $X$ is $(k,\nicefrac{2}{3})$-linked in $G$.
\end{proposition}

Finally, to now prove \zcref{thm:globalstructure}, we prove a stronger statement by induction.
The core structure of this proof is derived from one of the core proofs in \cite{RobertsonS1991Graph}, though in our peculiar situation we depart from it noticeably in the second major case.

\begin{theorem}\label{thm:globalstructure_induction}
    There exist functions $\mathsf{width}_{\ref{thm:globalstructure_induction}}, \mathsf{link}_{\ref{thm:globalstructure_induction}} \colon \mathbb{N} \to \mathbb{N}$ such that for every graph $G$, positive integer $t$, and vertex set $X \subseteq V(G)$ with $|X| \leq 3\mathsf{link}_{\ref{thm:globalstructure_induction}}(t) + 1$ one of the following holds:
    \begin{enumerate}
        \item $G$ has the parity handle of order $t$ as an odd minor,
        \item $G$ has the parity vortex of order $t$ as an odd minor, or
        \item there exists a tame, rooted \ocptd $(T,r,\beta,\alpha)$ for $G$ of width at most $\mathsf{width}_{\ref{thm:globalstructure_induction}}(t)$ such that $X \subseteq \alpha(r)$ and $|V(T)| \in \mathbf{O}(|V(G)|)$.
    \end{enumerate}
    In particular, $\mathsf{link}_{\ref{thm:globalstructure_induction}}(t), \mathsf{width}_{\ref{thm:globalstructure}}(t) \in \mathbf{O}(t^{3220})$ and there also exists an algorithm that, given $G$ and $k$ as above as input finds one of these outcomes in time $\max(f_{\ref{thm:ktminormodel}}(t),2^{\mathbf{poly}(t)})|V(G)|^6$.
\end{theorem}
\begin{proof}
    Our functions for this proof are:
    \begin{align*}
        \mathsf{link}_{\ref{thm:globalstructure_induction}}(t)     \coloneqq~& c_1 \mathsf{mesh}_{\ref{cor:corecorollary}}(t)^{20}, \text{ and} \\
        \mathsf{width}_{\ref{thm:globalstructure_induction}}(t)    \coloneqq~& \max(4\mathsf{link}_{\ref{thm:globalstructure_induction}}(t), \mathsf{cycles}_{\ref{cor:corecorollary}}(t)) + 1 ,
    \end{align*}
    where $c_1$ is the first constant from \zcref{thm:algogrid}.
    Using the estimates on our functions in \zcref{cor:corecorollary} we thus have $\mathsf{link}_{\ref{thm:globalstructure_induction}}(t),\mathsf{width}_{\ref{thm:globalstructure}}(t) \in \mathbf{O}(t^{3220})$ as promised.

    We proceed by induction on $|V(G) \setminus X|$ and note that if $|V(G)| \leq 3\mathsf{link}_{\ref{thm:globalstructure_induction}} + 1$, we may simply use a tree with a single vertex and assign all vertices to the unique bag in this decomposition to get our third desired outcome if we let $V(G)$ also be the apex set of this bag.
    Therefore, we may assume that $|V(G) \setminus X| \geq 3\mathsf{link}_{\ref{thm:globalstructure_induction}} + 2$.
    If $|X| \leq 3\mathsf{link}_{\ref{thm:globalstructure_induction}}$, we may add vertices arbitrarily to $X$ until we have $|X| = 3\mathsf{link}_{\ref{thm:globalstructure_induction}} + 1$ and then apply our induction hypothesis to be done.
    Thus we may further assume that $|X| = 3\mathsf{link}_{\ref{thm:globalstructure_induction}} + 1$.

    By now applying \zcref{prop_findsep} to $X$, we find one of two possible outcomes in $(2^{\mathsf{poly}(t)}m)$-time:
    \begin{enumerate}
        \item A $\nicefrac{2}{3}$-balanced separator $S$ for $X$ in $G$ with $|S| \leq \mathsf{link}_{\ref{thm:globalstructure_induction}}(t)$, or
        \item $X$ is $(\mathsf{link}_{\ref{thm:globalstructure_induction}}(t),\nicefrac{2}{3})$-linked in $G$.
    \end{enumerate}
    \textbf{Case 1:} There exists a $\nicefrac{2}{3}$-balanced separator $S$ for $X$ in $G$ with $|S| \leq \mathsf{link}_{\ref{thm:globalstructure_induction}}(t)$.

    Let $G_1', \ldots, G_\ell'$ be the components of $G - S$ and for each $i \in [\ell]$, let $G_i \coloneqq G[V(G_i') \cup S]$.
    Moreover, for each $i \in [\ell]$, let $X_i' \coloneqq (V(G_i) \cap X) \cup S$.
    By definition of balanced separators, we have
    \begin{align*}
        |X_i'|  &~\leq \lfloor \nicefrac{2}{3}(3\mathsf{link}_{\ref{thm:globalstructure_induction}}(t)+1) \rfloor + \mathsf{link}_{\ref{thm:globalstructure_induction}}(t) \\
                &~\leq 3\mathsf{link}_{\ref{thm:globalstructure_induction}}(t) .
    \end{align*}
    Thus we can construct a tame, rooted \ocptd $(T,r,\beta,\alpha)$ with the desired properties as follows.
    Introduce a vertex $r$, set $\beta(r) \coloneqq X \cup S$, and set $\alpha(r) \coloneqq \beta(r)$.
    For each $i \in [\ell]$ where $V(G_i) = X_i'$, we introduce a vertex $x_i$ to the tree, make it adjacent to $r$, set $\beta(x_i) \coloneqq X_i'$, and set $\alpha(x_i) \coloneq \beta(x_i)$.
    For each of the remaining $i \in [\ell]$, there exists some $v_i \in V(G_i) \setminus X_i'$, allowing us to set $X_i \coloneqq X_i' \cup \{ v_i \}$ whilst ensuring that $|V(G_i) \setminus X_i| < |V(G) \setminus X|$.
    This allows us to apply the induction hypothesis to $G_i$ and $X_i$, which either yields $\mathscr{H}_t$ or $\mathscr{V}_t$ as an odd minor, in which case we are done, or there exists a tame, rooted \ocptd $(T_i,r_i,\beta_i,\alpha_i)$ with the desired properties.
    Assuming we do not find $\mathscr{H}_t$ or $\mathscr{V}_t$ as an odd minor, we may then add $T_i$ to our tree, make $r_i$ adjacent to $r$, set $\beta(x) \coloneqq \beta_i(x)$, and set $\alpha(x) \coloneqq \alpha_i(x)$ for all $x \in V(T_i)$.
    This concludes the proof of our theorem in this case.

    \textbf{Case 2:} $X$ is $(\mathsf{link}_{\ref{thm:globalstructure_induction}}(t),\nicefrac{2}{3})$-linked in $G$.

    We start by applying \zcref{thm:algogrid} to find a $\mathsf{mesh}_{\ref{cor:corecorollary}}(t)$-mesh $M$ in $G$ such that the tangle $\mathcal{T}_M$ is a truncation of the tangle $\mathcal{T}_X$, which takes $2^{\mathbf{O}(k^{c_2})}|V(G)|^2|E(G)|\log(|V(G)|)$-time.
    Further, we apply \zcref{cor:corecorollary} to $M$, which takes $(f_{\ref{thm:ktminormodel}}(t) + \mathsf{poly}(k))|E(G)||V(G)|^2$-time.
    If this yields $\mathscr{H}_t$ or $\mathscr{V}_t$ as an odd minor, we are done.
    Thus we may suppose that instead we find a set $A \subseteq V(G)$ with $|A| \leq \mathsf{apex}_{\ref{cor:corecorollary}}(t)$ such that the large block $B$ of $\mathcal{T}_M - A$ is non-empty and has odd cycle packing number at most $\mathsf{cycles}_{\ref{cor:corecorollary}}$.
    
    Let $B_0,B_1, \ldots , B_\ell$ be the blocks of $G - A$, such that $B_0 = B$, and there exists some $r \in [0,\ell]$ with the property that $H_0 = \bigcup_{i=0}^r B_i$ is the component of $G-A$ that contains $B_0$ and thus $B_{r+1}, \ldots , B_\ell$ all reside in different components of $G-A$.
    Let $(T',\beta')$ be the block decomposition (see \zcref{prop:blockdecomposition}) of $H_0$, where we let $r \in V(T')$ be the vertex corresponding to $B_0$ and let $t_1', \ldots , t_y'$ be the neighbours of $r$ in $T'$.
    If $r$ has no neighbours we may skip ahead to treating the other blocks in the graph.
    Thus we may suppose that $y$ is a positive integer.
    For each $i \in [y]$, let $G_i \subseteq G - A$ be the graph induced by the component of $T' - r$ containing $t_i'$.

    Since $B_0$ is the large block of $\mathcal{T}_M - A$, we have that $\mathcal{T}_M$ is a truncation of the tangle $\mathcal{T}_X$, and
    \[ |A| + 1 \leq \mathsf{link}_{\ref{thm:globalstructure_induction}}(t) , \]
    we in particular know that $V(G_i) \cap X$ contains at most $\nicefrac{2}{3}|X| \leq 2\mathsf{link}_{\ref{thm:globalstructure_induction}}(t)$ vertices of $X$.
    As $B_0$ is a block, there exists a vertex in $V(B_0) \setminus V(G_i)$ for each $i \in [y]$.
    We apply our induction hypothesis to $G_i$ with $X_i \coloneqq (V(G_i) \cap X) \cup A \cup (V(B_0) \cap V(G_i))$ being our special set, whilst noting that
    \[ |X_i| = |(V(G_i) \cap X) \cup A \cup (V(B_0) \cap V(G_i))| = 2\mathsf{link}_{\ref{thm:globalstructure_induction}}(t) + \mathsf{link}_{\ref{thm:globalstructure_induction}}(t) + 1 = 3\mathsf{link}_{\ref{thm:globalstructure_induction}}(t) + 1 . \]
    Thus for each $i \in [y]$ we receive a tame, rooted \ocptd $(T_i,r_i,\beta_i,\alpha_i)$ of width at most $\mathsf{width}_{\ref{thm:globalstructure_induction}}(t)$ such that $X_i \subseteq \alpha_i(r_i)$.
    We then construct our tame, rooted \ocptd $(T_0',r_0',\beta_0',\alpha_0')$ for $G_0$ by starting with $r$, adding $T_i$ for each $i \in [y]$, making $r$ and $r_i$ adjacent, setting $\beta_0(r) = V(B_0) \cup A$, $\alpha_0(r) = A \cup X$, and $\beta_0(x) = \beta_i(x)$, as well as $\alpha_0(x) = \alpha_i(x)$ for all $i \in [y]$.
    The tameness of the resulting \ocptd is preserved by construction.

    If we now have $r = \ell$, we are done.
    Thus we may suppose that $r \neq \ell$ and there exists a positive $k$ such that $H_0, H_1, \ldots, H_k$ are the components of $G-A$.
    Again, due to $B_0$ being the large block of $\mathcal{T}_M - A$, the fact that $B_0$ is non-empty, and $\mathcal{T}_M$ being a truncation of the tangle $\mathcal{T}_X$, we know that $V(H_i) \setminus X$ contains at most $\nicefrac{2}{3}|X| \leq 2\mathsf{link}_{\ref{thm:globalstructure_induction}}(t)$ vertices of $X$.
    We can again apply our induction hypothesis on each $i \in [k]$ to find a tame, rooted \ocptd $(T_i,r_i,\beta_i,\alpha_i)$.
    This allows us to construct the tame, rooted \ocptd $(T,r,\beta,\alpha)$ for $G$ by adding $T_i$ to $T_0$ for each $i \in [k]$, setting $\beta(x) = \beta_i(x)$, and $\alpha(x) = \alpha_i(x)$ for each $i \in [0,k]$.
    Again, we achieve tameness by construction and it is easy to observe that, again by construction, we have $X \subseteq \alpha(r)$.
    Since $\alpha(r) = A \cup X$ and $|A \cup X| \leq 4\mathsf{link}_{\ref{thm:globalstructure_induction}}(t)$, we can confirm that the width of $(T,r,\beta,\alpha)$ is indeed bounded by $\mathsf{width}_{\ref{thm:globalstructure_induction}}(t)$.
    This completes our proof.
\end{proof}

\zcref{thm:globalstructure} is an immediate corollary of this result.
\section{Integer programming}\label{sec:integerprogramming}
In this section we seek to extend our results on \mis to solving integer programs. The main goal of this section is to prove \zcref{cor:IPsWork}, and along the way we seek to explain a deeper connection between \mis, integer programming, signed graphs, and totally $\Delta$-modular matrices. We first need to introduce signed graphs, which we discuss in detail here.

\subsection{Signed graphs}
A \textit{signed graph} is a pair $(G, \gamma)$ where $G$ is a (multi) graph and $\gamma: E(G) \rightarrow \mathbb{Z}_2$ is a labelling of the edges of $G$ in the two element group. $\gamma(e)$ is called the \textit{label} or \textit{sign} of an edge $e$. Edges with label 1 are called \textit{odd edges} while edges with label 0 are called \textit{even edges}. We think of the labelling of the edges as a more refined way of describing the parity of a cycle. That is, we call a cycle $C$ in $G$ \textit{even} if $\sum_{e \in E(C)} \gamma(e) = 0$, and otherwise we call the cycle \textit{odd}. It is most commonly natural to think of an unsigned graph $G$ as one where every edge is odd.

A signed graph $(G, \gamma')$ is obtained from $(G, \gamma)$ by \textit{shifting at a vertex} $v \in V(G)$ if
$$\gamma'(e) = \begin{cases}
    \gamma(e) & e \not\in \partial(v) \text{ or $e$ is a loop}\\
    \gamma(e) + 1 & e \in \partial(v) \text{ and $e$ is not a loop}
\end{cases}$$
where $\partial(v)$ is the set of edges incident to $v$. That is, $\gamma'$ is obtained by flipping the label of every nonloop edge incident to $v$. Note that shifting at a vertex $v$ does not change the parity of any cycle. Clearly if a cycle does not use the vertex $v$ then it has the same parity, and if a cycle does use the vertex $v$ then shifting will contribute 2 to aforementioned sum. More generally, a signed graph $(G, \gamma')$ is obtained from $(G, \gamma)$ by \textit{shifting} if it is obtained via a sequence of shiftings at vertices. Because we often only care about the parity of cycles, we consider signed graphs $(G, \gamma'), (G, \gamma)$ to be \textit{shifting equivalent} if $(G, \gamma')$ can be obtained from $(G, \gamma)$ by shifting. We note that any two signed graphs $(G, \gamma_1), (G, \gamma_2)$ have the same collection of odd cycles if and only if they are shifting equivalent. One direction follows from the fact that shifting does not change the parity of cycles, while the other direction can be seen by fixing a maximal forest $F \subseteq E(G)$ (that is, a spanning tree of each component of $G$) and shifting both such that the label of $e$ is 0 for each $e \in F$. This can be done for each tree by rooting the tree at a vertex $r$, shifting only at children of $r$ to make all edges in $\partial(r)$ have label 0, and recursing on each subtree. Then the label of every other edge in $G$ is exactly determined by the parity of fundamental cycle it creates, so the resulting graphs must necessarily have the same labelling.

We note the similarity between shifting and an \emph{even subdivision} as defined in \zcref{sec:preliminaries}. Both can be used to ``swap the parity'' of edges along a minimal edge cut, and both don't change the parity of any cycle.

\paragraph{Signed graph minors} We say $(H, \gamma_H)$ is a \textit{signed graph minor} of $(G, \gamma_G)$ if it can be obtained via a sequence of edge deletions, vertex deletions, shiftings, and contracting even edges. Note that because we may only contract even edges, every cycle in $(H, \gamma_H)$ corresponds to a cycle in $(G, \gamma_G)$ of the same parity. In particular we do not allow the contraction of odd loops (odd loops can still be deleted) even though this could be done in the corresponding matroid. We discuss the reason behind this in more detail below. Note that if $\gamma_G$ and $\gamma_H$ are both identically 1, then this corresponds exactly with the definition of odd-minors given in \zcref{sec:preliminaries} via shifting at every vertex in one side of an edge cut to make that edge cut the only even edges, and then contracting that edge cut. Thus a graph $G$ forbidding one of the parity breaking grids as an odd-minor is equivalent to the signed graph $(G, \gamma)$, where $\gamma(e) = 1$ for all $e \in E(G)$, forbidding the same grid as a signed graph minor where each edge in the grid has label 1. Note that the parity breaking grids are shifting equivalent to the signed graph in which all edges of the cylindrical grid are even and all other edges (denoted in red in \zcref{fig:ParityGridsIntro}) are odd.

\paragraph{Matroids} A matroid $\mathcal{M}$ is a pair $(S, \mathcal{I})$ where $S$ is \textit{ground set} of $\mathcal{M}$ and $\mathcal{I} \subseteq 2^S$ is the collection of \textit{independent sets} of $\mathcal{M}$ which obey the following properties
\begin{enumerate}
    \item $\varnothing \in \mathcal{I}$,
    \item $\mathcal{I}$ is downwards closed, that is $I \in \mathcal{I}$ and $J \subseteq I$ implies that $J \in \mathcal{I}$,
    \item if $I,J \in \mathcal{I}$ with $|J| < |I|$, then there exists an element $x \in I$ such that $J \cup \{x\} \in \mathcal{I}$.
\end{enumerate}
One should think of a these properties as being motivated by the properties of linear independence. More formally, given a matrix $A \in \mathbb{F}^{m \times n}$, where $\mathbb{F}$ is some arbitrary field, the \textit{linear matroid} $\mathcal{M}(A)$ is the matroid with ground set given by the columns $A$, and a set of columns is independent in $\mathcal{M}(A)$ exactly when they are linearly independent. Note that row operations preserve linear independence, and thus matrices $A$ and $A'$ have the same linear matroid when they are row equivalent.

Matroids also serve to generalize many properties of (multi) graphs in the following sense. The \textit{graphic matroid} of a graph $G$ is the matroid $\mathcal{M}(G)$ which has ground set $E(G)$ such that a set of edges are independent if they are acyclic. One can check that this obeys the properties above. Many properties of graphs can be generalized to properties of matroids. Given a graph $G$, the \textit{vertex-edge incidence matrix of $G$} is the matrix $A \in \{0, 1, 2\}^{V(G) \times E(G)}$ where $A_{v,e}$ is nonzero if and only if $e \in E(G)$ is incident to $v \in V(G)$, and the sum of every column is 2. That is, if $e \in E(G)$ is a loop at vertex $v \in V(G)$, then we will set $A_{v,e} = 2$, and otherwise the endpoints of $e$ both have entry 1. The \textit{edge-vertex incidence matrix of $G$} is the transpose of the vertex-edge incidence matrix of $G$. Then note that the graphic matroid is exactly the linear matroid of its vertex-edge incidence matrix over $GF(2)$. Note in particular that over $GF(2)$ the column corresponding to loops become the zero vector, which agrees with the matroidal property that adding the edge to any set makes it not independent.

\paragraph{Matroid minors} Matroids also have a notion of minors inspired by their connection to graph theory. Given a matroid $\mathcal{M} = (S, \mathcal{I})$ and a set $F \subseteq S$, the matroid $\mathcal{M} - F$ is the matroid obtained by \textit{deleting} $F$ where $\mathcal{M} - F$ has ground set $S \setminus F$ and independent sets $\{I \in \mathcal{I} : I \cap F = \varnothing\}$. The matroid $\mathcal{M}/F$ is the matroid obtained by \textit{contracting} $F$ where $\mathcal{M}/F$ has ground set $S \setminus F$ and independent sets $\{I : I \cup B \in \mathcal{I}\}$ where $B$ is a maximal independent set in $F$. One can verify that the choice of $B$ does not matter, and that this definition corresponds with the same operations in the graph setting (where we allow ourselves to keep multiple edges and loops). That is, given $F \subseteq E(G)$, $\mathcal{M}(G - F) = \mathcal{M}(G) - F$ and $\mathcal{M}(G/F) = \mathcal{M}(G)/F$. A matroid $\mathcal{M}'$ is a \textit{minor} of $\mathcal{M}$ if $\mathcal{M}'$ is formed from $\mathcal{M}$ via a sequence of deletions and contractions.

It will be convenient to consider the notion of deletion and contraction in a linear matroid. Let $A \in \mathbb{F}^{m \times n}$ be a matrix over the field $\mathbb{F}$ and let $F$ be a subset of the columns of $A$. Clearly $\mathcal{M}(A) - F$ is the linear matroid for the submatrix of $A$ obtained by deleting the columns $F$. To see $\mathcal{M}(A)/F$, fix a maximal independent set $B$ of $F$. We may reorder the columns such that $F$ corresponds to the first $|F|$ columns and $B$ corresponds to the first $|B|$ columns. Let $A'$ be formed from $A$ by row reducing to the form
$$A' = \left[\begin{array}{cc|c}
    I_{|B|} & C & D_1\\
    0 & 0 & D_2
\end{array}\right].$$
where $I_{|B|}$ is the $|B| \times |B|$ identity matrix, and $C$ has dimension $|B| \times |F \setminus B|$. Then the matroid $\mathcal{M}(A)/F$ is equal to $\mathcal{M}(D_2)$. That is to contract a set of columns, we first row reduce until those columns contain an identity matrix, and then delete the columns of $F$ and the rows of that identity matrix.

For a more in-depth introduction to matroids, see James Oxley's book~\cite{Oxley2011Matroid}, his short introduction~\cite{Oxley2003matroid}, or his even shorter introduction~\cite{OxleyBriefly}.

\paragraph{Signed incidence matrices} A matrix $A \in \{-2, -1, 0, 1, 2\}^{V(G) \times E(G)}$ is a \textit{signed vertex-edge incidence matrix} of a signed graph $(G, \gamma)$ if it has the following properties. For each column $e \in E(G)$ corresponding to a non-loop even edge, there is a 1 in the row of one end of $e$, a $-1$ in the row of the other end, and 0 in all other entries. If $e$ is an odd edge, both ends are either both 1 or both $-1$, and all other entries are 0. If $e$ is an even loop, then the column $e$ is the 0 vector. If $e$ is an odd loop, then there is a 2 or $-2$ in the row corresponding to the end of the loop. Similarly $A$ is a \textit{signed edge-vertex incidence matrix} of a signed graph $(G, \gamma)$ if its transpose is a signed vertex-edge incidence matrix of $(G, \gamma)$. Note that all signed vertex-edge incidence matrices of a signed graph $(G, \gamma)$ are the same up to scaling columns by $-1$, so in particular they all have the same linear matroid. Also note that shifting at a vertex in $G$ corresponds to multiplying a row of $A$ by $-1$. Note that if $(G, \gamma)$ is such that every edge is odd, then a signed vertex-edge incidence matrix of $(G, \gamma)$ is exactly the vertex-edge incidence matrix of $G$ (up to negating columns). Signed incidence matrices arise naturally in the study of totally-delta modular matrices, which we discuss in more detail below.

\paragraph{Signed-graphic matroids} The \textit{(frame) signed-graphic matroid} of a signed graph $(G, \gamma)$ is the matroid $\mathcal{M}(G, \gamma) = \mathcal{M}(A)$ over the field $\mathbb{R}$ where $A$ is a signed vertex-edge incidence matrix of $(G, \gamma)$. We note that signed graphs also correspond naturally to another matroid, that being their \textit{lift matroid} as appears in \cite{Gerards1995Tutt,GeelenG2005Regular}, but we will not be interested in this correspondence. Equivalently, a set of edges in a signed-graphic matroid is independent if and only if it does not contain an even cycle, or two odd cycles joined by a path (possibly of length 0). This implies that if $(G, \gamma)$ is such that every edge is even, then $\mathcal{M}(G, \gamma) = \mathcal{M}(G)$. In particular note that it is sometimes helpful to think of an unsigned graph as a signed graph with all even edges (such as when looking at its graphic matroid) and other times more helpful to think of it as a signed graph with all odd edges (such as when looking at odd minors and integer programming).

The minors of $\mathcal{M}(G, \gamma)$ correspond to the signed-graphic matroids of the signed graph minors of $(G, \gamma)$ except for at odd loops. To see this, note that contracting a nonloop edge $e \in E(G)$ corresponds to first possibly negating a row such that the entries of $e$ have different signs (so shifting such that $e$ is an even edge), adding the row of one end of $e$ to the row of the other (contracting $e$ in the graph), and then deleting the column $e$ and the row which is left nonzero in $e$. Note that contracting an even loop is the same as deleting it. Contracting an odd loop in the matroid is the same as deleting the column of that edge and the row of the vertex $v$ its incident to. This makes the other nonloop edges incident to $v$ have only one nonzero entry in the matrix, and because scaling columns doesn't affect the matroid, this turns all nonloop edges incident to $v$ into odd loops at their other end. This type of contraction is undesirable when one is interested in the odd cycles of $G$ or the subdeterminants of its signed vertex-edge incidence matrix $A$, so we disallow this type of contraction in signed graphs. Note that the other types of contraction can be done while preserving that $A$ is a signed vertex-edge incidence matrix and not scaling rows or columns. In particular if we maintain that $A$ is a signed vertex-edge incidence matrix of a signed graph, then the contraction of an even edge preserves subdeterminants of $A$, while the contraction of an odd edge does not.

For a more in-depth introduction to signed graphs and their matroids, see \cite{Zaslavsky1982Signed, Zaslavsky2013Signed}. Note that in the literature it is common to see the labelling over the two element group with multiplication instead of addition, and the words ``even'' and ``odd'' replaced with ``balanced'' and ``unbalanced'' respectively.

\subsection{Totally \texorpdfstring{$\Delta$}{Δ}-modular matrices} In this section we will be interested in \textit{integer (linear) programs} which are problems of the form
$$\max \{w^{\text{T}} x : Ax \leq b, x \in \mathbb{Z}^n\}$$
where $A \in \mathbb{Z}^{m \times n}, b \in \mathbb{Z}^m, w \in \mathbb{Z}^n$. Many combinatorial optimization problems can be modelled as integer programs, see \cite{Schrijver2003Combinatorial}. Most notable for our setting will be \mis for a graph $G$, can be formulated as
$$\max\{w^{\text{T}} x : Ax \leq \mathbf{1}, \mathbf{0} \leq x \leq \mathbf{1}, x \in \mathbb{Z}^{V(G)}\}$$
where $w \in \mathbb{Z}^{V(G)}$ is the weight of the vertices, $A$ is the edge-vertex incidence matrix of $G$, and $\mathbf{1}$ the vector with 1's in every entry. This implies that integer programming is NP-hard in general, but there has been a lot of work to solve integer programs with specific conditions on $A$, such as having a constant number of rows \cite{Papdimitriou1981complexity, EisenbrandW2019proximity}, a constant number of columns \cite{Lenstra1983Integer, Kannan1987Minkowskis, Dadush2012integer}, a particular block structure \cite{HemmeckeOR2011nfold, JansenLR2020nearlinear, CslovjecsekEHRW2021blockstructred}, or various graph-theoretic properties \cite{GanianOR2017goingbeyond, EisenbrandHKKLO2022algorithmic}.

A matrix $A$ is \textit{totally $\Delta$-modular} if each of its subdeterminants has absolute value at most $\Delta$, where a subdeterminant is the determinant of any square submatrix of $A$. There has been much study of integer programs where $A$ is a totally $\Delta$-modular matrix due to the following famous conjecture.

\begin{conjecture}[\cite{shevchenko1996qualitative}]
    The integer program $\max\{w ^{\text{T}} x : Ax \leq b, x \in \mathbb{Z}^n\}$ can be solved in polynomial time when $A$ is totally $\Delta$-modular for some constant $\Delta$.
\end{conjecture}

The case where $\Delta = 1$ is a classical result and follows from solving the linear programming relaxation $\max\{w^{\text{T}} x : Ax \leq b, x \in \mathbb{R}^n\}$, see Schrijver's books \cite{Schrijver2003Combinatorial}. The case when $\Delta = 2$ was solved by Artmann, Weismantel, and Zenklusen~\cite{ArtmannWZ2017Strongly}. The case when $A$ is a matrix with at most two nonzero entries per row was solved recently by Fiorini, Joret, Weltge, and Yuditsky~\cite{Fiorini2025Integer} via a reduction to solving \mis on graphs of bounded odd cycle packing number. 

We denote the absolute value of the largest entry in a matrix $A$ by $\|A\|_\infty$. The \textit{signed support} of a matrix $A$ is the matrix obtained by putting a 1 in every positive entry of $A$, a $-1$ in every negative entry, and a 0 in every 0 entry. We note that for matrices with at most two nonzero entries per row, having bounded subdeterminants is equivalent to the corresponding graph having bounded odd cycle packing number and not having too many large entries, which we formalize in the following lemma that arose from discussions with Rose McCarty \cite{McCarty2024Personal}. We remark that \zcref{itm:DMequivItem2} is similar to \cite[Lemma~4]{Fiorini2025Integer} but not the same.

\begin{lemma}\label{lem:deltamodularequivalence}
    Let $A \in \mathbb{Z}^{m \times n}$ have at most two nonzero entries per row. Let $\ocp(A)$ denote that odd cycle packing number of the signed graph whose edge-vertex incidence matrix is given by the signed support of $A$, and let $k$ be the number of columns of $A$ with entries outside $\{-1, 0, 1\}$. Then if $A$ is totally $\Delta$-modular,
    \begin{enumerate}
        \item $\|A\|_\infty \leq \Delta$,
        \item\label{itm:DMequivItem2} $k \leq 2\log_2\Delta$, and
        \item $\ocp(A) \leq \log_2\Delta$.
    \end{enumerate}
    Conversely, if $\Delta$ is the largest absolute value of a subdeterminant of $A$, then $\Delta \leq 2^{\ocp(A)}\|A\|_{\infty}^k$.
\end{lemma}
\begin{proof}
    We first show the reverse direction. Note that clearly this bound is best possible by considering $A$ a block diagonal matrix where each block is either an odd cycle or a single entry of value $\|A\|_\infty$.
    
    Consider an arbitrary square submatrix $B$ of $A$. Because $A$ has at most two nonzero entries per row and $B$ is square, $B$ has an average of at most two nonzero entries per column. If there exists a column with at most 1 nonzero entry, then by performing cofactor expansion on that column we can reduce to finding the determinant of a submatrix of $B$ with a single row and column deleted. Furthermore if that column had an entry of magnitude greater than 1, then this cofactor expansion can at most add a factor of $\|A\|_\infty$ to the determinant of $B$ while the corresponding submatrix has one less column with large entries. Otherwise this cofactor expansion does not increase the magnitude of the determinant of $B$. If $B$ has a column with at least 3 nonzero entries, then because the average number of nonzero entries is at most 2, there must be some column with at most 1 nonzero entry. Thus after repeating this process we may assume that $B$ has exactly two nonzero entries in each column. Let $(G, \gamma)$ be the signed graph whose edge-vertex incidence matrix is given by the signed support of $B$, so $G$ is 2-regular. Note that the determinant of $B$ is the product of the determinants for each submatrix corresponding a component of $G$, so we may assume $G$ is connected and thus $G$ is a cycle.

    Suppose $G$ is a cycle with $B \in \mathbb{Z}^{\ell \times \ell}$ the submatrix of $A$ whose signed support is the edge-vertex incidence matrix of $G$. We take all indices to be modulo $\ell$. By reordering the columns we may assume all nonzero entries are in $B_{i, i}$ and $B_{i,i+1}$ for $i \in [\ell]$. Note that the only permutations $\sigma \in S_\ell$ with all of $B_{i,\sigma(i)}$ nonzero are the identity and the cyclic permutation $\sigma_c(i) = i+1$. Thus
    $$\det B = \sum_{\sigma \in S_\ell}\operatorname{sgn}(\sigma)\prod_{i=1}^\ell B_{i,\sigma(i)} =  \prod_{i=1}^\ell B_{i,i} + \operatorname{sgn}(\sigma_c)\prod_{i=1}^\ell B_{i,i+1}.$$
    Note that the two terms in the sum have the same sign when $G$ is an odd cycle, and different signs when $G$ is an even cycle. Each term has magnitude at most $\|B\|_\infty^k$ for $k$ the number of columns with entries outside of $\{-1, 0, 1\}$, and each term has magnitude at least 1. Thus if $G$ is an even cycle, $|\det B| \leq \|B\|_\infty^k - 1$, and if $G$ is an odd cycle $|\det B| \leq 2\|B\|_\infty^k$.
    
    In the forward direction note that if $A$ is totally $\Delta$-modular, then $\|A\|_\infty \leq \Delta$. By the same argument as above, note that the subdeterminant given by a submatrix $B$ with signed support the edge-vertex incidence matrix of an odd cycle, the determinant of $B$ is at least 2 in absolute value. Thus $\ocp(A) \leq \log_2 \Delta$. We now also show $k \leq 2\log_2\Delta$ for $k$ the number of columns of $A$ with entries outside $\{-1, 0, 1\}$. We proceed by induction on $k$. If $k = 0$ we can take any nonzero entry of $A$ to certify that $\Delta \geq 1$. Otherwise fix a column $c$ and a row $r$ of $A$ such that $|A_{r,c}| \geq 2$. Form the submatrix $A'$ by deleting row $r$ and all columns in which that row is nonzero. Note that we have then removed at most 2 columns with large entries. For any submatrix of $A'$ with subdeterminant $\Delta'$, we can form a submatrix of $A$ with determinant at least $2\Delta'$ by adding row $r$ and column $c$. Then by computing the subdeterminant via cofactor expansion on $r$ we obtain the desired factor of 2, completing the proof by induction. All bounds given are best possible. To see this for the bound on $k$, consider $2A$ for $A$ the edge-vertex incidence matrix of a perfect matching.
\end{proof}

In particular the proof above implies that if $A \in \{-1, 0, 1\}^{m \times n}$, then the largest absolute value of a subdeterminant of $A$ is equal to $2^{\ocp(A)}$, which is well known for unsigned graphs \cite{GrossmanKS1995ontheminors}. Note that this holds even for $A \in \{-2, -1, 0, 1, 2\}^{m \times n}$ the signed edge-vertex incidence matrix of a signed graph $(G, \gamma)$ because rows with entries outside $\{-1, 0, 1\}$ occur exactly at odd loops. Thus if $A$ is the signed edge-vertex incidence matrix of a signed graph $(G,\gamma)$, $A$ is totally $\Delta$-modular if and only if $(G, \gamma)$ forbids the disjoint collection of $\lfloor \log_2 \Delta \rfloor + 1$ odd loops as a signed graph minor.

\subsection{Integer programs forbidding a signed graph minor}
The main result of Fiorini, Joret, Weltge, and Yuditsky~\cite{Fiorini2025Integer} implies in particular that if $A \in \{-1, 0, 1\}^{m \times n}$ is the signed edge-vertex incidence matrix of a signed graph $(G, \gamma)$ which forbids the disjoint collection of $k$ odd loops as a signed graph minor, for some constant $k$, then we can solve the integer program $\max\{w^{\text{T}} x : Ax \leq b, \ell \leq x \leq u, x \in \mathbb{Z}^n\}$ in polynomial time where $\ell, u$ are vectors with entries in $\mathbb{Z} \cup \{-\infty, \infty\}$. Note that in the context of integer programming, loops can be moved into the bounds $\ell \leq x \leq u$, so it suffices to consider $A \in \{-1, 0, 1\}^{m \times n}$. They do this by reducing the problem to solving \mis on a subdivision of a subgraph of $G$, and in fact their reduction works in more generality. The proof below is very similar to theirs, but we include it for completeness.

\begin{theorem}\label{thm:reductionfromIPtoindependentset}
    Let $T: \mathbb{N}^2 \rightarrow \mathbb{R}_+$ be increasing in both entries, and suppose that we can solve \mis in $T(|V(G)|,|E(G)|)$ time over signed graphs with all edges odd who forbid a collection of signed graph minors $\mathcal{G}$ each of minimum degree 2. Let $A \in \{-1, 0, 1\}^{m \times n}$ be the edge-vertex incidence matrix of a signed graph $(G, \gamma)$ which forbids the signed graphs in $\mathcal{G}$ as signed graph minors. Then we can solve the integer program $\max\{w^{\text{T}} x : Ax \leq b, \ell \leq x \leq u, x \in \mathbb{Z}^n\}$ in $\mathbf{O}(m^{\omega_{\ref{def:matrixmultconstant}} + 3}\log^3m + T(n + 2m, 3m))$ time, where $b \in \mathbb{Z}^m, w \in \mathbb{Z}^n$, and $\ell,u \in (\mathbb{Z} \cup \{-\infty, \infty\})^n$.
\end{theorem}

We note that the extra term of $m^{\omega_{\ref{def:matrixmultconstant}} + 3}\log^3 m$ in the runtime is dominated by solving the linear programming relaxation. This term is almost always dominated by the $T(n,m)$, with a notable exception being when $A$ is totally unimodular (so $\ocp(G) = 0$). We also note that the reduction only requires that you can solve \mis over signed graphs with all edges odd which appear in a family of signed graphs closed under subgraphs and parity-preserving subdivisions, that is under replacing an edge by a path of the same parity. The class of signed graphs forbidding some collection of signed graph minors each of minimum degree 2 is one such family.

\paragraph{Bounding the integer program} We first replace $\ell, u$ with vectors in $\mathbb{Z}^n$, forcing the polyhedron $\{x \in \mathbb{R}^n : Ax \leq b, \ell \leq x \leq u\}$ to be bounded. To do this we apply a proximity result of Cook, Gerards, Schrijver, and Tardos~\cite{CookGST1986sensitivity}.

\begin{theorem}[\cite{CookGST1986sensitivity}]
    Let $A$ be an $m \times n$ matrix with largest absolute value of a subdeterminant at most $\Delta$ and let $b,w$ be vectors such that $\{x \in \mathbb{Z} : Ax \leq b\}$ is nonempty and $\max\{w^{\text{T}} x : Ax \leq b\}$ exists. Then for any optimal solution $\overline{x}$ to $\max\{w^{\text{T}} x : Ax \leq b\}$, there exists an optimal solution $z^*$ to $\max\{w^{\text{T}} x : Ax \leq b, x \in \mathbb{Z}^n\}$ with $\|\overline{x} - z^*\|_\infty \leq n\Delta$.
\end{theorem}

We apply the theorem above to the system $\begin{bmatrix}
    A\\
    -I\\
    I
\end{bmatrix}x \leq \begin{bmatrix}
    b\\
    \ell\\
    u
\end{bmatrix}$, so the largest magnitude subdeterminant of this matrix is at most that of $A$. Note that we don't have a polynomial bound on $\Delta$, but due to Hadamard's inequality we can bound $\Delta$ by $n^{n/2}$. This implies we can solve the linear programming relaxation $\max\{w^{\text{T}} x : Ax \leq b, \ell \leq x \leq u\}$ to get optimal solution $\overline{x}$ in strongly polynomial time $\mathbf{O}(m^{\omega_{\ref{def:matrixmultconstant}} + 3}\log^3m)$ using the algorithm of Dadush, Natura, and V\'{e}gh~\cite{DadushNV2020revisiting}, which improves on the work of Tardos~\cite{Tardos1986strongly} using the linear program solver of van den Brand~\cite{Brand2020deterministic}. If the linear programming relaxation is infeasible, then we declare that the integer program is infeasible. If the linear programming relaxation is unbounded, then the integer program is either infeasible or unbounded. We can then repeat the process with $w = \textbf{0}$ to distinguish between the cases. In the result that we obtain the optimum $\overline{x}$, we then restrict our the problem to consider vectors $x$ with $\|x - \overline{x}\|_\infty \leq n(n^{n/2})$. Note that because $\overline{x}$ and $n(n^{n/2})$ can be represented in polynomially many bits, we can represent the new $\ell, u \in \mathbb{Z}^n$ in polynomial size.

\paragraph{Reduction to positive coefficients} We now construct an equivalent linear program in which every constraint is one of the following types.
\begin{align*}
    x_i &\geq \ell_i\\
    x_i &\leq u_i\\
    x_i + x_j &\leq b_{ij}\\
    x_i + x_j &= b_{ij}
\end{align*}
That is, we will force $A$ to only have entries in $\{0,1\}$ with some of the constraints set to equality. To do this we replace every constraint of the form $x_i - x_j \leq b_{ij}$ with the two constraints
\begin{align*}
    x_i + y_{ij} &\leq b_{ij}\\
    x_j + y_{ij} &= 0.
\end{align*}
Similarly we replace every constraint of the form $-x_i - x_j \leq b_{ij}$ with the three constraints
\begin{align*}
    z_{ij} + z'_{ij} &\leq b_{ij}\\
    x_i + z_{ij} &= 0\\
    x_j + z'_{ij} &= 0.
\end{align*}
Note that in the signed graph corresponding to $A$, this process corresponds to subdividing an even edge with 2 odd edges, and subdividing an odd edge into 3 odd edges. Thus if the original signed graph forbid some collection of signed graph minors of minimum degree 2, then this subdivision does as well. Also note that the resulting signed graph $(G, \gamma)$ is one in which every edge is odd, so we can now think of $A$ as the edge-vertex incidence matrix of the unsigned graph $G$. Note that we may now assume that $G$ is a simple graph by replacing any constraints arising from parallel edges by strictest of those constraints. Also note that $G$ now has at most $3m$ edges and $2m + n$ vertices.

\paragraph{Eliminating equality constraints} We now update the objective function so that we may replace the equality constraints with inequality. To do this we will penalize not being at equality by such a large amount that no optimal solution will be infeasible to the original program. Let $A = \begin{bmatrix}
    A_1\\
    A_2
\end{bmatrix}$ and $b = \begin{bmatrix}
    b_1\\
    b_2
\end{bmatrix}$ such that the constraints are $A_1x \leq b_1$ and $A_2x = b_2$. That is the current integer program is
\begin{equation}\label{eq:originalIP}
    \max\{w^{\text{T}} x: x \in K \cap  \mathbb{Z}^n\}
\end{equation}
\begin{equation*}
    K = \{x \in \mathbb{R}^n : Ax_1 \leq b_1, Ax_2 = b_2, \ell \leq x \leq u\}
\end{equation*}
and we wish to replace it with
\begin{equation}\label{eq:extendedIP}
    \max\{\overline{w}^{\text{T}} x : x \in K' \cap \mathbb{Z}^n\}
\end{equation}
\begin{equation*}
    K' = \{x \in \mathbb{R}^n : Ax \leq b, \ell \leq x \leq u\}
\end{equation*}
for some other $\overline{w}$. We define the constant
$$\mu = \max\{w^{\text{T}} x : \ell \leq x \leq u\} - \min\{w^{\text{T}} x : \ell \leq x \leq u\} + 1$$
and we set
$$\overline{w} = w + \mu \mathbf{1}^{\text{T}} A_2.$$
Clearly if $K' \cap \mathbb{Z}^n = \varnothing$, then $K \cap \mathbb{Z}^n = \varnothing$, so we may assume \eqref{eq:extendedIP} is feasible. We now prove the following lemma.
\begin{lemma}[Lemma 5~\cite{Fiorini2025Integer}]
    Let $x^*$ be any optimal solution to \eqref{eq:extendedIP}. Then if $K \cap \mathbb{Z}^n \neq \varnothing$, $x^* \in K$. If $x^* \in K$, then it is optimal for \eqref{eq:originalIP}. 
\end{lemma}
\begin{proof}
    Note that maximizing $\overline{w}^{\text{T}} x$ over $K \cap \mathbb{Z}^n$ is the same as maximizing $w^{\text{T}} x$ over $K\cap \mathbb{Z}^n$ because for any $x \in K$, $\mu\mathbf{1}^{\text{T}} A_2x = \mu\mathbf{1}^{\text{T}} b_2$ is constant. That is, suppose $x^*$ is optimal for \eqref{eq:extendedIP} and consider an arbitrary $x \in K \cap \mathbb{Z}^n$. Then
    $$w^{\text{T}} x = \overline{w}^{\text{T}} x - \mu\mathbf{1}^{\text{T}} b_2 \leq \overline{w}^{\text{T}} x^* - \mu\mathbf{1}^{\text{T}} b_2 = w^{\text{T}} x^*$$
    so $x^*$ is optimal for \eqref{eq:originalIP}.

    Now suppose for the sake of contradiction that $x^* \not\in K \cap \mathbb{Z}^n$, but there exists some $\tilde{x} \in K \cap \mathbb{Z}^n$. Then because $x^*$ satisfies $A_2x^* \leq b_2$ but not $A_2x^* = b_2$ and both $A_2, b_2$ have integer entries, there must be some inequality which has a gap of at least 1. In particular, $\mathbf{1}^{\text{T}} A_2x^* \leq \mathbf{1}^{\text{T}} b_2 - 1$. Thus
    $$\overline{w}^{\text{T}} x^* \leq w^{\text{T}} x^* + \mu(\mathbf{1}^{\text{T}} b_2 - 1) \leq \max\{w^{\text{T}} x : \ell \leq x \leq u\} + \mu(\mathbf{1}^{\text{T}} b_2 - 1).$$
    Because $\tilde{x} \in K \cap \mathbb{Z}^n$, it satisfies
    $$\overline{w}^{\text{T}} \tilde{x} = w^{\text{T}} \tilde{x} + \mu\mathbf{1}^{\text{T}} b_2 \geq \min\{w^{\text{T}} x : \ell \leq x \leq u\} + \mu\mathbf{1}^{\text{T}} b_2.$$
    However this implies we have
    \begin{align*}
        \overline{w}^{\text{T}} x^* &\leq \max\{w^{\text{T}} x : \ell \leq x \leq u\} + \mu(\mathbf{1}^{\text{T}} b_2 - 1)\\
        &= \max\{w^{\text{T}} x : \ell \leq x \leq u\} - \mu + \mu\mathbf{1}^{\text{T}} b_2\\
        &= \min\{w^{\text{T}} x : \ell \leq x \leq u\} - 1 + \mu\mathbf{1}^{\text{T}} b_2\\
        &\leq \overline{w}^{\text{T}} \tilde{x} - 1
    \end{align*}
    contradicting the optimality of $x^*$.
\end{proof}
Thus to solve \eqref{eq:originalIP}, we can solve \eqref{eq:extendedIP}. If we get that \eqref{eq:extendedIP} is infeasible, then so is \eqref{eq:originalIP}. Otherwise we can check whether our optimal solution is feasible for \eqref{eq:originalIP} in $\textbf{O}(m)$ time. If so, then we have the desired optimal. Otherwise we declare that \eqref{eq:originalIP} is infeasible.

\paragraph{Reduction to $\{0, 1\}$ variables} We have reduced the integer program to one of the form
\begin{equation}\label{eq:IPbefore01vars}
    \max\{w^{\text{T}} x : x \in K \cap \mathbb{Z}^{V(G)}\}
\end{equation}
\begin{equation*}
    K = \{x \in \mathbb{R}^{V(G)} : Ax \leq b, \ell \leq x \leq u\}
\end{equation*}
where $A$ is the edge-vertex incidence matrix of a graph $G$ and $\ell, u \in \mathbb{Z}^{V(G)}$. Furthermore the signed graph $(G, \gamma)$ where $\gamma(e) = 1$ for all $e \in E(G)$ forbids the signed graph minors in $\mathcal{G}$. We now show how to replace $\ell, u$ with $\mathbf{0}, \mathbf{1}$ respectively. We will then be able to remove unnecessary constraints such that $b = \mathbf{1}$ and obtain the \mis problem over a subgraph of $G$, which we can solve in polynomial time by assumption. To do this, we will solve the linear programming relaxation
\begin{equation}\label{eq:LPrelax}
    \max\{w^{\text{T}} x :x \in K\}
\end{equation}
and show that we can restrict the original integer program to values which are the closest integer to that linear program solution. We first need the following lemmas, whose proofs we include for completeness.

\begin{lemma}[Prop.\@ 11~\cite{ConfortiFHJW2019stableset}]\label{lem:verticesgraphpolyhedron}
    Let $B$ be the edge-vertex incidence matrix of a graph $H$. Then all vertices of the polyhedron $Q = \operatorname{conv}\{x \in \mathbb{Z}^{V(H)} : Bx \leq \mathbf{1}\}$ are contained in $\{0, 1\}^{V(H)}$.
\end{lemma}
\begin{proof}
    Let $x^*$ be a vertex of $Q$, and by definition $x^* \in \mathbb{Z}^{V(H)}$. Let $E(x^*) \subseteq E(H)$ be the rows of $B$ such that the corresponding inequalities in $Bx^* \leq \mathbf{1}$ are at equality. That is, for each edge $uv \in E(x^*)$, $x^*_u + x^*_v = 1$. Let $H'$ be a connected component of $H[E(x^*)]$. Note that if vertices $u,v$ share a common neighbour $w$ in $H'$, then the constraints $x^*_u + x^*_w = 1, x^*_w + x^*_v = 1$ implies that $x^*_u = x^*_v$. Note that because $x^*$ is integer, this implies that there are no odd cycles in $H'$. Thus $H'$ is bipartite, and there exists an integer $\alpha \geq 1$ such that $V'_\alpha = \{v \in V(H') : x^*_v = \alpha\}$ and $V'_{1 - \alpha} = \{v \in V(H') : x^*_v = 1 - \alpha\}$ is a bipartition of $H'$.

    Suppose for the sake of contradiction that $x^* \not\in \{0, 1\}^{V(H)}$. Then it must be that $\alpha \geq 2$. Let $r \in \mathbb{Z}^V$ be such that
    $$r_v = \begin{cases}
        1 & v \in V'_\alpha\\
        -1 & v \in V'_{1 - \alpha}\\
        0 & v \in V(H) \setminus V(H').
    \end{cases}$$
    We then argue that $x^* + r$ and $x^* - r$ are both integer points in $Q$, contradicting that $x^*$ is a vertex. Clearly all constraints with both ends in $V(H) \setminus V(H')$ are satisfied for both $x^* + r$ and $x^* - r$. For constraints given by edges in $E(x^*)$, they are at equality for both $x^* + r$ and $x^* - r$. Edges with exactly one end in $V(H')$ must have a constraint with a gap of at least 1, so they are satisfied for both $x^* + r, x^* - r$. Edges $uv$ with both ends in $V(H')$ which were not tight for $x^*$ must have both ends in the same side of the bipartition of $H'$. Thus $x_u + x_v \leq 1$ implies that $u,v \in V'_{1 - \alpha}$, and because $\alpha \geq 2$ the constraint has a slack of at least 2. Therefore the constraint is still satisfied for $x^* + r, x^* - r$.
\end{proof}

\begin{lemma}[Lemma 8~\cite{Fiorini2025Integer}]\label{lem:optimumgraphpolyhedron}
    Let $B$ be the an edge-vertex incidence matrix of a graph $H$, and let $w \in \mathbb{R}^{V(H)}$. Then if $\max\{w^{\text{T}} x : Bx \leq \mathbf{1}\}$ is attained at $\frac{1}{2}\mathbf{1}$, then $\max\{w^{\text{T}} x: Bx \leq \mathbf{1}, x \in \mathbb{Z}^{V(H)}\}$ is attained at a point in $\{0, 1\}^{V(H)}$.
\end{lemma}
\begin{proof}
    We may assume $H$ is connected. Let $Q = \operatorname{conv}\{x \in \mathbb{Z}^{V(H)} : Bx \leq \mathbf{1}\}$. It is well known that if $Q$ contains a vertex, then $\max\{w^{\text{T}} x : x \in Q\}$ is attained at a vertex. Thus by \zcref{lem:verticesgraphpolyhedron} we may assume that $Q$ does not contain a vertex, or equivalently $Q$ does not contain a line. Note that $Q$ contains a line exactly when there is a nonzero solution to $Bx = 0$ which gives the direction of the line. Such an equation has a nonzero solution if and only if $H$ is bipartite. If $H$ is bipartite, then setting $x_v = 1$ for one side of the bipartition and $x_v = -1$ for the other side is a valid solution. If $H$ is not bipartite, then extending an odd cycle to a spanning subgraph with $|V(H)|$ many edges gives a submatrix with nonzero determinant (compute the determinant via cofactor expansion on degree one vertices until $H$ is an odd cycle, which has nonzero determinant as in \zcref{lem:deltamodularequivalence}), so $B$ has full column rank. Hence we may assume $H$ is bipartite.

    Let $V(H) = L \cup R$ be the bipartition of $H$, and without loss of generality we may assume $\sum_{v \in L} w_v \geq \sum_{v \in R} w_v$. Let $x^* \in \{0, 1\}^{V(H)}$ be the characteristic vector of $L$. Note that clearly $x^* \in Q$. Then
    $$w^{\text{T}} x^* = \sum_{v \in L} w_v \geq \frac{1}{2}\sum_{v \in V(H)}w_v = \max\{w^{\text{T}} x : Bx \leq \mathbf{1}\} \geq \max\{w^{\text{T}} x : Bx \leq \textbf{1}, x \in \mathbb{Z}^{V(H)}\}$$
    so $x^*$ is the desired optimal solution.
\end{proof}

We are now to prove the following proposition. We call a solution to the linear program \eqref{eq:LPrelax} \textit{extremal} if it is a vertex of $K$. Note that because $K$ is bounded, if $K \neq \varnothing$ then there is always a solution to the integer program which is a vertex, and given an arbitrary solution we can always find one which is a vertex, for example by perturbing $w$ by a tiny amount.

\begin{proposition}[Prop.\@ 6~\cite{Fiorini2025Integer}]\label{prop:integralitygap}
    Suppose that \eqref{eq:IPbefore01vars} is feasible. Then for every extremal optimal solution $x^*$ of \eqref{eq:LPrelax}, there exists an optimal solution $\overline{x}$ to \eqref{eq:IPbefore01vars} such that $\|x^* - \overline{x}\|_\infty \leq 1/2$.
\end{proposition}
\begin{proof}
    Every extremal point in $K$ is half integer, see Prop. 2.1 in \cite{NemhauserT1974properties}, so $x^* \in \frac{1}{2}\mathbb{Z}^{V(G)}$. Let $V_0 = \{v \in V(G) : x_v \in \mathbb{Z}\}$ and $V_* = V(G) \setminus V_0$.

    We first show that $x^*$ is best possible on its integer coordinates. Let $z \in \mathbb{Z}^{V(G)}$ be an optimal solution to \eqref{eq:IPbefore01vars}. Then it must be that
    $$\sum_{v \in V_0} w_vz_v \leq \sum_{v \in V_0} w_vx^*_v.$$
    Otherwise consider $y \in \mathbb{R}^{V(G)}$ such that $y_v = x^*_v$ for $v \in V_*$ and $y_v = (1 - \varepsilon)x^*_v + \varepsilon z_v$ for $v \in V_0$. We now claim that we can choose $\varepsilon > 0$ small enough such that $y \in K$, and because $w^{\text{T}} y = w^{\text{T}} x + \varepsilon\sum_{v \in V_0}w_v(z_v - x_v) > w^{\text{T}} x$, we contradict that $x^*$ is optimal. Because $\ell \leq x^* \leq u$ and $\ell \leq z \leq u$, it must be that $\ell \leq y \leq u$. Consider any constraint of the form $x_{v} + x_{v'} \leq b_{vv'}$ which is at equality for $x^*$. Note that if this constraint is tight for $x^*$, then it must be that either both $v,v' \in V_0$, or both $v,v' \in V_*$. If both $v,v' \in V_*$, then $y_v + y_{v'} = x^*_v + x^*_{v'} = b_{vv'}$. If both $v,v' \in V_0$, then $(y_v,y_{v'})$ is the linear combination of $(x^*_v, x^*_{v'})$ and $(z_v, z_{v'})$, both of which satisfy the inequality. Therefore $y$ does not violate any tight constraints for any $\varepsilon$, so we can choose $\varepsilon$ small enough to not violate any of the loose constraints as desired.

    We now show that there exists an optimal solution which is obtained by rounding the half-integral coordinates of $x^*$ to one of the closest integers. Consider the graph $H$ with $V(H) = V_*$ and whose edges are those whose constraints are tight with both ends in $V_*$. Let $B$ be an edge-vertex incidence matrix of $H$. We first claim that for any $q$ satisfying $Bq \leq \mathbf{1}$, the vector $y$ such that $y_v = x^*_v$ for $v \in V_0$ and $y_v = \lfloor x^*_v\rfloor + q$ for $v \in V_*$ satisfies $y \in K$. Note that because $x^* \in K$ and $\ell,u$ are integer, $y \geq \lfloor x^* \rfloor \geq \ell$, and $y \leq \lceil x^* \rceil \leq u$. Consider any constraint of the form $x_v + x_{v'} \leq b_{vv'}$. If both $v,v' \in V_0$, then $y_v + y_{v'} = x^*_v + x^*_{v'} \leq b_{vv'}$. If $v \in V_0$ and $v' \in V^*$, then because $b$ is integer, $x^*_v + x^*_{v'} \leq b_{vv'} - \frac{1}{2}$, and hence $y_v + y_{v'} \leq x^*_v + x^*_{v'} + \frac{1}{2} \leq b_{vv'}$. If $v,v' \in V_*$ with the constraint tight for $x^*$, then note $\lfloor x^*_v\rfloor + \lfloor x^*_{v'}\rfloor \leq b_{vv'} - 1$ while $q_v + q_{v'} \leq 1$, so $y_v + y_{v'} \leq b_{vv'}$. If $v,v' \in V^*$ and the constraint is not tight, then $\lfloor x^*_v\rfloor + \lfloor x^*_{v'}\rfloor \leq b_{vv'} - 2$ and trivially $q_v + q_{v'} \leq 2$, so $y_v + y_{v'} \leq b_{vv'}$. Thus $y \in K$ as desired.
    
    We now argue that $\frac{1}{2}\mathbf{1}$ is an optimal solution to $\max\{(w|_{V_*})^{\text{T}} x : Bx \leq \mathbf{1}\}$. Indeed suppose there exists $y^* \in \mathbb{R}^{V_*}$ such that $By^* \leq \mathbf{1}$ and $(w|_{V_*})^{\text{T}} y^* > \frac{1}{2}(w|_{V_*})^{\text{T}} \mathbf{1}$. Then let $y \in \mathbb{R}^{V(G)}$ be such that $y_v = x^*_v$ for $v \in V_0$ and $y_v = \lfloor x^*_v \rfloor + y^*_v$ for $v \in V_*$. The $y\in K$ as above, and
    $$w^{\text{T}} y = \sum_{v \in V_0} w_vx^*_v + \sum_{v \in V_*}w_v(\lfloor x^*_v\rfloor + y^*_v) \geq \sum_{v \in V_0} w_vx^*_v + \sum_{v \in V_*}w_v\left(\lfloor x^*_v\rfloor + \frac{1}{2}\right) = w^{\text{T}} x^*$$
    contradicting the optimality of $x^*$.

    Therefore by \zcref{lem:optimumgraphpolyhedron}, there exists some $q \in \{0,1\}^{V_*}$ an optimal solution to $\max \{(w|_{V_*})^{\text{T}} x : Bx \leq \mathbf{1}, x \in \mathbb{Z}^{V_*}\}$. Let $\overline{x} \in \mathbb{Z}^{V(G)}$ be such that $\overline{x}_v = x^*_v$ if $v \in V_0$, and $\overline{x}_v = \lfloor x_v^*\rfloor +q_v$ if $v \in V_*$. As above $\overline{x} \in K$. Let $z \in \mathbb{Z}^{V(G)}$ be an optimal solution to \eqref{eq:IPbefore01vars}. Then note that $z' = z|_{V_*} -\lfloor x|_{V_*}\rfloor$ satisfies $Bz' \leq \mathbf{1}$. Thus the optimality of $q$ implies
    \begin{align*}
        \sum_{v \in V_*} w_v\overline{x}_v &= \sum_{v \in V_*} w_v(\lfloor x^*_v\rfloor + q_v)\\
        &\geq \sum_{v \in V_*}w_v(\lfloor x^*_v\rfloor + z'_v)\\
        &= \sum_{v \in V_*} w_vz_v.
    \end{align*}
    From above we have
    $$\sum_{v \in V_0} w_v \overline{x}_v = \sum_{v \in V_0} w_vx^*_v \geq \sum_{v \in V_0}w_vz_v$$
    and hence $w^{\text{T}} \overline{x} \geq w^{\text{T}} z$. Thus $\overline{x}$ is an optimal solution to \eqref{eq:IPbefore01vars} with $\|x^* - \overline{x}\|_\infty \leq 1/2$, completing the proof.
\end{proof}

Finally we reduce to \mis as follows. First, we compute an extremal optimal solution $x^*$ to \eqref{eq:LPrelax} in strongly polynomial time $\textbf{O}(m^{\omega_{\ref{def:matrixmultconstant}} + 3}\log^3m)$. Clearly if \eqref{eq:LPrelax} is infeasible, then so is \eqref{eq:IPbefore01vars}, and note that \eqref{eq:LPrelax} is never unbounded. We then translate the feasible region $K$ by $\lfloor x^* \rfloor$. That is, we now consider the integer program 
\begin{equation}\label{eq:IPin01vars}
    \max\{w^{\text{T}} x : x + \lfloor x^*\rfloor \in K \cap \mathbb{Z}^{V(G)}\}.
\end{equation}
Note that this has the same structure as \eqref{eq:IPbefore01vars} but with perhaps $\ell,u,$ and $b$ changed. We can update $\ell,u,b$ in $\textbf{O}(m)$ time. Then by \zcref{prop:integralitygap}, it suffices to find solutions to \eqref{eq:IPin01vars} where $x \in \{0,1\}^{V(G)}$. We then remove all trivial constraints. That is, for a constraint $x_v + x_{v'} \leq b_{vv'}$, if $b_{vv'} < 0$, we declare the problem infeasible. If $b_{vv'} = 0$, we set $x_v, x_{v'}$ to be 0. If $b_{vv'} > 1$, we can remove the constraint (and thus the edge from $G$). Similarly the constraints $\ell \leq x \leq u$ may result in fixing more variables or directly imply infeasibility. We can iterate through all constraints, either fixing variables or declaring infeasibility, in $\textbf{O}(m)$ time. For each variable that is fixed to 0, we delete the vertex from $G$. For each variable that is fixed to 1, we delete the vertex and all of its neighbours from $G$. Then the resulting integer program is exactly \mis on a subgraph of $G$ which we can solve in $T(|V(G)|, |E(G)|) \leq T(n + 2m, 3m)$ time by assumption.

\subsection{Integer programs forbidding a parity breaking grid}

If we combine the above reduction with our main result we obtain the following. Let $\gamma_1$ denote the function which maps every input to 1.

\begin{corollary}\label{cor:IPsWork}
    Let $(G, \gamma)$ be a signed graph which forbids $(\mathscr{H}_k, \gamma_1)$ and $(\mathscr{V}_k, \gamma_1)$ as signed graph minors. Let $A \in \{-1, 0, 1\}^{m \times n}$ be the signed edge-vertex incidence matrix of $G$. Then there exists a function $f(k)$ such that we can solve the integer program
    $$\max\{w^{\text{T}} x : Ax \leq b, \ell \leq x \leq u, x \in \mathbb{Z}^n\}$$
    in $n^{f(k)}$ time, where $w \in \mathbb{Z}^n$, $b \in \mathbb{Z}^m$, and $\ell,u \in (\mathbb{Z} \cup \{-\infty, \infty\})^n$.
\end{corollary}
\begin{proof}
    By \zcref{thm:reductionfromIPtoindependentset}, it suffices to solve \mis on $(G', \gamma_1)$ for $G'$ a graph with at most $n + 2m$ vertices and $3m$ edges and $(G', \gamma_1)$ forbidding $(\mathscr{H}_k, \gamma_1)$ and $(\mathscr{V}_k, \gamma_1)$ as signed graph minors. Because each of these graphs have every edge odd, $G'$ forbids $\mathscr{H}_k$ and $\mathscr{V}_k$ as odd-minors. Therefore by \zcref{thm:globalstructure}, $G'$ has \ocp-treewidth at most $f_{\ref{thm:globalstructure}}(k)$. Then by \zcref{thm:miswithboundedocptw}, we can solve \mis in time $n^{f(k)}$ as desired.
\end{proof}

\zcref{cor:IPsWork} can be seen as a partial answer to the below question, which seeks to extend \zcref{quest:MainIntro} to integer programs with at most two nonzero entries per row.

\begin{question}\label{ques:IPquestion}
    What is the largest minor-closed class of matrices with at most two nonzero entries per row such that we can solve the integer program $\max\{w^T x : Ax \leq b, \ell \leq x \leq u, x \in \mathbb{Z}^n\}$ in polynomial time for all matrices $A$ in the class?
\end{question}

The above problem is interesting for multiple minor relations. If we let $A \in \{0, 1\}^{m \times n}$ and let the minor relation be the normal minor relation for the graph whose incidence matrix is $A$, then this question is equivalent to finding the largest minor-closed class of graphs for which we can solve \mis. This question is answered by the Grid Theorem of Robertson and Seymour \cite{RobertsonS1986Grapha}, which implies that we can solve \mis in polynomial time on a minor-closed class if and only if the class forbids a planar graph. \zcref{cor:IPsWork} provides a partial answer to the question when we let $A \in \{-1, 0, 1\}^{m \times n}$ and let the minor relation be signed-graph minors of the signed-graph whose signed edge-vertex incidence matrix is $A$. \zcref{thm:reductionfromIPtoindependentset} implies that this case is once again equivalent to \mis, and thus the question is essentially equivalent to \zcref{quest:MainIntro}. It is natural to pose \zcref{ques:IPquestion} for matrices with entries outside $\{-1, 0, 1\}$ and under other minor relations, such as the minors of the linear matroid of $A^T$.

\zcref{ques:IPquestion} is of course interesting for matrices with more than two nonzero entries per row, but will likely require different techniques. We remark that the class $\mathcal{M}_\Delta$ of matroids with a representation as a $\Delta$-modular matrix $A$, that is all $\mathsf{rank(A) \times \mathsf{rank}(A)}$ subdeterminants have absolute value at most $\Delta$, is closed under matroid minors \cite[Prop 8.6.1]{GeelenNW2024Excluding} and duality \cite[Theorem 4.2]{OxleyW2022Two-Modular}. Much less is known about the class $\mathcal{M}_\Delta^t$ of matroids with a representation as a totally $\Delta$-modular matrix. In particular it is unknown whether $\mathcal{M}_\Delta^t = \mathcal{M}_\Delta$, though this holds for $\Delta \leq 2$.

\bigskip

\textbf{Acknowledgements.}
The authors thank Niloufar Fuladi, Micha{\l} Seweryn, and in particular, Rose McCarty -- who independently initiated this project with the fourth author before she and the fourth author joined their efforts with the remaining four authors -- for several helpful discussions.

\bibliographystyle{alphaurl}
\bibliography{literature}

\end{document}